\documentclass[11pt,reqno]{amsart}
\usepackage{graphicx}
\usepackage{amssymb}
\usepackage{hyperref}
\usepackage{amsthm}
\usepackage{amsfonts}
\usepackage{amsmath}
\usepackage{bm}
\usepackage{esint}
\usepackage{mathrsfs}
\usepackage[all]{xy}
\usepackage{color}
\usepackage{subfigure}
\usepackage{enumerate}                        
\usepackage{anysize}
\usepackage{tikz}
\usepackage{amscd}
\usepackage[letterpaper]{geometry}
\usepackage{geometry}
\newtheorem{theorem}{Theorem}[section]
\newtheorem{proposition}[theorem]{Proposition}
\newtheorem{claim}[theorem]{Claim}
\newtheorem{lemma}[theorem]{Lemma}

\newtheorem{corollary}[theorem]{Corollary}

\theoremstyle{remark}

\theoremstyle{definition}
\newtheorem{definition}[theorem]{Definition}
\newtheorem{remark}[theorem]{Remark}
\newtheorem{example}[theorem]{Example}

\numberwithin{equation}{section}
\numberwithin{figure}{section}

\newcommand{\dC}{\mathbb{C}}
\newcommand{\dD}{\mathbb{D}}
\newcommand{\dN}{\mathbb{N}}
\newcommand{\dR}{\mathbb{R}}
\newcommand{\dT}{\mathbb{T}}
\newcommand{\dZ}{\mathbb{Z}}

\newcommand{\TT}{\mathbb{T}^2}

\def \C {\mathbb C}
\def \Z {\mathbb{Z}}

\def \O {\mathcal O}

\def \p {\partial}
\def \bp {\bar\partial}

\DeclareMathOperator{\ALH}{ALH}
\DeclareMathOperator{\Area}{Area}

\DeclareMathOperator{\dl}{\underline{\delta}_1}
\DeclareMathOperator{\dr}{\underline{\delta}_2}
\DeclareMathOperator{\diam}{Diam}

\DeclareMathOperator{\dvol}{dvol}
\DeclareMathOperator{\el}{\underline{\epsilon}_1}
\DeclareMathOperator{\er}{\underline{\epsilon}_2}

\DeclareMathOperator{\HK}{HK}

\DeclareMathOperator{\Id}{Id}
\DeclareMathOperator{\I}{I}
\DeclareMathOperator{\II}{II}
\DeclareMathOperator{\III}{III}
\DeclareMathOperator{\IV}{IV}
\DeclareMathOperator{\V}{V}
\DeclareMathOperator{\VI}{VI}
\DeclareMathOperator{\Image}{Image}
\DeclareMathOperator{\InjRad}{InjRad}
\DeclareMathOperator{\Isom}{Isom}
\DeclareMathOperator{\K3}{K3}

\DeclareMathOperator{\Nil}{Nil}

\DeclareMathOperator{\pr}{pr}
\DeclareMathOperator{\Pic}{Pic}
\DeclareMathOperator{\rank}{rank}
\DeclareMathOperator{\Ric}{Ric}

\DeclareMathOperator{\Rm}{Rm}

\DeclareMathOperator{\SL}{SL}
\DeclareMathOperator{\Sp}{Sp}
\DeclareMathOperator{\SU}{SU}

\DeclareMathOperator{\TF}{tf}
\DeclareMathOperator{\Tr}{Tr}
\DeclareMathOperator{\Vol}{Vol}

\newcommand{\RR}{\mathbb{R}}

\allowdisplaybreaks[2]

\begin{document}
\title[Nilpotent structures and collapsing Ricci-flat metrics on K3 surfaces]{Nilpotent structures and collapsing Ricci-flat \\ metrics on K3 surfaces}
\date{}

\author{Hans-Joachim Hein}
\address{Department of Mathematics, Fordham University, Bronx, NY 10458} 
\email{hhein@fordham.edu}
\author{Song Sun}
\address{Department of Mathematics, University of California, Berkeley, CA 94720 \newline \hspace*{9pt} Department of Mathematics, Stony Brook University, Stony Brook, NY~11790} 
\email{sosun@berkeley.edu}
\author{Jeff Viaclovsky}
\address{Department of Mathematics, University of California, Irvine, CA 92697} 
\email{jviaclov@uci.edu}
\author{Ruobing Zhang}
\address{Department of Mathematics, Stony Brook University, Stony Brook, NY 11790}
\email{ruobing.zhang@stonybrook.edu}
\thanks{The first author was partially supported by NSF Grant DMS-1745517. The second author acknowledges the support by NSF Grant DMS-1708420, an Alfred P. Sloan Fellowship, and a grant from the Simons Foundation ($\sharp$488633, S.S.). The third author was partially supported by NSF Grant DMS-1405725.}
\begin{abstract}
We exhibit families of Ricci-flat K\"ahler metrics on K3 surfaces which collapse to an interval, with Tian-Yau and Taub-NUT metrics occurring as bubbles. There is a corresponding continuous surjective map from the $\K3$ surface to the interval, with regular fibers diffeomorphic to either $3$-tori or Heisenberg nilmanifolds.
\end{abstract}
\maketitle
\setcounter{tocdepth}{1}
\tableofcontents

\section{Introduction}

The main purpose of this paper is to describe a new mechanism by which Ricci-flat  metrics on $\K3$ surfaces can degenerate. It also suggests a new general phenomenon that could possibly occur also in other contexts of canonical metrics.

Recall in 1976, Yau's solution to the Calabi conjecture \cite{Yau} proved the existence of K\"ahler metrics with vanishing Ricci curvature, which are governed by the vacuum Einstein equation, on a compact K\"ahler manifold with zero first Chern class. These \emph{Calabi-Yau metrics} led to the first known construction of compact Ricci-flat Riemannian manifolds which are not flat. Examples of such manifolds exist in abundance, and these metrics often appear in natural families parametrized by certain complex geometric data, namely, their \emph{K\"ahler class} and their \emph{complex structure}.  As the complex geometric data degenerates, it is  a natural question to understand the process of singularity formation of the corresponding Ricci-flat  metrics.

From a geometric analytic point of view, the Einstein equation 
\begin{equation}
\Ric_g = \lambda g,\ \lambda \in \dR,
\end{equation}
can be made into an elliptic system in natural harmonic coordinates, so it inherits certain general properties of elliptic equations. However, the Einstein equation is strongly non-linear and might be highly degenerate. 
 One distinguished feature is that, except in dimension $4$, a satisfactory regularity theory has only been developed under an extra \emph{local volume non-collapsing} assumption, i.e., that there exists a uniform positive constant $v>0$ such that 
 \begin{equation}
 \Vol_g(B_1(x)) \geq v > 0. \label{e:non-collapsed}
 \end{equation}
It is known that a sequence of non-collapsed Einstein manifolds $(M_j^n, g_j, p_j)$ with uniformly bounded Ricci curvature  will converge in the Gromov-Hausdorff sense to a metric space with controlled singularities. More precisely, after passing to a subsequence,
\begin{equation}
(M_j^n, g_j, p_j) \xrightarrow{GH} (X_{\infty}^n, d_{\infty}, p_{\infty}),
\end{equation}
and there is a small singular set $\mathcal{S}\subset X_{\infty}^n$ such that 
$X_{\infty}^n\setminus \mathcal{S}$ is a smooth manifold endowed with an Einstein metric \cite{Anderson-harmonic, BKN,  Tian1990, ChC1, ChCT, CD, ChNa-quantitative}. By Cheeger and Colding's fundamental work in \cite{ChC}, each {\it tangent cone} for every point $x\in X_{\infty}^n$
is a {\it metric cone}. 
Recently, it was proved by Cheeger and Naber in \cite{ChNa-codim-4} that $\dim_{Haus}(\mathcal{S})\leq n-4$, which is optimal. The same estimate in the K\"ahler case was previously  proved in \cite{Cheeger-elliptic}, and also by Tian. The case of special holonomy was proved in \cite{Cheeger-Tian-special}.

Let $(M_j^n,g_j, p_j)$ be a sequence of Einstein manifolds
satisfying $|{\Ric_{g_j}}|\leq n-1$ but not necessarily the non-collapsing condition \eqref{e:non-collapsed} such that
$(M_j^n,g_j, p_j)$ $\xrightarrow{GH}$ $(X_{\infty}, d_{\infty}, p_{\infty})$.
To understand the degeneration of the metrics $g_j$ more precisely one studies them at an infinitesimal scale: assuming that curvature blows up around the points $p_j$, we choose rescaling factors $\lambda_j\to\infty$ such that 
\begin{equation}
  \label{ghlim}
 (M_j^n, \lambda_j^2 g_j, p_j)\xrightarrow{GH}(Y_{\infty}, \tilde{d}_{\infty}, \tilde{p}_{\infty})
 \end{equation}
 as $j \to \infty$ after passing to a subsequence. The limit in \eqref{ghlim} can depend upon the choice of rescaling factors $\lambda_j$, and we will call any such limit a {\textit{bubble limit}}. Note that in the non-collapsing case,  any bubble limit has Euclidean volume growth.

 Without the non-collapsing assumption \eqref{e:non-collapsed}, it is much more difficult to understand the degeneration of Einstein manifolds. 
In the general context of collapsed Einstein manifolds, there is no available metric cone structure at infinitesimal scales, and bubble limits will not necessarily have Euclidean volume growth. 
Moreover, due to the non-existence of a uniform Sobolev constant, classical regularity analysis cannot be applied directly to collapsed manifolds. 
However, in the case of Einstein $4$-manifolds, in a pioneering work, 
Cheeger and Tian proved the first $\epsilon$-regularity theorem without any non-collapsing assumption \cite{CheegerTian}. Their result implies in particular that, given a sequence of Ricci-flat  metrics $g_j$ on a fixed compact $4$-manifold $M^4$ with $(M^4, g_j, p_j)\xrightarrow{GH}(X_{\infty}^k, d_{\infty}, p_{\infty})$, there exists a finite singular set $\mathcal{S}\equiv\{q_{\beta}\}_{\beta=1}^N\subset X_{\infty}^k$ in the following sense: for any compact subset $\Omega_{\infty}\subset X_{\infty}^k\setminus\mathcal{S}$ there are regular regions $\Omega_j\subset M^4$ converging to $\Omega_{\infty}$ with uniformly bounded curvatures, a phenomenon which has been deeply studied in Riemannian geometry,  see for instance \cite{ChGr1, ChGr2, Fu1, Fu2,  ChFuG, Rong}.

 In this paper we focus on the case of K\"ahler-Einstein metrics in complex dimension $2$. In fact we will only consider Ricci-flat K\"ahler metrics  on \emph{$\K3$ surfaces}, i.e., simply-connected compact complex surfaces with vanishing first Chern class. The Ricci-flat K\"ahler metrics in this case have holonomy $\SU(2)\cong\Sp(1)$ so they are in fact \emph{hyperk\"ahler}. The main result of this paper presents a new gluing construction of metrics of this type, relying crucially on their hyperk\"ahler property.

\subsection{Gluing constructions of hyperk\"ahler $\K3$ surfaces}

In this section we recall the known gluing constructions of hyperk\"ahler metrics on $\K3$ surfaces in the literature.

\subsubsection{Kummer construction} We start with a flat  orbifold $\mathbb{T}^4/\Z_2$ given by the quotient of a  flat $4$-torus by the  involution $x\mapsto -x$. It has 16 orbifold singularities.  One can resolve these singularities by gluing 16 \emph{Eguchi-Hanson spaces} onto $X$, which are complete hyperk\"ahler ALE metrics defined on the cotangent bundle of $S^2$. By varying the flat structure on $\mathbb{T}^4/\Z_2$ and the gluing parameters,  one obtains an open set in the moduli space of all hyperk\"ahler metrics on the $\K3$ surface where the areas of the exceptional curves are small. As these areas go to zero the corresponding hyperk\"ahler metrics naturally converge back to the flat orbifold $\mathbb{T}^4/\Z_2$, and the Eguchi-Hanson spaces appear as bubbles under rescaling. For a rigorous proof we refer readers to \cite{LeBrun-Singer},  \cite{Donaldson-kummer} and the references therein. 
This is a typical example of singularity formation in the non-collapsing situation. In general the Gromov-Hausdorff limit will be an orbifold hyperk\"ahler $\K3$ surface, and the bubbles are ALE gravitational instantons, classified by Kronheimer in \cite{Kronheimer}.

\subsubsection{Codimension-$1$ collapse}

 In \cite{Foscolo} Foscolo constructed a family of hyperk\"ahler metrics on a $\K3$ surface that collapses to the flat orbifold $\mathbb{T}^3/\Z_2$. 
 The collapse has bounded curvature away from finitely many points, and is given by shrinking the fibers of an $S^1$-fibration. In the simplest case, curvature blow-up occurs at the 8 singular points of $\mathbb{T}^3/\Z_2$, where the bubbles are given by complete hyperk\"ahler spaces with cubic volume growth, which in this case are ALF-$D_2$ spaces. See \cite{CCII, minerbe} for a partial classification of hyperk\"ahler ALF spaces.
Let us also point out that the results of \cite{Foscolo} have motivated the study of codimension-$1$ collapse of $G_2$-manifolds to $3$-dimensional Calabi-Yau manifolds in \cite{FHN}.

\subsubsection{Codimension-$2$ collapse}

In  \cite{GW}, Gross and Wilson constructed a family of hyperk\"ahler 
metrics on the $\K3$ surface which collapse to a singular metric $d_{\infty}$ on a topological sphere $X_{\infty}^2\approx S^2$.
One starts from an elliptic $\K3$ surface, i.e., a $\K3$ surface that admits a holomorphic fibration over $\dC P^1$ with the general fibers being smooth elliptic curves. Moreover we assume the generic situation when there are exactly 24 singular fibers of type $I_1$. Using a combination of a gluing construction and Yau's estimates, \cite{GW} gave a fairly satisfactory picture describing the metric asymptotic behavior when the area of the fibers goes to zero. Away from the singular fibers, the metric is modeled on the Green-Shapere-Vafa-Yau hyperk\"ahler \emph{semi-flat   metrics} \cite{GSVY}, whose restrictions to the fibers are exactly flat;   in a neighborhood of each singular fiber the metric is modeled on the \emph{Ooguri-Vafa metric} (see \cite{GW} and \cite{OV}). The latter is an incomplete hyperk\"ahler metric constructed using the Gibbons-Hawking ansatz which we will recall in Section~\ref{s:model-space}.  When we rescale near the singular point of any singular fiber, the complete bubble that we obtain is $\dC^2$ endowed with the \emph{Taub-NUT metric}, which is K\"ahler with respect to the standard complex structure on $\dC^2$ and has cubic volume growth (see \cite{LeBrun, NUT, Taub}).

Notice that the limit metric $d_{\infty}$ on the topological sphere $X_{\infty}^2$ is non-smooth at the $24$ points corresponding to the singular fibers, but every tangent cone at $X_{\infty}^2$
is in fact isometric to $\dR^2$.  
Away from the singular points, $d_{\infty}$ gives a Riemannian metric on $X_{\infty}^2$ which satisfies a real Monge-Amp\`ere equation, an adiabatic limit of the Calabi-Yau equation. By hyperk\"ahler rotation, this family of hyperk\"ahler metrics also describes the geometry of the Calabi-Yau metrics on a polarized family of $\K3$ surfaces approaching a \emph{large complex structure limit}.

\subsubsection{Codimension-$3$ collapse with torus fibers} 
Here we start with two complete noncompact hyperk\"ahler $4$-manifolds with cylindrical ends. These were constructed by Tian and Yau \cite{TianYau} by removing smooth fibers from rational elliptic surfaces, and were proved in \cite{Hein} to converge to their $\mathbb{R} \times \mathbb{T}^3$ flat asymptotic models at an exponential rate. Such spaces are known as {\it $\ALH$-spaces} or \emph{half-$\K3$ surfaces} in the literature. It is then possible to glue together two $\ALH$ spaces to obtain a family of hyperk\"ahler metrics on $\K3$ which degenerates by developing a long neck modeled on $\mathbb{T}^3$ times an interval (see \cite{CCIII} for a rigorous proof). 
If we rescale  these metrics so that the rescaled diameter equals $1$, then the Gromov-Hausdorff limit is the unit interval and the bubbles are the Tian-Yau asymptotically cylindrical metrics at each endpoint. Gluing of asymptotically cylindrical geometric structures is a very familiar construction in geometry, see for example  \cite{Floer, KS} for anti-self-dual metrics in dimension $4$, and \cite{Kovalev} for holonomy $G_2$ metrics in dimension $7$.

\subsection{Main results}
\label{ss:main-results}

The main result of this paper gives a new gluing construction in which a family of hyperk\"ahler metrics on a $\K3$ surface collapses to a unit interval and generically the collapse happens along a $3$-dimensional \emph{Heisenberg nilmanifold} (i.e., a nontrivial $S^1$-bundle over $\mathbb{T}^2$). Part of our motivation was an attempt to understand the hyperk\"ahler metric degenerations corresponding to Type II complex structure degenerations of polarized $\K3$ surfaces.
A guiding example is when we have a family of quartic $\K3$ surfaces $Z_t$ in $\dC P^3$ defined by the equation $tq+f_1f_2=0$, where $q$ is a general quartic and $f_1$ and $f_2$ are general quadrics. So the general fiber is a smooth $\K3$ surface while the central fiber is a union of two quadric surfaces $X_1$ and $X_2$, intersecting transversally along an elliptic curve defined by $f_1=f_2=0$. We would like to understand the behavior of the Ricci-flat metrics on $Z_t$ in the cohomology class of $2\pi c_1(\O(1)|_{Z_t})$ as $t$ tends to zero. 

In general, given a del Pezzo surface $M$ and a smooth anti-canonical curve $D\subset M$, Tian-Yau  proved in \cite{TianYau} the existence of a hyperk\"ahler metric on $M\setminus D$, with interesting asymptotic geometry at infinity. 
Namely, outside a compact set the manifold is diffeomorphic to $N\times [0, \infty)$, where $N$ is an $S^1$-bundle over $D$ of degree $d=c_1(X)^2$, and the metric is modeled on a doubly-warped product so that as we move towards infinity the $S^1$-fibers shrink in size while the base torus $D$ expands. The volume growth rate of the hyperk\"ahler metric is $4/3$ and the curvature decays quadratically.  We call such a hyperk\"ahler metric  a \emph{Tian-Yau metric} throughout this paper, and for more precise details we refer to Section \ref{s:Tian-Yau}. The proof in \cite{TianYau} uses the Calabi ansatz in a neighborhood of infinity and then solves a Monge-Amp\`ere equation, so it is \emph{not} a priori clear whether these metrics are unique or canonical in a suitable sense. Nevertheless they provide candidates for the bubble limits of the degeneration that we would like to understand. 

Motivated by the above Type II degeneration picture, our basic idea was to glue together two Tian-Yau metrics to obtain hyperk\"ahler $\K3$ surfaces. 
However, an easy topological consideration shows that one cannot naively glue the ends of the two Tian-Yau metrics together to even match the topology of a $\K3$ surface. Geometrically, even though the end of a Tian-Yau metric is toplogically cylindrical, the metric itself is not. So we need to construct a \emph{neck region} that approximates the Tian-Yau ends on both sides.

A key novel ingredient of this paper is exactly to construct such a transition region. It is an incomplete hyperk\"ahler $4$-manifold that can be viewed as a doubly-periodic cousin of the \emph{Ooguri-Vafa metric}. Recall that the Ooguri-Vafa metric is the metric arising from the Gibbons-Hawking ansatz applied to a harmonic function on $S^1\times \mathbb{R}^2$ with a pole on $S^1\times \{0\}$; or equivalently, a harmonic function on $\dR\times \mathbb{R}^2$, periodic in the first variable, and with poles on $\Z\times \{0\}\subset \dR\times \{0\}$. Our neck metric is instead constructed by applying the Gibbons-Hawking ansatz to a harmonic function on the flat cylinder $\TT\times \mathbb{R}$  with finitely many poles (which we call the \emph{monopole points}) in $\TT \times \mathbb{R}$. This is equivalent to a harmonic function on $\dR^2\times \mathbb{R}$, doubly periodic in the first and second variables, and with poles on lattices. For more details of this construction we refer to Section \ref{s:model-space}. Here we point out that in analogy with the Ooguri-Vafa case, the resulting metric is \emph{incomplete} because $\TT \times \mathbb{R}$ is parabolic, hence admits no globally positive harmonic functions; moreover, the two ends of this neck metric indeed match up closely with the ends of Tian-Yau metrics. 

Our main theorem says that it is in fact possible to ``glue together'' two Tian-Yau metrics with a suitable neck region as above to construct families of Ricci-flat   metrics on $\K3$ with non-trivial nilpotent collapsing structure.
\begin{theorem}\label{t:codim-3}
  Let $b_+$, $b_-$ and $m$ be positive integers satisfying 
\begin{align}
1 \leq b_{\pm} \leq 9, \ 1 \leq m \leq b_+ + b_-.
\end{align}
Then there exists a family of hyperk\"ahler metrics $\hat{h}_{\beta}$ on a $\K3$ surface which collapse to the standard metric on the closed interval $[0,1]$, i.e.,
\begin{equation}
(\K3, \hat{h}_{\beta}) \xrightarrow{GH} ([0,1], dt^2),\ \beta \to \infty.
\end{equation}
Moreover, for each sufficiently large $\beta\gg 1$, there exist a finite set $\mathcal{S}\equiv\{0,t_1,\ldots, t_m,1\} \subset [0,1]$ and a continuous surjective map
\begin{equation}
F_{\beta} : \K3 \to [0,1]
\end{equation}
which is almost distance-preserving, i.e., for some constant $C_0>0$ independent of $\beta$
\begin{equation}
\Big||F_{\beta}(p)-F_{\beta}(q)|-d_{\hat{h}_{\beta}}(p,q)\Big|\leq \frac{C_0}{\beta} , \ \forall p,q\in \K3,
\end{equation}
 such that the following properties hold. 

\begin{enumerate}
\item (Regular collapsing regions)
Denote by $T_{\epsilon}(\mathcal{S})$ the $\epsilon$-tubular neighborhood of $\mathcal{S}$ and $\mathcal{R}_{\epsilon}\equiv [0,1] \setminus T_{\epsilon}(\mathcal{S})$. Then for every $\epsilon\in(0,10^{-2})$ and $k\in\dN$, there exists $C_{k,\epsilon}>0$
 such that \begin{equation}
\sup\limits_{F_{\beta}^{-1}(\mathcal{R}_{\epsilon})}|\nabla^k {\Rm_{\hat{h}_{\beta}}}|\leq C_{k,\epsilon},
\end{equation}
and for each $t \in \mathcal{R}_{\epsilon}$, $F_{\beta}^{-1}(t)$ is diffeomorphic to an $S^1$-fiber bundle over $\TT$. 
Furthermore, 
 \begin{align}
C_0^{-1} \beta^{-1}\leq \diam_{\hat{h}_{\beta}}(F_{\beta}^{-1}(t))\leq C_0 \beta^{-1},\quad  C_0^{-1} \beta^{-2}\leq \diam_{\hat{h}_{\beta}}(S^1)\leq C_0 \beta^{-2}.
\end{align}

\item (Bubbling regions) 
Denote by $F_{\beta}^{-1}(T_{\epsilon}(\mathcal{S}))\equiv\mathcal{S}_{\epsilon}^{-}\cup \bigcup\limits_{j=1}^m\mathcal{S}_{\epsilon}^j\cup\mathcal{S}_{\epsilon}^{+}$ the components of the singular pre-image. Then the following spaces occur as bubble limits:
\begin{enumerate}
\item 
For each $1\leq j\leq m$,
there exists an $x_{\beta,j}\in\mathcal{S}_{\epsilon}^j$  such that $F_{\beta}(x_{\beta,j}) \rightarrow t_j$, $|{\Rm_{\hat{h}_{\beta}}}|(x_{\beta,j}) \rightarrow \infty$ as $\beta \rightarrow \infty$,
and rescalings of the metrics near $x_{\beta,j}$ converge to Taub-NUT metrics. In fact, it is possible to have several distinct Taub-NUT bubbles coming out of the same component $\mathcal{S}_\epsilon^j$; see Theorem \ref{t:domain-wall-crossing} for a more precise statement.
\item 
There exist $x_{\beta,\pm}\in\mathcal{S}_{\epsilon}^{\pm}$ such that $F_\beta(x_{\beta,-})\to0$, $F_\beta(x_{\beta,+})\to1$,
 $|{\Rm_{\hat{h}_{\beta}}}|(x_{\beta,\pm}) \rightarrow \infty$ as $\beta \rightarrow \infty$,
and rescalings of the metrics near $x_{\beta,\pm}$ converge to Tian-Yau metrics on a del Pezzo surface of degree $b_{\pm}$, minus a smooth anti-canonical curve. 
\end{enumerate}

\end{enumerate}

\end{theorem}
\begin{remark}
 Figure \ref{mainfig} schematically shows the collapsing process when $b_-=b_+=7$, $m=1$, and $t_1=\frac{1}{2}$. In this case we have the maximal possible number $b_- + b_+ = 14$ of Taub-NUT bubbles coming out of one single singular pre-image component $\mathcal{S}^1_\epsilon = F_\beta^{-1}(t_1 - \epsilon, t_1 + \epsilon)$. 
\end{remark}

\begin{figure}
\begin{tikzpicture}[scale = 1.3]
\draw (0,0) to [out = 90, in = 180] (1,.5) to (2,.5);
\draw (2,.5) to [out = 0, in = 235] (4,2);
\draw (4,2) to [out = 55, in = 180] (6,3);
\draw (6,3) to [out = 0, in = 125] (8,2);
\draw (8,2) to [out = -55, in = 180] (10,.5);
\draw (10,.5) to (11,.5);
\draw (11,.5) to [out = 0, in = 90] (12,0);
\draw (0,0) to [out = 270, in = 180] (1,-.5) to (2,-.5);
\draw (2,-.5) to [out = 0, in = 125] (4,-2);
\draw (4,-2) to [out = -55, in = 180] (6,-3);
\draw (6,-3) to [out = 0, in = -125] (8,-2);
\draw (8,-2) to [out = 55, in = 180] (10,-.5);
\draw (10,-.5) to (11,-.5);
\draw (11,-.5) to [out = 0, in = -90] (12,0);
\draw[blue](2,0) ellipse (.2 and .5);
\draw[red](2,0) ellipse (.2 and .05);
\draw[blue, fill=gray] (3,0) ellipse (.1 and .82);
\draw[red](3,0) ellipse (.1 and .05);
\draw[blue](4,0) ellipse (.05 and 2);
\draw[red ](4,0) ellipse (.05 and .02);
\draw[blue](10,0) ellipse (.2 and .5);
\draw[red](10,0) ellipse (.2 and .05);
\draw[blue, fill=gray](9,0) ellipse (.1 and .82);
\draw[red](9,0) ellipse (.1 and .05);
\draw[blue](8,0) ellipse (.05 and 2);
\draw[red ](8,0) ellipse (.05 and .02);
\draw[red](5,0) ellipse (.03 and .02);
\draw[blue](5,0) ellipse (.03 and 2.82);
\draw[red](7,0) ellipse (.03 and .02);
\draw[blue](7,0) ellipse (.03 and 2.82);
\draw (.8,-.05) arc [ radius = .5, start angle = 45, end angle = 135];
\draw (1.2,-.05) arc [ radius = .5, start angle = 45, end angle = 135];
\draw (1.6,-.05) arc [ radius = .5, start angle = 45, end angle = 135];
\draw (11.1,-.05) arc [ radius = .5, start angle = 45, end angle = 135];
\draw (11.5,-.05) arc [ radius = .5, start angle = 45, end angle = 135];
\draw (11.9,-.05) arc [ radius = .5, start angle = 45, end angle = 135];
\node at (5.5,2) {$\times$};
\node at (5.5,-2) {$\times$};
\node at (5.5,1) {$\times$};
\node at (5.5,0) {$\times$};
\node at (5.5,-1) {$\times$};
\node at (6.5,2) {$\times$};
\node at (6.5,-2) {$\times$};
\node at (6.5,1) {$\times$};
\node at (6.5,0) {$\times$};
\node at (6.5,-1) {$\times$};
\node at (6,2) {$\times$}; 
\node at (6,-2) {$\times$}; 
\node at (6,1) {$\times$}; 
\node at (6,-1) {$\times$};
\draw (0,-5) -- (12,-5);
\draw (0,-5.1)   -- (0,-4.9);
\draw (12,-5.1)   -- (12,-4.9);
\draw (6,-5.1)   -- (6,-4.9);
\draw[->, very thin] (0,-.3) -- (0,-4.5);
\draw[->, very thin] (1.25,-.6) -- (.2,-4.5);
\draw[->, very thin] (3,-.9) -- (3,-4.5);
\draw[->, very thin] (6,-3.1) -- (6,-4.5);
\draw[->, very thin] (5,-2.9) -- (5.8,-4.5);
\draw[->, very thin] (12,-.3) -- (12,-4.5);
\draw[->, very thin] (10.75,-.6) -- (11.8,-4.5);
\draw[->, very thin] (9,-.9) -- (9,-4.5);
\draw[->, very thin] (6,-3.1) -- (6,-4.5);
\draw[->, very thin] (7,-2.9) -- (6.2,-4.5);
\node at (1,.8) {$X_{b_-}$};
\node at (11,.8) {$X_{b_+}$};
\node at (6,3.3) {$\mathcal{N}$};
\node at (0,-5.5) {$0$};
\node at (6,-5.5) {$\frac12$};
\node at (12,-5.5) {$1$};
\draw[|->] (0,4) -- (4,4);
\node[align = left, below] at (4,4) {$z_-$};
\node[align = left, below] at (2.7,4) {$T_-$};
\draw[fill] (3,4) circle [radius=1pt];
\draw[<->] (2,5) -- (10,5);
\node[align = right, below] at (2,5) {$z$};
\node[align = right, below] at (10,5) {$z$};
\draw[dashed, very thin] (3,5) -- (3,.8);
\node[above] at (3,5) {$-T_-$};
\draw[fill] (3,5) circle [radius=1pt];
\draw[|->] (12,4) -- (8,4);
\node[align = left, below] at (8,4) {$z_+$};
\draw[dashed, very thin] (9,5) -- (9,.8);
\node[above] at (9,5) {$T_+$};
\draw[fill] (9,5) circle [radius=1pt];
\node[align = left, below] at (9.3,4) {$T_+$};
\draw[fill] (9,4) circle [radius=1pt];
\draw (6,5-.1) -- (6,5+.1);
\node[below] at (6,4.9) {$0$};
\end{tikzpicture}
\caption{The vertical arrows represent collapsing to a one-dimensional interval. The red circles represent the $S^1$ fibers and the blue curves represent the base $\TT$s of the nilmanifolds. The $\times$s are the monopole points in the neck region $\mathcal{N}$. The gray regions are in the ``damage zones''.}
\label{mainfig}
\end{figure}
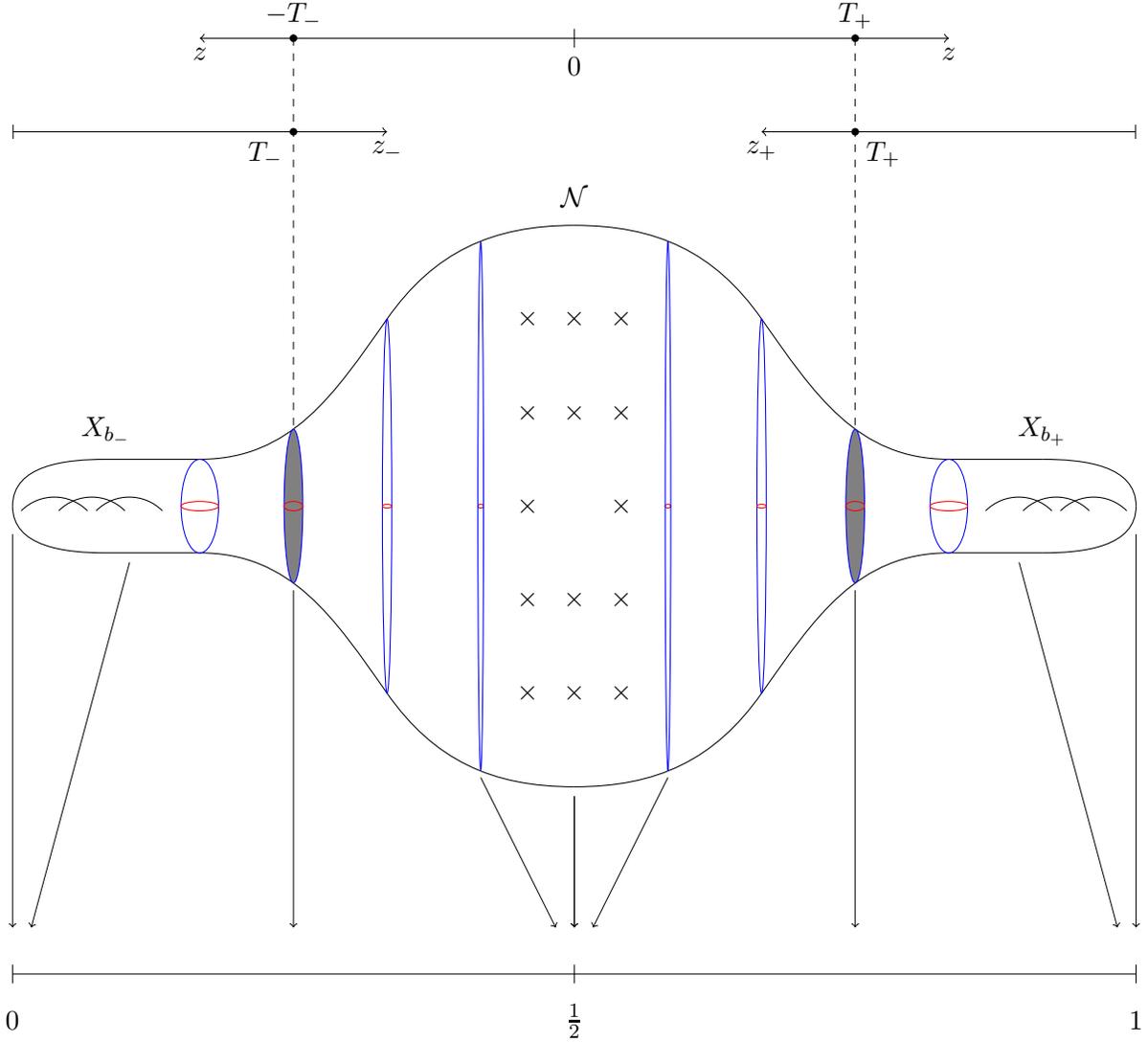

\begin{remark}
The Riemannian geometry of the regular collapsing regions is actually completely understood.  For each $t \in \mathcal{R}_{\epsilon}$, $F_{\beta}^{-1}(t)$  is a $3$-dimensional Heisenberg nilmanifold if the $S^1$-bundle is nontrivial, and is diffeomorphic to $\mathbb{T}^3$ otherwise. Furthermore, the universal cover of a regular preimage $F_{\beta}^{-1}(t_j+ \epsilon ,t_{j+1} - \epsilon)$ converges to a hyperk\"ahler manifold $(\widetilde{U}_{\infty},\tilde{g}_{\infty})$ with a Heisenberg or Euclidean group of isometries according to whether the $S^1$-bundle is nontrivial or trivial. An explicit expression for $\tilde{g}_{\infty}$ in the Heisenberg case may be found in Section \ref{ex:model-space}.  
In particular, our construction gives a concrete example of Lott's recent work classifying the regular regions in collapsing $4$-manifolds with almost Ricci-flat  metrics (see \cite{lott} for more details).  
\end{remark}

\begin{remark} The volume of the hyperk\"ahler metrics in Theorem~\ref{t:codim-3} is comparable to $\beta^{-4}$. If one scales these metrics to have unit volume instead of unit diameter, it is not hard to see that the possible pointed Gromov-Hausdorff limits are either $\TT \times \mathbb{R}$ or $\TT \times [0, \infty)$, depending on the basepoint (see Section \ref{ss:rescaled-limits}). Thus under this scaling, the nilmanifolds disappear and one sees only a simple $S^1$-collapse. However, the bubbles remain the same.\end{remark}
 
For the precise definition of a $3$-dimensional 
Heisenberg nilmanifold, see Section \ref{heisnil}. These are $S^1$-bundles over $\TT$, and thus they have a degree which is only well-defined up to sign. However, if one specifies a projection to an oriented $\TT$, then the degree is a well-defined integer. We denote by $\Nil_b^3$ a $3$-dimensional nilmanifold of {\textit{degree}} $b$, where  we will always have a certain projection to an oriented $\TT$ in mind.
Let $t_0 = 0$, $t_{m +1} = 1$, and let 
$d_j$ be the degree of a nilpotent fiber $\Nil_{d_j}^3$ on the  interval
$(t_j + \epsilon ,t_{j+1} - \epsilon)$, $j = 0, \dots, m$, with $d_0 = b_-$ and  $d_{m} = - b_+$.
The following {\it Domain Wall Crossing Theorem} describes the possible jumps of the degrees of the nilmanifolds
upon crossing the singular regions.

\begin{theorem} \label{t:domain-wall-crossing}
Given any $m$-tuple of positive integers $(w_1,\ldots, w_{m})$ satisfying 
\begin{equation}
\sum\limits_{j=1}^{m} w_j=b_-+b_+,
\end{equation}
 there exist examples in Theorem \ref{t:codim-3} with $d_{j} - d_{j+1} = w_{j+1}$, $j = 0, \dots, m-1$.  
Furthermore, near each singular point $t_j$,  exactly $w_j$ Taub-NUT bubbles occur. 
\end{theorem}

\begin{remark} This domain wall crossing phenomenon has been studied in the physics literature, namely, it arises in Type IIA massive superstring theory, see \cite{Hull}. 
\end{remark}

\begin{remark} It is possible to generalize our construction to obtain bubble-trees of ALF-$A_k$ metrics at the interior points with several levels of scaling (by taking clusters of monopole points in the neck region which conglomerate in the limit at various rates). Similarly, one could take some monopole points to have higher multiplicity, in which case one would obtain collapsing sequences of Ricci-flat  metrics on {\textit{orbifold}} K3 surfaces with $A_k$-type singularities. However, for simplicity we do not list all of these numerous possibilities here in this paper, but see Remark \ref{rem:ALF} below. 
\end{remark}

\begin{remark} 
This paper is primarily concerned with the construction of the hyperk\"ahler metrics. To actually identify the collapsing limits as polarized degenerations of complex structures in various cases takes more work, and will be discussed in a forthcoming paper \cite{HSVZ2}.
\end{remark}

\begin{remark}
In 1987 R.~Kobayashi proposed a conjectural mechanism for Ricci-flat metrics on $\K3$ surfaces to degenerate into unions of Tian-Yau and Taub-NUT spaces. See Cases (i) and (ii) on p.223 in \cite{Koba}. Our main result in this paper verifies Kobayashi's expectation at the level of hyperk\"ahler structures. However, at least in certain cases, Kobayashi also proposed an identification of these limits with type II polarized degenerations of complex structures on $\K3$ surfaces.
\end{remark}

\subsection{Gluing hyperk\"ahler triples}
\label{ss:outline}
We will adopt the general description of a hyperk\"ahler metric in terms of a triple of three symplectic forms due to Donaldson \cite{Donaldson}, a description which was also used, for example, in \cite{FineLotaySinger, Foscolo}.  
Namely, we will not directly construct a Ricci-flat   metric on $\mathcal{M}$. Instead we will glue together triples of symplectic forms, which we will then perturb to obtain a hyperk\"ahler triple, which will then yield in particular a Ricci-flat K\"ahler metric.

Let $M^4$ be an oriented $4$-manifold with a volume form $\dvol_0$. A triple of $2$-forms $\bm{\omega}=(\omega_1,\omega_2,\omega_3)$ is called a definite triple if the matrix $Q=(Q_{ij})$ defined by 
 \begin{equation}\frac{1}{2}\omega_i\wedge\omega_j=Q_{ij}\dvol_0\end{equation} is positive.
Given a definite triple $\bm{\omega}$, the associated volume form is defined as
\begin{equation}
\dvol_{\bm{\omega}} = (\det(Q))^{\frac{1}{3}} \dvol_0,
\end{equation}
which is independent of the choice of volume form $\dvol_0$. We denote by $Q_{\bm{\omega}}\equiv (\det(Q))^{-\frac{1}{3}}Q$
the renormalized matrix with unit determinant.
 Furthermore, a definite triple $\bm{\omega}=(\omega_1,\omega_2,\omega_3)$ is called a hyperk\"ahler triple if $d\omega_1=d\omega_2=d\omega_3=0$ and the renormalized coefficient matrix satisfies
\begin{align}
Q_{\bm{\omega}}=\Id.\label{e:id-eq}\end{align}
Note that equation \eqref{e:id-eq} is equivalent to
\begin{equation}
\frac{1}{2}\omega_i\wedge\omega_j =\frac{1}{6}\delta_{ij}(\omega_1^2+\omega_2^2+\omega_3^2),
\end{equation}
for every $1\leq i\leq j\leq 3$.

 Each definite triple $\bm{\omega}$ defines a Riemannian metric $g_{\bm{\omega}}$ such that each $\omega_j$ is self-dual with respect to $g_{\bm{\omega}}$. If moreover
$\bm{\omega}=(\omega_1,\omega_2,\omega_3)$ is a hyperk\"ahler triple, then $g_{\bm{\omega}}$ is a hyperk\"ahler metric. Furthermore, $\omega_2 + i\omega_3$  is a holomorphic $2$-form with respect to the complex structure determined by $\omega_1$. See Section \ref{s:model-space} for more concrete examples of such triples. 
 
Since in our construction each piece of the manifold $\mathcal{M}$ is hyperk\"ahler, hence carries a hyperk\"ahler triple,  we will first show that there exists a glued triple of closed $2$-forms $\bm{\omega}$ on $\mathcal{M}$ which is very close to being a hyperk\"ahler triple, i.e.,
\begin{equation}
\|Q_{\bm{\omega}}-\Id\|_{C^0(\mathcal{M})} <\epsilon
\end{equation}
for some sufficiently small constant $\epsilon>0$. A precise statement will be proved in Section \ref{s:approx-triple}.
Starting from such a glued approximately hyperk\"ahler triple $\bm{\omega}$, 
the goal of the perturbation procedure is to find a triple of closed $2$-forms $\bm{\theta}=(\theta_1,\theta_2,\theta_3)$ such that $
\underline{\bm{\omega}}\equiv\bm{\omega}+\bm{\theta}$ is an actual hyperk\"ahler triple on $\mathcal{M}$. Precisely, we will solve the system
\begin{equation}
\frac{1}{2}(\omega_i+\theta_i)\wedge(\omega_j+\theta_j)=\delta_{ij} \dvol_{\bm{\omega}+\bm{\theta}}, \label{e:perturbation}
\end{equation}
which is equivalent to
\begin{equation}
\frac{1}{2}(\omega_i + \theta_i)\wedge (\omega_j+ \theta_j)=\frac{1}{6}\delta_{ij}\sum\limits_{j=1}^3(\omega_j+\theta_j)^2. \label{e:perturbed-hkt-eq}
\end{equation}
We expand equation \eqref{e:perturbed-hkt-eq}:
\begin{equation}
\frac{1}{2}(\omega_i\wedge\omega_j + \omega_i\wedge\theta_j + \omega_j\wedge\theta_i + \theta_i\wedge \theta_j) =\frac{1}{6}\delta_{ij}\sum\limits_{j=1}^3\Big(\omega_j^2 + \theta_j^2+ 2\omega_j\wedge \theta_j\Big).\label{e:expansion}
\end{equation}
Now split $\bm{\theta}$ into its self-dual and anti-self-dual parts with respect to $g_{\bm{\omega}}$, 
$\bm{\theta}=\bm{\theta}^++\bm{\theta}^-$. We define a matrix $A=(A_{ij})$ by $\theta^+_i = \sum\limits_{j=1}^3 A_{ij} \omega_j$ and also define the matrix $S_{\bm{\theta}^-}=(S_{ij})$ by 
\begin{equation}
\frac{1}{2}\theta_i^{-}\wedge\theta_j^{-}=S_{ij}\dvol_{\bm{\omega}},\ 1\leq i\leq j\leq 3.
\end{equation}
 Then in terms of the volume form $\dvol_{\bm{\omega}}$ given by the approximately hyperk\"ahler triple $\bm{\omega}$, the expansion \eqref{e:expansion} can be rewritten as the matrix equation
\begin{align}
\begin{split}
&(Q_{\bm{\omega}}+Q_{\bm{\omega}}A^T+AQ_{\bm{\omega}}+AQ_{\bm{\omega}}A^T)+S_{\bm{\theta}^-} \\
&=\frac{1}{3}\Id\cdot\Big(\Tr(Q_{\bm{\omega}})+\Tr(AQ_{\bm{\omega}}A^T)+\Tr(S_{\bm{\theta}^-})+\Tr(AQ_{\bm{\omega}})+\Tr(Q_{\bm{\omega}}A^T)\Big).
\end{split}
\end{align}
For any $3\times 3$ real matrix $B$ denote by \begin{equation}\TF(B) = B - \frac{1}{3}\Tr(B)\Id\end{equation} the trace-free part of $B$. Then we get

\begin{equation}
\TF(Q_{\bm{\omega}}A^T + Q_{\bm{\omega}}A + AQ_{\bm{\omega}}A^T) = \TF(-Q_{\bm{\omega}}-S_{\bm{\theta}^-}).\label{e:lftf}
\end{equation}
For simplicity, in our context, we always identify a $3\times 3$-matrix with a triple of self-dual $2$-forms.
Then observe that a solution of following gauge-fixed system
\begin{align}
d^+ \bm{\eta} + \bm{\xi} = \mathfrak{F}_0\Big(\TF(-Q_{\bm{\omega}}-S_{d^{-}\bm{\eta}})\Big), \
d^* \bm{\eta} = 0,
\label{e:elliptic-system}
\end{align} 
is also a solution of \eqref{e:lftf}. Here $\mathfrak{F}_0$
denotes the local inverse near zero to the local diffeomorphism $\mathfrak{G}_0: \mathscr{S}_0(\dR^3) \to \mathscr{S}_0(\dR^3)$ on the space of trace-free symmetric $3\times 3$-matrices defined by 
\begin{align}\mathfrak{G}_0(A) = \TF(Q_{\bm{\omega}}A^T + AQ_{\bm{\omega}}+ AQ_{\bm{\omega}}A^T).               \end{align} Moreover, 
$\bm{\eta}=(\eta_1,\eta_2,\eta_3)$ with $\eta_j\in\Omega^1(\mathcal{M})$, 
$\bm{\xi}= (\xi_1, \xi_2, \xi_3)$ with 
$\xi_i\in\mathcal{H}^+(\mathcal{M})$ (the space of self-dual harmonic $2$-forms with respect to $g_{\bm{\omega}}$),  and $d^{\pm}\bm{\eta}$ is the self-dual or anti-self-dual part of  $d \bm{\eta} = \bm{\theta} - \bm{\xi}$ with respect to $g_{\bm{\omega}}$, respectively.

The linearization of the elliptic system \eqref{e:elliptic-system} at $\bm{\eta} = 0$ is 
\begin{equation}
\mathscr{L} = (\mathscr{D} \oplus \Id)\otimes \dR^3 : (\Omega^1(\mathcal{M}) \oplus \mathcal{H}^+(\mathcal{M}) )\otimes \dR^3 \longrightarrow ( \Omega^0(\mathcal{M}) \oplus \Omega^2_+(\mathcal{M}))\otimes \dR^3,
\end{equation}
where
\begin{equation}
\mathscr{D}\equiv d^* + d^+: \Omega^1(\mathcal{M}) \longrightarrow (\Omega^0(\mathcal{M}) \oplus \Omega^2_+(\mathcal{M}))
\end{equation}
is a Dirac-type operator.

In order to solve the elliptic system \eqref{e:elliptic-system} we will use the implicit function theorem (see Lemma~\ref{l:implicit-function}). This requires finding a bounded right inverse to the linearized operator $\mathscr{L}$. 

The first step is to define certain weighted Banach norms whose setup requires a careful understanding of the collapsing geometries in our construction. The construction of a suitable weight function will be more complicated than in almost all known gluing constructions because our neck region is topologically nontrivial. Geometrically, a crucial step is to assign to each $p\in\mathcal{M}$ an appropriate scale $r_p$ such that the geodesic ball $B_{r_p}(p)$ captures enough geometric information and curvature remains uniformly bounded for most points in $B_{r_p}(p)$. This is equivalent to finding a canonical rescaling factor for each $p\in\mathcal{M}$ which reflects the collapsing behavior of the glued metric near $p$. This leads to a decomposition of $\mathcal{M}$ (see Section \ref{ss:notations} and \ref{ss:rescaled-limits})
into 9 different regions, labeled $\I$, $\II$, $\III$, $\IV_{\pm}$, $\V_{\pm}$, $\VI_{\pm}$, 
with each of these regions exhibiting different regularity and convergence behaviors.

Second, in order to prove uniform boundedness of the right inverse by contradiction (Proposition \ref{p:injectivity-of-D}), we then need to analyze the kernel of the linearized operator in suitable weighted spaces on the building blocks of the gluing construction. Here a crucial ingredient is a new Liouville theorem for half-harmonic 1-forms, i.e., $1$-forms $\phi$ satisfying
\begin{equation} 
d^*\phi= 0, \ \ d^+\phi=0,\label{e:half-harmonic}
\end{equation}
on a Tian-Yau space, see Theorem \ref{t:liouville-1-form}.

\subsection{Outline of the paper}

In Section \ref{s:model-space}, we give some background on the Gibbons-Hawking ansatz. We also define the Heisenberg nilmanifolds, and describe the {\textit{Calabi model space}} used in the Tian-Yau construction
from the Gibbons-Hawking perspective. Lastly, we construct a harmonic function whose associated Gibbons-Hawking ansatz defines the neck region $\mathcal{N}$. 

The Calabi model space will be described in Section \ref{s:Tian-Yau} from a complex geometric perspective, which will be used to obtain the precise asymptotic behavior of the complete hyperk\"ahler Tian-Yau spaces. The main result is Proposition~\ref{p:TY-asym} which roughly states that a complete Tian-Yau space is exponentially asymptotic to a Calabi model space, up to any arbitrary order of derivatives.

In Section \ref{s:liouville-functions}, we will establish a Liouville type theorem for harmonic functions 
which says that any harmonic function of sufficiently small exponential growth on a complete hyperk\"ahler Tian-Yau space
 has to be a constant.  To prove this, the asymptotics proved in Section \ref{s:Tian-Yau} will be crucial. These will allow us to reduce our problem to a question about harmonic functions on the Calabi model space, which will be solved using separation of variables. Specifically, the Laplace equation on the model space will be reduced to certain linear ODEs, and the quantitative analysis of the harmonic functions reduces to some delicate estimates for \emph{Hermite functions}. The 
estimates in Section \ref{s:liouville-functions} do not require any advanced theory in hypergeometric functions. Instead the ODE solutions are defined by exponential integrals, which has an advantage that all the calculations in Section \ref{s:liouville-functions} are in fact self-contained.
 This step is already quite involved because it requires us to develop some new elliptic theory on a model space which is a doubly warped product rather than just a cylinder.

Using this work, we will then prove  the aforementioned Liouville theorem for half-harmonic 1-forms on a complete hyperk\"ahler Tian-Yau space in Section \ref{s:liouville-1-form} (see Theorem \ref{t:liouville-1-form}). This will later be crucial in the proof of the key uniform estimate (Proposition \ref{p:injectivity-of-D}) for our gluing construction. 
An important observation here is that equation \eqref{e:half-harmonic} is equivalent to the $(0,1)$-component of $\phi$ satisfying $\bp^*\phi^{0, 1}=\bp\phi^{0, 1}=0$ for any choice of compatible integrable complex structure, which will allow us to invoke tools from complex geometry.  
More precisely, since the equation $\bp\phi^{0, 1}=0$ depends only on the complex structure, we may study it using a smooth K\"ahler metric on the compactified Fano manifold. The Kodaira vanishing theorem then ensures that $\phi^{0, 1}=\bp f$ for some function $f$.  Such an $f$ is unique only modulo holomorphic functions, so a priori there is no growth estimate on $f$. However,  using the construction of the Tian-Yau metrics and basic elliptic estimates we may find a solution $f$ satisfying {\textit{some}} growth estimates.  The equation $\bp^*\phi^{0, 1}=0$ then implies that $f$ is harmonic with respect to the Tian-Yau metric. However at this point we are not yet able to conclude using the Liouville theorem for harmonic functions  that $f$ must be constant because our growth estimate for $f$ is too weak. We overcome this problem by using separation of variables on the Calabi model space. This final step hinges on the fact that the asymptotic decay rate of the complex structure of the Tian-Yau manifold relative to the Calabi model is much faster than the asymptotic decay rate of the hyperk\"ahler metric.

In Section \ref{s:approx-triple} we will complete the construction of the neck region and construct
a closed almost hyperk\"ahler triple on a manifold $\mathcal{M}$. We will also describe some topological invariants of $\mathcal{M}$. Note that at this stage it would be difficult to show directly that $\mathcal{M}$ is actually {\em{diffeomorphic}} to $\K3$, a fact which will follow easily once we have shown it admits a hyperk\"ahler metric.

Section \ref{s:gluing-space} will focus on the geometry the manifold $\mathcal{M}$. 
We provide two different methods for analyzing the curvature of the approximate metric. One way is to apply some new types of $\epsilon$-regularity theorems for collapsed Einstein manifolds due to Naber and the fourth author \cite{NaberZhang}, which will yield curvature control once a simple topological condition is verified (see theorem \ref{t:collapsed-eps-reg} in Section \ref{s:gluing-space}). On the other hand, in Lemma~\ref{l:curvature-estimates} we will also give a direct curvature estimate in the collapsed regions. 
The remainder of this section will then describe all of the rescaled Gromov-Hausdorff limits for basepoints in each of the various regions.

In Sections \ref{s:weight-analysis} and \ref{s:existence}, we will set up the weight analysis package and prove the main technical theorems.
The first step is to define weighted H\"older spaces which are consistent with the different collapsing behaviors in different regions of $\mathcal{M}$. The remainder of Section \ref{s:weight-analysis} will then consist of proving the weighted Schauder estimate in Proposition \ref{p:weighted-Schauder-estimate}. Existence of a bounded right inverse to the linearized operator will be proved in Section \ref{s:existence} (see Proposition~\ref{p:injectivity-for-L}), where we will then complete the proofs of Theorems \ref{t:codim-3} and \ref{t:domain-wall-crossing}.

\subsection{Acknowledgements} We would like to thank Olivier Biquard, Simon Donaldson, Lorenzo Foscolo, Mark Gross, Mark Haskins, and Xiaochun Rong for various helpful comments and discussions.
We would also like to thank Bobby Acharya and David Morrison for pointing out to us some interpretations of our work from the physical perspective. 

\section{The Gibbons-Hawking ansatz and the model space}
\label{s:model-space}

We first recall the Gibbons-Hawking construction of $4$-dimensional hyperk\"ahler structures with an $S^1$ symmetry. Let $(U, g_U)$ be a $3$-dimensional parallelizable flat manifold. Then $h_U=e_1^2+ e_2^2+dz^2$, where $e_1, e_2, dz$ are global parallel $1$-forms. Let $V$ be a positive harmonic function on $U$,  so $*dV$ is a closed 2-form. We assume further that the de Rham class $[\frac{1}{2\pi}*dV]\in H^2(U, \Z)$ so $*dV$ is the curvature form of a unitary connection $-i\theta$ on a circle bundle $\pi: \mathfrak{M}\rightarrow U$. 
Then
\begin{equation}g=V\pi^*h_U+V^{-1}\theta\otimes \theta\end{equation}
is a hyperk\"ahler metric on $\mathfrak{M}$ invariant under the natural $S^1$ action. This is called the \emph{Gibbons-Hawking ansatz}. The corresponding triple of symplectic forms $\bm{\omega}=(\omega_1,\omega_2,\omega_3)$ is given by
\begin{equation}
\label{hktd}
\begin{split}
\omega_1=dz\wedge\theta+Ve_1\wedge e_2,\\
\omega_2=e_1\wedge\theta+Ve_2\wedge dz,\\
\omega_3=e_2\wedge \theta+Vdz\wedge e_1,\\
\end{split}
\end{equation}
and satisfies
\begin{equation}
\frac{1}{2}\omega_i\wedge\omega_j=\delta_{ij}\dvol_{\bm{\omega}}.
\end{equation} 

By definition the metric $g$ depends not only on the harmonic function $V$ but also on the choice of a connection $\theta$ with curvature form $\ast dV$. One can check that gauge equivalent connections lead to isometric metrics; so in essence $g$ depends only on the gauge equivalence class of $\theta$. Different gauge equivalence classes differ by tensoring with a flat connection, and the set of isomorphism classes of flat connections is given by $H^1(U, \dR)/H^1(U, \dZ)$.

Conversely, any $4$-dimensional hyperk\"ahler metric admitting a tri-holomorphic Killing symmetry is locally given by the Gibbons-Hawking ansatz. To see this we notice that in the formula above, we can intrinsically interpret $V^{-1}$ as the norm-squared of the Killing field, and the projection map $\pi$ as the hyperk\"ahler moment map.

To get more interesting examples one often allows $V$ to have isolated poles $\mathcal{P}_k \equiv \{p_1, \ldots, p_k\}$ such that near each $p_j$, $V$ can be written as $\frac{1}{2r_j}+h_j$ where $r_j$ is the distance function to $p_j$ and $h_j$ is a smooth harmonic function. Then the corresponding metric $g$ is defined on a manifold $\mathfrak{M}$ admitting a projection $\pi: \mathfrak{M}\rightarrow U$,  such that  $\pi$ is a circle bundle over $U\setminus \mathcal{P}_k$ and near each point $p_j$, $\pi$ is modeled on the Hopf fibration 
\begin{equation}\pi: \dC^2\rightarrow \dR^3,\ (z_1, z_2)\mapsto \Big(\frac{|z_1|^2-|z_2|^2}{2}, Re(z_1z_2), Im(z_1 z_2)\Big). \end{equation}

\begin{example}\label{ex:taub-nut}
Let $U=\dR^3$. If $V=\sigma$ (a positive constant), then $(\mathfrak{M},g)$ is a flat product $\dR^3\times S^1$. If $V=\frac{1}{2r}$, then $(\mathfrak{M},g)$ is flat Euclidean space $\dR^4$ and the map $\pi$ is exactly the Hopf fibration. If $V_{\sigma}=\sigma+\frac{1}{2r}$ then $(\mathfrak{M},g)$ is the \emph{Taub-NUT space}. This is again diffeomorphic to $\RR^4$ but has cubic volume growth and is asymptotic to an  $S^1$ fibration over $\dR^3\setminus K$ at infinity where the length of the $S^1$ fibers approaches a positive constant. Notice that as $\sigma$ varies, these  metrics are isometric up to dilation.  This is most easily seen using the above intrinsic description. We take the metric $g_1$ constructed using $V_1=1+\frac{1}{2r}$ and rescale $g_{\sigma}=\sigma^{-1}g_1$. Then the length of the $S^1$ orbits becomes $(\sigma V_1)^{-1/2}$ and the hyperk\"ahler moment map becomes $\pi_{\sigma}=\sigma^{-1}\pi$. Thus, $g_\sigma$ can be written in Gibbons-Hawking form with potential $\sigma V_1 = \sigma +\frac{1}{2r_{\sigma}}$. See Lemma \ref{l:rescaled-Taub-NUT} for more details. 

By taking multiple poles, we similarly obtain other hyperk\"ahler manifolds which are asymptotic to quotients of either $\dR^4$ or Taub-NUT space by cyclic groups. These are usually referred to in the literature as ALE and ALF spaces of $A_{k-1}$ type. In particular, see \cite{minerbe} for a complete theory of ALF-$A_{k-1}$ spaces. 
\end{example}

\begin{example}\label{ex:ooguri-vafa}
Let $U=S^1\times \mathbb{R}^2$ and let $V$ be a Green's function with exactly one pole on $S^1\times \{0\}$. In \cite{GW} $V$ is constructed by passing to the universal cover $\widetilde{U}=\dR^3$, where the lifted function $\widetilde{V}$ is a periodic Green's function constructed using a Weierstrass series. $V$ is only positive in a certain bounded open set in $U$. The corresponding hyperk\"ahler metric on this bounded open set is called the \emph{Ooguri-Vafa metric}. With one particular choice of a compatible complex structure, $\mathfrak{M}$ becomes a holomorphic elliptic fibration over a disc $\dD\subset \dR^2=\dC$, and the singular fiber has monodromy of type $I_1$.  The Ooguri-Vafa metric plays a crucial role in the work of Gross-Wilson  \cite{GW} on collapsing Calabi-Yau metrics on elliptic $\K3$ surfaces with exactly 24 singular fibers of type $I_1$.
\end{example}

In the following subsections, we will consider Gibbons-Hawking spaces $(\mathfrak{M},g)$ whose base is a large open subset of a flat cylinder $U\equiv \mathbb{T}^2_{xy} \times\dR_z$. In Section \ref{ex:model-space}, we take $V$ to be linear in $z$. This yields the model space at infinity of the Tian-Yau metrics \cite{TianYau} as well as of the gravitational instantons with $r^{4/3}$ volume growth and $r^{-2}$ curvature decay from \cite{Hein}. In this situation, $\mathfrak{M}$ is diffeomorphic to the product of a line and a $3$-dimensional Heisenberg nilmanifold. We begin by collecting together some useful basic facts about Heisenberg nilmanifolds in Section \ref{heisnil}. In Section \ref{s:greens-flat-cylinder} we then consider a doubly-periodic analog of Example \ref{ex:ooguri-vafa} over $\mathbb{T}^2$ times a bounded interval, and show that this is asymptotic to the model space of Section \ref{ex:model-space} near the ends of the interval. Ultimately this new metric will serve as the neck region in our gluing construction.

\subsection{The Heisenberg nilmanifolds}
\label{heisnil}

In this subsection, we will define $3$-dimensional Heisenberg nilmanifolds.
Recall the $3$-dimensional Heisenberg group is 
\begin{equation}
H(1,\dR) \equiv\left\{\begin{bmatrix}
1 & x &t \\
0 & 1 &  y\\
0 & 0 & 1
\end{bmatrix}: \ x,y,t\in\dR\right\}.
\end{equation}
To define 
a Heisenberg nilmanifold, let us define a co-compact group action on $H(1,\dR)$.
First, we define a lattice $\Lambda \equiv \epsilon \Z\langle 1, \tau\rangle\subset \dR^2_{x, y}=\dC$ by 
choosing  
\begin{equation}
\tau_1=Re(\tau), \ \tau_2=Im(\tau),
\end{equation}
which is generated by 
\begin{align}
\begin{bmatrix}1 & \epsilon & 0 \\
0 & 1 & 0 \\
0 & 0 & 1 \\
\end{bmatrix}
,\  
\begin{bmatrix}
1 & \epsilon \tau_1 & 0   \\
0 & 1 & \epsilon \tau_2 \\
0 & 0 & 1 \\
\end{bmatrix}\in H(1,\dR).
\end{align}
 Then immediately $A =\Area(\dR_{x,y}^2/\Lambda)= \epsilon^2 \tau_2$.
For $b \in \dZ_+$,  
the Heisenberg nilmanifold $\Nil^3_b(\epsilon, \tau)$ of {\it degree} $b$ 
is the quotient of $H(1,\dR)$ by the left action  generated by
\begin{align}
\begin{bmatrix}1 & \epsilon & 0 \\
0 & 1 & 0 \\
0 & 0 & 1 \\
\end{bmatrix}
,
\begin{bmatrix}
1 & \epsilon \tau_1 & 0   \\
0 & 1 & \epsilon \tau_2 \\
0 & 0 & 1 \\
\end{bmatrix}
,
\begin{bmatrix}
1 & 0 & \frac{A}{b} \\
0 & 1 & 0 \\
0 & 0 & 1 \\
\end{bmatrix}
.
\end{align}
Note that 
these transformations are
\begin{align}
(x,y,t) &\mapsto (x + \epsilon, y, t + \epsilon y),\\
(x,y,t) &\mapsto (x + \epsilon \tau_1, y+ \epsilon \tau_2, t + \epsilon \tau_1 y ),\\
(x,y,t) &\mapsto \Big( x, y, t + \frac{A}{b} \Big).
\end{align}
The forms 
\begin{align}\label{e:wtf}
dx, dy, \theta_b \equiv \frac{2 \pi b}{A} (dt - x dy)
\end{align}
are a basis of left-invariant $1$-forms. 

It is clear that $\Nil^3_b$ is the total space of a  degree $b$ circle fibration
\begin{align}
\label{cirfib}
  S^1 \longrightarrow \Nil^3_b \xrightarrow{\ \pi \ } \TT.
\end{align}
The following result will be needed later in Proposition \ref{p:topological-invariants} to determine the Betti numbers of $\mathcal{M}$.
 \begin{proposition}
   \label{nilbn} For $\Nil^3_b$, we have  $b_1(\Nil^3_b) = b_2(\Nil^3_b) = 2$,
   and the de Rham cohomology group $H^1(\Nil^3_b)$ is generated by $\pi^* dx$ and $\pi^* dy$.
\end{proposition}
\begin{proof}
  The Gysin sequence associated to \eqref{cirfib} yields 
\begin{align}
    0 \rightarrow H^1(\TT) \xrightarrow{\pi^*} H^1(\Nil^3_b) \rightarrow
    H^0(\TT) \xrightarrow{\cup e} H^2(\TT) \rightarrow \cdots
\end{align}
Since the Euler class $e$ of the bundle is $b$ times a generator of $H^2(\TT)$, the mapping
$\cup e : \RR \cong H^0(\TT) \rightarrow H^2(\TT) \cong \RR$ is just multiplication by $b$,
so this mapping is an isomorphism. Consequently, $\pi^* : H^1(\TT) \rightarrow H^1(\Nil^3_b)$ is also an isomorphism. 
Since $\Nil^3_b$ is a compact orientable $3$-manifold, Poincar\'e duality implies that $b_1 = b_2$. 
\end{proof}

For $b \in \dZ_+$, we define the Heisenberg nilmanifold $\Nil^3_{-b}$ 
to be the quotient of $H(1,\dR)$ by the action generated by 
\begin{align}
(x,y,t) &\mapsto (x +  \epsilon, y, t - \epsilon y),\\
(x,y,t) &\mapsto (x + \epsilon \tau_1, y+ \epsilon \tau_2, t - \epsilon \tau_1 y ),\\
(x,y,t) &\mapsto \Big(x, y, t -  \frac{A}{b}\Big).
\end{align}
Note that the generated action is conjugate to the previous action by the mapping $(x,y,t) \mapsto (-x,-y,-t)$.  The forms 
\begin{align}
dx, dy, \theta_{-b} \equiv \frac{2 \pi b}{A} (dt + x dy)
\end{align}
are a basis of left-invariant $1$-forms.

\subsection{The model space}
\label{ex:model-space}  

Consider a $2$-torus $\TT$ with a flat  metric of area $A$ and let $U=\TT_{x, y}\times \mathbb{R}_{z>0}$, where we have fixed a choice of an orthogonal frame $\{e_1, e_2\}$ on $\TT$ such that $g_{\TT}=A(e_1^2+e_2^2)$.
Let $V(z)=\frac{2\pi bz}{A}$ for a positive integer $b>0$. Fixing a connection form $\theta$ such that $d\theta = \frac{2\pi b}{A}  {\rm dvol}_{\TT}$, the corresponding hyperk\"ahler Gibbons-Hawking metric $g$ is given by
\begin{align}\label{e:wtf666}
g = \frac{2 \pi bz}{A} ( g_{\TT} +dz^2) + \frac{A}{2\pi b z} \theta^2.
\end{align}
The Gibbons-Hawking space $(\mathfrak{M},g)$ has one complete end as $z\rightarrow\infty$ and one incomplete end as $z\rightarrow 0$. Moreover, for each $z_0>0$, the level set $\{z = z_0\}$ is a Heisenberg nilmanifold $\Nil^3_b(\epsilon, \tau)$ with a $z_0$-dependent left-invariant metric, where $\tau$ denotes the modulus of our flat $2$-torus in the upper half-plane and $A = \epsilon^2 \tau_2$ as in Section \ref{heisnil}.
Making the substitution $z = (3/2) s^{2/3}$, and then scaling appropriately, the Gibbons-Hawking metric $g$ takes the form 
\begin{align}
ds^2 + s^{2/3} g_{\TT} + s^{-2/3} \Big( \frac{A}{3\pi b }\theta \Big)^2.
\end{align}
In this form, it is easy to see that the volume growth is $\sim s^{4/3}$ 
and that $|{\Rm}| \sim s^{-2}$ as $s \rightarrow \infty$. One can also show using the Chern-Gauss-Bonnet formula that the $L^2$ norm of $\Rm$  is  finite. 

Note that if we had instead taken $b = 0$, the Gibbons-Hawking metric would be the product of $\RR$ with a flat $3$-torus, i.e., the type of geometry known as ALH geometry in the literature.

View the flat torus $\mathbb{T}^2$ as an elliptic curve $E$ of modulus $\tau$ with respect to the complex structure $J$ defined by $Je_1 = e_2$. Then there is exactly one $g$-parallel complex structure $J_0$ on the total space $\mathfrak{M}$ that makes the projection map to $E$ holomorphic. This choice of complex structure realizes $\mathfrak{M}$ as an open subset of the total space of a degree $b$ holomorphic line bundle $L$ over $E$ (more precisely, as a tubular neighborhood of the zero section of $L$ with the zero section removed). The Ricci-flat K\"ahler form with respect to $J_0$ is then given by  (see Proposition~\ref{Calabihyperkahler})
\begin{equation}\omega_0=\frac{2}{3}i\partial\bar\partial (-{\log |\xi|}_{h}^2)^{3/2},
\end{equation}
where $h$ is a hermitian metric on $L$ whose curvature form is a multiple of the flat K\"ahler form on $E$, and where the tautological section $\xi$ cuts out the zero section of $L$. This is an example of the \emph{Calabi construction}, and we call the corresponding metric a \emph{Calabi model metric} of degree~$b$. We will discuss this construction (in all dimensions) in more detail in Section \ref{s:Tian-Yau}. In complex dimension $2$, the Gibbons-Hawking ansatz actually recovers the Calabi ansatz for \emph{all} degree $b$ holomorphic line bundles over $E$. Indeed, any two such line bundles differ by a degree 0 line bundle and $\Pic^0(E)=H^1(E, \O_E)/H^1(E, \dZ)$; on the other hand, the gauge equivalence classes of flat connections are parametrized by $H^1(E, \dR)/H^1(E, \dZ)$. 

A very convenient property for our purposes is that if we do change the connection $1$-form $\theta$ by a flat connection, then the associated Gibbons-Hawking metric \eqref{e:wtf666} actually only changes by a diffeomorphism (although this diffeomorphism necessarily breaks the $S^1$-bundle structure on $\mathfrak{M}$). To see this, fix any choice of $\theta$ such that $d\theta=\frac{2\pi b}{A} \dvol_E$. Note that we can always arrange by parallel transport in the $z$-direction that the $dz$-component of $\theta$ vanishes.  Then we only need to consider the case that $\theta$ gets changed by the pullback (under the projection $\mathfrak{M} \stackrel{\pi}{\to} U \to E$) of a parallel $1$-form $\eta$ on $E$. Write $\eta = v \,\lrcorner\,\dvol_E$, where $v$ is a parallel vector field on $E$. Let $\hat v$ be the $\theta$-horizontal lift of $v$ to $\mathfrak{M}$ and let $\hat f_s$ be the $1$-parameter group of diffeomorphisms generated by $\hat v$, which covers a $1$-parameter group of translations on $E$. Then
\begin{align}\begin{split}\frac{d}{ds}(\hat f_s^*\theta)&=\hat v\,\lrcorner\, \hat{f}_s^*(d\theta)+d(\hat v\,\lrcorner\,\hat{f}_s^*\theta)=\frac{2\pi b}{A} (\hat v\,\lrcorner\, \pi^*({\rm dvol}_E)) = \frac{2\pi b}{A} \pi^*\eta,
\end{split}
\end{align}
using the fact that $d\hat{f}_s|_x (\hat{v}|_x) = \hat{v}|_{\hat{f}_s(x)}$ for all $x \in \mathfrak{M}$. Thus,
\begin{equation}\hat f_s^* \theta=\theta+\frac{2\pi bs}{A} \pi^*\eta.\end{equation}
 This shows that any two possible choices of $\theta$ with vanishing $dz$-component differ by the translation action of $E$ on itself, lifted to $\mathfrak{M}$, and then the corresponding hyperk\"ahler metrics \eqref{e:wtf666} are clearly isometric as well. We note that with the particular choice $\theta = \theta_b$ from \eqref{e:wtf}, these diffeomorphisms can be written down explicitly as follows: for all $p, q \in \RR$, the mapping 
\begin{align}
\varphi(x,y,t,z) = (x-q, y +p, t + p x, z)
\end{align}
descends to a diffeomorphism of $ {\rm Nil}^3_b(\epsilon,\tau)_{x,y,t} \times \RR_z$ which satisfies 
\begin{align}
\varphi^* \theta_b = \theta_b + \frac{2 \pi b}{A}( p dx + q dy). 
\end{align}

\begin{remark}\label{boahey}
The above observations reflect the fact that $E$ acts transitively by pullback on its own ${\rm Pic}^b$ for $b > 0$ (for instance, a holomorphic line bundle of degree $1$ on $E$ is uniquely isomorphic to $\mathcal{O}_E(x)$ for some point $x \in E$). Moreover, the total spaces of all degree $b > 0$ holomorphic line bundles on $E$ are actually biholomorphic as complex manifolds. All of this is false in the classical ALH case $b = 0$. In particular, for $b = 0$ the above parameters $p,q$ are actual moduli of the metric, corresponding to flat  metrics on $\mathbb{T}^3$ that do not split isometrically as $S^1 \times \mathbb{T}^2$.  
\end{remark}

The Calabi construction provides the model at infinity of the Tian-Yau metrics in our context. These metrics will also be studied in more detail in Section \ref{s:Tian-Yau}. Note that every Tian-Yau metric comes with its own preferred choice of a connection form $\theta$ determined by the Chern connection of the normal bundle of the compactifying divisor. The above gauge fixing construction then allows us to choose a new coordinate system at infinity that identifies this $\theta$ with the standard connection form $\theta_b$ (see the proof of Proposition \ref{Calabihyperkahler} and also equation \eqref{thetaass}).

\begin{remark}
We can choose a different $g$-parallel complex structure $J_1$ on $\mathfrak{M}$  such that $J_1\theta= zdx$. With respect to $J_1$ we can view $\mathfrak{M}$ as a holomorphic elliptic fibration over a punctured disc in $\mathbb{C}$. The monodromy of this fibration is given by the matrix \begin{equation}\begin{bmatrix}1 & b\\ 0 & 1\end{bmatrix}\in \SL(2,\dZ).\end{equation}
See \cite{Scott} for more details. 
This gives a different compactified model for $\mathfrak{M}$ where the compactifying divisor is a singular fiber of Kodaira type $I_b$. The $J_1$-K\"ahler form of our hyperk\"ahler model metric is then given by an appropriate \emph{semi-flat ansatz} \cite{GSVY, GW}, and provides the model at infinity for the gravitational instantons with volume growth $\sim r^{4/3}$ and curvature decay $\sim r^{-2}$ constructed in \cite{Hein}. We will pursue this observation further in \cite{HSVZ2}, connecting the complete hyperk\"ahler metrics of \cite{TianYau} and of \cite{Hein} by global hyperk\"ahler rotations.
\end{remark}

\subsection{Green's function on a flat cylinder}\label{s:greens-flat-cylinder} The neck region in our gluing construction is given by a doubly-periodic analog of the Ooguri-Vafa metric. To construct this metric using the Gibbons-Hawking ansatz, we first need to construct a Green's function $V_{\infty}$ on $(\TT \times \mathbb{R}, g_0)$, where $\TT$ is any flat $2$-torus. We also need to determine the asymptotics of $V_\infty$ as $|z| \to \infty$ in order to ensure that the neck matches the Calabi model space from Section \ref{ex:model-space} on both ends.

It is known that the flat cylinder $\TT\times\dR$ is {\it parabolic}, namely, it does not admit a positive Green's function. In fact,  Cheng and Yau
proved that 
any complete non-compact 
Riemannian manifold with $\Vol(B_R(p))\leq CR^2$ must be parabolic. See theorem 1 and corollary 1 in \cite{Cheng-Yau}.

We now construct a particular sign-changing Green's function $V_{\infty}$.
Fix a point $p$ on $\TT \times \mathbb{R}$ with $z(p)=0$. For $R>0$ let $V_R$ denote the unique function on $\TT \times [-R, R]$ satisfying
 \begin{align}
\begin{split}
-\Delta_{g_0} V_R &= 2\pi \delta_p, \ z\in(-R,R),\\
V_R&=0, \ \hspace{16pt}  z=\pm R.
\end{split}
\end{align}
The normalization of the right-hand side is precisely chosen in such a way that $V_R = \frac{1}{2r} + O(1)$ near $p$, where $r$ denotes $g_0$-distance to $p$. Thus, the $4$-dimensional Gibbons-Hawking metric associated with $V_R$ (or $V_R + C$ for any constant $C$) extends smoothly across $p$; cf. Example \ref{ex:taub-nut}.

By the maximum principle, one can see that $V_R$ is an increasing family as $R\to\infty$. Define 
\begin{equation}
C_R\equiv\sup\limits_{\partial B_1(p)}V_R.
\end{equation} 
Then it follows from the parabolicity of $\mathbb{T}^2 \times \mathbb{R}$ that $C_R \to \infty$ as $R \to \infty$. 
By a result of Li and Tam (see theorem 1 in \cite{Li-Tam}), $V_R(x,y,z)-C_R$ converges to a function $V_\infty(x,y,z)$ uniformly on compact subsets of $(\TT\times\dR)\setminus \{p\}$. Moreover, $V_\infty \leq 0$  on the complement of $B_1(p)$ and
\begin{equation}
-\Delta_{g_0} V_\infty=2\pi \delta_p\ \text{on}\ \TT\times\dR.
\end{equation}
Notice also that $V_{\infty}$ is symmetric in $z$ by construction.

\begin{theorem}\label{t:harmonic-function-cylinder}
There are constants $\beta_-,\beta_+\in\dR$ and $k_-, k_+ \in \dR$ 
with 
\begin{equation}k_- = -k_+= \frac{\pi}{\Area_{g_0}(\TT)} > 0\label{e:slope-relation}
\end{equation}   such that   for all $k\in \dN$,
\begin{align}
\begin{split}
|\nabla^k_{g_0}(V_\infty(z)-(k_{-}z+\beta_-))|&=O(e^{\sqrt{\lambda_1} z}),\ z\to-\infty, \\
|\nabla^k_{g_0}(V_\infty(z)-(k_+z + \beta_+))|&=O(e^{-\sqrt{\lambda_1} z}),\ z\to+\infty, \\
\end{split}
\end{align}
where $\lambda_1>0$ is the smallest eigenvalue of $-\Delta_{\TT}$. 
\end{theorem}
\begin{proof}
Consider the fiberwise average
\begin{equation}
\mathcal{V}_{\infty}(z) \equiv \frac{1}{{\rm Area}_{g_0}(\TT)}\int_{\TT \times \{z\}} V_\infty(x,y,z) \dvol_{g_0}(x,y).
\end{equation}
This is well-defined, smooth in $z$ for $z \neq 0$, and continuous at $z = 0$.
 For $z\neq 0$ we have 
 \begin{equation}\mathcal{V}_{\infty}''(z)=\int_{\TT \times \{z\}} \frac{d^2}{dz^2}V_{\infty}=-\int_{\TT\times\{z\}} \Delta_{\TT} V_{\infty}=0.\end{equation} 
 This implies that $\mathcal{V}_\infty(z)$ is a piecewise linear function. Since $V_R(x,y,z) \leq -C_R$ for $|z| \geq R$ with $C_R \to \infty$ as $R \to \infty$, it follows that $\lim\limits_{z\to +\infty}V_\infty(x,y,z) = -\infty$, and hence for all $z > 0$ that 
 \begin{equation}\label{666} \mathcal{V}_{\infty}'(z) = const \equiv k_+  < 0.\end{equation} 
Denote $D_0\equiv \diam_{g_0}(\TT)$. Choose $R_0 > 10D_0$ large enough so that $V_{\infty}(x,y,z) \leq -1$ in $(\TT\times\dR)\setminus B_{R_0}(p)$. 
Then
for any fixed $q\in (\TT\times\dR)\setminus B_{2R_0}(p)$ and $r\in(2D_0,4D_0)$, we can apply the Harnack inequality to the harmonic function $V_{\infty}$, which is negative in the geodesic ball 
$B_{r}(q)\subset (\TT\times\dR)\setminus B_{R_0}(p)$. More precisely, passing to the universal cover and applying the standard Harnack inequality for positive harmonic functions on a fixed ball in $\mathbb{R}^3$, we see that  there is a uniform constant $C_0>0$ depending only on $D_0>0$ such that for all $w_1,w_2\in B_r(q)$, 
\begin{equation}
\frac{1}{C_0}\leq \frac{V_{\infty}(w_1)}{V_{\infty}(w_2)} \leq C_0.\label{e:harnack}
\end{equation}
Since the fiber average $\mathcal{V}_{\infty}(z)$ of $V_\infty$ is linear in $z$ with slope $k_+ < 0$, \eqref{e:harnack} yields that  
 \begin{equation}
-C_2 z\leq V_\infty(x,y,z) \leq - C_1 z  \label{e:V-asymp}
 \end{equation}
for $z\gg 1$, where the constants $C_1$ and $C_2$ depend only on the constants $C_0$ and $k_+$.

We denote by $\Lambda_{\TT}=\{\lambda_j\}_{j=1}^{\infty}$ the positive spectrum of $-\Delta_{\TT}$ and expand $V_\infty$ according to the eigenfunctions of $\Delta_{\TT}$ along the torus fiber  $\TT\times \{z\}$ for each fixed $z>0$. This yields
 \begin{equation}V_\infty(x,y,z)=(k_+ z + \beta_+) + \sum_{j=1}^{\infty} f_j (z)h_j(x, y),\end{equation}
 where $k_+ < 0$ is the constant of \eqref{666} and where
 \begin{align}
 f_j''(z)=\lambda_j f_j(z), \
 -\Delta_{\TT} h_j= \lambda_j h_j, \ 
 \int_{\TT} |h_{j}|^2=1.
 \end{align}
Immediately, 
 \begin{equation}f_j(z)=c_j e^{-\sqrt{\lambda_j} z}+c_j^*e^{\sqrt{\lambda_j} z}. \end{equation}
 Notice that
 \begin{equation}\int_{\TT\times\{z\}}|V_\infty|^2=(k_+ z + \beta_+)^2 +\sum_{j=1}^{\infty} |f_j(z)|^2.\end{equation}
By the linear growth property \eqref{e:V-asymp}, we obtain that $c_j^*=0$ for all $j\in\dZ_+$. Therefore, 
 \begin{equation}\int_{\TT\times \{z\}} |V_\infty-(k_+z + \beta_+)|^2=\sum_{j=1}^{\infty} |c_j|^2 e^{-2\sqrt{\lambda_j} z}=O(e^{-2\sqrt{\lambda_1} z})\ \text{as}\ z \to +\infty,\end{equation}
where $\lambda_1>0$ is the minimum of $\Lambda_{\TT}$. To see the $O(e^{-2\sqrt{\lambda_1} z})$ estimate, note that the series converges for $z = 1$. Applying elliptic regularity to the harmonic function $\widehat{V}_{\infty}\equiv V_{\infty}-(k_+ z + \beta_+)$, 
\begin{equation}
\|\widehat{V}_{\infty}\|_{W^{2,2}(B_{r/2}(q))} \leq C \|\widehat{V}_{\infty}\|_{L^{2}(B_{r}(q))} \leq Ce^{-2\sqrt{\lambda_1} z}
\end{equation}
for all balls $B_r(q)$ as above, 
where $C$ depends only on the diameter and on the injectivity radius of $\TT$. By the $3$-dimensional Sobolev embedding $W^{2,2}\hookrightarrow C^{0,\frac{1}{2}}$, 
\begin{equation}
|V_\infty(x,y,z)-(k_+ z + \beta_+)|=O(e^{-\sqrt{\lambda_1} z})\ \text{as}\  z\to+\infty. \end{equation}
Standard elliptic regularity then shows that for any $k\in\mathbb{N}$,  
\begin{equation}
|\nabla_{g_0}^k (V_\infty(x,y,z)-(k_+ z + \beta_+))|=O(e^{-\sqrt{\lambda_1} z})\ \text{as}\  z\to+\infty. 
\end{equation}
The same argument applies in the case $z\rightarrow-\infty$. 

Now we prove the slope relation \eqref{e:slope-relation}.
Fix $R>0$. Then by Green's formula,
\begin{equation}\mathcal{V}_{\infty}'(R)-\mathcal{V}_{\infty}'(-R)=\int_{\TT \times [-R, R]}\Delta V_\infty.\end{equation}
Thus, by the definition of $k_+$ and the analogous definition of $k_-$, 
\begin{equation}
k_+\Area_{g_0}(\TT) - k_-\Area_{g_0}(\TT) =-2\pi.
\end{equation}
It follows that
\begin{equation}
k_-  - k_+ = \frac{2\pi}{\Area_{g_0}(\TT)}.
\end{equation}
Since $V_{\infty}$ is symmetric in $z$, it holds that $k_- = -k_+$ and the claim follows.
\end{proof}
 
It is straightforward to use superposition to extend the above construction to the case of multiple poles. Precisely, we have the following corollary.
\begin{corollary} 
\label{coro:gf}
Let $(\TT\times\dR, g_0)$ be a flat cylinder with a flat product metric $g_0$. Given a finite set $\mathcal{P}_{m_0}\equiv\{p_1, \ldots, p_{m_0}\}\subset\TT\times\dR$, there is a sign changing Green's function $V_{\infty}$ with
\begin{equation}
-\Delta_{g_0} V_{\infty} = 2\pi \sum\limits_{k=1}^{m_0} \delta_{p_k},
\end{equation}
and there are linear functions $L_{\pm}(z) = k_{\pm} z + \beta_{\pm}$ with
\begin{equation}
k_- = - k_+ = \frac{\pi m_0}{\Area_{g_0}(\TT)} > 0
\end{equation}
 such that for all $k \in \mathbb{N}$,
\begin{align}
\begin{split}
|\nabla^k_{g_0}(V_\infty(z)-L_-(z))|&=O(e^{\sqrt{\lambda_1} z}),\ z\to-\infty, \\
|\nabla^k_{g_0}(V_\infty(z)-L_+(z))|&=O(e^{-\sqrt{\lambda_1} z}),\ z\to+\infty, \\
\end{split}
\end{align}
where $\lambda_1>0$ is the smallest eigenvalue of $-\Delta_{\TT}$. 
\end{corollary}

\begin{remark}
The asymptotics in Theorem \ref{t:harmonic-function-cylinder} have an intuitive electro-magnetic interpretation. Namely, for $|z|$ large, the electric potential determined by the union of point charges on $\Lambda\times \{z_0\}$ is approximated by an evenly distributed charge on the corresponding plane $\dC\times \{z_0\}$ in the universal cover. The electric potential of this uniformly charged plate corresponds to the linear term in $V_{\infty}$.
Further, for any number of charged plates parallel to the $xy$-plane located at 
$z = z_i$, $i = 1, \ldots, m_0$, we can add together 
the corresponding $V_i$ to get the potential $V=V_1 + \dots + V_{m_0}$, and the potential at large distances looks like the potential due to $m_0$ uniformly charged plates. 
\end{remark}

In Section \ref{s:approx-triple}, we will use $V_{\infty}$ to construct a potential $V_{\beta}$ which is positive in a large region, and such that $[\frac{1}{2\pi}* dV_{\beta}]\in H^2(U,\dZ)$ is a {\textit{integral}} class,  which will then be used to define our neck metric. 

\section{The asymptotic geometry of Tian-Yau spaces} \label{s:Tian-Yau}

In this section we review the Tian-Yau construction \cite{TianYau} of complete Ricci-flat K\"ahler metrics on the complement of a smooth anti-canonical divisor in a smooth Fano manifold. In complex dimension $2$ these metrics are hyperk\"ahler and will be used in our gluing construction in Section \ref{s:approx-triple}.

The Ricci-flat   metrics  constructed in \cite{TianYau} are asymptotic to the \emph{Calabi model space} at infinity. We first give the definition of the latter. Let $D$ be an $(n-1)$-dimensional compact K\"ahler manifold with trivial canonical bundle and let $L\rightarrow D$ be an ample line bundle. We fix a nowhere vanishing holomorphic $(n-1)$-form $\Omega_D$ on $D$ with 
\begin{equation}\label{lalala}
\frac{1}{2}\int_{D} i^{(n-1)^2}\Omega_D\wedge\overline{\Omega}_D=(2\pi c_1(L))^{n-1}.\end{equation} By Yau's resolution of the Calabi conjecture \cite{Yau}, there exists  a unique Ricci-flat K\"ahler metric $\omega_D\in 2\pi c_1(L)$ satisfying the equation 
\begin{equation}
\omega_D^{n-1}=\frac{1}{2}i^{(n-1)^2} \Omega_D\wedge\overline{\Omega}_D. 
\end{equation}
 Up to scaling there exists a unique hermitian metric $h$ on $L$ whose curvature form is $-i\omega_D$. We now fix a choice of $h$. Then the Calabi model space is the subset $\mathcal C$ of the total space of $L$ consisting of all elements $\xi$ with $0<|\xi|_h < 1$, endowed with a nowhere vanishing holomorphic volume form $\Omega_\mathcal C$ and a Ricci-flat K\"ahler metric $\omega_\mathcal C$ which is incomplete as $|\xi|_h \to 1$ and complete as $|\xi|_h \to 0$. The holomorphic volume form $\Omega_{\mathcal C}$ is uniquely determined by the equation 
 \begin{equation}Z \lrcorner\ \Omega_{\mathcal C}=p^*\Omega_D\end{equation}
 where $p: \mathcal C\rightarrow D$ is the bundle projection and  $Z$ is the holomorphic vector field generating the natural $\dC^*$-action on the fibers of $p$. The metric  $\omega_{\mathcal{C}}$ is given by the \emph{Calabi ansatz}
 \begin{equation}\label{calabiansatz}\omega_{\mathcal C}=\frac{n}{n+1} i\p\bp (-{\log |\xi|_h^2})^{\frac{n+1}{n}} \end{equation}
and satisfies the Monge-Amp\`ere equation
 \begin{equation}\omega_{\mathcal C}^n=\frac{1}{2}i^{n^2} \Omega_{\mathcal C}\wedge\overline\Omega_{\mathcal C},\end{equation}
hence is Ricci-flat. Define $z=(-{\log |\xi|_h^2})^{1/n}$. It is easy to check that $z$ is the $\omega_{\mathcal{C}}$-moment map for the natural $S^1$-action on $L$. Then the $\omega_{\mathcal{C}}$-distance function $r$ to a fixed point in $\mathcal C$ satisfies
 \begin{equation}
C^{-1}z^{\frac{n+1}{2}}\leq r\leq C z^{\frac{n+1}{2}}
 \end{equation}
uniformly for all $z \geq 1$. 

If $n=2$ then $D=E$ is an elliptic curve and the Calabi model space is hyperk\"ahler and agrees with the Gibbons-Hawking construction from Section \ref{ex:model-space}. To see this,  we choose $A=2\pi \deg(L)$ and a flat K\"ahler form $\omega_{E}=Ae_1\wedge e_2$ and a holomorphic 1-form $\Omega_E=A^{1/2}(e_1+ie_2)$ on $E$ such that $\omega_E=\frac{i}{2}\Omega_E\wedge\overline{\Omega}_E$. Then the Calabi ansatz produces a holomorphic $2$-form $\Omega_{\mathcal C}$ and a Ricci-flat K\"ahler form $\omega_{\mathcal{C}}$ on $\mathcal C$ such that $\omega_{\mathcal C}^2=\frac{1}{2}\Omega_{\mathcal C}\wedge\overline{\Omega}_{\mathcal C}$. In particular, 
\begin{equation}\label{e:calabitriple}
\bm{\omega}_{\mathcal{C}}\equiv(\omega_{\mathcal C}, Re(\Omega_{\mathcal C}), Im(\Omega_{\mathcal C}))
\end{equation}
 is a hyperk\"ahler triple. 

 \begin{proposition} \label{Calabihyperkahler}
The hyperk\"ahler structure $\bm{\omega}_{\mathcal{C}}$ is diffeomorphism equivalent to the one given by the  Gibbons-Hawking ansatz with $V = z$ on ${\rm Nil}^3_b(\epsilon,\tau) \times (0, \infty)$ as in Section \ref{ex:model-space}, where $b = \deg(L)$, $\tau$ is the modulus of $E$ in the upper half-plane, $A = 2\pi b = \epsilon^2 Im(\tau)$, and $\theta = \theta_b$. 
\end{proposition}

\begin{proof}
The natural $S^1$-action on $\mathcal C$ obviously preserves ${\bm \omega}_{\mathcal C}$, so ${\bm \omega}_{\mathcal C}$ must be given by the Gibbons-Hawking construction. It suffices to determine the hyperk\"ahler moment map $\pi =(\pi_1, \pi_2, \pi_3)$ and the potential $V$. We already know that the $\omega_{\mathcal{C}}$-moment map is given by $\pi_1 = (-{\log |\xi|_h^2})^{1/2} = z$, and it is easy to check that the $\Omega_{\mathcal{C}}$-moment map $\pi_2 + i\pi_3$ equals the bundle projection $p: \mathcal{C} \to E$ 
followed by the Abel-Jacobi isomorphism $E \to \mathbb{C}/\text{\emph{periods}}$, $x \mapsto \int_{x_0}^x i\Omega_E$, for an arbitrary but fixed basepoint $x_0 \in E$. Also, $V^{-1}$ is the norm-squared of the Killing field, so that
\begin{equation*}
V^{-1}=(-{\log |\xi|^2_h})^{-1/2}=z^{-1}. 
\end{equation*}
The Calabi construction provides us with an explicit realization $\mathfrak{M} = \mathcal{C}$ of the total space of the $S^1$-bundle and with a specific choice of connection form $\theta$ given by the Chern connection of $L$ with respect to $h$. The diffeomorphism equivalence to the model ${\rm Nil}_b^3(\epsilon,\tau) \times (0,\infty)$ with connection form $\theta_b$ (after rotating $\Omega_{\mathcal C}$ to $e^{i\alpha}\Omega_{\mathcal C}$ if necessary) follows from the discussion before Remark \ref{boahey}.
\end{proof}

\begin{remark}
In \cite{TianYau} the volume growth and curvature decay rates of the $n$-dimensional Calabi metric $\omega_{\mathcal{C}}$ are estimated as $O(r^{\frac{2n}{n+1}})$ and $O(r^{-\frac{2}{n+1}})$, respectively. When $n = 2$, this suggests that $|{\rm Rm}|$ is borderline \emph{not} in $L^2$. However, while the volume estimate is sharp for all $n$, the curvature estimate is sharp only for $n \geq 3$. For $n = 2$, the leading term in the asymptotic expansion of the curvature vanishes because the Calabi-Yau metric on an elliptic curve is flat, and the true curvature decay rate of $\omega_{\mathcal{C}}$ for $n = 2$ is $O(r^{-2})$. This was also pointed out by R.~Kobayashi in \cite{Koba} but is perhaps most easily seen in the Gibbons-Hawking picture.
\end{remark}

We now explain the Tian-Yau construction \cite{TianYau} of complete Ricci-flat K\"ahler metrics asymptotic to a Calabi ansatz at infinity. Let $M$ be a smooth Fano manifold of complex dimension $n$, let $D\in |K_M^{-1}|$ be a smooth divisor, and let $L$ denote the holomorphic normal bundle to $D$ in $M$. Then $D$ has trivial canonical bundle and $L$ is ample, so in particular we can choose a holomorphic volume form $\Omega_D$ on $D$ which satisfies \eqref{lalala}. We fix a defining section $S$ of $D$, so that $S^{-1}$ can be viewed as a holomorphic $n$-form $\Omega_X$ on $X=M\setminus D$ with a simple pole along $D$. After scaling $S$ by a nonzero complex constant, we may assume that $\Omega_D$ is the residue of $\Omega_X$ along $D$. (In practice this means that $\Omega_X$ is asymptotic to $\Omega_{\mathcal{C}}$ near $D$ with respect to a suitable diffeomorphism between tubular neighborhoods of $D$ in $M$ and of the zero section in $L$.) Lastly, we fix a hermitian metric $h_M$ on $K_M^{-1}$ whose curvature form is strictly positive on $M$ and restricts to the unique Ricci-flat K\"ahler form $\omega_D \in 2\pi c_1(L)$ on $D$. Then
 \begin{equation}\omega_X \equiv\frac{n}{n+1}i\p\bp (-{\log |S|_{h_M}^2})^{\frac{n+1}{n}}  \end{equation}
defines a K\"ahler form on a neighborhood of infinity in $X$. In fact, by multiplying $h_M$ by a sufficiently small positive constant, we can arrange that $\omega_X$ is defined and strictly positive on all of $X$. As expected, $\omega_X$ is then complete and asymptotic to $\omega_{\mathcal{C}}$, where the hermitian metric $h$ used in \eqref{calabiansatz} is simply the restriction of $h_M$ to $K_M^{-1}|_D = L$. In particular, $\omega_X$ is asymptotically Ricci-flat and the $\omega_X$-distance function $r_X$ to any fixed basepoint in $X$ can be uniformly estimated by
 \begin{equation}C^{-1}(-{\log |S|^2_{h_M}})^{\frac{n+1}{2n}}\leq  r_X \leq C(-{\log |S|^2_{h_M}})^{\frac{n+1}{2n}}\ \text{as}\ |S|_{h_M} \to 0. \end{equation}
The following theorem is proved in \cite{TianYau} by solving a Monge-Amp\`ere equation with reference metric $\omega_X$. The exponential decay statement follows from Proposition 2.9 in \cite{Hein}. 

\begin{theorem}[\cite{TianYau, Hein}] \label{t:hein}
There is a smooth function $\phi$ on $X$ such that  $\omega_{TY}\equiv \omega_X+i\p\bp\phi$ is a complete Ricci-flat K\"ahler metric on $X$ solving the Monge-Amp\`ere equation
 \begin{equation}\omega_{TY}^n=\frac{1}{2}i^{n^2}\Omega_X\wedge\overline\Omega_X.  \end{equation}
Moreover,  there is a constant $\delta_0 = \delta_0(M,D) >0$ such that for all integers $k\geq 0$,
 \begin{equation}\label{lalilu}|\nabla_{g_X}^k \phi|_{g_X} = O(e^{-\delta_0r_X^{\frac{n}{n+1}}})\ \text{as} \ r_X \to \infty.
\end{equation}
\end{theorem}
We remark that polynomial decay estimates were obtained in \cite{KK}.

To make this result useful for our gluing construction in this paper, we need to replace $\omega_X$ by the Calabi model metric $\omega_{\mathcal C}$ in \eqref{lalilu}. This amounts to estimating the convergence of $\omega_X$ to $\omega_{\mathcal C}$. In the following, the big-O notation will mean with respect to the limit $z \rightarrow \infty$, unless otherwise indicated. 
\begin{proposition}\label{p:TY-asym}
There is a diffeomorphism $\Phi: \mathcal C\setminus K'\rightarrow X\setminus K$, where $K \subset X$ is compact and $K' = \{|\xi|_h \geq \frac{1}{2}\}$, such that the following hold uniformly for all large enough values of $z$.
\begin{enumerate}
\item[(a)]  We have the complex structure asymptotics
\begin{equation}
|\nabla_{g_{\mathcal C}}^k(\Phi^*J_{X}-J_{\mathcal C})|_{g_\mathcal C}=O(e^{-(\frac{1}{2}-\epsilon)z^n})\ \text{for all} \ k \geq 0, \epsilon > 0.
\end{equation}

\item[(b)] We have the holomorphic $n$-form asymptotics
\begin{equation}\label{e:n-form-asympt}
|\nabla_{g_{\mathcal{C}}}^k(\Phi^*\Omega_X-\Omega_{\mathcal C})|_{g_{\mathcal C}}=O(e^{-(\frac{1}{2}-\epsilon)z^n})\ \text{for all} \ k \geq 0, \epsilon > 0.
\end{equation} 

\item[(c)] There is some positive constant \begin{equation}\label{tritratrullala}\underline{\delta}>0\end{equation} such that for all $k\geq 0$ the Ricci-flat K\"ahler metric $\omega_{TY}$ satisfies the asymptotics
\begin{equation}|\nabla_{g_\mathcal C}^k(\Phi^*\omega_{TY}-\omega_{\mathcal C})|_{g_{\mathcal C}}=O(e^{-\underline{\delta} z^{n/2}}).\end{equation} 
\end{enumerate}
\end{proposition}

\begin{proof}
One can prove using an argument due to Donaldson that items (a) and (b) are equivalent. The point is that $\Omega_X$ uniquely determines $J_X$ because $J_X$ is determined by knowing the subspace $\Lambda^{1,0}_{J_X} \subset \Lambda^1_{\mathbb{C}}X$, and we have $\Lambda^{1,0}_{J_X}X = {\rm ker}\,T$, where $T: \Lambda^1_{\mathbb{C}} X \to \Lambda^{n+1}_{\mathbb{C}} X$ is the $\C$-linear map defined by $T\alpha = \Omega_X \wedge \alpha$. See Lemma 2.14 in \cite{CH1} for details.

Item (b) can be proved by following the steps of a similar estimate in the asymptotically conical case in Section 2.2 of \cite{CH2}. Fix any background hermitian metric $g$ on $M$. Via $g$-orthogonal projection, the holomorphic normal bundle $L = N_D = T^{1,0}M|_D/T^{1,0}D$ is naturally isomorphic to the $g$-orthogonal complement $(T^{1,0}D)^\perp \subset T^{1,0}M$ as a $C^\infty$ complex line bundle, and the $g$-normal exponential map defines a diffeomorphism from a  neighborhood of the zero section in $(T^{1,0}D)^\perp$ to a neighborhood of $D$ in $M$. Let $\Phi$ be the composition of these two maps. Then $\Phi$ is a diffeomorphism from a neighborhood of the zero section in $L$ to a neighborhood of $D$ in $M$, and the restriction of $\Phi$ to the zero section is ${\rm Id}_D$. Note that $\Phi$ is almost never holomorphic, but in generic situations $\Phi$ will be one of  the ``most holomorphic'' diffeomorphisms between tubular neighborhoods of $D$ in $L$ and in $M$. In any case,  $\Phi$ turns out to be good enough to obtain the asymptotics \eqref{e:n-form-asympt}.

Fix a point on $D$. Let $(z_1,\ldots,z_{n-1},w)$ be local holomorphic coordinates on $M$ centered at this point such that $D$ is locally cut out by $w = 0$. Then $(z_1, \ldots, z_{n-1}, w)$ may also be viewed as local holomorphic coordinates on $L$ corresponding to the normal vector $w(\frac{\partial}{\partial w} + T^{1,0}D) \in L$ based at the point $(z_1, \ldots, z_{n-1},0) \in D$. In these coordinates  we may write 
\begin{align}
\Omega_X &= \biggl(\frac{f(z)}{w} + g(z,w)\biggr) dz_1 \wedge \ldots \wedge dz_{n-1} \wedge dw,\label{e:form-on-X}\\
\Omega_{\mathcal{C}} &= \frac{f(z)}{w}dz_1 \wedge \ldots \wedge dz_{n-1} \wedge dw,
\end{align}
where $f,g$ are local holomorphic functions with $f(z) \neq 0$ for all $z$. In order to compare $\Phi^*\Omega_X$ to $\Omega_{\mathcal{C}}$, we define new $C^\infty$ complex coordinates $(z_1', \ldots, z_{n-1}', w)$ on $M$ by 
\begin{align}\label{e:new-coord}
z_i'(z,w)= z_i - a_i(z) w, \ \text{where}\ a_i(z) =  \frac{\partial(z_i \circ \Phi)}{\partial w}\biggr|_{(z,0)}.
\end{align}
Using the fact that $d\Phi$ is complex linear at $w = 0$ and that $\Phi_*(\frac{\partial}{\partial w}) = \frac{\partial}{\partial w} + \sum a_i(z) \frac{\partial}{\partial z^i}$ at $w = 0$, it is easy to check that these new coordinates satisfy
\begin{align}
\Phi^*dz_i' = dz_i \ \text{and} \ \Phi^*dw = dw \ \text{at} \ w = 0.
\end{align}
By Taylor expansion, it follows directly from this that 
\begin{align}\label{e:coord-exp}
\Phi^*z_i' = z_i + A_{i}w^2 + B_i w\overline{w} + C_i \overline{w}^2  \ \text{and} \ \Phi^*w = w + A w^2 + B w\overline{w} + C \overline{w}^2
\end{align}
with smooth functions $A_i, B_i, C_i$ and $A,B,C$. We now express the coordinates $(z,w)$ in \eqref{e:form-on-X} in terms of $(z', w)$ using \eqref{e:new-coord}, and then use \eqref{e:coord-exp} to compare $\Phi^*\Omega_X$ to $\Omega_{\mathcal{C}}$. The first step yields
\begin{align}
\Omega_X = \frac{f(z')}{w} dz_1' \wedge \ldots \wedge dz_{n-1}' \wedge dw + \Upsilon \wedge dw,
\end{align}
where $\Upsilon$ extends to a smooth complex $(n-1)$-form on a neighborhood of $D$ in $M$. Then
\begin{align}\label{e:yetanotherequn}
\begin{split}
&\Phi^*\Omega_X - \Omega_{\mathcal{C}} \\  
=&\frac{(A' + B'\frac{\overline{w}}{w} + C'\frac{\overline{w}^2}{w^2})(\Upsilon' \wedge dw)+ (A'' + B'' \frac{\overline{w}}{w} + C''\frac{\overline{w}^2}{w^2})( \Upsilon'' \wedge d\overline{w}) + (w\Theta' + \overline{w}\Theta'' + \frac{\overline{w}^2}{w}\Theta''')}{1 + Aw + B\overline{w} + C \frac{\overline{w}^2}{w}}\\
&+ \Phi^*\Upsilon \wedge (dw + w \phi' + \overline{w} \phi''),
\end{split}
\end{align}
where $A',B',C', A'', B'', C'',\Upsilon', \Upsilon'',\Theta', \Theta'', \Theta''',\phi', \phi''$ extend to smooth complex functions, $(n-1)$-forms, $n$-forms, and $1$-forms on a neighborhood of $D$ in $L$, respectively. The reason for writing the right-hand side of \eqref{e:yetanotherequn} in this way is that a smooth complex $n$-form is small with respect to $g_{\mathcal{C}}$ if it either contains an explicit factor of $w$ or $\overline{w}$ in front, or if it splits off a wedge factor of $dw$ or $d\overline{w}$. Unfortunately the right-hand side of \eqref{e:yetanotherequn} is not smooth at the divisor but all non-smooth terms are due to factors of $\overline{w}/w$, which satisfy the same estimates as smooth functions.

It remains to prove appropriate estimates on $|\nabla^k_{g_{\mathcal C}} F|_{g_{\mathcal C}}$ for all $k \geq 0$, where $F$ is either a smooth function on a neighborhood of $D$ in $L$, or $F = \overline{w}/w$. To begin, note that
\begin{align}
|\nabla_{g_{\mathcal C}}^k z_i|_{g_{\mathcal{C}}} &= O(1)\ \text{for all} \ k \geq 0,\label{hatschi1}\\
|w| &= O(e^{- \frac{1}{2}z^n}), \ |\nabla_{g_{\mathcal C}}^k w|_{g_{\mathcal{C}}} = O(e^{-(\frac{1}{2}-\epsilon)z^n})\ \text{for all}\  k \geq 1, \epsilon > 0, \label{hatschi2}\\
|w^{-1}| &= O(e^{\frac{1}{2}z^n}), \ |\nabla_{g_{\mathcal C}}^k w^{-1}|_{g_{\mathcal{C}}} = O(e^{(\frac{1}{2}+\epsilon)z^n})\ \text{for all}\  k \geq 1, \epsilon > 0.\label{hatschi10}
\end{align}
Here the bound $|z_i| = O(1)$ is clear, and the bounds $|w^{\pm 1}| = O(e^{\mp \frac{1}{2}z^n})$ follow from the definition of the moment map $z = (-{\log |\xi|_h^2})^{1/n}$ together with the fact that $|\xi|_h^2 = |w|^2 e^{-\phi}$ with $\phi$ independent of $w,\overline{w}$. The higher derivative bounds in \eqref{hatschi1}--\eqref{hatschi10} follow from these pointwise bounds by using elliptic estimates for holomorphic functions on a K\"ahler manifold of $C^\infty$ bounded geometry (these estimates apply here because $g_{\mathcal{C}}$ is Ricci-flat K\"ahler of bounded curvature). Note that the $\epsilon$-terms in \eqref{hatschi2}--\eqref{hatschi10} are necessary because the sup of $|w|$ over a $g_{\mathcal{C}}$-ball of radius 1 is $O(|w|^{1-\epsilon})$ for every $\epsilon > 0$ but is not $O(|w|)$, unlike on a cylinder $\C^*_w \times D$ with model metric $|d \log w|^2 + g_D$.

We now prove by induction that for all smooth functions $F$ on a neighborhood of $D$ in $L$,
\begin{equation}\label{hatschi3}
|F| = O(1), \ |\nabla_{g_{\mathcal C}}^k F|_{g_{\mathcal{C}}} = O(e^{\epsilon z^n}) \ \text{for all} \ k \geq 1, \epsilon > 0.
\end{equation}
Indeed, the pointwise bound is clear, and for $k \geq 1$ we apply $\nabla_{g_{\mathcal{C}}}^{k-1}$ to the expansion
\begin{equation}
dF = \frac{\partial F}{\partial w} dw  + \frac{\partial F}{\partial \overline{w}}d\overline{w} + 
\sum_{i=1}^{n-1} \frac{\partial F}{\partial z_i} dz_i + \sum_{i=1}^{n-1} \frac{\partial F}{\partial \overline{z}_i}d\overline{z}_i,
\end{equation}
using the inductive hypothesis to control $\nabla_{g_{\mathcal{C}}}^{k-1}$ of the partials of $F$ on the right-hand side and using \eqref{hatschi1}--\eqref{hatschi2} to control $\nabla_{g_{\mathcal{C}}}^{k-1}$ of $dz_i, d\overline{z}_i, dw, d\overline{w}$. This proves \eqref{hatschi3}. By using \eqref{hatschi2}--\eqref{hatschi10} we can then prove in a similar manner that 
\begin{align}\label{hatschi4}
\left|\frac{\overline{w}}{w}\right| = O(1), \ \left|\nabla^{k}_{g_{\mathcal{C}}} \left(\frac{\overline{w}}{w}\right) \right|_{g_{\mathcal{C}}} = O(e^{\epsilon z^n})  \ \text{for all} \ k \geq 1, \epsilon > 0.
\end{align}
Taken together, \eqref{hatschi3} and \eqref{hatschi4} allow us to estimate all contributions to \eqref{e:yetanotherequn} in all $C^k$ norms with respect to $g_{\mathcal{C}}$, proving item (b). 

To prove item (c), notice that in local coordinates $\xi = (z_1, \ldots, z_{n-1}, w) \in L$ as above,
 \begin{equation}
\Phi^*|S|^2_{h_M}=(1+G)|\xi|^2_{h} = (1 + G)|w|^2 e^{-\phi}
\end{equation}
with smooth real-valued locally defined functions $G$ and $\phi$, where $G$ vanishes at $w = 0$ and $\phi$ does not depend on $w, \overline{w}$. Notice that $G =  F w + \overline{Fw} $ for some smooth complex-valued locally defined function $F$. This structure of the $\Phi^*J_X$-K\"ahler potential of $\Phi^*\omega_X$, together with \eqref{hatschi2}, \eqref{hatschi3}, and item (a), makes it possible to prove that for all $k \geq 0, \epsilon > 0$, 
\begin{equation}
|\nabla^k_{g_{\mathcal{C}}}(\Phi^*\omega_X-\omega_{\mathcal C})|_{g_\mathcal C}=O(e^{-(\frac{1}{2}-\epsilon)z^n}).
\end{equation}
Similarly by Theorem \ref{t:hein}  we get for some $\underline{\delta}>0$ depending on $\delta_0$ that for all $k \geq 0$,
 \begin{equation}
|\nabla^k_{g_{\mathcal{C}}} (\Phi^*\omega_{TY} - \Phi^*\omega_X)|_{g_\mathcal C}=O(e^{-\underline{\delta} z^{n/2}}). \end{equation}
This completes the proof of item (c).
\end{proof}

\begin{remark}
The decay of the Ricci-flat K\"ahler form $\omega_{TY}$ to its asymptotic model $\omega_{\mathcal{C}}$ is weaker than the decay of the complex structure and of the holomorphic volume form. The reason is that the latter is obtained by an explicit computation where the errors admit an expansion in terms of  $|w| \sim e^{-z^n/2}$. On the other hand, the decay rate of $\omega_{TY}$ depends on an analysis of the Tian-Yau solution of the Monge-Amp\`ere equation, which is related to the fact that the decay rate of harmonic (not necessarily holomorphic) functions on the Calabi model space (see Section \ref{s:liouville-functions}) is in general only $O(e^{-\delta z^{n/2}})$.  It is an interesting question if $O(e^{-\delta z^{n/2}})$ decay of the K\"ahler form is indeed optimal. This is a global question  because one can easily construct Tian-Yau solutions outside a compact set with leading term equal to any given decaying harmonic function which is not pluriharmonic.
\end{remark}
When $n=2$, the Tian-Yau metric is hyperk\"ahler, and the corresponding hyperk\"ahler triple $\bm{\omega}_{TY}$ is determined by $\omega_{TY}$ and $\Omega_X$. Let $\bm{\omega}_{\mathcal{C}}$ be the Calabi hyperk\"ahler triple defined in \eqref{e:calabitriple}.

\begin{corollary}
\label{coro:tydiff}
Under the same diffeomorphism $\Phi$ as in Proposition \ref{p:TY-asym} we have that
 \begin{equation}\label{e:TYtripleconv}
|\nabla^k_{g_{\mathcal C}}(\Phi^*\bm{\omega}_{TY}-\bm{\omega}_{\mathcal C})|_{g_\mathcal C}=O(e^{-\underline{\delta} z}) \end{equation}
for all $k \geq 0$, where $\underline{\delta}>0$ is the same constant as in \eqref{tritratrullala}.
\end{corollary}

In our gluing construction we need the following refinement of Corollary \ref{coro:tydiff}.

\begin{lemma}
\label{lemma:diff}
There exists a triple of $1$-forms $\bm{a}$ with $\partial_z \,\lrcorner\, \bm{a} = 0$ such that 
\begin{align}
\label{atycon}
\Phi^*\bm{\omega}_{TY}-\bm{\omega}_{\mathcal C}=d\bm{a}
\end{align}
and such that for all $k \geq 0$ and all $\epsilon > 0$,
\begin{align}
\label{atycond}
| \nabla^k_{g_\mathcal C}\bm{a}|_{g_\mathcal C}  = O(e^{-(\underline{\delta}-\epsilon)z}).
\end{align}
\end{lemma}

\begin{proof}
By Proposition \ref{Calabihyperkahler}, the Calabi space $\mathcal{C}$ is diffeomorphic to $\Nil^3_b \times (0, \infty)$ in such a way that the Calabi metric $g_{\mathcal{C}}$ becomes the Gibbons-Hawking metric \eqref{e:wtf666}. Ignoring this diffeomorphism, we have a $2$-form triple ${\bm \phi}$ and a $1$-form triple ${\bm \psi}$ such that $\partial_z \,\lrcorner\,{\bm \phi} = 0$, $\partial_z \,\lrcorner\,{\bm \psi} = 0$, and
\begin{equation}\Phi^* \bm{\omega}_{TY}-\bm{\omega}_{\mathcal{C}}= \bm{\phi} + dz \wedge \bm{\psi}. \end{equation}
Since $\bm{\omega}_{TY}, \bm{\omega}_{\mathcal C}$ are closed, it follows that
\begin{equation}\label{argh42} d_{\Nil^3_b}\bm{\phi} =0,\;\, \p_z\bm{\phi} - d_{\Nil^3_b}\bm{\psi}=0.\end{equation}
Now we define the $1$-form triple
\begin{equation}
\bm{a} \equiv - \int_{z}^\infty {\bm \psi}\,dz. \label{e:def-a}\end{equation}
Thanks to \eqref{e:TYtripleconv}, this integral exists and satisfies \eqref{atycond}. Note that this is not completely obvious because the $1$-form basis $dx, dy, dt -  xdy$ on ${\rm Nil}^3_b$ is not parallel with respect to $g_{\mathcal{C}}$. However, this effect is absorbed by the $\epsilon$ in \eqref{atycond} because all error terms are at worst polynomial in $z$. Property \eqref{atycon} now follows in a standard manner by using \eqref{argh42}.
\end{proof}

\section{Liouville theorem for harmonic functions}
\label{s:liouville-functions}

In this section and the next, we will set up some technical tools for the gluing construction. One of the crucial technical ingredients in analyzing the linearized operator is to establish a Liouville theorem on the complete non-compact hyperk\"ahler manifolds that arise in our context.  

Our main goal in this section is to prove a Liouville theorem for harmonic functions with a small enough exponential growth rate, on a complete Riemannian $4$-manifold $(X^4, g)$ with non-negative Ricci curvature which is asympotically Calabi in the sense of Definition \ref{d:asymptotic-Calabi}. This is a necessary step towards proving our Liouville theorem for half-harmonic $1$-forms in Section \ref{s:liouville-1-form}.

\begin{definition}\label{d:asymptotic-Calabi} Given some constant $\delta>0$,
a complete Riemannian manifold $(X^4, g)$ is said to be $\delta$-\emph{asymptotically Calabi}
 if there exist a compact subset $K\subset X$ and a Calabi model space $(\mathcal{C}, g_{\mathcal C})$ as defined in Section \ref{s:Tian-Yau}, and a diffeomorphism 
 \begin{equation}\Phi: \mathcal{C} \setminus K'\rightarrow X\setminus K\end{equation} 
with $K' = \{|\xi|_h \geq \frac{1}{2}\} \subset \mathcal C$ such that for all $k \geq 0$,
\begin{equation}|\nabla_{g_{\mathcal C}}^k(\Phi^*g-g_{\mathcal C})|_{g_{\mathcal C}}=O(e^{-\delta z}) \ \text{as} \ z \to \infty,
\end{equation}
 where $z = (-{\log |\xi|_h^2})^{1/2}$ denotes the natural moment map coordinate on $(\mathcal{C}, g_{\mathcal{C}})$. 
\end{definition}

\begin{example}
According to Proposition \ref{p:TY-asym}, any complete hyperk\"ahler Tian-Yau space $(X^4, g)$ as in Theorem \ref{t:hein} is $\underline{\delta}$-asymptotically Calabi for some appropriate constant $\underline{\delta}>0$.
\end{example}

The following is our main result in this section.

\begin{theorem}
\label{t:liouville-theorem-functions}
Let $(X^{4}, g)$ be a complete Riemannian $4$-manifold which is $\delta$-asymptotically Calabi for some $\delta > 0$ and which has $\Ric_g \geq 0$. Then there exists an $\ell_0\in(0,1)$ depending on $(X^4,g)$ such that if $u$ is a harmonic function on $(X^{4}, g)$ with $u = O(e^{ \ell_0 z})$ as $z \to \infty$, then $u$ is a constant.
 \end{theorem}

The proof heavily relies on the elliptic theory of the Laplace operator on the Calabi model space. We begin with a careful study of this model operator.

\subsection{Separation of variables on the model space}
\label{ss:separation-of-variables}

We work with a Calabi model space $\mathcal C$ with a smooth divisor $D$ defined as in Section \ref{s:Tian-Yau}. The Calabi metric is given
 by 
 \begin{equation}\omega_\mathcal C=\frac{n}{n+1}\sqrt{-1}\p\bp (-\log |\xi|_{h}^2)^{\frac{n+1}{n}},
 \end{equation}
which is well-defined for $|\xi|_h<1$. In order to carry out separation of variables, we  will study the local representation of the Laplace operator $\Delta_{\mathcal{C}}$ on $\mathcal C$.

We choose local holomorphic coordinates $\underline z=\{z_i\}_{i=1}^{n-1}$ on the smooth divisor $D$, and fix a local holomorphic trivialization  $e_0$ of the line bundle $L$  with $|e_0|^2=e^{-\psi}$,  where $\psi:D \to \dR$ is a smooth function. So we get local holomorphic coordinates $(\underline z,w)\equiv(z_1,\ldots,z_{n-1}, w)$ on $\mathcal C$ by writing a point  $\xi\in \mathcal C$ as $\xi=w e_0(\underline z)$. Then $|\xi|_h^2=|w|^2e^{-\psi}$. We may assume $\psi(0)=1$, $d\psi(0)=0$ and $\sqrt{-1}\partial\bar\partial\psi=\omega_0$. Let $\pi: \mathcal C\rightarrow D$ be the obvious projection map. Then  we obtain 
 \begin{equation}
 \label{Calabiformula}
 \omega_\mathcal C=(-\log |\xi|^2_h)^{\frac{1}{n}}\omega_D+\frac{1}{n}(-\log |\xi|_h^2)^{\frac{1}{n}-1} \sqrt{-1}(\frac{dw}{w}-\p\psi)\wedge (\frac{d\bar w}{\bar w}-\bp\psi).
 \end{equation}
 Let $u$ be a $C^2$-function in the Calabi space $\mathcal{C}$, the Laplacian at  points in the fiber $\pi^{-1}(0)$ is given by 
\begin{equation}\Delta_{\mathcal C}u= (-\log |\xi|^2_h)^{-\frac{1}{n}} \sum_{i=1}^{n-1}\frac{\partial^2u}{\partial z_i \partial \bar{z}_i} +n(-\log |\xi|_h^2)^{-\frac{1}{n}+1} |w|^2 \frac{\partial^2 u}{\partial w\partial \bar w}.\end{equation}
Now denote $\varrho\equiv|\xi|_h$, then we can write 
\begin{equation}w=\varrho e^{\frac{\psi}{2}+\sqrt{-1}\theta},\end{equation} where $\p_\theta$ generates the natural $S^1$-rotation on the total space of $ L$. Then it is straightforward to check that  
\begin{equation}
\frac{\p \varrho}{\p w}=\frac{\varrho}{2w},\ \frac{\p\theta}{\p w}=\frac{1}{2\sqrt{-1}w}, \frac{\p \varrho}{\p z_i}=-\frac{1}{2}\varrho\p_{z_i}\psi, \frac{\p \theta}{\p z_i}=0 \end{equation}
and 
\begin{equation}
|w|^2\frac{\p^2}{\p w\p \bar w} u=\frac{1}{4}(\varrho^2u_{\varrho\varrho}+\varrho u_{\varrho}+u_{\theta\theta}).
\end{equation}
 For fixed $r_0\in (0, 1)$, the level set $Y^{2n-1}\equiv\{\varrho=r_0\}$ is equipped with the induced Riemannian metric given by 
\begin{equation}h_0=(-\log r_0^2)^{\frac{1}{n}} g_D+\frac{1}{n} (-\log r_0^2)^{\frac{1}{n}-1}r_0^2(d\theta-\frac{1}{2}d^c\psi)\otimes (d\theta-\frac{1}{2}d^c\psi).\label{e:h_0}\end{equation}
Now we consider a smooth function $\phi\in C^{\infty}(Y^{2n-1})$ with \begin{equation}\mathcal L_{\p_\theta}\phi=\sqrt{-1}j\phi\label{e:S1-action}\end{equation} for some integer $j$. Replacing $\phi$ by $\bar\phi$ if necessary we may assume $j\geq 0$. Then $\phi$ is induced by a smooth section $\hat \phi$ of $(L^*)^{\otimes j}$. Precisely, if we locally write $ \hat\phi(\underline z)=\Phi(\underline z)(e_0(\underline z)^*)^{\otimes j}$, then 
\begin{equation}\phi(\underline z, w)=w^j\Phi(\underline z)|_{\varrho=r_0}=r_0^je^{j(\frac{\psi}{2}+\sqrt{-1}\theta)}\Phi(\underline z).\end{equation}
Now let $\hat\phi$ be a non-zero eigen-section of the $\bp$-Laplace operator, i.e. 
\begin{equation}\Delta_{\bp}\hat\phi=\hat{\lambda}\hat\phi.\end{equation}  By Kodaira-Nakano formula $\Delta_{\bp}=\Delta_{\p}+j(n-1)$, so we have $\hat{\lambda}\geq j(n-1)$. 
By a direct calculation, we get that on $\pi^{-1}(0)$,
\begin{equation}
 \sum_{i=1}^{n-1} \frac{\partial^2 \phi}{\partial z_i\partial \bar{z}_i}=\Big(-\hat{\lambda}+\frac{j(n-1)}{2}\Big)\phi. \label{e:phi-eq}
 \end{equation}
 Moreover, by the local expression of $h_0$ as in \eqref{e:h_0}, one can directly check that on $\pi^{-1}(0)\cap Y^{2n-1}$, 
\begin{align}
\begin{split}
\Delta_{h_0}\phi
&=(-\log r_0^2)^{-\frac{1}{n}}\sum_{i}\frac{\p^2\phi}{\p z_i\p \bar z_i}+n(-\log r_0^2)^{-\frac{1}{n}+1} r_0^{-2} \phi_{\theta\theta}\\
&= \Big((-\log r_0^2)^{-\frac{1}{n}}(-\hat{\lambda}+\frac{j(n-1)}{2})-j^2n(-\log r_0^2)^{-\frac{1}{n}+1} r_0^{-2}\Big)\phi.
\end{split}
\end{align}

Now suppose a smooth function $u(\varrho, z)$ on the Calabi space $\mathcal{C}$ is of the form $u\equiv f(\varrho) \phi(y)$, where  $\phi$ is a function on $Y^{2n-1}$ satisfying \eqref{e:S1-action} and \eqref{e:phi-eq}.
In polar coordinates,  we obtain
 \begin{align}
 \Delta_{\mathcal C}  u
 &=\phi(y)\cdot\Big((-\log \varrho^2)^{-\frac{1}{n}}((-\hat{\lambda}+\frac{j(n-1)}{2})f-\frac{n-1}{2}{\varrho}f_{\varrho})
 \nonumber\\
 &\ \ \ \ +\frac{n}{4}(-\log {\varrho}^2)^{1-\frac{1}{n}}({\varrho}^2f_{{\varrho}{\varrho}}+{\varrho}f_{\varrho}-j^2f)\Big)
 \nonumber\\
 &=\phi(y)\cdot(-\log \varrho^2)^{-\frac{1}{n}}\Big(\frac{n}{4}(-\log {\varrho}^2) ({\varrho}^2 f_{{\varrho}{\varrho}}+{\varrho}f_{\varrho}-j^2f )
 \nonumber\\
 &\ \ \ \ -\frac{n-1}{2}{\varrho} f_{\varrho} -(\hat{\lambda}-\frac{j(n-1)}{2}) f\Big). 
  \end{align}
  Notice this formula is now independent of the choice of local holomorphic coordinates. So $u$ is harmonic if and only if 
 \begin{equation}\frac{n}{4}(-\log {\varrho}^2) (\varrho^2 f_{{\varrho}{\varrho}}+\varrho f_{\varrho}-j^2f)-\frac{n-1}{2}{\varrho} f_{\varrho} -(\hat{\lambda}-\frac{j(n-1)}{2}) f=0.\end{equation}
Denote $z=(-\log \varrho^2)^{\frac{1}{n}}$, then we get 
 \begin{equation}f_{zz}-(n(\hat{\lambda}-\frac{j(n-1)}{2})+\frac{j^2n^2}{4}z^{n})z^{n-2}f=0.\end{equation}
In this section, we will also analyze the Poisson equation
\begin{equation}
\Delta_{\mathcal{C}} u = v.
\end{equation}
Suppose now $v \equiv \zeta(\varrho)\cdot \phi(y)$, then the same separation of variables gives the following ODE 
 \begin{equation}f_{zz}-(n(\hat{\lambda}-\frac{j(n-1)}{2})+\frac{j^2n^2}{4}z^{n})z^{n-2}f=z^{n-1}\cdot \zeta.\end{equation}
We remark that a similar separation of variables was carried out in \cite{KK}, but we will need stronger estimates on solutions in order to prove Theorem \ref{t:liouville-theorem-functions}.

For our application we focus on the case $n=2$. So the corresponding ODEs become 
\begin{equation}f_{zz}-(\lambda+j^2z^2)f=0 \end{equation}
and
\begin{equation}
f_{zz}-(\lambda+j^2z^2)f=z \cdot \zeta ,
\end{equation}
where \begin{equation}\lambda\equiv 2\hat{\lambda}-j\geq j.\end{equation} 
We have assumed $j\geq 0$ in the above discussion, but notice that the Laplace operator is a real operator, so the ODEs we get for $j$ and $-j$ are the same. 
Denote $z_0 \equiv (-\log r_0^2)^{\frac{1}{2}}$, then we notice that each eigenvalue of $\Delta_{h_0}$ can be represented by
\begin{equation}
\Lambda = \frac{\lambda}{2z_0}+\frac{2z_0 \cdot j^2}{r_0^{2}}.
\end{equation}

With the above computations, we are ready to set up the ODE system. 
Now we fix some $r_0\in (0, 1)$, and define $(Y^3, h_0)$ to be the level set $\{r=r_0\}$ endowed with the induced Riemannian metric $h_0$. The above computations tell us that the eigenvalues of $Y^3$ is given by linear combinations of $\hat{\lambda}$ and $j$.  
Below we will parametrize our summation in terms of eigenvalues of $(Y^3, h_0)$ (counted with multiplicity), but we shall keep in mind that we have further split the eigenspaces of $\Delta_{h_0}$ according to the $S^1$ action hence an eigenvalue is naturally written in terms of a linear combination of $\hat{\lambda}$ and $j$.

We denote by $\{\Lambda_{k}\}_{k=1}^{\infty}$
the spectrum of $\Delta_{h_0}$ and let $\{\varphi_k\}_{k=1}^{\infty}$ be the eigenfunctions which are homogeneous under the $S^1$ action and with
\begin{equation}
-\Delta_{h_0} \varphi_k = \Lambda_k \cdot \varphi_k.
\end{equation}
In the above notations, one can compute that in the case $n=2$, 
\begin{equation}
\Lambda_k = \frac{\lambda_k}{2z_0}+\frac{2z_0 \cdot j_k^2}{r_0^{2}}.\label{e:fiber-eigenvalues}
\end{equation}
First, we carry out separation of variables for harmonic functions   on $\Delta_{\mathcal{C}}$.
Let $u$ be a harmonic function on the model space $\mathcal{C}$, namely,
\begin{equation}
\Delta_{\mathcal{C}} u = 0
\end{equation}
For every fixed $z$, we can write the $L^2$-expansion along the fiber $Y^3$,
\begin{equation}
u(z, \bm{y}) = \sum\limits_{k=1}^{\infty} u_k(z) \cdot \varphi_k(\bm{y}).
\end{equation}
The above computations tell us that for each $k\in\dZ_+$, there are numbers $j_k\in\dN$ and $\lambda_k \geq n j_k$ such that the function $u_k(z)$ satisfies the differential equation
\begin{equation}
\frac{d^2u_k(z)}{dz^2} - (j_k^2 z^2 + \lambda_k) u_k(z) = 0,\ z\geq 1. \label{e:separation-harmonic}
\end{equation}
We also consider the Poisson equation 
\begin{equation}
\Delta_{\mathcal{C}} u = v.
\end{equation}
Take the $L^2$-expansion of $v$ in the direction of the cross section $Y^3$,
\begin{equation}
v(z,\bm{y}) = \sum\limits_{k=1}^{\infty}\xi_k(z) \cdot \varphi_k(\bm{y}),
\end{equation}
then the same procedure of separation of variables leads to a differential equation 
\begin{equation}
\frac{d^2u_k(z)}{dz^2} - (j_k^2 z^2 + \lambda_k) u_k(z) =\xi_k(z)\cdot z.\label{e:separation-poisson}\end{equation}

We end this subsection by giving a model example of the fiber $Y^3$.  
  
\begin{example}
[The spectrum of a Heisenberg manifold] In our interested context, $Y^3$ is a Heisenberg nilpotent manifold. We consider a simple example that $Y^3\equiv H(1,\dZ)\setminus H(1,\dR)$ with
\begin{equation}
H(1,\dR) \equiv\left\{\begin{bmatrix}
1 & x &t \\
0 & 1 &  y\\
0 & 0 & 1
\end{bmatrix}: \ x,y,t\in\dR\right\}.
\end{equation}
and
\begin{equation}
H(1,\dZ)\equiv\left\{\begin{bmatrix}
1 & m & p \\
0 & 1 &  n\\
0 & 0 & 1
\end{bmatrix}: \ m,n,p\in\dZ\right\}.
\end{equation}
In this case, $Y^3$
is a Heisenberg manifold of degree $1$. As a $\TT$ bundle over $S^1$, its  monodromy is given by $(\begin{smallmatrix} 1 & 1 \\ 0 & 1 \end{smallmatrix})\in\SL(2,\dZ)$.
So it is standard that the spectrum consists of two classes of eigenvalues
\begin{equation}
\mathfrak{T}\equiv\Big\{4\pi^2(k^2+\ell^2) \Big| k,\ell\in\dZ\Big\}
\ and\ 
\mathfrak{S}\equiv\Big\{2\pi|m|(2h+1+2\pi|m|) \Big| m\in \dZ\setminus\{0\},h\in\dN\Big\}.
\end{equation}
Detailed discussions can be found in \cite{DS} and \cite{Gordon-Wilson}. So
we can see that the above eigenvalues coincide with the form \eqref{e:fiber-eigenvalues}.
\end{example}  
  
In the following subsections, we will analyze the convergence and regularity issues of the formal solutions \eqref{e:separation-harmonic}
and \eqref{e:separation-poisson}.

\subsection{Uniform estimates for the fundamental solutions}

\label{ss:c0-estimate}

A crucial step in applying the method of separation of variables is to prove the 
 $C^0$-regularity of a formal solution obtained from the above separation of variables. 
Specifically, in our context, to prove such a $C^{0}$-regularity result, first we need to obtain 
 some effective estimates for the fundamental solutions to the linear differential equation (see \eqref{e:separation-harmonic})
 \begin{equation}
\frac{d^2u(z)}{dz^2} - (j^2 z^2 +\lambda) u(z) = 0, 
\end{equation}
which arises from the harmonic functions on the Calabi manifold $(\mathcal{C}, g_{\mathcal{C}})$. In our context, we always require 
\begin{equation}
j \geq 0, \ h\geq 0,\ z>1. 
\end{equation}
There are two different cases to analyze. 

The first case is much simpler, i.e. $j=0$ and the ODE becomes
\begin{equation}
\frac{d^2u(z)}{dz^2} = \lambda \cdot u(z). \label{e:exp-ode}
\end{equation}
Further, if $\lambda=0$, the solutions to \eqref{e:exp-ode} are linear.  
If $\lambda>0$, the above equation has two linearly independent solutions $e^{\sqrt{\lambda}z}$ and  $e^{-\sqrt{\lambda}z}$. All the required estimates in this case are standard and straightforward. 
Geometrically, the ODE analysis for \eqref{e:exp-ode} arises naturally from the flat cylindrical geometry and corresponding gluing constructions. 

So in our case, we only focus on the case $j\in\dZ_+$ which is substantially much more technically involved.
In the case $j\in\dZ_+$, we have already shown in
Section \ref{ss:separation-of-variables} that 
$\lambda$ and $j$ satisfy the relation
\begin{equation}
\lambda\geq j\geq 1.
\end{equation} Hence for each pair of $\lambda$ and $j$ satisfying the above, we choose $h\geq0$ such that  
\begin{equation}
\lambda = (2h+1)j.
\end{equation}
From now on, we focus on the differential equation for every $j\in\dZ_+$ and $h\geq 0$,
 \begin{equation}   
\frac{d^2u(z)}{dz^2} = j(jz^2 + 2h + 1)u(z).\label{e:model-ode}
\end{equation}
We will simplify the above equation by the following transformations. 
Let 
\begin{equation}
y =   \sqrt{j} z,\ V(y) =  u\Big(\frac{y}{\sqrt{j}}\Big),
\end{equation}
then $V(y)$ satisfies
\begin{equation}
\frac{d^2 V(y)}{dy^2} = (y^2 + (2h+1)) V(y).
\end{equation}
Further, we make the transformation
\begin{equation}
V(y) = e^{-\frac{y^2}{2}} Q(y),
\end{equation}
then $Q$ sovles the differential equation
\begin{equation}
\frac{d^2Q(y)}{dy^2}-2y\frac{dQ(y)}{dy} - 2(h+1)  Q(y)=0. \label{e:Hermite-ODE}
\end{equation}
Notice that equation \eqref{e:Hermite-ODE} is invariant under the change of variables $y\mapsto -y$. 
Given $y>1$ and $h\geq0$, we define the following exponential integral
\begin{equation}
H_{-h-1}(y) \equiv  \int_0^{\infty} e^{-t^2-2ty} t^{h}dt.
\end{equation}
Straightforward calculations show that for each given $h\geq 0$, the functions
 $H_{-h-1}(y)$ and $H_{-h-1}(-y)$ are linearly independent solutions to \eqref{e:Hermite-ODE}. In fact, 
the above solutions coincide with the usual Hermite functions up to a constant (see \cite{Lebedev} for more details). 
Eventually, we obtain two solutions to \eqref{e:model-ode},
\begin{equation}
\mathcal F(z)=e^{-\frac{jz^2}{2}}H_{-h-1}(-\sqrt{j}z)=e^{-\frac{y^2}{2}}\int_0^\infty e^{-t^2+2ty+h\log t}dt
 \end{equation}
 and
\begin{equation}
\mathcal U(z)=e^{-\frac{jz^2}{2}} H_{-h-1}(\sqrt{j}z)=e^{-\frac{y^2}{2}}\int_0^\infty e^{-t^2-2ty+h\log t}dt.
\end{equation}
The lemma below shows that $\mathcal{F}$ and $\mathcal{U}$  are two linearly independent solutions.

\begin{lemma} \label{l:wronskian}The Wronskian is a constant given by 
\begin{equation}
\mathcal W(\mathcal F, \mathcal U)=2^{-h}\sqrt{j\pi}\Gamma(h+1)> 0. 
\end{equation}
In particular, $\mathcal{F}$ and $\mathcal{U}$ are linearly independent.

\end{lemma}
\begin{proof}
First observe that $\mathcal{W}(\mathcal{F}(z),\mathcal{U}(z))=\mathcal F'(z)\mathcal U(z)-\mathcal F(z)\mathcal U'(z)$
is a constant. In fact,
\begin{align}
\frac{d}{dz}\mathcal{W}(\mathcal{F}(z),\mathcal{U}(z)) & =  \Big(\mathcal{F}''(z)\mathcal{U}(z)+\mathcal{F}'(z)\mathcal{U}'(z)\Big)-\Big(\mathcal{F}(z)\mathcal{U}''(z) + \mathcal{F}'(z)\mathcal{U}'(z)\Big)
\nonumber\\
&= j(jz^2+2h+1)(\mathcal{F}(z)\mathcal{U}(z)-\mathcal{F}(z)\mathcal{U}(z)) =0.
\end{align}
Hence $\mathcal{W}(\mathcal{F}(z),\mathcal{U}(z))$ has to be a constant. Now we evaluate it at $z=0$, we get
\begin{equation*}
\mathcal W(\mathcal F, \mathcal U)=2\mathcal F'(0)\mathcal U(0). 
\end{equation*}
Now 
\begin{equation*}
\mathcal F'(0)=2\sqrt{j}\int_0^\infty e^{-t^2}t^{h+1}dt=\sqrt{j}\Gamma(\frac{h}{2}+1), 
\end{equation*}
and similarly
\begin{equation*}
\mathcal U(0)=\frac{1}{2}\Gamma(\frac{h}{2}+\frac{1}{2}). 
\end{equation*}
Applying Legendre duplication formula 
\begin{equation}
\frac{\Gamma(t)\Gamma(t+\frac{1}{2})}{\Gamma(2t)}=\frac{\sqrt{\pi}}{2^{2t-1}},\ t>0,
\end{equation}
 we have
\begin{equation*}
\mathcal W(\mathcal F, \mathcal U)=2^{-h}\sqrt{j\pi}\Gamma(h+1)>0. 
\end{equation*}
\end{proof}

The regularity of the formal solutions obtained from the above separation of variables requires very precise uniform estimates for the fundamental solutions $\mathcal F$ and $\mathcal U$. We will use the \emph{Laplace Method}, which is inspired by \cite{SS} in a different context. Again we denote $y=\sqrt{j} z$ and define 
\begin{align}
\begin{split}
F(t)&\equiv-t^2+2ty+h\log t
\\ 
U(t)&\equiv-t^2-2ty+h\log t.
\end{split}
\end{align}
Straightforward computations tell us that both $F$ and $U$ are strictly concave when $h\geq 0$. For fixed $y$, let $t_0$ and $s_0$ be the unique (positive) critical points of $F$ and $U$ respectively. It is straightforward that
\begin{align}
\begin{split}
t_0 &= \frac{y}{2} + \sqrt{\frac{h^2}{2}+\frac{y^2}{4}}\\
s_0 &= -\frac{y}{2} + \sqrt{\frac{h^2}{2}+\frac{y^2}{4}}.
\end{split}
\end{align}
  
\begin{lemma}
\label{l:exp-bound}
The following uniform estimates hold for all $j\in\dZ_+$ and $h\geq 0$,
\begin{equation}
\mathcal{F}(z) \leq (1+\sqrt{\pi}) e^{-\frac{jz^2}{2}+F(t_0(z))},
\end{equation}

\begin{equation}
\mathcal{U}(z) \leq (1+\sqrt{\pi}) e^{-\frac{jz^2}{2}+U(s_0(z))}.
\end{equation}

\end{lemma}

\begin{proof}

By the definition of $\mathcal{F}$ and $\mathcal{U}$, it suffices to prove 
\begin{equation}
\int_0^\infty e^{F(t)}dt\leq (1+\sqrt{\pi}) e^{F(t_0)},\end{equation}
and
\begin{equation}
\int_0^\infty e^{U(t)}dt\leq (1+\sqrt{\pi})e^{U(s_0)}.
\end{equation}
We only prove the first inequality and the second can be proved in exactly the same way. In fact, the second can be proved exactly the same way. Denote $a\equiv 2y$. For $\epsilon\in (-1, 1]$ we have \begin{align}
 F(t_0(1+\epsilon)) - F(t_0) 
&= -\epsilon(\epsilon+2)t_0^2 + \epsilon a t_0 + h \log (1 + \epsilon) 
\nonumber\\
&= -\epsilon^2t_0^2+h (\log (1+\epsilon)-\epsilon)\nonumber\\
&\leq  -\epsilon^2 t_0^2 -h(\frac{\epsilon^2}{2}-\frac{\epsilon^3}{3})
\leq -\epsilon^2(t_0^2 + \frac{h}{6}).
\end{align}
The above computations imply that under the transformation $t=t_0(1+\epsilon)$,
\begin{align}
\int_{0}^{2t_0}\exp(F(t))dt 
&\leq \int_{0}^{2t_0}\exp\Big( F(t_0)  - \epsilon^2(t_0^2 + \frac{h}{6})\Big) dt 
\nonumber\\
&= t_0\exp(F(t_0))\int_{-1}^1 \exp\Big(-(t_0^2 + \frac{h}{6})\epsilon^2\Big) d\epsilon
\nonumber\\
&=\frac{t_0}{\sqrt{t_0^2 + \frac{h}{6}}}\exp(F(t_0))\int_{-\sqrt{t_0^2 + \frac{h}{6}}}^{\sqrt{t_0^2 + \frac{h}{6}}} \exp(-\tau^2) d\tau
\nonumber\\
&\leq  \sqrt{\pi}\exp(F(t_0)).
\end{align}
In addition, let $t>2t_0$, then
\begin{align}
F(t) - F(2t_0) \leq F'(2t_0) (t - 2t_0), 
\end{align}
and hence \begin{align}
\int_{2t_0}^{\infty}\exp(F(t))dt \leq \exp(F(2t_0))\int_{2t_0}^{\infty}\exp\Big(F'(2t_0) (t - 2t_0)\Big)dt=\frac{\exp(F(2t_0))}{-F'(2t_0)}.
\end{align}
It can be directly computed that 
\begin{equation}
F'(2t_0) = -4t_0 + a + \frac{h}{2t_0} = -\frac{4t_0^2 + h}{2t_0}<0,
\end{equation}
then
\begin{align}
\int_{2t_0}^{\infty}\exp(F(t))dt  
\leq  \frac{2t_0}{4t_0^2+h}\exp(F(2t_0))
\leq  \frac{2t_0}{4t_0^2+h}\cdot\frac{\exp(F(t_0))}{\exp(t_0^2 + \frac{h}{6})}
\leq \exp(F(t_0)).
\end{align}
Combining the above calculations, 
\begin{equation}
\int_{0}^{\infty}\exp(F(t))dt \leq (1 + \sqrt{\pi}) \exp(F(t_0)).
\end{equation}
\end{proof}

Apply the same method as in Lemma \ref{l:exp-bound}, we have the following asymptotic property of $\mathcal{F}$ and $\mathcal{U}$.
\begin{lemma}\label{l:asymp-F-U} For fixed $j\in\dZ_+$ and $h> 0$, we have the following asymptotic formula

 \begin{equation}\label{Fasymptotics}
\lim\limits_{z \to + \infty }\frac{\mathcal{F} (z)}{\sqrt{\pi}e^{\frac{jz^2}{2}} (\sqrt{j} z)^{h}} = 1
 \end{equation}
and
\begin{equation}\label{Uasymptotics}
\lim\limits_{z\to+\infty}\frac{\mathcal{U}(z)}{e^{-\frac{jz^2}{2}}\Gamma(h+1) (2\sqrt{j}z)^{-h-1}} =1.
\end{equation}
 \end{lemma}

\begin{proof}
For simplicity, we will calculate the asymptotic behavior in $y$.
For fixed $h$ and $j$, as $y\rightarrow \infty$, it is straightforward that 
\begin{equation}
\begin{split}
t_0&=y+ \frac{h}{2y}+O(y^{-2})\\
s_0&=\frac{h}{2y}+O(y^{-2})\end{split}\label{e:t-s-asymp}
\end{equation}
which implies that \begin{equation}
\begin{split}
F(t_0)&=y^2+h\log y+O(y^{-1})
\\
 U(s_0)&=-h+h\log h-h\log (2y)+O(y^{-1}).\end{split}\label{e:F-U-asymp}
\end{equation}

First, we prove the asymptotics for $\mathcal{F}$.
As in the proof of Lemma \ref{l:exp-bound}, we get 
\begin{align}
\int_{0}^\infty e^{F(t)}dt
&=
e^{F(t_0)}\int_0^{2t_0} e^{F(t)-F(t_0)}dt + \int_{2t_0}^{\infty}e^{F(t)}dt
\nonumber\\
&=
t_0e^{F(t_0)}\int_{-1}^1 e^{-t_0^2\epsilon^2+h(\log(1+\epsilon)-\epsilon)}d\epsilon
+ \int_{2t_0}^{\infty}e^{F(t)}dt.
\end{align}
Notice that
\begin{equation}
\lim_{t_0\rightarrow \infty}t_0\int_{-1}^{1}e^{-t_0^2\epsilon^2+h(\log(1+\epsilon)-\epsilon)}d\epsilon=\sqrt{\pi}, 
\end{equation}
and 
\begin{equation}
\lim_{t_0\rightarrow \infty}e^{-F(t_0)}\int_{2t_0}^\infty e^{F(t)}dt=0.  
\end{equation}
Moreover, by \eqref{e:t-s-asymp}, $\lim\limits_{y\to+\infty}t_0(y)\to \infty$.
It follows that 
\begin{equation}
\lim\limits_{z\to\infty}\frac{\mathcal{F}(z)}{\sqrt{\pi}e^{-\frac{jz^2}{2}+F(t_0(z))}} =1. \end{equation}
Combining the above limit and \eqref{e:F-U-asymp}, the proof of \eqref{Fasymptotics} is complete. 

In the case $j\in\dZ_+$ and $h>0$,
we will prove the asymptotic behavior of $\mathcal{U}$ and we write
\begin{align}
\int_0^\infty e^{U(t)}dt
&=
s_0e^{U(s_0)}\int_{-1}^\infty e^{-s_0^2\epsilon^2+h(\log(1+\epsilon)-\epsilon)}d\epsilon\nonumber\\
&= s_0e^{U(s_0)}\int_{-1}^{\infty}e^{-s_0^2\epsilon^2} \cdot (1+\epsilon)^h e^{-h\epsilon}d\epsilon.
\end{align}
We claim that
\begin{equation}
\lim\limits_{s_0\to0}\int_{-1}^{\infty}(e^{-s_0^2\epsilon^2}-1) \cdot (1+\epsilon)^h e^{-h\epsilon}d\epsilon = 0.
\end{equation}
In fact, it is straightforward that for any $s_0> 0$, 
\begin{equation}
-1\leq e^{-s_0^2\epsilon^2}-1\leq 0
\end{equation}
and for any fixed $h>0$, 
\begin{equation}
\int_{-1}^{\infty} (1+\epsilon)^h e^{-h\epsilon}d\epsilon <\infty.
\end{equation}
Applying the dominated convergence theorem,
\begin{equation}
\lim\limits_{s_0\to0}\int_{-1}^{\infty}(e^{-s_0^2\epsilon^2}-1) \cdot (1+\epsilon)^h e^{-h\epsilon}d\epsilon = 0.
\end{equation}
This completes the proof the the claim.
Next, by the definition of the gamma function,
\begin{equation}
\int_{-1}^\infty (1+\epsilon)^h e^{-h\epsilon}d\epsilon=e^h\int_0^\infty e^{-hs}s^hds=e^hh^{-h-1}\Gamma(h+1).
\end{equation}
Therefore,
\begin{equation}
\lim\limits_{s_0\to0}\int_{-1}^{\infty}e^{-s_0^2\epsilon^2+h(\log(1+\epsilon)-\epsilon)}d\epsilon=e^h h^{-h-1}\Gamma(h+1).
\end{equation}
Since $\lim\limits_{y\to+\infty}s_0=0$ and $U$ yields to the asymptotic property \eqref{e:F-U-asymp}, eventually we obtain
\eqref{Uasymptotics}. 
\end{proof}

\begin{lemma}\label{l:quotient-bound} There is an absolute constant $C_0>0$ independent of $j\in\dZ_+$ and $h\geq 0$ such that the following uniform estimate holds for all $z\geq 1$, 
\begin{equation}0<\frac{e^{\widehat{F}(z)+\widehat{U}(z)}}{\mathcal{W}(z)}\leq C_0,\label{e:uniform-quotient}\end{equation}
where \begin{equation}\widehat{F}(z)=-\frac{jz^2}{2}+F(t_0(z))\end{equation}
and \begin{equation}\widehat{U}(z)=-\frac{jz^2}{2}+U(s_0(z)).\end{equation}

\end{lemma}

\begin{proof}

The proof is based on Lemma \ref{l:exp-bound}. 
First, we discuss the case $h=0$ and $j\in\dZ_+$. Direct computations give that $t_0=y$ and $s_0 =0$, then by definition we have that $\widehat{F}(z)=\frac{jz^2}{2}$ and $\widehat{U}(z)=-\frac{jz^2}{2}$. Therefore, \eqref{e:uniform-quotient} immediately follows.

Next, we prove the case $j\in\dZ_+$ and $h>0$.
We notice that
\begin{align}
t_0 - s_0 &= \frac{a}{2} \\
t_0^2 + s_0^2 &= \frac{a^2}{4} + h \\
t_0 s_0 &= \frac{h}{2},
\end{align}
by elementary calculations,
\begin{align}
F(t_0) + U(s_0) 
&= -(t_0^2 + s_0^2) + a(t_0-s_0) + h\log(t_0s_0)  \notag \\
&= \frac{a^2}{4} - h + h\log(\frac{h}{2}) = jz^2 - h + h\log(\frac{h}{2}).
\end{align}
Immediately we have that
\begin{equation}
e^{\widehat{F}(z) + \widehat{U}(z)} \leq e^{-h+h\log (\frac{h}{2})}. \label{e:bound-sum}
\end{equation}
Combining \eqref{e:bound-sum} and Lemma \ref{l:wronskian},
\begin{align}
\frac{e^{\widehat{F}(z)+\widehat{U}(z)}}{\mathcal{W}(z)} 
\leq  
\frac{e^{-h+h\log (\frac{h}{2})}}{\sqrt{j\pi}2^{-h}\Gamma(h+1)}
\leq C_0.
\end{align}
This proves the lemma. 
\end{proof}

A key technical point of this section is to construct a well-behaved solution of the Poisson equation
\begin{equation}
\Delta_{\mathcal{C}} u = v
\end{equation}
by applying separation of variables and the uniform estimate on the ODE solutions.
For this purpose, we need the following monotonicity.

\begin{lemma}
\label{l:monotonicity}
Let 
$\widehat{F}(z)$ and $\widehat{U}(z)$ be the function
defined in Lemma \ref{l:quotient-bound}, then 
$\widehat{F}(z) - \eta z$ is increasing for $z>2\eta$ and 
$\widehat{U}(z)+\eta z$ is decreasing for $z>2\eta$.
\end{lemma}

\begin{proof}

Let $y=\sqrt{j}z$ and $a = 2y$, then by definition,
\begin{equation}
\widehat{F} = -\frac{a^2}{8} - (t_0(a))^2 + at_0(a) + h\log (t_0(a))
\end{equation}
and
\begin{equation}
\widehat{U} = -\frac{a^2}{8} - (s_0(a))^2 - as_0(a) + h\log (s_0(a)).
\end{equation}
We show that $\widehat{F}$ is increasing in $a$ and $\widehat{U}$ is decreasing in $a$.
Indeed,
\begin{align}
\frac{d\widehat{F}}{da} = -\frac{a}{4} + t_0(a) + \Big(-2t_0(a) + a + \frac{h}{t_0(a)} \Big) t_0'(a) = \sqrt{\frac{h}{2}+\frac{a^2}{16}} \geq \frac{a}{4}. 
\end{align}
So the monotonicity of $\widehat{F}(z)-\eta z$ immediately follows when $z>2\eta$.
Similarly, the monotonicity of $\widehat{U}+\eta z$ follows from the computation
\begin{align}
\frac{d\widehat{U}}{da} = -\frac{a}{4} - s_0(a) + \Big(-2s_0(a) - a + \frac{h}{s_0(a)}\Big) s_0'(a)
= -\sqrt{\frac{h}{2}+\frac{a^2}{16}} \leq -\frac{a}{4}.
\end{align}
\end{proof}

\subsection{Asymptotics of harmonic functions}

Let  $(\mathcal{C},g_{\mathcal{C}})$ be a Calabi model space with a cross section $Y^3$. We will show that any harmonic function $u$ with a slow exponential growth must be linear.
\begin{proposition}[Harmonic functions with slow exponential growth]\label{p:harmonic-function-decay} 
Suppose $u$ is harmonic on the Calabi model space $(\mathcal{C},g_{\mathcal{C}})$, i.e.,
\begin{equation}
\Delta_{g_{\mathcal{C}}}u=0,
\end{equation}
 If $u=O(e^{\delta z})$ for some $\delta\in(0, \underline{\delta})$, where 
$\underline{\delta}>0$ depends only on the Calabi model space $(\mathcal C, g_{\mathcal{C}})$.
then there are constants $a_0,b_0\in\dR$ such that 
\begin{equation}
u(z)=a_0z+b_0+ O(e^{-\underline{\delta} z}).
\end{equation}

\end{proposition}

\begin{proof}

To start with, we choose a closed Sakaki manifold $Y^3$
which is  given by the level set $\{\varrho=r_0\}$ in the Calabi space $\mathcal{C}$.
Denote by $\{\Lambda_k\}_{k=1}^{\infty}$ be the spectrum of $(Y^3,h_0)$, where $h_0$ is the induced 
Riemannian metric from $g_{\mathcal{C}}$.
Let $\varphi_k\in C^{\infty}(Y^3)$ be the eigenfunctions satisfying
\begin{equation}
\begin{split}-\Delta_{h_0}\varphi_k&=\Lambda_k \varphi_k\\
\mathcal{L}_{\partial_{\theta}}\varphi_k&=\sqrt{-1} j_k \varphi_k,\ j_k\in \dN.
\end{split}
\end{equation}
As computed in Section \ref{ss:separation-of-variables}, separation of variables gives the following expansion,
 \begin{equation}
 u(z,\bm{y})=\sum\limits_{k=1}^{\infty}u_{k}(z)\cdot\varphi_{k}(\bm{y})\label{e:u-phi}
 \end{equation}
where $\bm{y}\in Y^3$, and $u_{k}$ satisfies the equation 
\begin{equation}
\frac{d^2u_k(z)}{dz^2} - (j_k^2 z^2 + \lambda_k) u_k(z) = 0,\ z\geq 1,
\end{equation}
for some $j_k\in\dN$
and $\lambda_k\geq1$. Note that, in Section \eqref{ss:separation-of-variables}, we have shown the relations 
\begin{equation}
\Lambda_k=\frac{\lambda_k}{z_0}+\frac{2z_0\cdot j_k^2}{r_0^2}
\end{equation}
and   $\lambda_k \geq j_k$, where $z_0\equiv(-\log r_0^2)^{\frac{1}{2}}$. 
So all the estimates obtained in the previous sections directly apply here.

Since the harmonic function $u$ is smooth, so the convergence \eqref{e:u-phi} is in the $C^\infty$ topology in any compact subset of $\mathcal{C}$.  
Immediately, for $k\in\dZ_+$ and for some fixed $z_0>1$,
\begin{align}
|u_{k}(z_0)| & = \Big|\int_{Y^3} u\cdot\varphi_{k} \dvol_{h_0}\Big| = \Big| \int_{Y^3} u\cdot \frac{(-\Delta_{h_0})^{K_0}\varphi_{k}}{(\Lambda_{k})^{K_0}}\dvol_{h_0}\Big| \nonumber\\
& \leq \frac{1}{(\Lambda_{k})^{K_0}} \Big|\int_{Y^3}|(-\Delta_{h_0})^{K_0}u| \cdot \varphi_{k}\dvol_{h_0}\Big| \leq \frac{C_1}{(\Lambda_{k})^{K_0}}, \label{e:u_{k}-uniform-bound}
\end{align}
where 
\begin{equation}C_1\equiv\|u\|_{C^{2K_0}(Y^3\times\{z_0\})}\cdot(\Vol_{h_0}(Y^3))^{1/2}.
\end{equation}

Before discussing the asymptotic behavior of the harmonic function $u$, let us give a more precise expression for each ODE solution $u_k$ under the growth condition $u=O(e^{\delta z})$ for $0<\delta<\underline{\delta}$.
First, for every $k\in\dZ_+$,
there exist constants $C_k$ and $C_k^*$ such that\begin{equation}
u_k(z)= C_k \cdot \mathcal{U}_k(z) + C_k^* \cdot  \mathcal{F}_k(z).
\end{equation}
The growth condition on $u$ gives the growth of $u_k$. Indeed, by assumption for any sufficiently large $z\in(2 z_0,+\infty)$ with $z_0>10^6$, it holds that \begin{equation}|u(z)|\leq C_0\cdot e^{\delta z},\label{e:u-asym-condition}\end{equation}
which implies that 
\begin{equation}
|u_k(z)|=\Big|\int_{Y^3}u\cdot\varphi_{k}\Big|\leq C_0\cdot(\Vol_{h_0}(Y^3))^{1/2}e^{\delta_0z}.\label{e:u_k-growth}
\end{equation}

There are two cases to analyze:

First, we consider the case  $j_k =0$, then $u_k$ satisfies the linear equation
\begin{equation}
u_k''(z) - \lambda_k \cdot u_k(z) = 0.\end{equation}
We only consider $\lambda_k>0$. Otherwise, the solution is just a linear function. 
In this case, we pick the fundamental solutions 
\begin{equation}
\mathcal{F}_k(z) \equiv e^{\sqrt{\lambda_k} \cdot z} \ \text{and}\ \mathcal{U}_k(z) \equiv e^{-\sqrt{\lambda_k} \cdot z},\ 
\end{equation}
We define
\begin{equation}
\underline{\delta} \equiv \min\Big\{\sqrt{\lambda_k} \Big| k \in \dZ_+\Big\}> 0. \label{e:underline-delta}
\end{equation}
If we choose $\delta\in(0,\underline{\delta})$,  then \eqref{e:u_k-growth} implies
\begin{equation}C_k^*=0\end{equation} for each $k\in\dZ_+$ which satisfies $j_k=0$, and hence
\begin{equation}
u_k(z)=C_ke^{-\sqrt{\lambda_k} \cdot z}.\label{e:exp-neg}
\end{equation}

Next, we consider the case $k\in\dZ_+$ such that $j_k\neq 0$. 
 Lemma \ref{l:asymp-F-U} implies that $\mathcal{F}_k$ is growing and  $\mathcal{U}_k$ is decaying. 
Therefore, apply \eqref{e:u_k-growth} again, we have \begin{equation}
 C_k^* = 0
 \end{equation}
for every $k\in\dZ_+$ which satisfies $j_k \neq 0$.

Combining the above two cases, we conclude that if $\delta\in(0,\underline{\delta})$, then for every $k\in\dZ_+$, there exists some constant $C_k\in\dR$ such that
\begin{equation}
u_k(z) = C_k \cdot \mathcal{U}_k(z)
\end{equation}
and hence
\begin{equation}
u(z,\bm{y}) = \sum\limits_{k=1}^{\infty} C_k \cdot \mathcal{U}_k(z) \cdot \varphi_k(\bm{y}).
\end{equation}
By definition, in our context $\mathcal{U}_{m,h}(z)>0$,  so there  is no harm to assume $u_{k}(z_0)\neq 0$.

Now we are in a position to estimate the upper bound of the harmonic function $u$ which satisfies $u=O(e^{\delta z})$ with $0<\delta<\underline{\delta}$.
We still separate in two cases. 
First, we consider $k\in\dZ_+$ with $j_k=0$.
For fixed $z_0>10^6$, we apply \eqref{e:exp-neg}, then for every sufficiently large  $z\in(2 z_0,+\infty)$,
\begin{equation}
\Big|\frac{u_k(z)}{u_k(z_0)}\Big|=\Big|\frac{\mathcal{U}_k(z)}{\mathcal{U}_k(z_0)}\Big|=e^{-\sqrt{\lambda_k}\cdot (z-z_0)}\leq e^{-\underline{\delta}(z-z_0)},
\end{equation}
and hence
\begin{align}
\Big|\sum_{\substack{k>0\\ j_k\geq 1}} u_k(z) \cdot \varphi_k(\bm{y}) \Big|
&=
\Big|\sum_{\substack{k>0\\ j_k=0}}\frac{u_k(z)}{u_k(z_0)}\cdot u_k(z_0)\cdot \varphi_{k}(\bm{y})\Big|
\nonumber\\
&\leq 
\sum_{\substack{k>0\\ j_k=0}}\Big|\frac{u_k(z)}{u_k(z_0)}\Big| \cdot |u_k(z_0) |\cdot |\varphi_{k}(\bm{y})| \leq C
e^{-\underline{\delta}z/2}\cdot\sum_{\substack{k>0\\ j_k=0}} \frac{1}{(\Lambda_{k})^{K_0-1}}.
\end{align}
Next, let $k\in\dZ_+$ satisfy $j_k \geq 1$.
For fixed $z_0>10^6$ and take $z\in(2 z_0,\infty)$, 
 then we have
\begin{equation}
\Big|\frac{u_{k}(z)}{u_{k}(z_0)}\Big|=\Big|\frac{\mathcal{U}_{k}(z)}{\mathcal{U}_{k}(z_0)}\Big|\leq e^{-\frac{j_k(z^2-z_0^2)}{2}}\leq Ce^{-\frac{3z^2}{8}}.
\end{equation}
Taking the sum, 
\begin{align}
\Big|\sum_{\substack{k>0\\ j_k\geq 1}} u_k(z) \cdot \varphi_k(\bm{y}) \Big|
&=
\Big|\sum_{\substack{k>0\\ j_k\geq 1}}\frac{u_k(z)}{u_k(z_0)}\cdot u_k(z_0)\cdot \varphi_{k}(\bm{y})\Big| \nonumber\\
& \leq  
\sum_{\substack{k>0\\ j_k\geq 1}}\Big|\frac{u_k(z)}{u_k(z_0)}\Big| \cdot |u_k(z_0)| \cdot |\varphi_{k}(\bm{y}) | \leq Ce^{-\frac{3 z^2}{8}}\sum_{\substack{k>0\\ j_k\geq 1}} \frac{1}{(\Lambda_k)^{K_0-1}}.
\end{align}
The estimates in the above two cases imply that
\begin{align}
|u(z,\bm{y})| 
&= 
\Big|\sum_{k=1}^{\infty} u_k(z) \cdot \varphi_k(\bm{y}) \Big| 
\nonumber\\
&\leq 
C\Big(
e^{-\underline{\delta}z/2}\sum_{\substack{k>0\\ j_k=0}} \frac{1}{(\Lambda_{k})^{K_0-1}} + 
e^{-\frac{3 z^2}{8}}\sum_{\substack{k>0\\ j_k\geq 1}} \frac{1}{(\Lambda_k)^{K_0-1}}\Big)
\leq  Ce^{-\underline{\delta}z/2}\sum_{k=1}^{\infty} \frac{1}{(\Lambda_{k})^{K_0-1}}.
\end{align}
We can choose $K_0\geq 3$, applying Weyl's law, then the above numerical series converges and
hence
\begin{equation}
|u(z,\bm{y})| \leq C.
\end{equation}
Therefore,  there exists sufficiently large $N_0>10^6$ such that for all $z\in(N_0,\infty)$ 
\begin{equation}
|u(z)-(a_0z+b_0)|\leq Ce^{-2\pi z},
\end{equation}
where $C$ depends on $\Vol_{h_0}(Y^3)$ and $\|u\|_{C^{K_0}(Y^3\times \{z_0\})}$ for some fixed $z_0>10^6$ and $K_0\geq 3$.

\end{proof}

Let $u$ be a harmonic function on the Calabi model space 
$(\mathcal{C},g_{\mathcal{C}})$, then there is an expansion of $u$,
\begin{align}
u(z,\bm{y})= \sum\limits_{k=1}^{\infty} \Big(C_k \mathcal{U}_k(z) + C_k^* \mathcal{F}_k(z)\Big)\varphi_k(\bm{y}).
\end{align}
Combining those growing and decaying components, we have the following decomposition of $u$,
\begin{equation}
u=u_{\mathfrak{p}} + u_{\mathfrak{n}} + (az+b),
\end{equation}
where
\begin{equation}
u_{\mathfrak{p}}(z,\bm{y})\equiv \sum_{\substack{k\geq 1\\ (\lambda_k,j_k )\neq (0,0)}} C_k^* \cdot \mathcal{F}_k(z) \varphi_k(\bm{y})
\end{equation}
and
\begin{equation}
u_{\mathfrak{n}}(z,\bm{y})\equiv \sum_{\substack{k\geq 1\\ (\lambda_k,j_k )\neq (0,0)}} C_k\cdot  \mathcal{U}_k(z) \varphi_k(\bm{y}).
\end{equation}

\begin{lemma}\label{l:decay-harmonic}
Let $u$ satisfy $\Delta_{h_0}u=0$ and assume $u=u_{\mathfrak{n}}$, then  $u_{\mathfrak{n}}$ is a harmonic function with $u=O(e^{-\delta z})$ for some $\delta >0$.
\end{lemma}

\begin{proof} The proof of the lemma is identical to the convergence arguments in Proposition \ref{p:harmonic-function-decay}. 

\end{proof}

\subsection{Regularity and asymptotics for Poisson equation}

With the above lemmas, the following estimate for the solutions of the 
non-homogeneous equation immediately follows.

\begin{lemma}
\label{l:error-estimate}

Let $(\mathcal{C},g_{\mathcal{C}})$ be the model space with a fixed fiber $(Y^3,h_0)$. Let $K_0\geq 1$ and let      
 $\xi \in C^{2K_0}(\mathcal{C})$ satisfy the expansion
 \begin{equation}
 \xi(z, \bm{y}) = \sum\limits_{k=1}^{\infty} \xi_k (z) \cdot \varphi_k (\bm{y}).
 \end{equation}
In addition, assume that there is some $\eta_0  \neq  0$ such that for every $0\leq m\leq 2K_0$,
\begin{equation}|\nabla^m \xi(z,\bm{y})|=O(e^{\eta_0 z}),\end{equation}
then for every  $z \geq 1$ and $k\in\dZ_+$,\begin{equation}|\xi_{k}(z)|\leq  \frac{ Ce^{\eta_0 z}}{(\Lambda_{k})^{K_0}},\end{equation}
where the constant $C>0$ is independent of $k$ and $z$.
\end{lemma}

\begin{proof}
 The estimate will be proved by the standard integration by parts. Since the eigenfunctions $\varphi_k$ satisfy 
 \begin{equation}
 -\Delta_{h_0}\varphi_{k}=\Lambda_{k}\varphi_{k}
 \end{equation} and $\|\varphi_k\|_{L^2(Y^3)}=1$, we have that
\begin{align} 
\Big|\xi_{k}(z)\Big|&=\Big|\int_{Y^3}\xi\cdot\varphi_k\Big|
=\Big|\int_{Y^3}\xi\cdot\frac{(-\Delta_{h_0})^{K_0}\varphi_k}{(\Lambda_k)^{K_0}}\Big|\nonumber\\
&\leq \frac{1}{(\Lambda_k)^{K_0}}\int_{Y^3}|\Delta_{h_0}^{K_0}\xi|\cdot|\varphi_k| \leq  \frac{Q_{2K_0}\cdot \Vol_{h_0}(Y^3)^{1/2} \cdot e^{\eta_0 z}}{(\Lambda_k)^{K_0}},
\end{align}
where $Q_{2K_0}$ depends only on the asymptotic bound of $\nabla^{2K_0}\xi$.
The proof is done.

\end{proof}

\begin{lemma}\label{l:coefficients-estimate}
Consider the inhomogeneous ordinary differential equation
\begin{equation}
\frac{d^2u_k(z)}{dz^2} - (j_k^2 z^2 + \lambda_k) u_k(z) = \xi_k(z)\cdot z,\ z\geq 10^6,
\end{equation}
where $\underline{\delta}>0$ is the constant defined in \eqref{e:underline-delta}.
Assume that the function $\xi_k(z)$ satisfies the following property:
 there are constants \begin{equation}\eta_0\in(-\underline{\delta}/2,\underline{\delta}/2)\setminus\{0\}\end{equation} and $Q_k>0$ such that 
\begin{equation}
|\xi_{k}(z)| \leq Q_k \cdot  e^{\eta_0 z}.
\end{equation}
Let $u_{k}(z)$ be the particular solution defined by
\begin{equation}
u_{k}(z)\equiv\frac{\mathcal{G}_k(z)+\mathcal{D}_k(z) }{\mathcal{W}_{k}(z)},\end{equation}
where 
\begin{equation}
\mathcal{D}_k(z) \equiv \mathcal{F}_k(z)\int_z^{\infty}\mathcal{U}_k(r)\cdot\Big(\xi_{k}(r)\cdot r\Big)dr,
\end{equation}

\begin{equation}
\mathcal{G}_k(z) \equiv \mathcal{U}_k(z)\int_1^{z}\mathcal{F}_k(r)\cdot\Big(\xi_{k}(r)\cdot r\Big)dr
\end{equation}
and $\mathcal{W}_k$ is the Wronskian
\begin{equation}
\mathcal{W}_k(z) \equiv \mathcal{W}\Big(\mathcal{F}_k(z),\mathcal{U}_k(z)\Big).\end{equation}
Then there are constants $C_0>0$ and $\eta_0<\eta<\eta_0 +  \underline{\delta}/10$ which are independent of $k$ such that the particular solution $u_k$ satisfies the uniform estimate 
\begin{equation}
|u_{k}(z)| \leq C_0 \cdot Q_k\cdot e^{\eta z}.
\end{equation}
\end{lemma}

\begin{proof}

We will prove that there exists some constant $\eta_0<\eta<\eta_0 + \underline{\delta}/10$ such that
\begin{equation}
\frac{\mathcal{D}_k(z)}{\mathcal{W}_{k}(z)}\leq C_0\cdot Q_k\cdot  e^{\eta z} \label{e:bdd-decaying-part}
\end{equation}
and 
\begin{equation}
\frac{\mathcal{G}_k(z)}{\mathcal{W}_{k}(z)}\leq C_0 \cdot Q_k \cdot e^{\eta z}, \label{e:bdd-growing-part}
\end{equation}
where the positive constant $C_0>0$ is independent of the index $k$.

We prove \eqref{e:bdd-decaying-part} and \eqref{e:bdd-growing-part} in two different cases.

In the first case, $k\in\dZ_+$ satisfies $j_k=0$. The fundamental solutions have an explicit form
\begin{equation}
\mathcal{F}_k(z) \equiv e^{\sqrt{\lambda_k} \cdot z}\end{equation} 
and
\begin{equation} 
\mathcal{U}_k(z) \equiv e^{-\sqrt{\lambda_k} \cdot z}. 
\end{equation}
 Immediately,
 \begin{equation}
 \mathcal{W}_k(z)=\mathcal{W}(\mathcal{F}_k(z),\mathcal{U}_k(z))=2\sqrt{\lambda_k}
 \end{equation}
and hence for $\eta>\eta_0$,
\begin{align}
\frac{|\mathcal{D}_k(z)|}{|\mathcal{W}_k(z)|}
= \frac{\mathcal{F}_k(z)}{\mathcal{W}_k(z)} \int_{z}^{\infty} \mathcal{U}_k(r) |\xi_k(r)\cdot r|dr \leq \frac{Q_ke^{\sqrt{\lambda_k} \cdot z}}{\sqrt{\lambda_k}}\int_z^{\infty}e^{(-\sqrt{\lambda_k} + \eta) \cdot r}dr \leq   C_0\cdot Q_ke^{\eta z}.\end{align}

Similarly, 
\begin{align}
\frac{|\mathcal{G}_k(z)|}{|\mathcal{W}_k(z)|}
= \frac{\mathcal{U}_k(z)}{\mathcal{W}_k(z)} \int_{z_0}^{z} \mathcal{F}_k(r) |\xi_k(r)\cdot r|dr
\leq \frac{Q_k e^{-\sqrt{\lambda_k} \cdot z}}{\sqrt{\lambda_k}}\int_{z_0}^{z}e^{(\sqrt{\lambda_k} + \eta) \cdot r}dr
\leq  C_0\cdot Q_ke^{\eta z}.
\end{align}

In the latter case $j_k\in\dZ_+$ and $k\in\dZ_+$, we will prove the uniform estimates.
A crucial point is to apply the monotonicity in Lemma \ref{l:monotonicity}. In fact, 
 \begin{align}
\frac{\mathcal{D}_k(z)}{\mathcal{W}_{k}(z)}
&= \frac{\mathcal{F}_k(z)}{\mathcal{W}_k(z)} \int_{z}^{\infty} \mathcal{U}_k(r) \xi_k(r)\cdot rdr \nonumber \\
&\leq\frac{C_0e^{\widehat{F}_k(z)}}{\mathcal{W}_k(z)}\int_z^{\infty}e^{\widehat{U}_k(r)} \xi_k(r)\cdot r dr
\leq \frac{C_0e^{\widehat{F}_k(z)}}{\mathcal{W}_k(z)}\int_z^{\infty}e^{\widehat{U}_k(r)+\eta' r} dr,
\end{align}
where $\eta'>\eta_0$. We choose  $\epsilon\in( \underline{\delta}/100, \underline{\delta}/10)$ and denote $\eta \equiv \eta' + \epsilon$, then
by Lemma  \ref{l:monotonicity}
\begin{align}
\frac{e^{\widehat{F}_k(z)}}{\mathcal{W}_k(z)}\int_z^{\infty}e^{\widehat{U}_k(r)+\eta' r} dr
&= \frac{e^{\widehat{F}_k(z)}}{\mathcal{W}_k(z)}\int_z^{\infty}e^{\widehat{U}_k(r)+\eta r} \cdot e^{-\epsilon r}dr
\nonumber\\
& \leq \frac{C_0\cdot Q_k\cdot e^{\widehat{F}_k(z)+\widehat{U}_k(z)+\eta z}}{  \mathcal{W}_k(z)}\int_z^{\infty}e^{-\epsilon r}dr \nonumber \\
& \leq  C_0\cdot Q_k\cdot\frac{e^{\widehat{F}_k(z)+\widehat{U}_k(z)+\eta z}}{\mathcal{W}_k(z)}
\leq C_0\cdot Q_k\cdot e^{\eta z}.
\end{align}
The proof of \eqref{e:bdd-decaying-part} is done.

Next, for the estimate \eqref{e:bdd-growing-part}, 
\begin{align}
\frac{\mathcal{G}_k(z)}{\mathcal{W}_k(z)}
&=
\frac{\mathcal{U}_k(z)}{\mathcal{W}_k(z)} \int_{1}^{z} \mathcal{F}_k(r) \xi_k(r)\cdot rdr \nonumber \\
&\leq  \frac{C_0e^{\widehat{U}_k(z)}}{ \mathcal{W}_k(z)} \int_1^z e^{\widehat{F}_k(r) + \eta' r} dr 
\leq \frac{C_0\cdot Q_k z\cdot e^{\widehat{U}_k(z) + \widehat{F}_k(z) + \eta' z} }{\mathcal{W}_k(z)}
\leq  C_0\cdot Q_k \cdot e^{\eta z}. \label{e:G/W}
\end{align}
This completes the proof of the proposition.

\end{proof}

\begin{lemma}
[Uniform estimate for eigenfunctions]\label{l:eigenfunction-bound} Let $\{\varphi_{k}\}_{k=1}^{\infty}$  be the eigenfunctions of $\Delta_{h_0}$ on $(Y^3,h_0)$ with $\|\varphi_{k}\|_{L^2(Y^3)}=1$, then there exists $C>0$ which depends only on the metric $h_0$ such that
\begin{equation}
\|\varphi_k\|_{L^{\infty}(Y^3)}\leq C\cdot\Lambda_k.
\end{equation}
\end{lemma}

\begin{proof}
 The proof follows from the standard elliptic regularity. Indeed, the eigenfunction $\varphi_k$ satisfies the elliptic equation 
\begin{equation}-\Delta_{h_0}\varphi_{k}=\Lambda_k\cdot\varphi_k.\end{equation} 
 It follows from the standard  elliptic regularity that there exists some constant $C>0$ depending only the metric $h_0$ such that
\begin{equation}
\|\varphi_k\|_{W^{2,2}(Y^3)}\leq C\cdot\Lambda_k\cdot \|\varphi_k\|_{L^2(Y^3)}=C\cdot\Lambda_k.
\end{equation}
Applying the Sobolev embedding theorem, 
\begin{equation}
\|\varphi_k\|_{C^{0,\frac{1}{2}}(Y^3)} \leq C\|\varphi_k\|_{W^{2,2}(Y^3)}\leq C\cdot\Lambda_k,
\end{equation}
where $C>0$ depends only on the metric $h_0$. The proof is complete.
\end{proof}

\begin{proposition}[Sovability of Poisson Equation]\label{p:surjectivity}  Let $(\mathcal{C}, g_{\mathcal{C}})$ be the Calabi space, there is some constant $\underline\delta>0$ which depends only on $\mathcal{C}$ such that the following property holds: given any 
 \begin{equation}\eta_0 \in(-\underline\delta,\underline\delta)\setminus
 \{0\},\end{equation} 
 if $v\in C^{3K_0,\alpha}(\mathcal{C})$ for $K_0\geq 3$
and $v(z,\bm{y})=O(e^{\eta_0 z})$, then the equation
\begin{equation}
\Delta_{g_0}u=v\label{e:possion-eq}
\end{equation}
has a solution $u\in C^{3K_0+2,\alpha}(\mathcal{C})$ with 
\begin{equation}
u(z,\bm{y})=O(e^{\eta z})
\end{equation}
for any $\eta>\eta_0$.

\end{proposition}

\begin{proof}
The proof of the proposition is constructive. 
The basic strategy is to apply separation of variables to construct a solution to the equation \eqref{e:possion-eq}.
Given a function $v$ and for any fixed $z\geq 1$, there is an expansion over the fiber $Y^3$,
\begin{equation}
v(z,\bm{y}) = \sum\limits_{k=1}^{\infty} v_k(z) \varphi_k(\bm{y}).
\end{equation}
  Separation of variables enables us to construct 
a formal solution 
\begin{equation}
u(z,\bm{y})=\sum\limits_{k=1}^{\infty}u_{k}(z)\varphi_k(\bm{y})
\end{equation}
to the equation \eqref{e:possion-eq},
where $u_{k}$ are the particular solutions in Lemma \ref{l:coefficients-estimate}. 
Since a priori the above series is defined in the $L^2$-topology along each fiber $Y^3\times\{z\}$, we need to verify the higher order convergence of the series, which will indicate that $u$ is a regular solution to \eqref{e:possion-eq}.

First, we will show that the above series converges in the $C^0$-topology and thus $u$ is a $C^0$-function. The main point is to reduce the uniform convergence to the convergence of certain numerical series involving only in the eigenvalues $\{\Lambda_k\}_{k=1}^{\infty}$ of a definite fiber $(Y^3,h_0)$. Indeed, Lemma \ref{l:error-estimate} guarantees that the solutions satisfy all the conditions in Lemma \ref{l:coefficients-estimate}. Since we have obtained in Lemma \ref{l:coefficients-estimate} the uniform estimate for the ODE solutions $u_k$
and also in Lemma \ref{l:eigenfunction-bound}
the uniform estimate for the eigenfunctions, the $L^2$-expansion has the following bound,
\begin{align}
|u(z,\bm{y})| 
\leq  \sum\limits_{k=1}^{\infty}|u_{k}(z)| \cdot |\varphi_k(\bm{y})| 
\leq C\sum\limits_{k=1}^{\infty}\frac{e^{\eta z}}{(\Lambda_k)^{K_0-1}}. 
\label{e:numerical-series}
\end{align}
Since the spectrum of Laplacian $\{\Lambda_k\}_{k=1}^{\infty}$ obeys Weyl's 
law on $(Y^3,h_0)$, it follows that for sufficiently large $k$,
\begin{equation}
C_0^{-1} k^{\frac{2}{3}}\leq|\Lambda_k| \leq C_0 k^{\frac{2}{3}},
\end{equation}
where $C_0>0$ depends only on  $h_0$.
Plugging the above asymptotics into \eqref{e:numerical-series}, we have that
\begin{equation}
\sum\limits_{k=1}^{\infty}\frac{1}{(\Lambda_k)^{K_0-1}} \leq C\sum\limits_{k=1}^{\infty} \frac{1}{k^{\frac{4}{3}}} < \infty 
\end{equation}
and hence
\begin{equation}
|u(z,\bm{y})| \leq Ce^{\eta z}.
\end{equation}
Therefore, 
$u\in C^0(\mathcal{C})$ and $u$ exponentially decays. 

Next, we will apply the standard elliptic regularity on the Calabi manifold
to show that $u\in C^2$ and thus $u$ is a regular solution.
For the expansions
\begin{align}
u(z,\bm{y})=\sum\limits_{k=1}^{\infty}u_k(z)\varphi_k(\bm{y}), \
v(z,\bm{y})=\sum\limits_{k=1}^{\infty}v_k(z)\varphi_k(\bm{y}),
\end{align}
we denote by
\begin{align}
U_N(z,\bm{y})\equiv\sum\limits_{k=1}^{N}u_k(z)\varphi_k(\bm{y}), \
V_N(z,\bm{y})\equiv\sum\limits_{k=1}^{N}v_k(z)\varphi_k(\bm{y})
\end{align}
the partial sums of $u$ and $v$ respectively.
Immediately,
\begin{equation}
\Delta_{g_{\mathcal{C}}}U_N=V_N.
\end{equation}
For every $\bm{x}\equiv (z,\bm{y})\in \mathcal{C}$, we will apply the elliptic regularity on the ball $B_2(\bm{x})\subset\mathcal{C}$ to obtain the higher regularity of $u$. For this purpose, first we prove the following claim.

\begin{claim}
As $N\to\infty$,
$\|V_N-v\|_{C^0(B_2(\bm{x}))}\to 0$.
\end{claim}
\begin{proof}
The proof of the claim follows from basically from Weyl's law. 
For the partial sum of $v$, 
\begin{equation}
V_N \equiv \sum\limits_{j=1}^Nv_{j} \varphi_j=\sum\limits_{j=1}^N\Big(\int_{Y^3}v\cdot\varphi_j\dvol_{h_0}\Big)\varphi_j
=\sum\limits_{j=1}^N\Big(\int_{Y^3}v\cdot\frac{(-\Delta_{h_0})^{K_0}\varphi_j}{(\Lambda_j)^{K_0}}\dvol_{h_0}\Big)\varphi_j.
\end{equation}
Applying integration by parts,
\begin{align}
\|V_N\|_{L^{\infty}(B_1(p_0))}
&\leq \sum\limits_{j=1}^N \Big(\frac{1}{(\Lambda_j)^{K_0}}\int_{Y^3}|\Delta_{h_0}^{K_0}v|\cdot |\varphi_j|\dvol_{h_0}\Big)\|\varphi_j\|_{L^{\infty}(B_1(p_0))}\notag\\
&\leq V_0\cdot\|v\|_{C^{2K_0}(Y^3\times\{z_0\})}\cdot\sum\limits_{j=1}^N\frac{1}{(\Lambda_j)^{K_0-1}},
\end{align}
where $V_0=\Vol_{h_0}(Y^3)$.
Notice that,  the spectrum $\{\Lambda_j\}_{j=1}^{\infty}$ satisfies the Weyl's law on $(Y^3,h_0)$, so in particular for sufficiently large $j$,
\begin{equation}
C_0^{-1} j^{\frac{2}{3}}\leq|\Lambda_j| \leq C_0 j^{\frac{2}{3}}.
\end{equation}
Since $K_0\geq 3$, 
\begin{equation}
\|V_N\|_{L^{\infty}(B_1(p_0))} \leq C \|v\|_{C^{2K_0}(Y^3\times\{z_0\})}\cdot\sum\limits_{j=1}^N\frac{1}{j^{\frac{4}{3}}} \leq C.
\end{equation}
The proof of the claim is done. 
\end{proof}

The proof of the higher order convergence is exactly the same. In fact, we just need to replace $\|v\|_{C^{2K_0}}$ with the higher order norm $\|v\|_{C^{2K_0+m}}$ with $m\leq K_0$.
Since $\Delta_{g_{\mathcal{C}}} U_N=V_N$, the  standard $W^{2,p}$- implies that regularity 
for every $1<p<\infty$, $\|U_N\|_{W^{2,p}(B_1(\bm{x}))}\leq C_{p,\bm{x}}$. By assumption $v\in C^{3K_0}(\mathcal{C})$ with $K_0\geq 3$, we have 
$\|V_N\|_{C^2(B_2(\bm{x}))}\leq C_{\bm{x}}$. Hence the regularity of $u$ will be improved as follows, for every $1<p<\infty$, 
\begin{equation}
\|U_N\|_{W^{4,p}(B_1(\bm{x}))}\leq C_{p,\bm{x}}  ( \|U_N\|_{W^{2,p}(B_{3/2}(\bm{x}))} + \|V_N\|_{W^{2,p}(B_2(\bm{x}))} )\leq C_{p,\bm{x}}.
\end{equation}
Now taking $p>4$ and applying the Sobolev embedding, 
\begin{equation}
\|U_N\|_{C^{3,\alpha}(B_1(\bm{x}))}\leq C_{p,\bm{x}},\ \alpha \equiv 1-\frac{4}{p},
\end{equation}
which implies that $U_N$ converges to a smooth solution $u$.  Then applying the standard Schauder estimate and bootstrapping, the statement of the proposition just follows.

\end{proof}

\subsection{Proof of the Liouville theorem}

\label{ss:proof-liouville-functions}

With the above technical preparations, we complete the proof of the main result in this section, Theorem~\ref{t:liouville-theorem-functions}.
We need the following lemma which is an immediate corollary of the Bochner formula and the maximum principle.
\begin{lemma}\label{l:max}
Let $(M^n,g,p)$ be a complete non-compact manifold with $\Ric_g \geq 0$. Let $\omega$ be a harmonic $1$-form on $(M^n,g)$, i.e., $\Delta_{H}\omega=0$ and assume that
\begin{equation}\lim\limits_{d_g(x,p)\to+\infty}|\omega(x)|= 0, \label{e:asymp}
\end{equation}
then $\omega \equiv 0$ on $M^n$.
\end{lemma}
\begin{proof}
Since $\omega$ is harmonic, by Bochner's formula,
\begin{equation}
\frac{1}{2}\Delta_g|\omega|^2 = |\nabla \omega|^2 + \Ric_g(\omega,\omega) \geq 0,
\end{equation}
then $|\omega|^2$ is subharmonic.  
Given the asymptotic property \eqref{e:asymp}, applying the maximum principle to the above subharmonic function $|\omega|^2$, 
we have $\omega \equiv 0$ on $M^n$.

\end{proof}

\begin{proof}[Proof of Theorem \ref{t:liouville-theorem-functions}]
Let $u$ satisfy $\Delta_{g}u=0$
 $(X^{4}, g)$. We also assume that $u$ satisfies the asymptotic behavior
\begin{equation}
u = O(e^{ \ell_0 z}) \label{e: exp-growth}
\end{equation}
for some $\ell_0\in(0,1)$.
The main part of the proof is to determine a positive number $\ell_0>0$ such that if \eqref{e: exp-growth} holds, 
then $u$ has at most linear growth at infinity, which enables us to apply Lemma \ref{l:max}.

By assumption,
there is a diffeomorphism
\begin{equation}
\Phi: X^{4}\setminus K \longrightarrow [10^2,+\infty)\times Y^{3}
\end{equation}
such that for all $k\geq 0$
\begin{equation}
\|g-\Phi^*g_{\mathcal C}\|_{C^k}  \leq C e^{-\delta z}.\label{e:C1-exp-close}
\end{equation}
To obtain an accurate growth order of $u$, we will study the equation of $u$ in terms of the metric $g_{\mathcal{C}}$ on the model space $[10^2,+\infty)\times Y^3
$.

First, we will show that a harmonic function on $(X^{4}, g)$ with exponential growth
is well behaved in terms of the model metric $g_{\mathcal{C}}$ near infinity. Preciesly, we will prove the following claim.

\begin{claim} Assume that $(X^4,g)$ is $\delta$-asymptotically Calabi. Let $\hat{\delta} \in (0, \delta/10)$
such that $u$ satisfies
\begin{align}
\begin{split}
\Delta_{g} u & = 0 \\
u &= O(e^{\hat{\delta} z}),
\end{split}
\end{align}
then for every fixed $k\in \dZ_+$, let $\bm{x}_0 \in [T_0(k),+\infty)\times Y^{3}$ 
with $T_0(k)\geq 100^{k^3}>0$ and denote
 $z_0 \equiv z(\bm{x}_0)$, we have
 \begin{equation}
 \|\nabla^k \Delta_{g_{\mathcal{C}}} u(\bm{x}_0)\| \leq C(k,g) \cdot e^{-\frac{\delta z_0}{2}}.
 \end{equation}
\end{claim}

\begin{proof}
Denoting $\phi\equiv  (\Delta_{g} - \Delta_{g_{\mathcal{C}}}) u$, 
then $\Delta_{g} u = 0$ implies 
\begin{equation}
\Delta_{g_{\mathcal{C}}} u + \phi = 0.\label{e:difference}
\end{equation}
We will show that 
for each $k\in\dN$ we have 
\begin{equation}
\|\nabla^k \phi(\bm{x}_0)\| \leq C(k,g) \cdot e^{-\frac{\delta z_0}{2}},
\end{equation} 
where $C(k,g)>0$ depends only on $k\in\dN$ and the curvature bound of the cutoff region $(X^4\setminus K, g)$.

The higher order derivative estimate will be proved by the $W^{k,p}$-estimate for harmonic functions on the complete space $(X^4, g)$. Since the metric $g$ is collapsing near the infinity, the standard elliptic estimate cannot be directly applied. To overcome this difficulty, 
we will scale up the metric $\tilde{g}=\lambda^2 g$ such that $B_1(\bm{x_0})$ is non-collapsing for $\tilde{g}$ which guarantees the elliptic estimate holds in terms of the rescaled metric $\tilde{g}$. For fixed $\bm{x}_0\in [T_0(k),+\infty)$, we take 
\begin{equation}\lambda= z_0^{\frac{1}{2}}\end{equation} and hence there is some constant $v_0>0$ which is independent of the $z$-coordinate such that
\begin{equation}\Vol_{\tilde{g}}(B_1(\bm{x_0}))\geq v_0 >0.\end{equation} 
By explicit calculation on the model space $\mathcal C$ using \eqref{Calabiformula}  one easily sees that
curvatures are uniformly bounded in a ball of definite size of radius,  i.e. 
\begin{equation}
\sup\limits_{B_2(\bm{x_0})}\|\Rm\|_{\tilde{g}}\leq\Lambda_0,
\end{equation}
where $\Lambda_0>0$ is independent of the $z$-coordinate.
It follows that for every $k\in\dN$ and $1<p<\infty$, there exists $C(k,v_0,\Lambda_0,p)>0$ such that under the rescaled metric $\tilde{g}$,
\begin{equation}
\|u\|_{W_{\tilde{g}}^{k+2,p}(B_{1}(x_0))}\leq C\|u\|_{W_{\tilde{g}}^{k,p}(B_{1+\frac{1}{k^{2}}}(x_0))},
\end{equation}
which implies that for every $k\in\dZ_+$,
\begin{equation}
\|u\|_{W_{\tilde{g}}^{k,p}(B_{1}(x_0))}\leq C\sup\limits_{B_3(\bm{x}_0)}|u|.\label{e:integral-estimate}
\end{equation}
Therefore, for every $k\in\dZ_+$ and sufficiently large $p\in(1,\infty)$, applying the Sobolev embedding on $(B_{4/3}(\bm{x}_0),\tilde{g})$, there exists $C(k,p,v_0,\Lambda_0)>0$ such that
\begin{equation}
\sup\limits_{B_1(\bm{x}_0)}|\nabla^{k}u|_{\tilde{g}} \leq C\|\nabla^{k+1}u\|_{L^{p}(B_{4/3}(x_0))}.\end{equation}
By \eqref{e:integral-estimate} and the growth assumption on $u$, there is some constant $C>0$ such that
\begin{equation}
\sup\limits_{B_1(\bm{x}_0)}|\nabla^{k}u|_{\tilde{g}}\leq C\sup\limits_{B_2(\bm{x}_0)}|u| \leq Ce^{\hat{\delta} z_0}.
\end{equation}
In terms of the original metric $g$, we have
\begin{equation}
|\nabla^{k}u(\bm{x}_0)|_{g}\leq \sup\limits_{B_{1/\lambda}(\bm{x}_0)}|\nabla^{k}u|_{g}\leq C\cdot  z_0^{2k} e^{\hat{\delta} z_0}< Ce^{\delta' z_0},\label{e:ck-elliptic}
\end{equation}
where $\delta'\in\Big(\hat{\delta},(1+10^{-3})\hat{\delta}\Big)$.

Next, by \eqref{e:C1-exp-close}, there is some constant $\delta>0$ such that 
\begin{equation}
\|\Phi^*g_{\mathcal{C}}-g\|_{C^{k}(B_2(\bm{x}_0))}\leq C_k e^{-\delta z}.\label{e:delta-close}
\end{equation}
 then the elliptic estimate \eqref{e:ck-elliptic} and \eqref{e:delta-close} imply that 
\begin{equation}
| \phi (\bm{x}_0)| =  |(\Delta_{g}-\Delta_{g_{\mathcal{C}}})u(\bm{x}_0)| \leq C e^{-\frac{\delta\cdot z_0}{2}}
\end{equation}
and similarly
\begin{equation}
|\nabla^{k} \phi (\bm{x}_0)|  \leq C_k e^{-\frac{\delta\cdot z_0}{2}}.
\end{equation}

\end{proof}

The above error estimate enables us to  construct a harmonic function with respect to the model metric $g_{\mathcal{C}}$ on $[10^2,+\infty)\times Y^3$ which has at most linear growth and is exponentially close to the original function $u$. 
Let 
\begin{equation}
\ell_0 \in \Big(0, \min\{\frac{\delta}{10^2},\underline{\delta} \}\Big),\label{e:def-l_0}
\end{equation}
where $\underline{\delta}>0$ is the constant in Proposition \ref{p:harmonic-function-decay}.
By assumption the harmonic function $u$ satisfies the asymptotic behavior,
\begin{equation}
u=O(e^{\ell_0z}).
\end{equation}
Then applying the above claim and Proposition \ref{p:surjectivity}  on $[T_0,+\infty)\times Y^3$, there exists a solution to the equation
\begin{equation}
\Delta_{g_{\mathcal{C}}}v=\phi\label{e:surjective}
\end{equation}
such that 
\begin{equation}v=O(e^{-\ell z})\end{equation} for some $\ell\in(-\delta/2,0)$.
Therefore, combine \eqref{e:difference}
and \eqref{e:surjective}, we have
\begin{align}
0 = \Delta_{g}(u) = \Delta_{g_{\mathcal{C}}} (u + v), \label{e:model-harmonic}
\end{align}
and $u + v = O( e^{\ell_0 z})$. 
Since $0< \ell_0 <1$ has been specified in \eqref{e:def-l_0}, now we are in a position to apply Proposition \ref{p:harmonic-function-decay} to $u + v$, which shows that 
\begin{equation}( u + v )= a z + b + O(e^{-\underline{\delta} z}),\end{equation} 
and hence in the non-compact part $[T_0,+\infty)\times Y^3$,
\begin{align}
u = a z + b + O(e^{-\delta'' z}), \ \delta''\equiv\min\{\ell,\underline{\delta}\}.\label{e:linear-growth} 
\end{align}

The above asymptotics immediately implies that
\begin{equation}
|du|_{g}\to0\ \text{as} \ z\to\infty.\label{e:decay-of-|dz|}
\end{equation}
 Let $\Delta_H$ be the Hodge-Laplacian on $(X^4, g)$. Since $\Delta_{g}u=0$, it holds that
 \begin{equation}
 \Delta_H(du)=dd^*(du)=-d\Delta_{g}u=0.
 \end{equation}
Since the complete space $(X^4,g)$ satisfies $\Ric_{g}\geq 0$, and $|du|$ satisfies  the decay property \eqref{e:decay-of-|dz|}, applying Lemma \ref{l:max} implies that \begin{equation}
|du|_{g}\equiv0 \ \text{on}\ X^4.
\end{equation} 
Therefore, $u$
has to be a constant.
\end{proof}

\section{Liouville theorem for half-harmonic  1-forms}
\label{s:liouville-1-form}
We call a $1$-form $\gamma$  \emph{half-harmonic} if it is in the kernel of the operator $\mathscr{D}\equiv d^+\oplus d^*$. If the manifold is compact, then such a form is automatically harmonic, but this is no longer true on a non-compact manifold. The main result of this section is the following. 
\begin{theorem}
\label{t:liouville-1-form}
Let $(X^{4}=M\setminus D, \omega)$ be given by the Tian-Yau construction where $M$ is a smooth Fano manifold of complex dimension $2$ and $D$ is a smooth anti-canonical divisor. Then there is some positive constant $\delta_h>0$ which depends only on the geometry of $X^4$ such that if a 1-form $\gamma$ on $X$ satisfies
\begin{equation} \label{eqnN1}
d^+\gamma=d^*\gamma=0
\end{equation}
and 
\begin{equation}
|\gamma|_\omega=O(e^{\delta_h r^{\frac{2}{3}}})
\end{equation}
as $r \to \infty$, 
where $r$ is the distance function on $X$ to a fixed point, then $\gamma = 0$. 
\end{theorem}
\begin{proof}
We begin with an interpretation of the equation (\ref{eqnN1}) in terms of complex geometric data. Notice in the Tian-Yau construction we have a preferred complex structure on $X$ induced from $M$, which we denote by $J$. With respect to $J$ we can write $\gamma=\gamma^{1,0}+\gamma^{0, 1}$ with $\gamma^{1,0}=\overline{\gamma^{0,1}}$. Then by the K\"ahler identities we have 
\begin{equation}
d^+\gamma=0 \Longleftrightarrow 
\begin{cases}
\bar\partial \gamma^{0, 1}=0\\
\sqrt{-1}(\bp^* \gamma^{0, 1}-\partial^* \gamma^{1, 0})=0
\end{cases}
\end{equation}
and
\begin{equation}
d^*\gamma=0 \Longleftrightarrow \bp^* \gamma^{0, 1}+\p^* \gamma^{1, 0}=0. 
\end{equation}
Thus, equation (\ref{eqnN1}) is equivalent to 
\begin{equation}
\bp \gamma^{0, 1}=0, \ \bp^*\gamma^{0, 1}=0.
\end{equation}
The theorem follows from Theorem \ref{t:liouville-theorem-functions} once we prove that there exists some small $\delta > 0$ and a smooth function $f = O(e^{\delta z})$ such that $\bar\partial f = \gamma$ (note that $\Delta f = \bar\partial^*\gamma = 0$).

 We next give a brief outline of the proof. In Step 1, we will construct a solution $f$ to $\bar\partial f = \gamma$ such that $f = O(e^{\epsilon z^2})$ for all $\epsilon > 0$. This is done using a complex geometric argument which amounts to an application of H\"ormander's weighted $L^2$ estimates for the $\bar\partial$-operator. Interestingly it does not seem to be possible to obtain the required improvement $f = O(e^{\delta z})$ using only this type of method, owing to the fact that the function $z^a$ is plurisubharmonic on the Calabi model space $(\mathcal{C},g_{\mathcal{C}})$ if and only if $a \geq 2$.

To overcome this problem we use the elliptic theory on $(\mathcal{C},g_{\mathcal{C}})$ developed in Section \ref{s:liouville-functions}. Thanks to the bound $f = O(e^{\epsilon z^2})$ for all $\epsilon > 0$ from Step 1 and the $O(e^{-(\frac{1}{2}-\epsilon)z^2})$ complex structure asymptotics of Proposition \ref{p:TY-asym}, it follows that $\bar\partial_{\mathcal{C}} f = O(e^{\delta z})$ on $(\mathcal{C},g_{\mathcal{C}})$. In particular, since $\Delta = {\rm tr}(\sqrt{-1}\partial\bar\partial)$, the Poisson equation estimates of Proposition \ref{p:surjectivity}  imply that $f$ can be decomposed into an $O(e^{\delta z})$ part $f_1$ and a \emph{$g_{\mathcal{C}}$-harmonic} part $f_2$ which is $O(e^{\epsilon z^2})$ for all $\epsilon > 0$ (see Step 2 for details). Observe that it would not be possible to compare $\Delta_{TY}f$ and $\Delta_{\mathcal{C}}f$ directly because $g_{TY}$ and $g_{\mathcal{C}}$ are only asymptotic at rate $O(e^{-\delta z})$, which is too slow to beat the $O(e^{\epsilon z^2})$ growth of $f$ from Step 1.

Step 3 analyzes the $g_{\mathcal{C}}$-harmonic part $f_2$ of $f$. It is clear from Section \ref{s:liouville-functions} that $f_2 = O(e^{Cz})$ for some large constant $C$. The required improvement $f_2 = O(e^{\delta z})$ comes from the first-order equation $\bar\partial_{\mathcal{C}} f_2 = O(e^{\delta z})$ satisfied by $f_2$ (in addition to $\Delta_{g_{\mathcal{C}}} f_2 = 0$). Technically this is done using separation of variables for the $\bar\partial_{\mathcal{C}}$-operator but the underlying idea can be easily explained: being of $O(e^{Cz})$ rather than $O(e^{Cz^2})$ growth, the leading terms of the harmonic function $f_2$ must be $S^1$-invariant, but on $S^1$-invariant functions the $\bar\partial_{\mathcal{C}}$-operator directly controls the radial derivative $\frac{\partial}{\partial z}$. 

Step 4 concludes the proof by appealing to Theorem \ref{t:liouville-theorem-functions}.
 \medskip\

\noindent \textbf{Step 1.} In this step, we prove the following proposition.

\begin{proposition} \label{p:propN2}
There is a smooth function $f$ on $X$ with $\bp f=\gamma$ and $|f|=O(e^{\epsilon z^2})$ for all $\epsilon>0$. 
\end{proposition}

\begin{remark}
We also have $\Delta_\omega f=0$, but at this point we cannot apply Theorem \ref{t:liouville-theorem-functions} directly to conclude that $f$ is a constant since this would require stronger control, $|f|=O(e^{\delta z})$. 
\end{remark}

\begin{proof}[Proof of Proposition \ref{p:propN2}] We work on the compact manifold $M$. Let $S$ be a holomorphic section of $K_M^{-1}$ with $S^{-1}(0)=D$, and let $h$ be a smooth hermitian metric on $K_M^{-1}$ whose curvature form $\omega_h$ is a K\"ahler form on $M$ with positive Ricci curvature. By Theorem \ref{t:hein} near $D$ we have 
\begin{equation}C^{-1} \sqrt{-1}\partial\bar\partial (-{\log |S|^2_h})^{3/2}\leq \omega_{TY}\leq C\sqrt{-1}\partial\bar\partial (-{\log |S|^2_h})^{3/2}.
\end{equation}
By a straightforward computation this implies that 
\begin{equation}\label{e:1986}\omega_{TY}\leq C|S|_h^{-2}\omega_h\end{equation}
and hence, trivially,
\begin{equation}\label{e:1987}|\gamma|_{\omega_h}\leq C^{-1} |S|_h^{-1}|\gamma|_{\omega_{TY}}=O(|S|_h^{-1-\epsilon})\end{equation} for any $\epsilon>0$. 
Define $\alpha=\gamma\otimes S$. This is a section of $\Lambda_M^{0,1} \otimes K_M^{-1}$ which lies in $L^p_{\omega_h}(M, \Lambda^{0, 1}_M\otimes K_M^{-1})$ for all $p\geq 1$. Since $\bp\gamma=0$, one can directly check that $\bp \alpha=0$ in the distributional sense. Now notice that $H^{1}(M, K_M^{-1})=H^{1}(M, K_M \otimes L)=0$ by the Kodaira vanishing theorem applied to the ample line bundle $L = K_M^{-2}$. Thus, we can define $\beta=\bp^*\Delta_{\bp}^{-1}\alpha$ with respect to $\omega_h$. It follows from elliptic regularity that $\beta\in W^{1, p}_{\omega_h}(M, K_M^{-1})$ for all $p\geq 1$, so that $\beta\in C^\alpha_{\omega_h}(M, K_M^{-1})$ for all $\alpha<1$. Moreover by local regularity we know $\beta$ is smooth outside $D$ and $\bp\beta=\alpha$. Let $f=\beta \otimes S^{-1}$, then on $X$ we have $\bp f=\gamma$. The immediate estimate  we get is that  for some constant $C>0$,
\begin{equation}
f=O(|S|_h^{-1})=O(e^{Cz^2}).\label{e:exp-square}\end{equation}
The lemma below allows us to improve \eqref{e:exp-square} to the growth order $e^{\epsilon z^2}$ for any $\epsilon>0$. The key point is that the estimate \eqref{e:1986} can be improved to almost $O(1)$ in directions tangential to $D$. 

\begin{lemma}
Denote $\beta_0:=\beta|_D$, then $\bp \beta_0=0$, i.e. $\beta_0$ is a holomorphic section of $K_M^{-1}|_D$. 
\end{lemma}
\begin{proof}
 We choose a finite cover $D=\bigcup_{k=1}^{N_0} O_k$ such that for each $k$ there exists a local holomorphic coordinate system $(z,w)$ on some domain $U_k \subset M$ such that $U_k \cap D = O_k = \{w = 0\}$. We will show that  $\bp \beta_0=0$ in every $O_k\subset D$ in the distributional sense. Let  $\psi$ be a smooth section of $\Lambda^{0,1}_D \otimes (K_M^{-1}|_D)$ with compact support in $O_k$. It suffices to show that 
$\langle\beta_0,\bp^*\psi\rangle_{O_k}=0$.
To this end, write $\psi(z) = \sigma(z) d\overline{z} \otimes (dz \wedge dw)^{-1}$ for some smooth function $\sigma \in C^{\infty}_0(O_k,\mathbb{C})$ and use this to define the trivial extension $\hat{\psi}(z,w) = \sigma(z) d\overline{z} \otimes (dz \wedge dw)^{-1}$ for all $(z,w) \in U_k$. Denote by $O_k(\tau)$ the slice $\{w = \tau\}$ in $U_k$, which is a complex submanifold of $M$, and equip $O_k(\tau)$ with the restriction of the K\"ahler metric $\omega_h$ from $M$. Notice that $\hat{\psi}$ restricts to a smooth section of $\Lambda_{O_k(\tau)}^{0,1} \otimes (K_M^{-1}|_{O_k(\tau)})$ with compact support in $O_k(\tau)$. Since $\bp \beta=\alpha$ and $\beta\in W^{1,p}\cap C^{\alpha}$ for any $p\geq 1$, it follows that
\begin{align}\label{e:weak-solution}
\langle\beta_0, \bp^*{\psi}\rangle_{O_k}
=\lim\limits_{\tau\to0}\langle\beta, \bp^*\hat{\psi}\rangle_{O_k(\tau)}=\lim\limits_{\tau\to0}\langle\bp\beta,\hat{\psi}\rangle_{O_k(\tau)}=\lim\limits_{\tau\to0}\langle\alpha,\hat{\psi}\rangle_{O_k(\tau)}.
\end{align}
Notice that \begin{equation}|\gamma(\p_{\bar{z}})|\leq |\gamma|_{\omega_{TY}} |\p_{\bar{z}}|_{\omega_{TY}}\leq |\gamma|_{\omega_{TY}} (-\log |S|^2_h)^{\frac{1}{4}}=O(|S|_h^{-\epsilon}). \end{equation}
Since $\alpha = \gamma \otimes S$, it then follows that $|\alpha(\p_{\bar{z}})| = O(|S|_h^{1-\epsilon}) \rightarrow 0$ uniformly as $w\rightarrow 0$. Using \eqref{e:weak-solution}, it follows that
\begin{equation}
\langle\beta_0, \bp^*{\psi}\rangle_{O_k}
=0,
\end{equation}
 as desired. By standard elliptic regularity, $\beta_0$ is a holomorphic section. 
\end{proof}

Since $M$ is Fano we have $H^1(M, \mathcal O_M)=0$ so by a standard exact sequence (\cite[p.139]{GH}) the restriction map $H^0(M, K_M^{-1})\rightarrow H^0(D, K_M^{-1}|_D)$ is surjective. This means we can find some $\beta_1$ $\in$ $H^0(M, K_M^{-1})$ such that $\beta_1|_{D}=\beta_0|_D$. Let $f=(\beta-\beta_1)\otimes S^{-1}$. Then we still have $\bp f=\gamma$ on $X$ but now since $\beta-\beta_1=0$ on $D$ and $\beta\in C^\alpha_{\omega_h}(M, \mathbb{C})$ for all $\alpha<1$, we finally obtain  Proposition \ref{p:propN2}.\end{proof}

\noindent \textbf{Step 2.} 
Let $f=u+\sqrt{-1}v$ be the smooth function constructed by Proposition \ref{p:propN2} with $\Delta_{\omega_{TY}}u=\Delta_{\omega_{TY}}v=0$.
In this step, we reduce the problem to a question on the Calabi model space through the diffeomorphism $\Phi:
(\mathcal{C}\setminus K',\omega_{C},J_{C})\rightarrow (X\setminus K,\omega_{TY}, J_{TY})$ chosen in Proposition \ref{p:TY-asym}.  
The main point is to obtain the decomposition  $u=u_1+u_2$ and $v=v_1+v_2$ such that $u_1=O(e^{\delta z})$, $v_1=O(e^{\delta z})$
and $\Delta_{\omega_{\mathcal{C}}}u_2=\Delta_{\omega_{\mathcal{C}}}v_2=0$. The growth estimates for $u_2$ and $v_2$ will be shown in Step 3.

The  idea of the proof of Step 2 is as follows. First, we will estimate 
$\Delta_{\omega_{\mathcal{C}}} u$ and $\Delta_{\omega_{\mathcal{C}}} v$ and all of their derivatives. Specifically, we will prove that they have slow exponential growth rates (as shown in \eqref{e:D-Laplace-growth}). Then applying Proposition \ref{p:surjectivity}, we can  construct solutions to the Poisson equations 
\begin{align}
\Delta_{\omega_{\mathcal{C}}}u_1 = \Delta_{\omega_{\mathcal{C}}}u, \ 
\Delta_{\omega_{\mathcal{C}}}v_1 = \Delta_{\omega_{\mathcal{C}}}v,
\end{align}
such that $u_1=O(e^{\delta z})$ and $v_1=O(e^{\delta z})$
This completes the desired decomposition of $u$ and $v$.

To obtain the derivative estimates for $\Delta_{\omega_{\mathcal{C}}} u$ and $\Delta_{\omega_{\mathcal{C}}} v$, we will prove the derivative estimates for $du + J_{\mathcal{C}}dv$. 
To start with, by the assumption on $\gamma$ and the first order equation given by Step 1,
\begin{equation}
du+J_{TY}dv=Re(\gamma)=O(e^{C \delta_h z}).\label{e:real-part-gamma}
\end{equation} 
Applying the asymptotic estimate for $J_{TY}$ in Proposition \ref{p:TY-asym}, we can convert the above growth control  to the corresponding estimate for $du + J_{\mathcal{C}}dv$.

In fact, applying Item (b) of Proposition \ref{p:TY-asym},
for any $\epsilon>0$ and for any $k\in\dN$,
\begin{equation}
|\nabla_{g_\mathcal C}^k(\Phi^*J_{TY}-J_{\mathcal C})|_{g_\mathcal C}=O(e^{(-\frac{1}{2}+\epsilon)z^2}).
\end{equation}
We also need derivative estimates for $u$ and $v$ with respect to the model metric $\omega_{\mathcal{C}}$.
Notice that by Step 1,  $f=u+\sqrt{-1}v=O(e^{\epsilon z^2})$ for any $\epsilon>0$, which implies that 
\begin{align}
|u| + |v| &\leq 2 |f|=O(e^{\epsilon r^{4/3}})=O(e^{\epsilon z^2}),
\end{align}
 where $z$ is the natural coordinate on $\mathcal C$.
Since $u$ and $v$ satisfy
$\Delta_{\omega_{TY}} u = \Delta_{\omega_{TY}} v = 
0,
$ by applying the same $W^{k,p}$-estimate as in the proof of Theorem \ref{t:liouville-theorem-functions}, we have for all $\epsilon>0$ and $k\geq 1$,
\begin{equation}
|\nabla^k u|_{\omega_{TY}}=O(e^{\epsilon z^2}), \ |\nabla^kv|_{\omega_{TY}} =O(e^{\epsilon z^2}).
\end{equation} 
Since the asymptotic order of  harmonic functions $u$ and $v$ is dominated by
$e^{\epsilon z^2}$ and the asymptotic order of the metric $\omega_{TY}$ is $e^{-\underline{\delta} z}$, so in terms of the model metric we have
 \begin{equation}
|\nabla^k u|_{\omega_{\mathcal{C}}}\leq C e^{\frac{\epsilon z^2}{2}}, \ 
|\nabla^kv|_{\omega_{\mathcal{C}}} \leq Ce^{\frac{\epsilon z^2}{2}}.
\end{equation} 
Now we apply the assumption $|\gamma|_{\omega_{TY}}=O(e^{C\delta_h z})$ and the above elliptic regularity to \eqref{e:real-part-gamma}, we get for $k\in\dN$,
\begin{align}
|\nabla^k(du+J_{\mathcal C}dv)|_{\omega_\mathcal C}
&=  \Big|\nabla^k(du+J_{TY}dv) + \nabla^k\Big((J_{\mathcal{C}}-\Phi^*J_{TY})dv\Big)\Big|_{\omega_\mathcal C}\nonumber\\
&=
O(e^{C_k\delta_h z}) + O (e^{-\frac{z^2}{4}}) =O(e^{C_k\delta_h z}). 
\label{e:system-derivative-estimate}
\end{align}

Now we proceed to prove the derivative estimates for $\Delta_{\omega_{\mathcal{C}}}u$ and $\Delta_{\omega_{\mathcal{C}}}v$ by making use of the system
\begin{equation}du+J_{TY}dv=Re(\gamma).\label{e:re-gamma}\end{equation}
The advantage of the above equation is that $\Delta_{\omega_{TY}}= \Tr_{\omega_{\mathcal{C}}}(d J_{TY} d)$ so that the behavior $\Delta_{\omega_{TY}}$ will follow from the asymptotics of $J_{TY}$.
 In fact, taking the differential of \eqref{e:re-gamma}, 
\begin{equation}
dJ_{TY}du=dJ_{TY}Re(\gamma), \ dJ_{TY}dv=dRe(\gamma).
\end{equation}
Then using Item $(a)$ of Proposition \ref{p:TY-asym},  similar to the above we have for all $k\in\dZ_+$
\begin{equation}
|\nabla^kdJ_{\mathcal C}du|=O(e^{C_k\delta_h z}), \
 |\nabla^kdJ_{\mathcal C}dv|=O(e^{C_k\delta_h z}).
 \end{equation}
Taking the trace, then we obtain 
\begin{equation}
|\nabla^k\Delta_{\omega_\mathcal C}u|=O(e^{C_k\delta_h z}), \
 |\nabla^k\Delta_{\omega_\mathcal C}v|=O(e^{C_k\delta_h z}).
\label{e:D-Laplace-growth}
 \end{equation}
Applying the linear theory for $\Delta_{\omega_\mathcal C}$ in Proposition \ref{p:surjectivity}, if $\delta_h\ll\underline\delta$, then we choose two solutions $u_1$ and $v_1$ provided by Proposition  \ref{p:surjectivity}
\begin{equation}
\Delta_{\omega_{\mathcal{C}}} u_1 = \Delta_{\omega_{\mathcal{C}}} u\
\text{and} \
\Delta_{\omega_{\mathcal{C}}} v_1 = \Delta_{\omega_{\mathcal{C}}} v
\end{equation}
such that $u_1$, $v_1$ satisfy
\begin{equation}
\label{eqn5-18}
|\nabla^ku_1|_{\omega_{\mathcal C}}=O(e^{C_k\delta_h z}), \ 
|\nabla^k v_1|_{\omega_{\mathcal C}}
=O(e^{C_k\delta_h z}).
\end{equation}
So we have finished the proof of the decomposition $u=u_1+u_2$ and $v=v_1+v_2$ such that
\begin{equation}
\Delta_{\omega_{\mathcal C}}u_2=\Delta_{\omega_\mathcal C}v_2=0.
\end{equation}
We also obtain that 
\begin{equation}|du_2+J_{\mathcal C}dv_2|_{\omega_{\mathcal C}}=O(e^{C\delta_h z})\end{equation}
and
\begin{equation} \label{eqn5-21}
|u_2|=O(e^{\epsilon z^2}), \ 
|v_2|=O(e^{\epsilon z^2}).
\end{equation}

\vspace{2mm}
\noindent
\textbf{Step 3.} Now we estimate the harmonic functions $u_2$ and $v_2$ with respect to the model metric $\omega_{\mathcal{C}}$ using separation of variables. The goal is to improve the growth order of $u_2$ and $v_2$ from $O(e^{\epsilon z^2})$ for all $\epsilon > 0$ to $O(z)$, using the fact that they also satisfy a first-order equation. 

\begin{proposition} \label{prop5-6}
We have
\begin{align}
|u_2|=O(z), \
|v_2|=O(z).
\end{align}
\end{proposition}
\begin{remark}
The operator $(u, v)\mapsto du+J_{\mathcal C}dv$ has a kernel which consists of pairs $(u, v)$ such that $u+\sqrt{-1}v$ is holomorphic. In our case this is eliminated since we have the growth control (\ref{eqn5-21}). Notice that the smallest growth rate of a non-constant holomorphic function on $\mathcal C$ is $e^{\frac{1}{2}z^2}. $
\end{remark}
\begin{proof}[Proof of Proposition \ref{prop5-6}]
We denote 
$$\psi=du_2+J_\mathcal Cd v_2=O(e^{C\delta_h z}), $$
  Then we have the following expansion as in Section 4.1: let $\{\Lambda_k\}_{k=1}$ be the spectrum of $Y^3$ and $\{\varphi_k\}_{k=1}^{\infty}$ are the corresponding eigenfunctions on $Y$ with $\mathcal L_{\p_\theta}\varphi_k=\sqrt{-1}j_k\varphi_k$, 
\begin{equation}
u_2=\sum\limits_k f_k(z)\varphi_k(z_\alpha, \theta), \
 v_2=\sum\limits_k g_k(z)\varphi_k(z_\alpha, \theta),
\end{equation}
 which implies that
\begin{equation}
du_2=\sum\limits_k(f_k'(z)\cdot\varphi_k\cdot dz+f_k(z)\cdot d\varphi_k), \
dv_2=\sum\limits_k(g_k'(z) \varphi_k\cdot dz+g_k(z)\cdot d\varphi_k).
\end{equation}
On $\mathcal C$ by the definition in Section \ref{ss:separation-of-variables}, 
we have $J_{\mathcal C}(zdz)=d\theta$, so we have 
\begin{align}
\begin{split}
\sum\limits_k \big(f_k'(z)-\sqrt{-1}j_k \cdot z \cdot g_k(z) \big)\varphi_k&=\psi(\p_z) \\
\sum\limits_k \big(z^{-1}\cdot g_k'(z)+\sqrt{-1}j_k \cdot f_k(z) \big)\varphi_k&=\psi(\p_\theta).
\end{split}
\end{align}
This implies that for each $k$, 
\begin{equation}\label{eqn5-26}
\begin{split}
f_k'(z)-\sqrt{-1}j_k\cdot z \cdot  g_k(z) =O(e^{C\delta_hz})\\
z^{-1}\cdot g_k'(z)+\sqrt{-1}j_k \cdot f_k(z)=O(e^{C\delta_hz}).
\end{split}
\end{equation}

There are three different cases.

If $j_k\neq 0$, then we can write $f_k$ and $g_k$ are given by a linear combination of one growing solution $\mathcal F_k$ and one decaying solution $\mathcal U_k$. Using the analysis in Section \ref{s:liouville-functions} we know that the asymptotic order of $\mathcal F_k$ is $e^{\frac{j_k z^2}{2}}$ (see Lemma \ref{l:asymp-F-U}). The control (\ref{eqn5-21}) then implies both $f_k$ and $g_k$ can only be a multiple of the decaying solution $\mathcal U_k = O(e^{-\frac{j_k z^2}{2}})$. 

If $j_k=0$ and $\lambda_k\neq0$, then $f_k$ and $g_k$ are given by linear combinations of the exponential functions of the form $e^{\sqrt{\lambda_k}z}$ and $e^{-\sqrt{\lambda_k} z}$. Let $\underline{\delta}>0$ be the positive constant given in Proposition \ref{p:harmonic-function-decay},  we use (\ref{eqn5-26}) and the fact that $\psi=O(e^{C\delta_h z})$ to conclude that, if $C\delta_h<\underline{\delta}$, then both $f_k$ and $g_k$ must be proportional to the decaying solutions. 

If $j_k=0$ and $\varphi_k$ is constant, then $f_k$ and $g_k$ are linear functions in $z$. Now since $u_2$ and $v_2$ are harmonic functions on $\mathcal C$, by Lemma \ref{l:decay-harmonic}, we conclude that 

\begin{equation}
|u_2|=O(z), \
|v_2|=O(z).
\end{equation}
This completes the proof of Proposition \ref{prop5-6}.\end{proof}

\noindent\textbf{Step 4.} We now complete the proof of Theorem \ref{t:liouville-1-form}. By Proposition \ref{prop5-6} and (\ref{eqn5-18}),
\begin{equation}
|u|=O(e^{C\delta_h z}), \
|v|=O(e^{C\delta_h z}). 
\end{equation}
If we further choose $C\delta_h \leq \ell_0$, where $\ell_0$ is the constant of Theorem \ref{t:liouville-theorem-functions}, we conclude that $u$ and $v$ must be constant, hence $\gamma=0$.
\end{proof}

\section{Construction of the approximate hyperk\"ahler triple}

\label{s:approx-triple}

In this section, we will obtain a closed manifold $\mathcal{M}$ by gluing two pieces of hyperk\"ahler Tian-Yau spaces with a {\it neck region} which satisfies appropriate topological balancing condition, such that $\mathcal{M}$ has the same homological invariants as the $\K3$ surface (see Proposition \ref{p:topological-invariants}).  Moreover, we will
construct a closed {\it definite triple} $\bm{\omega}^{\mathcal{M}}$ on $\mathcal{M}$ which is very close to an ${\rm{SU}}(2)$-structure. This will be perturbed to a {\it hyperk\"ahler triple} $\bm{\omega}^{\HK}$ in Section~\ref{s:existence}.

First, we briefly describe the geometry of the Tian-Yau spaces $(X_{b_-}^4,g_{b_-})$ and $(X_{b_+}^4,g_{b_+})$ for fixed $b_{\pm}\in\{1,\ldots, 9\}$.
Denote by $\Nil^3_{b_{\pm}}$ the Heisenberg nilpotent $3$-manifolds with $\deg(\Nil^3_{b_{\pm}})=b_{\pm}$. It follows from Proposition \ref{Calabihyperkahler} and Corollary \ref{coro:tydiff} that there exists a coordinate system on the end of the Tian-Yau spaces $(X_{b_{\pm}}^4 , g_{b_{\pm}})$ 
\begin{align}
\Phi^{TY}_{\pm} :  [\zeta_0^{\pm}, \infty) \times  \Nil^3_{b_{\pm}} \rightarrow X_{b_{\pm}}^4,
\end{align}
such that
\begin{align} 
(\Phi^{TY}_{\pm})^* g_{b_{\pm}} &= 
V_{\pm}( g_{\TT} + dz_{\pm}^2 ) + V_{\pm}^{-1} \theta_{b_{\pm}}^2 + O(e^{-\underline{\delta}_{\pm} z_{\pm}}), 
\end{align}
as $z_{\pm} \to \infty$,
and the area of each $2$-torus fiber $\TT$ is $A_-$ and $A_+$ respectively. Moreover, the harmonic potential $V_{\pm}$ admits the expansion in coordinates  
\begin{equation}
V_{\pm} = \frac{2 \pi}{A_{\pm}} b_{\pm} z_{\pm} + c_{\pm},
\end{equation}
as  $z_{\pm} \to \infty$,
and $\theta_{b_{\pm}}$ is the fixed connection $1$-form on the model space satisfying 
\begin{equation}
d \theta_{b_{\pm}} = \frac{2 \pi b_{\pm}}{A_{\pm}} \dvol_{\TT}.  
\end{equation}
From now on, we rescale the above Tian-Yau metrics such that the above $2$-tori have the common area $A$, which we assume satisfies $\frac{1}{2} \leq A \leq 1$.
We are free to make the following assumptions
\begin{align}
\label{cpmass}
c_{\pm} &= 0, \\
\label{thetaass}
\theta_{b_{\pm}} &= \frac{ 2 \pi b_{\pm}}{A} ( dt - x dy).
\end{align}
The normalization \eqref{cpmass} can be arranged by applying a translation in the $z_{\pm}$ coordinates (since $b_{\pm} \neq 0$ by assumption). The normalization \eqref{thetaass} can be arranged using the discussion of gauge transformations given in Section \ref{s:model-space}.

Notice that, we have fixed a scale of the space such that the slope of the above linear functions are uniquely determined.
In our gluing construction, for fixed parameters $T_->0$ and $T_+>0$, 
we choose the cutoff region in $X_{b_{\pm}}^4$,
\begin{align}
X_{b_{-}}^4 (T_-) &\equiv X_{b_{-}}^4 \setminus \Big((T_-, +\infty)\times \Nil_{b_{-}}^3\Big)\\
X_{b_{+}}^4 (T_-) &\equiv X_{b_{+}}^4 \setminus \Big((T_+, +\infty)\times \Nil_{b_{+}}^3\Big).
\end{align}

\subsection{The neck region: a doubly periodic analogue of the Ooguri-Vafa metric}
\label{ss:neck-region}
Given $m_0 \in \dZ_+$,
the neck region $(\mathcal{N}_{m_0}^4,g_{GH})$ is obtained from a Gibbons-Hawking space based on $\TT \times \mathbb{R} $ with $m_0$ monopole points. The Gibbons-Hawking metric is determined by a harmonic function on $\TT\times\dR$.
 Notice that, the topology of $\K3$ requires 
\begin{equation}m_0 = b_- + b_+.\end{equation} 
Consider the flat cylinder $(\TT \times \mathbb{R}, g_0)$ and denote by $\mathcal{P}_{m_0}\equiv\{p_1,\ldots, p_{m_0}\}\subset \TT\times\dR$ a finite set of monopoles, let $V_{\infty}$ be the Green's function from Corollary \ref{coro:gf}, which  satisfies
\begin{equation}
-\Delta_{g_0} V_{\infty} = 2\pi\sum\limits_{m=1}^{m_0} \delta_{p_{m}}.  
\end{equation}
In the gluing procedure, we need to modify the above Green's function $V_{\infty}$ such that the metric on the neck matches up with the metric of the Tian-Yau parts. 
For this purpose, we analyze the asymptotic behavior of $V_{\infty}$ at the two ends of the neck region. By  Corollary \ref{coro:gf},   
 there are bounded harmonic functions $h_{+}$ and $h_{-}$ on $\TT\times\dR$ such that
\begin{align}
V_{\infty}(z)  = 
\begin{cases}
\frac{\pi(b_- + b_+)}{A} z  + \beta_- +   h_-,  & z\ll-1,
\\
-\frac{\pi(b_- + b_+)}{A}  z  + \beta_+  + h_+, &  z \gg 1,
\end{cases}
\end{align}
and  $h_{\pm}$ satisfy the asymptotic behavior $h_{\pm}(\bm{x}) = O(e^{-\epsilon_0 |z(\bm{x})|})$.

For fixed $\beta>0$, we define a new harmonic function on $\TT\times\dR$,
\begin{equation}
V_{\beta}(z) \equiv V_{\infty}(z) + k z + \beta
\end{equation}
such that the metric on the neck can be glued with the metric on the Tian-Yau space.
First, we need to match the slopes, that is, the slope parameter $k$ should be chosen as 
\begin{align}
k = \frac{\pi (b_- - b_+)}{A}.
\end{align}
Then near to the two ends of the neck region, the harmonic function $V_{\beta}$ can be written as
\begin{align}
\label{vbasy}
V_{\beta}(z) = \begin{cases}
 \frac{2\pi b_-}{A} z + \beta_- + \beta + h_- , &  z \ll -1,
\\
-\frac{2\pi b_+}{A} z + \beta_+ + \beta + h_+ , &  z \gg 1.
\end{cases} 
\end{align}
Using this potential, we next define the neck metric through the Gibbons-Hawking ansatz.
Letting $\mathcal{P}_{m_0}$ denote the set of monopole points, 
we note that $H^2( (\TT \times \mathbb{R}) \setminus \mathcal{P}_{m_0}, \dZ)$ has dimension $m_0 +1$ with generators being small spheres around the monopole points, and any torus of the form $\TT \times \{z'\}$ where $z'$ is any value of $z$ for which there are no monopole points. It is easy to see that the $2$-form $\frac{1}{2\pi}* dV_{\beta}$ attains integer values on these cycles, which implies that the cohomology class $[\frac{1}{2\pi}*dV_{\beta}]$ lies in the image of the natural inclusion 
\begin{align}
 H^2((\TT \times \mathbb{R}) \setminus \mathcal{P}_{m_0}, \dZ) 
\hookrightarrow  H^2((\TT \times \mathbb{R}) \setminus \mathcal{P}_{m_0}, \dR).
\end{align} 
Therefore, we let $\mathcal{N}_{m_0}^4$ be the total space of the $S^1$-bundle over $(\TT \times \mathbb{R}) \setminus \mathcal{P}_{m_0}$  
corresponding to the class $[\frac{1}{2\pi}*dV_{\beta}]$, completed by adding finitely many points corresponding to $\mathcal{P}_{m_0}$. Choose a connection $1$-form $\theta$ on $\mathcal{N}_{m_0}$ so that 
\begin{align}
d \theta = * dV_{\beta}.
\end{align}
Then applying the Gibbons-Hawking construction to $V_\beta$, we obtain a smooth hyperk\"ahler triple $\bm{\omega}^N$ over $\mathcal{N}_{m_0}^4$, which induces an incomplete hyperk\"ahler metric over the part in $\TT\times \mathbb{R}$ where $V_\beta$ is strictly positive.

For parameters $T_->0$ and $T_+>0$, we define,
\begin{equation}
\mathcal{N}_{m_0}^4(-T_-, T_+) \equiv \pi^{-1} \big(  U \cap \{-T_- < z < T_+\} \big),
\end{equation}
where $\pi : \mathcal{N}_{m_0}^4 \rightarrow U \equiv (\TT \times \mathbb{R}) \setminus \mathcal{P}_{m_0}$ is the bundle projection.

\begin{proposition}\label{gaugep}
There is a diffeomorphism 
\begin{align}
\Phi^N_{-} : (-T_-, -T_- +1, ) \times \Nil^3_{b_-}(\epsilon, \tau) \rightarrow \mathcal{N}_{m_0}^4(-T_-, -T_- +1),
\end{align} 
which preserves the $z$-coordinate, such that 
\begin{equation}
\label{pneqn}
(\Phi^N_{-})^*\theta= \theta_{b_-} + O(e^{ - \delta_N T_-}),
\end{equation}
as $T_- \rightarrow \infty$. 
Similarly, there is a diffeomorphism
\begin{align}
\Phi^N_{+} : (T_+ -1, T_+) \times \Nil^3_{-b_+}(\epsilon, \tau) \rightarrow\mathcal{N}_{m_0}^4(T_+ -1, T_+),
\end{align} 
which preserves the $z$-coordinate, such that 
\begin{equation}
(\Phi^N_{+})^*\theta= \theta_{-b_+} + O(e^{- \delta_N T_+}),
\end{equation}
as $T_+ \rightarrow \infty$. 
Furthermore, there exist triples of $1$-forms on the ends of the neck such that  
\begin{align}
\label{pneqn-}
(\Phi^N_{-})^*\bm{\omega}^N - \bm{\omega}_{b_-} = d ( \bm{a}^N_-)\\
(\Phi^N_{+})^*\bm{\omega}^N - \bm{\omega}_{-b_+} = d ( \bm{a}^N_+),
\end{align}
with 
\begin{align}
|\nabla^k \bm{a}^N_{\pm}| \leq C e^{- \delta_N |z|}
\end{align}
for any integer $k \geq 0$ and $\epsilon>0$, where $\delta_N>0$ is a  uniform constant in Proposition \ref{p:TY-asym}, $ \bm{\omega}_{b_-}$ and  $\bm{\omega}_{-b_+}$ are the hyperk\"ahler triples on the corresponding Calabi model spaces. 
\end{proposition}
\begin{proof}
We just deal with the negative end of the neck, the positive end is similar. 
Let 
\begin{align}
U_- = \{ p \in (\TT \times \mathbb{R} ) \setminus \{p_1, \dots, p_{m_0}\} \ | \ -\infty <  z < - T_-/2 \}.
\end{align}
Note that $U_-$ deformation retracts to $\TT$, so $H^2(U_-, \dZ) = \dZ$. The neck is a circle bundle over $U = (\TT \times \mathbb{R}) \setminus \{p_1, \dots, p_{m_0}\}$, and call the restriction to $U_-$ by $\mathcal{N}_{U_-}$. Note this bundle has Euler number $b_-$.

Over $U_-$ there exists another $S^1$-bundle explicitly identified with an open subset of the model space 
\begin{align}
\mathcal{N}_{b_-} = (-\infty, -T_- +1, ) \times \Nil^3_{b_-}(\epsilon, \tau),
\end{align}
with connection form $\theta_{b_-}$, which has curvature form  $- \frac{2 \pi b_-}{A} dx \wedge dy$, so this bundle also has Euler number $b_-$. 
From the exponential sheaf sequence, $H^1(U_-, \mathcal{E}^*) \cong H^2(U_-, \dZ)$, so there exists a bundle equivalence
\begin{align}
H : \mathcal{N}_{b_-} \rightarrow \mathcal{N}_{U_-}
\end{align}
which covers the identity map on the base, and such that the pullback bundle $H^*\mathcal{N}_{U_-} =  \mathcal{N}_{b_-}$.
The $1$-forms $H^* \theta$ and $\theta_{b_-}$ are therefore both connection forms on 
$\mathcal{N}_{b_-}$. Note that 
\begin{align}
d (H^* \theta) = H^* d \theta = H^*( * dV) = * dV_{\beta},
\end{align}
and
\begin{align}
d \theta_{b_-} = - \frac{2 \pi b_-}{A} dx \wedge dy. 
\end{align}
From the asymptotics on $V_{\beta}$ in \eqref{vbasy}, we have 
\begin{align}
d ( H^* \theta - \theta_{b_-}) = O( e^{\delta' z}), 
\end{align}
as $z \rightarrow - \infty$. By same method from the proof of Lemma \ref{lemma:diff}, we conclude that 
\begin{align}
d ( H^* \theta - \theta_{b_-}) = d a,
\end{align}
where $a = O ( e^{\delta' z})$, as $z \to - \infty$.
Therefore 
\begin{align}
d( H^* \theta - \tilde{\theta}_{b_-}) = 0
\end{align}
where $\tilde{\theta}_{b_-} = \theta_{b_-} + a$. Since 
$H^* \theta$ and $\tilde{\theta}_{b_-}$ are two connections with the same curvature form,
and since $H^1(U_-, \dR) \cong H^1(\TT, \dR) \cong  \dR \oplus \dR$, we conclude that
\begin{align}
H^* \theta -  \tilde{\theta}_{b_-} = df + p dx + q dy, 
\end{align}
for some function $f : U_- \rightarrow \dR$, and constants $p,q \in \dR$. 

Next, there exists a gauge transformation, that is, a mapping 
$G:  \mathcal{N}_{b_-} \rightarrow  \mathcal{N}_{b_-}$, covering the identity map, given by fiber rotation by $e^{if}$, so that 
\begin{align}
\label{p61e}
G^* H^* \theta -  \tilde{\theta}_{b_-} = p dx + q dy.
\end{align}
Then, by the discussion in Subsection \ref{ex:model-space}, there exists a mapping 
$R:  \mathcal{N}_{b_-} \rightarrow  \mathcal{N}_{b_-}$ which is the lift of a rotation on the torus, so that $R^* \theta_{b_-} = \theta_{b_-} - p dx - q dy$. Pulling back \eqref{p61e},
\begin{align}
R^* G^* H^* \theta -  R^* \tilde{\theta}_{b_-} = R^*(p dx + q dy). 
\end{align}
Since $R$ covers a rotation on the torus, the right hand side is invariant under $R$, 
so this can be rewritten as
\begin{align}
R^* G^* H^* \theta - \theta_{b_-}' = 0,
\end{align}
where
$\theta_{b_-}' = \theta_{b_-} + O(e^{\delta z})$ as $z \rightarrow - \infty$.
Then we define $\Phi^N_- = H \circ G \circ R$. The $z$ coordinate is not affected because $H$ and $G$ both cover the identity map, and $R$ covers a rotation on the torus.

Next, it follows from \eqref{pneqn} that the leading terms of the hyperk\"ahler triple on the neck agree with the model hyperk\"ahler triple for $z \ll 0$ (note we can allow $V$ to become negative, the triple is still defined). 
The same method from the proof of Lemma \ref{lemma:diff} then yields \eqref{pneqn-}.

\end{proof}

\subsection{The attaching maps and constraints}
\label{constraints}

We next define the ``attaching maps'' which will be used to construct the manifold $\mathcal{M}$. 
Let 
\begin{align}
DZ_- = \Phi^{TY}_- \Big( (T_-, T_- + 1) \times \Nil^3_{b_-} \Big) \subset X_{b_-}^4,
\end{align}
using the above coordinates on the end of $X_{b_-}^4$. Define
\begin{align}
\Psi_- : DZ_- \rightarrow \mathcal{N}_{m_0}^4
\end{align}
by $\Psi_- (\Phi^{TY}_-)^{-1} (z_-, p) = \Phi^N_-(z_- - 2 T_-, p)$.

Simlarly, let 
\begin{align}
DZ_+ =  \Phi^{TY}_- \Big( (T_+, T_+ + 1) \times \Nil^3_{b_+}  \Big) \subset X_{b_+}^4,
\end{align}
using the above coordinates on the end of $X_{b_+}^4$.
Define
\begin{align}
\Psi_+ : DZ_+ \rightarrow \mathcal{N}_{m_0}^4
\end{align}
be defined by
$\Psi_+  (\Phi^{TY}_+)^{-1}(z_+, p) = \Phi^N_+(2 T_+  - z_+, \psi(p))$, where $\psi_+: \Nil^3_{b_+} \rightarrow \Nil^3_{-b_+}$ is the diffeomorphism given by 
\begin{align}
\psi (x,y,t) = (-x,-y, -t).
\end{align}
We obtain the manifold $\mathcal{M}$ by gluing the pieces together using the attaching maps:
\begin{equation}
\mathcal{M} \equiv X_{b_-}^4(T_- +1) \bigcup_{\Psi_-} \mathcal{N}_{m_0}^4(-T_-, T_+) \bigcup_{\Psi_+} X_{b_+}^4(T_+ +1).
\end{equation}
The manifold $\mathcal{M}$ carries an orientation compatible with both Tian-Yau pieces, and we will fix this orientation in the following. 

Next, we want the potentials to agree up to the constant term in the damage zones after identifying the corresponding regions by the attaching maps. 
On $DZ_-$  we have 
\begin{align}
\Psi_-^* V_{\beta} &= \frac{2 \pi}{A} b_- (z_- - 2 T_- ) + \beta_- + \beta\\
& = \frac{2 \pi}{A} b_- z_- + \Big( -  \frac{2 \pi}{A} b_- (2 T_-) + \beta_- + \beta \Big),
\end{align}
which we want to equal to the leading terms of $V_-$, 
so we must have 
\begin{align}
\label{c1eq}
0 =  -  \frac{2 \pi}{A} b_- (2 T_-) + \beta_- + \beta.
\end{align}
Similarly, on the other damage zone $DZ_+$ we have 
\begin{align}
\Psi_+^* V_{\beta} &= -\frac{2 \pi}{A} b_+ (2T_+ - z_+) + \beta_+ + \beta\\
& = \frac{2 \pi}{A} b_+ z_+ + \Big(  -\frac{2 \pi}{A} b_+ (2T_+) + \beta_+ + \beta  \Big),
\end{align}
which we want to equal to the leading terms of $V_+$,
so we must have 
\begin{align}
\label{c2eq}
0 =  -  \frac{2 \pi}{A} b_+ (2 T_+) + \beta_+ + \beta .
\end{align}

\begin{remark}
If both $b_+ = 0$ and $b_- = 0$, then there is no constraint. This is the already known $\ALH \# \ALH$ gluing \cite{CCIII}, so we do not need to analyze this case further. 
\end{remark}

To summarize: the gluing procedure requires 
\begin{align}
\frac{4\pi b_\pm}{A}\cdot T_\pm &= \beta_\pm + \beta.
\end{align}
Immediately, the above constraints give $1$ free parameter $\beta>0$. So $T_->0$ and $T_+>0$ are completely determined by $\beta>0$ in the case $b_->0$ and $b_+>0$, i.e., 
\begin{align}
T_\pm &= \frac{A(\beta + \beta_\pm )}{4\pi b_\pm}.
\label{e:time-constraint}
\end{align}

\begin{remark}
We emphasize that we are fixing all the other gluing parameters so that only $\beta$ varies. We will prove some effective estimates in Section \ref{s:weight-analysis} and Section \ref{s:existence} which give uniform etimates for the linearized gluing operator $\mathscr{L}$ (defined in Section \ref{ss:outline}) and for sufficiently large $\beta\gg1$.
We also note that the estimates are unifrom  as long as other parameters vary in compact sets.
\end{remark}

\subsection{Gluing definite triples and topology of $\mathcal{M}$}

We have hyperk\"ahler triples 
\begin{align}
\begin{split}
\bm{\omega}^- &\equiv(\omega_1^-, \omega_2^-, \omega_3^-) \ \mbox{ on } X_{b_-}^4,\\
\bm{\omega}^N &\equiv (\omega_1^N, \omega_2^N, \omega_3^N) \ \mbox{ on } \mathcal{N}_{m_0}^4(-T_-,T_+),\\
\bm{\omega}^+ &\equiv (\omega_1^+, \omega_2^+, \omega_3^+) \ \mbox{ on } X_{b_+}^4.
\end{split}
\end{align}
We assume that $\omega_1$ is the K\"ahler form with respect to which the
tori are holomorphic on all three pieces. 
Next, we will glue these triples in the damage zones $DZ_-$ and $DZ_+$, to get a definite triple on $\mathcal{M}$. In this section, we show that we can moreover obtain a \textit{closed} definite triple, which is also very close to an ${\rm{SU}}(2)$-structure.

\begin{proposition}\label{p:gluing-definite-triple} There exist smooth triples of $1$-forms $\bm{a}^{\pm} \in \Omega^1(DZ_{\pm}) \otimes \dR^3$ satisfying
\begin{align}
\bm{\omega}^{\pm} + d \bm{a}^{\pm} =  \Psi_{\pm}^* \bm{\omega}^N  \mbox{ in } DZ_{\pm},
\end{align}
such that for any $k\in\dN$,
\begin{align}
| \nabla^k \bm{a}^{\pm}| \leq C_k e^{-\delta z_{\pm}} \mbox{ in } DZ_{\pm},
\end{align}
where $\delta>0$ and $C_k>0$ are uniform constants independent of $\beta$.
\end{proposition}
\begin{proof}
This follows upon combining Lemma \ref{lemma:diff} and Proposition \ref{gaugep}. 
\end{proof}

Let $\phi_{\pm}$ be cutoff functions such that 
\begin{align}
\phi_{\pm} =\begin{cases}
0 &\mbox{ for } z_{\pm} \leq T_{\pm}\\
 1 &\mbox{ for } z_{\pm} \geq T_{\pm} +1\\
\end{cases}.
\end{align}
Then we define 
\begin{align}
\bm{\omega}^{\mathcal{M}} = 
\begin{cases}
\bm{\omega}^-  &\mbox{ on }  X_{b_-}^4(T_-),
\\
\bm{\omega}^- + d \big( \phi_- \bm{a}^- \big)  &\mbox{ on } X_{b_-}^4(T_-,T_-+1),
\\
\bm{\omega}^N & \mbox{ on }\mathcal{N}_{m_0}^4(-T_- +1, T_+ -1),
\\
\bm{\omega}^+ + d \big( \phi_+ \bm{a}^+ \big) &\mbox{ on } X_{b_+}^4(T_+,T_+ + 1),
\\
\bm{\omega}^+  &\mbox{ on }  X_{b_+}^4(T_+).\\
\end{cases}\label{e:def-approx-triple}
\end{align}
\begin{corollary}
\label{c:sutt}The triple $\bm{\omega}^{\mathcal{M}}$ 
is a closed definite triple on $\mathcal{M}$. Furthermore, for any $k\in\dN$, there is some constant $C_k>0$ independent of the gluing parameter $\beta>0$ such that
\begin{align}
\Vert Q_{\bm{\omega}} - \Id \Vert_{C^{k}(\mathcal{M})} \leq C_k e^{-\delta_{q}\beta},
\end{align}
where $\delta_q>0$ is a uniform constant independent of $\beta$ and 
$Q_{\bm{\omega}}=(Q_{ij})$ is defined by 
\begin{equation}\frac{1}{2}\omega_i \wedge \omega_j=Q_{ij}\dvol_{\bm{\omega}^{\mathcal{M}}}.
\end{equation}
Here the norm is measured with respect to $g_{\beta}$, the Riemannian metric associated to $\bm{\omega}^{\mathcal{M}}$.
\end{corollary}

\begin{proof}This follows from Proposition \ref{gaugep} and Proposition \ref{p:gluing-definite-triple}.
\end{proof}

We next analyze some topological properties of the manifold $\mathcal{M}$. Note that we do not yet know that $\mathcal{M}$ is diffeomorphic to the $\K3$ surface. 
\begin{proposition}\label{p:topological-invariants} The compact oriented manifold $\mathcal{M}$ has the following 
topological properties:
\begin{align}
b_1(\mathcal{M}) = 0, \ \chi(\mathcal{M}) = 24, 
\ b_2^+(\mathcal{M}) = 3, \ b_2^-(\mathcal{M}) = 19.
\end{align}
\end{proposition}
\begin{proof}
We write the manifold $\mathcal{M}$ as the union of open sets $U \cup V$,
where  
\begin{align}
U = \mathcal{N}_{m_0}^4 (-T_-, T_+), \ V = X_{b_-}(T_- +1)  \sqcup  X_{b_+} (T_- +1),
\end{align}
where $X_{b_{\pm}} = Q_{b_{\pm}} \setminus \TT$, $m_0=b_- + b_+$, with $Q_{b_{\pm}}$ a del Pezzo surface
of degree $b_{\pm}$. Clearly, $U \cap V $ deformation retracts onto $\Nil_{b_-}^3 \sqcup \Nil_{b_+}^3$. 

Next, we claim that the de Rham cohomology $H^1(X_{b_{\pm}}) = 0$. To see this, we use the long exact sequence of a pair in de Rham cohomology
\begin{align}
  \label{lesp}
  \cdots \rightarrow H^k_c(Q_{b_{\pm}} \setminus \TT)
  \rightarrow H^k(Q_{b_{\pm}}) \rightarrow H^k (\TT) \xrightarrow{\phi}  H^{k+1}_c(Q_{b_{\pm}} \setminus \TT)
  \rightarrow \cdots,
\end{align}
see \cite[Chapter 11]{Spivak}.
Since $H^3( Q_{b_{\pm}}) = 0$, \eqref{lesp} yields an exact sequence 
\begin{align}
  \dots \rightarrow H^2( Q_{b_{\pm}}) \xrightarrow{i^*} H^2(\TT) \rightarrow H^3_c( Q_{b_{\pm}} \setminus \TT) \rightarrow 0.
  \end{align}
  Here the mapping $i^* : H^2( Q_{b_{\pm}})  \rightarrow H^2(\TT)$ is just the pullback under inclusion, which is dual to the mapping on homology $i_* :  H_2(\TT; \RR)\rightarrow   H_2( Q_{b_{\pm}};\RR)$. Since $\TT$ is a complex submanifold of a K\"ahler manifold, this latter mapping is injective, so the mapping $i^*$ is surjective, and
  by Poincar\'e duality  we conclude that
  \begin{align}
  H^1( Q_{b_{\pm}} \setminus \TT) \cong  H^3_c(Q_{b_{\pm}} \setminus \TT) = 0.
  \end{align}
Since we just showed that $H^1( X_{b_{\pm}}) = 0$, the Mayer-Vietoris sequence in cohomology for the pair $\{U, V\}$ yields an exact sequence 
\begin{align}
  \label{mvpair}
  0 \rightarrow H^1(\mathcal{M}) \rightarrow H^1(\mathcal{N}_{m_0}) \xrightarrow{i^*}
H^1 (\Nil_{b_-}^3 \sqcup \Nil_{b_+}^3) \cong H^1(\Nil_{b_-}^3) \oplus H^1 (\Nil_{b_+}^3).
  \end{align}
  The mapping $i^*$ is the pullback under inclusion of the two nilmanifold fibers of the neck at each end.
We claim that this mapping is injective. 
To see this, let $\mathcal{P}_{m_0}\equiv 
\{p_1,\ldots, p_{m_0}\}$ denote the monopole points in $B = \TT \times (-T_-, T_+)$, where $m_0=b_-+b_+$.
Then there are $\tilde{p}_j \in \mathcal{N}_{m_0}^4$ such that 
$\mathcal{N}_{m_0}^4 \setminus \widetilde{\mathcal{P}}_{m_0}$ is a circle bundle over $B \setminus \mathcal{P}_{m_0}$, 
\begin{align}
\label{neckfib}
  S^1 \longrightarrow \mathcal{N}_{m_0}^4 \setminus \widetilde{\mathcal{P}}_{m_0}
  \xrightarrow{\ \pi \ } B\setminus\mathcal{P}_{m_0}.
\end{align}
The Gysin sequence of \eqref{neckfib} begins with
\begin{align}
  \label{gys1}
  0 \rightarrow H^1 (  B\setminus\mathcal{P}_{m_0}) \xrightarrow{\pi^*}
H^1(  \mathcal{N}_{m_0}^4 \setminus \widetilde{\mathcal{P}}_{m_0}) \rightarrow \cdots
  \end{align}
  It is easy to see inclusion induces an isomorphism
 $H^1 (  B\setminus\mathcal{P}_{m_0}) \cong H^1(B) \cong \RR \oplus \RR$,
and similarly, $H^1 (   \mathcal{N}_{m_0}^4 \setminus \widetilde{\mathcal{P}}_{m_0})
\cong H^1 (  \mathcal{N}_{m_0}^4)$.  Then \eqref{gys1} becomes
\begin{align}
  \label{gys2}
  0 \rightarrow \mbox{span}\{dx, dy\} \xrightarrow{\pi^*}
H^1(  \mathcal{N}_{m_0}^4 ) \rightarrow \cdots
\end{align}
Together with Proposition \ref{nilbn}, and the exact sequence \eqref{mvpair}, we conclude that $i^* \pi^* dx$ and $i^* \pi^* dy$ are both nontrivial and are linearly independent in
$H^1(\Nil^3_{b_{-}} \sqcup \Nil^3_{b_{+}} )$, so $i^*$ is injective as claimed. Then \eqref{mvpair} implies that $b_1(\mathcal{M}) = 0$.
Since $\mathcal{M}$ is a compact orientable $4$-manifold, Poincar\'e duality also implies that $b_3(\mathcal{M}) =0$.

Next, it follows from the fibration \eqref{neckfib}
that $\chi( \mathcal{N}_{m_0}^4\setminus \widetilde{\mathcal{P}}_{m_0} ) = 0$, and therefore \begin{align}\chi(\mathcal{N}_{m_0}^4 ) = \# {\mbox{ of monopole points }} = m_0 = b_- + b_+.\end{align}

For a Tian-Yau space, it follows that 
\begin{align}
\chi(X_{b}^4) = \chi ( Q_b \setminus \TT) = \chi(Q_b) - \chi(\TT) = \chi(Q_b),
\end{align}
where $Q_b$ is a degree $b$ del Pezzo surface, so
\begin{align}
\chi(X_b^4) = \chi( Q_b \setminus \TT) = 12 - b
\end{align}
Note also that $\chi(\Nil_{b_-}^3) =\chi(\Nil_{b_+}^3) = 0$ since it is an orientable 3-manifold.
Then we have
\begin{align}
\chi (\mathcal{M}) = \chi( X_{b_-}^4 ) + \chi(\mathcal{N}) + \chi(X_{b_+}^4)
= (12 - b_-)  +  (b_+  + b_-) + (12 - b_+) = 24.
\end{align}
Since we have shown above that $b_1(\mathcal{M}) = b_3(\mathcal{M}) = 0$, this proves that $b_2(\mathcal{M})=22$.

Next, as we constructed in \eqref{e:def-approx-triple} the approximate definite triple $\bm{\omega}^{\mathcal{M}}\equiv (\omega_1, \omega_2, \omega_3)$, 
which are everywhere non-zero self-dual 2-forms forming a basis of $\Lambda^2_+$ at every point. This implies the bundle $\Lambda^2_+(\mathcal{M})$ is a trivial rank $3$ bundle.
Also, $\omega_1$ being non-zero everywhere means that there is an 
almost complex structure ($\omega_1 / |\omega_1|$ is a unit norm self-dual 2-form, which is equivalent to an orthogonal almost complex structure). By Corollary \ref{c:sutt}, for $\beta \gg 1$, the rank 2 subbundle $V \subset \Lambda^2_0$, given by the orthogonal complement of $\omega_1 / |\omega_1|$ is trivial.
Then $0 = c_1 ( V \otimes \dC) = c_1(T\mathcal{M},J)^2$, and the Hirzebruch signature theorem implies that
\begin{align}
2 \chi(\mathcal{M}) + 3 \tau(\mathcal{M})  = \int_{\mathcal{M}} c_1^2 = 0,
\end{align}
from which it follows that $\tau(\mathcal{M})  = -16$. Therefore, $b_2^+(\mathcal{M})=3$ and $b_2^-(\mathcal{M})=19$.

\end{proof}

\section{Geometry and regularity of the approximate metric}

\label{s:gluing-space}
In this section, we will give a detailed analysis of the geometry of $(\mathcal{M}, g_{\beta})$.

\subsection{Notations}
\label{ss:notations}

Since the arguments in the next sections are very tedious and involved, in this subsection we will list some fixed constants and make necessary conventions which will be frequently used in the later proofs. Throughout the rest of the paper, the notation $a_j \to a$ will implicitly mean the limit as $j \rightarrow \infty$, unless otherwise noted. 

\subsubsection{Tian-Yau spaces and their asymptotic rates}

To start with,
for two positive integers 
\begin{equation}
b_-, b_+ \in \{1,2,\ldots, 9\},
\end{equation}
let $(X_{b_{-}}^4,g_{b_-}, q_{-})$ and 
$(X_{b_{_+}}^4,g_{b_+}, q_{+})$
be fixed hyperk\"ahler Tian-Yau spaces
with reference points $q_{-}\in X_{b_-}^4$ and $q_{+}\in X_{b_+}^4$ such that their degrees are $b_{-}$ and $b_{+}$ respectively. See Section \ref{s:Tian-Yau} for the definition of a Tian-Yau space and the natural coordinate outside a large compact subset. 
On $(X_{b_{-}}^4, g_{b_-}, q_-)$ and $(X_{b_{+}}^4, g_{b_+}, q_+)$, 
there are diffeomorphisms
\begin{equation}
\Phi_\pm: [\zeta_0^\pm,+\infty)\times \Nil_{b_\pm}^3 \rightarrow X_{b_\pm}^4 \setminus K_\pm
\end{equation}
between the Gibbons-Hawking space which models over a flat cylinder $\TT\times\dR$. 
We define the definite constants $D_0^\pm$ 
by
\begin{align}
D_0^{\pm} \equiv \text{The distance between} \ q_\pm \ \text{and the level set}\  \{\bm{x}\in X_{b_\pm}^4|z_\pm(\bm{x})=\zeta_0^\pm\}.
\end{align}
Proposition \ref{p:TY-asym} shows that 
there are some positive constants 
\begin{equation}
\dl>0 ,\ \dr>0
\end{equation}
such that for any $k\in\dN$,
\begin{equation}
|\nabla_{g_{\mathcal{C}}^\pm}^k(\Phi^*\omega_{TY}^{\pm}-\omega_{\mathcal C}^\pm)|_{g_{\mathcal{C}}^\pm}=O(e^{-\dl z_\pm}),
\end{equation} 
where $\omega_{TY}^{\pm}$, $\omega_{\mathcal{C}}^\pm$ 
are the K\"ahler forms on the Tian-Yau spaces $X_{b_\pm}^4$,and the Calabi model spaces respectively.

\subsubsection{Some notations about the neck region}

Now we fix some parameters in the neck region for the convenience of our discussions in the later sections.

Let $\mathcal{P}_{m_0}\equiv\{p_1,\ldots, p_{m_0}\}$ be the set of monopoles on the flat cylinder $(\dT^2\times\dR,g_0)$ with coordinates $(x,y,z)$ such that 
\begin{enumerate} 
\item $z(p_1)=0$.

\item There are definite constants 
\begin{equation}
\iota_0>0,\ T_0>0
\end{equation}
 such that for all $k\neq l$, we have
\begin{equation} \iota_0  \leq d_{g_0}(p_k,p_l) \leq T_0.\end{equation}

\end{enumerate}
Around each monopole $p_m \in \mathcal{P}_{m_0}$, we define the associated distance function 
 \begin{equation}
 d_m(\bm{x})\equiv d_g(p_m,\bm{x}),\ p_{m}\in\mathcal{P}_{m_0},\ \bm{x}\in(\mathcal{M},g).
 \end{equation}
In our proof, the following notations will also be needed. 
We fix definite constants 
\begin{equation}
\iota_0'>0,\ T_0'>0\end{equation}
 such that for all $1\leq m< l \leq m_0$, then in terms of the Gibbons-Hawking metric of the neck region, we have 
\begin{equation}
\iota_0' \cdot (\beta)^{\frac{1}{2}}\leq d_{g}(p_m, p_l) \leq  T_0' \cdot (\beta)^{\frac{1}{2}}.
\end{equation}
We have already defined in Section \ref{s:approx-triple} the Gibbons-Hawking metric in the neck region $\mathcal{N}_{m_0}^4$. Given a gluing parameter $\beta>0$,  by Theorem \ref{t:harmonic-function-cylinder},  
 the defining Green's function $V_{\beta}$ satisfies the asymptotic property that there are constants 
 \begin{equation}
\el>0, \ \er>0
 \end{equation}
such that for any $k\in\dN$ we have
\begin{equation}
\Big|\nabla^k \Big(V_{\beta}- \Big(\frac{2\pi b_-}{A}z + \beta\Big) \Big) \Big| = O(e^{ \el z}), \qquad  z\to-\infty
\end{equation}
and 
\begin{equation}
\Big|\nabla^k \Big(V_{\beta}- \Big(-\frac{2\pi b_+}{A}z + \beta\Big)\Big) \Big|= O(e^{-\er z}), \qquad z\to+\infty.
\end{equation}

The following functions defined on $\mathcal{N}_{m_0}^4$ as well as on the Tian-Yau pieces are crucial in analyzing the rescaled limits and the definition of the weight function in the next section, which naturally comes from the construction of the model metric:

\begin{enumerate}

\item On the negative part of the neck region, we define the function  
\begin{equation}
L_-(\bm{x})\equiv \Big(\frac{2\pi b_-}{A} \cdot z(\bm{x}) +\beta\Big)^{\frac{1}{2}},\ -T_-\leq z(\bm{x})<0 \end{equation}

\item  On the positive part of the neck region, we define the function  
\begin{equation}
L_+(\bm{x})\equiv \Big(-\frac{2\pi b_+}{A} \cdot z(\bm{x}) +\beta\Big)^{\frac{1}{2}},\ 0\leq z(\bm{x})<T_+ \end{equation}

\item For $\bm{x}\in\mathcal{M}$ located in the end region of $X_{b_{-}}^4$ and  satisfy $\zeta_0^{-} \leq z_{-}(\bm{x}) \leq T_-$, we define
\begin{equation}\underline{L}_-(\bm{x}) \equiv \Big(\frac{2\pi b_{-}}{A} \cdot z_{-}(\bm{x})  \Big)^{\frac{1}{2}}\end{equation}

\item For $\bm{x}\in\mathcal{M}$ located in the end region of $X_{b_{+}}^4$ and  satisfy  $\zeta_0^{+} \leq z_+(\bm{x}) \leq T_+$, we define
\begin{equation}\underline{L}_+(\bm{x}) \equiv \Big(\frac{2\pi b_{+}}{A} \cdot z_{+}(\bm{x})\Big)^{\frac{1}{2}}.\end{equation}

\end{enumerate}

\subsubsection{Subdivision of the manifold $\mathcal{M}$}

\label{sss:subdivision}

Fix a gluing parameter $\beta>1$,  the manifold $(\mathcal{M},g_{\beta})$ will be divided into the following $9$ regions depending on the different collapsing behaviors of metric $g$: 

\begin{align*}
&\I: \ d_m(\bm{x}) \leq \beta^{-\frac{1}{2}} \ \text{for some}\ 1\leq m\leq m_0
\\
& \hspace{1cm} \mbox{(in the neck, very close to a monopole point) }\\
&\II: \ 2\beta^{-\frac{1}{2}}\leq d_m(\bm{x})\leq \frac{\iota_0'}{4}\cdot(\beta)^{\frac{1}{2}}  
\text{ for some}\ 1\leq m\leq m_0
\\
& \hspace{1cm} \mbox{(in the neck, not close, but not too far from any monopole point)  }\\
&\III: \  z(\bm{x})\in[-m_0T_0,m_0T_0] \ \text{and } d_m(\bm{x}) \geq  \frac{\iota_0'}{2} \cdot (\beta)^{\frac{1}{2}} \text{ for all}\  1\leq m \leq m_0
\\
& \hspace{1cm} \mbox{(in a bounded region of the neck, but far from any monopole point)}\\
  &\IV_{-} : \  z(\bm{x})\in[-T_-/2,-2m_0T_0]  \\
&\hspace{1cm} \mbox{(in the negative end region of the neck)}\\
&\IV_{+}: \ z(\bm{x})\in[2m_0T_0, T_+/2] \\
&\hspace{1cm} \mbox{(in the positive end region of the neck)}\\
&\V_{-}: \  \bm{x}\in X_{b_-}^4\ \text{and}\ 2 \zeta_0^{-} \leq z_-(\bm{x})\leq T_-
\\
&\hspace{1cm} \mbox{(in the end region of $X_{b_-}^4$)}\\
&\V_{+}: \  \bm{x}\in X_{b_+}^4\ \text{and} \ 2 \zeta_0^{+} \leq z_+(\bm{x})\leq T_+
\\
&\hspace{1cm} \mbox{(in the end region of $X_{b_+}$)}\\
&\VI_{-}: \ \bm{x}\in \overline{B_{D_0^-}(q_-)}\subset  X_{b_-}^4
\\
&\hspace{1cm} \mbox{(in the bounded part of $X_{b_-}^4$)}\\
&\VI_{+}: \ \bm{x}\in \overline{B_{D_0^+}(q_+)}\subset  X_{b_+}^4
\\
&\hspace{1cm} \mbox{(in the bounded part of $X_{b_+}^4$)}.
\end{align*}
We note that for $\bm{x} \in\IV_{\pm}$, we have 
\begin{align}
 T_0' (\beta)^{\frac{1}{2}} \leq  d_m(\bm{x})\leq R_{\pm} \text{ for all}\ 1\leq m \leq m_0,
\end{align}
where
\begin{align}
R_- &\equiv \sup\Big\{d_{g}(x,p_1) \Big| -T_- \leq z(\bm{x}) \leq 0 \Big\},
\\
R_+ &\equiv \sup\Big\{d_{g}(x,p_1) \Big| 0 \leq z(\bm{x}) \leq T_+ \Big\}.
\end{align}
Immediately, there is some constant $C>0$ (independent of $\beta$) such that
\begin{equation}
 C^{-1} \beta^{\frac{3}{2}}\leq R_{\pm} \leq C \beta^{\frac{3}{2}}.
\end{equation}
\begin{remark}
\label{re:subdivision}
Notice that the above regions do not completely cover the manifold $\mathcal{M}$. However, each gap region shares the geometric behavior with the adjacent regions in the above subdivision. Therefore, the curvature estimates and the rescaled geometries in each gap region will be the same as in the adjacent regions, so we will ignore these gap regions in the following.
\end{remark}

\subsection{Regularity of the approximate metrics}
\label{ss:regularity}

In this subsection, we prove uniform curvature estimates on $\mathcal{M}$ which will be crucial in showing that certain rescalings of the approximate metric have bounded curvature. We will show two different ways to understand the regularity. 

The first way is to directly compute the curvature tensors. 
Since the Gibbons-Hawking metric has an explicit form in terms of the defining harmonic function, the curvature estimates just follow from straightforward calculations. 
The following lemma gives sharp curvature estimates for every point on $\mathcal{M}$.
\begin{lemma}
\label{l:curvature-estimates}

The following uniform curvature estimates hold for every point in $\mathcal{M}$:
\begin{enumerate}
\item 
Let $r(\bm{x})$ denote the Euclidean distance to the monopole points, then there exists constants $C>0$ so that
such that for each $1\leq m\leq m_0$
for every $\bm{x}\in B_{r_0}(p_m)$ with 
$r_0\equiv \frac{1}{2}\InjRad_{g_0}(\TT)$,
the following curvature estimates hold,
\begin{align}
|\Rm|(\bm{x}) \leq
\begin{cases}
C\beta,  & 0\leq r(\bm{x})  < \beta^{-1},
\\
\frac{C}{\beta^2 r(\bm{x})^3}, & \beta^{-1}\leq  r(\bm{x}) <  r_0.
\end{cases}\label{e:curvature-estimate-r}
\end{align}
In terms of the intrinsic distance function with respect to the Riemannian metric $g_j$, 
\begin{align}
|\Rm|(\bm{x}) \leq
\begin{cases}
C\beta,  &  0\leq r(\bm{x})  < \beta^{-1},
\\
\frac{C}{\beta^{\frac{1}{2}} d_{m}(\bm{x})^3}, &  \beta^{-1}\leq  r(\bm{x}) <  r_0.
\end{cases}\label{e:curvature-estimate-dist}
\end{align}

\item If $\bm{x}$ is in the neck region but has some definite distance away from the monopoles,
the following curvature estimates hold for some uniform constant $C>0$,
\begin{align}|\Rm|(\bm{x}) \leq 
\begin{cases}
\frac{C}{\beta^2 z(\bm{x})}, & \frac{r_0}{10} < |z(\bm{x})| < \beta,\\
\frac{C}{\beta^3}, & -T_1\leq z(\bm{x}) \leq -\beta \text{and}\ \beta \leq z(\bm{x}) <T_2.
\end{cases}
\end{align}

\item 
For $\bm{x} \in X_{b_{\pm}} (T_{\pm}) \subset \mathcal{M}$, there is a constant $C$ so that
\begin{align}
|\Rm|(\bm{x}) \leq
\begin{cases}
C   & d(\bm{x}) < \zeta_{\pm}
\\
\frac{C}{d(\bm{x})^2} & d (\bm{x}) \geq \zeta_{\pm},
\end{cases}\label{e:curvature-estimate-TY}
\end{align}
where $d(\bm{x})$ is the distance to a base point in $X_{b \pm}$.

\item 
For $\bm{x} \in DZ_{\pm} \subset \mathcal{M}$, there is a constant $C>0$ so that 
\begin{align}
|\Rm|(\bm{x}) \leq \frac{C}{\beta^{3}}.
\end{align}
\end{enumerate}

\end{lemma}

\begin{remark}
The curvature estimates in Lemma \ref{l:curvature-estimates} are sharp in the following sense. The second estimate in \eqref{e:curvature-estimate-r} and \eqref{e:curvature-estimate-dist} corresponds to the curvature behavior of the Taub-NUT metric which is exactly of cubic decay. The curvature estimate in \eqref{e:curvature-estimate-TY} is sharp as well because the curvatures decay quadratically in the end of a complete Tian-Yau space. 
\end{remark}

\begin{proof}
The proof only requires straightforward calculations, so we only sketch the calculations. 
We use the following formula for the pointwise norm squared of the curvature of a Gibbons-Hawking metric
\begin{align}
\label{ghcf}
|\Rm|^2 = \frac{1}{2} V_{\beta}^{-1} \Delta^2 ( V_{\beta}^{-1}),
\end{align}
see \cite{GW}. We just need to consider the case of $1$ monopole point located at the origin,  the case of several monopole points follows easily from this case. Let $r_0\equiv\frac{1}{2}\InjRad_{g_0}(\TT)$, then we have the expansion 
\begin{align}
V_{\beta}(\bm{x}) = \frac{1}{2r(\bm{x})} + \beta + h(\bm{x}), \ x\in B_{r_0}(0^3),
\label{e:V-close-to-0}\end{align}
where $h$ is a bounded harmonic function. 

First, we estimate the curvature in the case $r(\bm{x})<\frac{1}{\beta}$.  By \eqref{e:V-close-to-0},
\begin{align}
V_{\beta}^{-1}(\bm{x}) = \frac{2r}{1 + 2r \beta + 2r h }, 
\end{align}
so it follows that
\begin{align}
V_{\beta}^{-1} = \frac{2r}{1 + 2r \beta + 2r h } 
\leq  C r ( 1 - 2r \beta + 4r^2 \beta^2 + 8 r^3 \beta^3 ) 
\end{align}
for $r < \beta^{-1}$, then 
\begin{align}
|V_{\beta}^{-1} \Delta^2 (V_{\beta}^{-1}) | \leq C r \beta^3,
\end{align}
and the first claimed estimate follows from this. 

Before showing the curvature estimates in other regions, we relate the intrinsic distance function $d_m(\bm{x})$ and the Euclidean radial function $r(\bm{x})$. By directly estimating the integral of $\sqrt{V_{\beta}}$, we have that
\begin{align}
\begin{split}
\frac{1}{C'}\cdot \sqrt{r(\bm{x})}&\leq d_m(\bm{x}) \leq C' \sqrt{r\bm(x)}, \ \hspace{17pt}   r(\bm{x})<\beta^{-1},
\\
\frac{1}{C'}\cdot\beta^{\frac{1}{2}}\cdot r(\bm{x})&\leq d_m(\bm{x}) \leq C'\cdot \beta^{\frac{1}{2}}\cdot r(\bm{x}), \  r(\bm{x}) \geq \beta^{-1},
\end{split}
\end{align}
 where $C'>0$ is some universal constant. So the first part of the curvature estimate in \eqref{e:curvature-estimate-dist} immediately follows.

Next, let $\bm{x}\in B_{r_0}(0^3)$ satisfy $ r(\bm{x}) \geq \beta^{-1}$. Substituting \eqref{e:V-close-to-0} into \eqref{ghcf}, then similar expansion formula shows that for some uniform constant $C>0$,
\begin{equation}
|\Rm|(\bm{x})\leq\frac{C}{\beta^2 r^3(\bm{x})}.
\end{equation}
Correspondingly in terms of the intrinsic distance function, the curvature estimate turns out to be
\begin{equation}|\Rm|(\bm{x})\leq \frac{C}{\beta^{\frac{1}{2}} d_{m}(\bm{x})^3}.
\end{equation}

The above in fact covers the curvature estimates in Region I and Region II.

From now on, we consider the case that $\bm{x}$ is in the neck region satisfying $\frac{r_0}{10}\leq |z(\bm{x})| \leq \beta$. In this case, the harmonic function $V_{\beta}$ has the expansion,
\begin{align}
V_{\beta}(\bm{x})=
\begin{cases}
\frac{2\pi b_-z(\bm{x})}{A}+h_-(\bm{x})+\beta, & -T_-\leq z(\bm{x})\leq -\zeta_0
\\
h(\bm{x})+\beta, & -\zeta_0\leq z(\bm{x})\leq \zeta_0,
\\
-\frac{2\pi b_+z(\bm{x})}{A} + h_+(\bm{x}) + \beta, & \zeta_0 \leq z(\bm{x}) \leq T_+.
\end{cases}
\end{align}
We apply the above expansion to the curvature formula \eqref{ghcf}, then we obtain the following curvature estimate
\begin{equation}
|\Rm|(\bm{x}) \leq \frac{C}{\beta^2 z(\bm{x})},
\end{equation}
where $C>0$ is a uniform curvature estimate. 
Similarly, one can calculate that in the damage zones, 
\begin{equation}
|\Rm|(\bm{x}) \leq \frac{C}{\beta^3}
\end{equation}
for some uniform constant $C>0$. Note that the cutoff function and its derivatives up to third order are uniformly bounded, the curvature of the glued metric is therefore also of order $\beta^{-3}$ in the damage zone region.

Next, we recall from Section \ref{ex:model-space} that 
for the model spaces, the defining harmonic functions are $V_-(\bm{x}) = \frac{2\pi b_- z(\bm{x})}{A}$ and $V_+(\bm{x}) = \frac{2\pi b_+ z(\bm{x})}{A}$, so \eqref{ghcf} implies that 
\begin{align}
|\Rm|(\bm{x}) \leq  \frac{C}{d(\bm{x})^2}, 
\end{align} 
for some uniform constant $C>0$, so the complete end of the model space has exactly inverse quadratic curvature decay. It follows from Proposition \ref{p:TY-asym} that the Tian-Yau metric does also.

\end{proof}

The curvature estimates in Lemma \ref{l:curvature-estimates} relies on the explicit formulas of the Gibbons-Hawking ansatz. 
For the sake of conceptually understanding the collapsing behavior, we introduce the following $\epsilon$-regularity theorem for collapsed Einstein manifolds due to Naber and the fourth author of this paper (see \cite{NaberZhang} for more details).   
 
\begin{theorem}
[Naber-Zhang, \cite{NaberZhang}]\label{t:collapsed-eps-reg}
Let $(M^n, g, p)$ satisfy $\Ric_g \equiv \lambda g$ and $|\lambda|\leq n-1$.
 Given a manifold $(Z^k,z^k)$ with $k=\dim(Z^k)< n$, there are  uniform constants $\delta_0>0$, $w_0>0$ and $C_0>0$ which depend only on $n$ and the geometry of $B_1(z^k)$ such that the following property holds: 
if 
\begin{equation}
d_{GH}(B_2(p), B_2(z^k)) < \delta_0,
\end{equation}
then 
the group $\Gamma_{\delta_0}(p) \equiv \Image[\pi_1(B_{\delta_0}(p))\to \pi_1(B_2(p))]$ has a nilpotent subgroup $\mathcal{N}$ of index bounded by $w_0$ such that 
$\rank(\mathcal{N}) \leq n-k$.

Furthermore, if $\rank(\mathcal{N})=n-k$, then $\sup\limits_{B_1(p)}|\Rm|\leq C_0$. Conversely, if  $\sup\limits_{B_3(p)}|\Rm|\leq C_0$, then $\rank(\mathcal{N})=n-k$.

\end{theorem}

\begin{remark}
Given a finitely generated nilpotent group $\mathcal{N}$, let $\mathcal{N}=\mathcal{N}_0\rhd \mathcal{N}_1\rhd\ldots\rhd\mathcal{N}_m=\{e\}$ be the lower central series with abelian factor groups $\mathcal{N}_{j-1}/\mathcal{N}_j$, where $\mathcal{N}_{j+1}\equiv[\mathcal{N},\mathcal{N}_j]$ are the commutator subgroups. Then the nilpotent rank of $\mathcal{N}$ is defined as the sum of the ranks of the abelian factors, i.e.
\begin{equation}
\rank(\mathcal{N})\equiv \sum\limits_{j=1}^m \rank(\mathcal{N}_{j-1}/\mathcal{N}_j).
\end{equation}

\end{remark}

\begin{remark}
If the Einstein assumption is replaced with bounded Ricci curvature, then the uniform curvature bound can be replaced with bounded $C^{1,\alpha}$-covering geometry for any $0<\alpha<1$. This can be used in analyzing the regularity of the damage zones.  
\end{remark}

In fact, theorem \ref{t:collapsed-eps-reg} has a quick proof in the special case of {\it codimension-1 collapse} which exactly applies in our case. For the  readers' convenience, we give the statement and the proof here.
\begin{lemma}\label{l:codim-1}
Let $(M_j^n,g_j,p_j)$ be a sequence of Einstein manifolds with $|\Ric_{g_j}| \leq \epsilon_j \to 0$ such that 
\begin{equation}
(M_j^n,g_j,p_j)\xrightarrow{GH}\dR^{n-1}\label{e:converge-euclidean}
\end{equation}
and $\Gamma_{2}(p_j)\equiv\Image[\pi_1(B_2(p_j))\to\pi_1(M_j^n)]$ is of infinite order. Then for any $R>0$,
\begin{equation}
\sup\limits_{B_R(p_j)}|\Rm| \leq \frac{C_0(n)}{R^2}.\label{e:curvature-est}
\end{equation}

\end{lemma}

\begin{remark}
Simple rescaling and contradicting arguments imply theorem \ref{t:collapsed-eps-reg} in the case $k=n-1$, which is an effective version of the lemma.
\end{remark}

\begin{proof}
Let $(\widetilde{M_j^n},\tilde{g}_j,\Gamma_j, \tilde{p}_j)$
be the Riemannian universal covers of $(M_j^n,g_j)$ which converge to the limit product space $(\dR^{n-1}\times Y, \tilde{d}_{\infty},\Gamma_{\infty},\tilde{p}_{\infty})$ in the equivariant Gromov-Hausdorff topology, where $\Gamma_j\equiv\pi_1(M_j^n)$ and $\Gamma_j\to\Gamma_{\infty}\leq \Isom(\dR^{n-1}\times Y)$. See Section 3 of \cite{Fukaya-Yamaguchi} for the precise definition of the equivariant Gromov-Hausdorff convergence. 
In summary, we have the following diagram
\begin{equation}
\xymatrix{
(\widetilde{M_j^n}, \tilde{g}_j,  \tilde{p}_j)
\ar[rr]^{eqGH}\ar[d]_{\pr_j} &   & \dR^{n-1}\times Y \ar [d]^{\pr_{\infty}}
\\
 (M_j^n, g_j ,  p_j )\ar[rr]^{GH} &  & \dR^{n-1},
 }
\end{equation}
where the covering maps $\pr_j:\widetilde{M_j^n}\to M_j^n$ converge to a natural projection map $\pr_{\infty}: \dR^{n-1}\times Y \to \dR^{n-1}$. 

The main part is to prove the claim that $Y$ is isometric to $\dR$. 

Applying Cheeger-Colding's quantitative splitting theorem (see \cite{ChC}), the convergence assumption \eqref{e:converge-euclidean} implies that for any fixed $R>0$, there are harmonic splitting maps $\Phi_j\equiv(u_j^{(1)},\ldots, u_j^{(n-1)}): B_{10R}(p_j)\to \dR^{n-1}$ 
which realize the Gromov-Hausdorff maps
such that 
\begin{equation}
\sum\limits_{\alpha,\beta=1}^{n-1}\fint_{B_{5R}(p_j)}
|\langle\nabla u_j^{(\alpha)}, \nabla u_j^{(\beta)}\rangle - \delta_{\alpha\beta}| + \sum\limits_{\alpha=1}^{n-1}\fint_{B_{5R}(p_j)}
|\nabla^2 u_j^{(\alpha)}|^2 \to0.
\end{equation}
Let $\widetilde{\Phi}_j\equiv(\tilde{u}_j^{(1)},\ldots, \tilde{u}_j^{(n-1)})$
be the lifted harmonic functions on the universal covers, then the volume comparison theorem implies that
\begin{equation}
\sum\limits_{\alpha,\beta=1}^{n-1}\fint_{B_{5R}(\tilde{p}_j)}
|\langle\tilde{\nabla} \tilde{u}_j^{(\alpha)}, \tilde{\nabla} \tilde{u}_j^{(\beta)}\rangle - \delta_{\alpha\beta}| + \sum\limits_{\alpha=1}^{n-1}\fint_{B_{5R}(\tilde{p}_j)}
|\tilde{\nabla}^2 \tilde{u}_j^{(\alpha)}|^2 \to 0.
\end{equation}
By the definition of the splitting maps, $Y$ is the Gromov-Hausdorff limit of the level sets of the lifted splitting maps $\widetilde{\Phi}_j^{-1}(0^{n-1})$.
Since $\Gamma_{2}(p_j)\leq\pi_1(M_j^n)$ is of infinite order which acts on $\widetilde{M}_j^n$ isometrically and discretely, the limit space $Y$ must be non-compact.

On other hand hand, notice that 
$\widetilde{\Phi}_j^{-1}(0^{n-1})$ is invariant under the deck transformation group $\Gamma_j$ and $\Gamma_j$ converge to some limiting group $\Gamma_{\infty}\leq \Isom(\dR^{n-1}\times Y)$ such that $\pr_{\infty}$ is given by the quotient $(\dR^{n-1}\times Y)/\Gamma_{\infty}=\dR^{n-1}$. Hence $\Gamma_{\infty}\leq\Isom(Y)$ and $\Gamma_{\infty}$ acts homogeneously on $Y$. 

Therefore, by standard arguments, the noncompact homogeneous space $Y$ admits a line (see \cite{ChGr} or lemma 2.4 in \cite{NaberZhang}).  
The Ricci curvature assumption implies that $Y$ is isometric to $\dR$. 
This completes the proof of the claim.

The curvature estimate \eqref{e:curvature-est} immediately follows from the $\epsilon$-regularity theorem for noncollapsed Einstein manifolds (for example see Section 7 in \cite{ChC1}).

\end{proof}

\subsection{Rescaled geometries}

\label{ss:rescaled-limits}

In this subsection, we will focus on the rescaled geometry of each region defined in Section \ref{sss:subdivision}, which can be viewed as a geometric preparation for defining the weighted H\"older space. 
In this direction, a necessary technical preparation is to rescale $(\mathcal{M}, g)$ by correctly choosing
some rescaled metric $\tilde{g} = \lambda^2 g$
 such that the weighted H\"older space in the rescaled space is much easier to analyze.  

In our context, we will discuss a sequence $(\mathcal{M}_j, g_j)$ with a sequence of gluing parameters $\beta_j \to \infty$. In the remaining part of this section, we will specify the following way of rescaling which will be consistent with the definition of the weight function. 
For every $\bm{x}_j \in \mathcal{M}_j$, we will choose the rescaling factors $\lambda_j>0$ and the corresponding  rescaled metric $\tilde{g}_j = \lambda_j^2 g_j$ we have the convergence 
\begin{equation}
(\mathcal{M}_j, \tilde{g}_j, \bm{x}_j) \xrightarrow{GH} (\mathcal{M}_{\infty}, \tilde{g}_{\infty}, \bm{x}_{\infty}).
\end{equation}
In the meanwhile, for applying the delicate tools in analysis, necessarily we need to improve the Gromov-Hausdorff convergence to some convergence with higher regularity. 
To this end,
we will select subdomains $U_j\subset \mathcal{M}_j$ which is of almost full measure such that the above convergence keeps Riemann curvatures uniformly bounded in $U_j$. The main tool of proving the curvature estimates is given by Lemma \ref{l:curvature-estimates} and Lemma \ref{l:codim-1}.

Our main task is to appropriately define the rescaling factors $\lambda_j>0$ which depends on the different regions in the definition of the weight function. 
The primary scenario is the following: while $d_{g_j}(p_m,\bm{x}_j)$ is increasing and $\bm{x}_j$ is moving from the monopoles in the neck region to the Tian-Yau pieces, the limiting geometries of the rescaled limits vary in a natural way. First, around the monopoles, the rescaled limit is the standard Taub-NUT space such that the $S^1$-fiber at infinity equals $1$. The advantage of rescaling in this way is that the local geometry around the monopoles can be captured in the rescaled limit.
When $d_{g_j}(p_m,\bm{x}_j)$ is increasing, the length of the $S^1$ of the Taub-NUT space is decreasing such that the rescaled limit will collapse to $\dR^3$. When $\bm{x}_j$ is farther from the monopoles, the size of the $\TT$-fiber will be shrinking such that the rescaled limit will become $\dT^2 \times \mathbb{R}$. Eventually when $\bm{x}_j$ is located in the Tian-Yau pieces, we choose the original scale so that we will obtain a complete Tian-Yau space.

\vspace{0.5cm}

\noindent
{\bf Region $\I$:} 

\vspace{0.5cm}

In this subsection, 
we focus on the blowing-up geometry around each monopole in the neck region $\mathcal{N}_{m_0}^4(-T_-, T_+)$.
We will prove that, by correctly rescaling the Gibbons-Hawking metric defined in the above section, 
the blowing-up limit around each monopole is the Taub-NUT space.

Let $(\TT\times\dR, g_0)$ be a cylinder with a flat product metric $g_0$. Given a constant $\beta>0$, let $V_{\beta}$ be a harmonic function such that 
\begin{align}
\begin{split}
-\Delta_{g_0} V_{\beta} &=  2\pi \sum\limits_{m = 1}^{m_0} \delta_{p_m} \\
V_{\beta}(\bm{x}) &= \frac{1}{2|\bm{x}-p_m|} + h_m(p) + \beta, \  \bm{x} \in B_{r_0}(p_m),
\end{split}
\end{align}
where each $h_m$ is a bounded harmonic function, 
$r_0 \equiv \frac{1}{2}\min\{d_0, i_0\}$ and
\begin{align}
\begin{split}
d_0 &\equiv \min\limits_{1\leq m<l\leq m_0} d_{g_0}(p_m, p_l) \\
i_0 &\equiv \InjRad_{g_0}(\TT \times \mathbb{R}).
\end{split}
\end{align}
Let $(\mathcal{N}_{m_0}^4, g_{\beta})$ be the Gibbons-Hawking space defined by
\begin{equation}
g_{\beta} \equiv V_{\beta} g_{\TT\times\dR} + V_{\beta}^{-1}\theta^2,
\end{equation}
where $\theta$ is a connection $1$-form with
\begin{equation}
d\theta = *d V_{\beta}.
\end{equation}
Given any positive constant $\sigma>0$, we define the rescaled metric as follows, 
\begin{align}
\begin{split}
\lambda_{\sigma,\beta} &\equiv \sigma \cdot \beta^{\frac{1}{2}}
\\
\tilde{g}_{\sigma,\beta} &\equiv (\lambda_{\sigma,\beta})^2 g_{\beta}.
\end{split}
\end{align}
Then we have the following useful lemma.

\begin{lemma}
\label{l:rescaled-Taub-NUT}
For every monopole point $p_m\in \mathcal{P}_{m_0}$ and for every fixed positive constant $\sigma>0$, we have the following $C^{\infty}$-convergence
\begin{equation}
(\mathcal{N}_{m_0}^4, \tilde{g}_{\sigma,\beta}, p_{m})\xrightarrow{C^{\infty}} (\dR^4, \tilde{g}_{\sigma,\infty}, \tilde{p}_{m,\infty})\ \text{as}\ \beta \to +\infty,
\end{equation}
 such that $(\dR^4, \tilde{g}_{\sigma,\infty}, \tilde{p}_{m,\infty})$ is a Ricci-flat Taub-NUT space with\begin{equation}
\tilde{g}_{\sigma,\infty} = G_{\sigma}\cdot g_{\dR^3} + (G_{\sigma})^{-1} \theta^2 
\end{equation}
and 
\begin{equation}
G_{\sigma}(p) = \frac{1}{2d_0(p,0^3)} + \frac{1}{\sigma^2},
\end{equation}
where $d_0$ is the distance function in the Euclidean space $\dR^3$.
\end{lemma}

\begin{remark}
When $\sigma\to\infty$, the above family of Taub-NUT spaces will converge to $\dR^4$ with the standard Euclidean metric. When $\sigma\to0$, the above family of Taub-NUT spaces will converge to $\dR^3$ with the standard Euclidean metric.
\end{remark}

\begin{proof}
To prove this lemma, we need to rescale both the metric and the coordinates. 
We choose the pull-back region $\pi^{-1}(B_{r_0}^{g_0}(p_m))$ with $r_0>0$ defined as the above, then for every $\bm{x}\in \pi^{-1}(B_{r_0}^{g_0}(p_m))$ with $\pi(\bm{x})=(x,y,z)$, 
\begin{equation}
V_{\beta}(\bm{x})=\frac{1}{2\sqrt{x^2+y^2+z^2}} + h(\bm{x}) + \beta,
\end{equation}
where $h$ is a bounded harmonic function on $\dR^3$.
Let us denote the rescaled coordinates by
\begin{align}
\tilde{x}_{\beta} \equiv \gamma_{\beta}\cdot  x, \
\tilde{y}_{\beta} \equiv \gamma_{\beta}\cdot  y, \ 
\tilde{z}_{\beta} \equiv \gamma_{\beta}\cdot  z, 
\end{align}
and we choose
\begin{equation}
\gamma_{\beta} \equiv    \sigma^2 \cdot \beta.
\end{equation}
So the rescaled metrics $\tilde{g}_{\sigma,\beta}$ converge to  
\begin{equation}
\tilde{g}_{\sigma,\infty} = 
G_{\sigma}\cdot  g_{\dR^3} + (G_{\sigma})^{-1} \theta^2 
\end{equation}
such that 
\begin{equation}
G_{\sigma}(p) = \frac{1}{2d_0(p,0^3)} + \frac{1}{\sigma^2},
\end{equation}
where $d_0$ is the distance function in the Euclidean space $\dR^3$. This tells us that $\tilde{g}_{\sigma,\infty}$ is a Taub-NUT metric, and the proof is complete.

\end{proof}
Returning to the analysis of Region $\I$: In this case, we choose $\lambda_j \equiv (\beta_j)^{\frac{1}{2}}$ and the corresponding metric $\tilde{g}_j = \lambda_j^2 g_j$. Applying Lemma \ref{l:rescaled-Taub-NUT}, the rescaled spaces converge to the standard Taub-NUT space, i.e.,
\begin{equation}
(\mathcal{M}, \tilde{g}_j, p_{m}) \xrightarrow{GH} (\dR^4, \tilde{g}_{\infty}, p_{m,\infty}),
\end{equation}
where the length of the $S^1$-fiber at infinity equals $1$.
By the regularity theory of non-collapsing Einstein manifolds, the above convergence can be improved to $C^{\infty}$ everywhere.

\vspace{0.5cm}

\noindent
{\bf Region $\II$:}

\vspace{0.5cm}

We will analyze the convergence rescaled spaces for every fixed reference point $\bm{x}_j$ in Region $\II$. To understand the geometries of the rescaled limits, we will break down this region in three different cases which depend on the distance of a reference point $\bm{x}_j$ to the monopoles:

\begin{enumerate}
\item[(a)] There is a uniform constant $\sigma_0>0$ such that $2\beta_j^{-\frac{1}{2}} \leq d_{m}(\bm{x}_j) \leq \frac{1}{\sigma_0}\cdot \beta_j^{-\frac{1}{2}}$.

\item[(b)] The distance to a pole $d_m(\bm{x}_j)$ satisfies 
\begin{align}
\frac{d_m(\bm{x}_j)}{\beta_j^{-\frac{1}{2}}}\to \infty,  \mbox{ and }
\frac{d_m(\bm{x}_j)}{\beta_j^{\frac{1}{2}}} \to 0.
\end{align}

\item[(c)] There is some uniform constant $C_0>0$ such that
\begin{equation}
0< C_0 \cdot \beta_j^{\frac{1}{2}}\leq d_m(\bm{x}_j)\leq \frac{\iota_0'}{4}\cdot\beta_j^{\frac{1}{2}}.
\end{equation}

\end{enumerate}

In Case (a) and Case (b), we choose
\begin{equation}
\lambda_j \equiv \frac{1}{d_m(\bm{x}_j)}
\end{equation}
and define the rescaled metric $\tilde{g}_j = \lambda_j^2 g_j$. Immediately, $d_{\tilde{g}_j}(p_{m},\bm{x}_j)=1$.
In Case (c), we denote $d_j\equiv \min\limits_{1\leq m\leq m_0}d_m(\bm{x}_j)$ and define the following rescaled metric by $\tilde{g}_j\equiv \lambda_j^2 g_j$ and 
\begin{equation}
\lambda_j \equiv \frac{1}{d_j}. 
\end{equation}

Now we proceed to describe the rescaled limits in each of the above cases.
Applying Lemma \ref{l:rescaled-Taub-NUT} to Case (a), the rescaled spaces converge to a Ricci-flat Taub-NUT space with a monopole $p_{m,\infty}$, i.e., 
\begin{equation}
(\mathcal{M}, \tilde{g}_j, \bm{x}_j) \xrightarrow{GH} (\dR^4, \tilde{g}_{\infty}, \bm{x}_{\infty}),
\end{equation}
where $d_{\tilde{g}_{\infty}}(\bm{x}_{\infty},p_{m,\infty})=1$ and
the $S^1$-fiber at infinity has length at least $\sigma_0>0$.
Moreover, the above convergence is $C^{\infty}$ everywhere.

In Case (b), we have the convergence 
\begin{equation}
\Big(\mathcal{M}\setminus B^{g_j}_{\beta_j^{-\frac{1}{2}}}(p_m), \tilde{g}_j, \bm{x}_j\Big) \xrightarrow{GH} (\dR^3\setminus\{0^3\}, g_{\dR^3}, \bm{x}_{\infty}),
\end{equation}
where $g_{\dR^3}$ is the standard Euclidean metric in $\dR^3$.
In terms of the rescaled metrics $\tilde{g}_j$, the diameters of the fibers converge in the following way,
\begin{align}
\label{s1diam}
\diam_{\tilde{g}_j}(S^1) \leq C\cdot\beta_j^{-\frac{1}{2}}\cdot \frac{1}{d_m(\bm{x}_j)} \to 0.
\end{align}
Denote  $\gamma_j\equiv\frac{\beta_j^{\frac{1}{2}}}{d_m(\bm{x}_j)}$ and choose the rescaled coordinates 
\begin{equation}
x_j \equiv\gamma_j\cdot x, \
y_j \equiv \gamma_j\cdot y, \ 
z_j \equiv \gamma_j\cdot z,
\end{equation}
then one can check that the metric tensor $\tilde{g}_j$ in terms of the rescaled coordinates converges to the Euclidean metric $dx_{\infty}^2 + dy_{\infty}^2 + dz_{\infty}^2$, where $(x_j,y_j,z_j)$ converges to 
$(x_{\infty},y_{\infty},z_{\infty})$ with $|dx_{\infty}|=|dy_{\infty}|=|dz_{\infty}|=1$.
Therefore, by \eqref{s1diam}, the rescaled Gromov-Hausdorff limit is the punctured Euclidean space $\dR^3\setminus\{0^3\}$. Moreover, applying 
Lemma \ref{l:curvature-estimates} (or Lemma~\ref{l:codim-1}), 
it follows that the sequence converges with uniformly bounded curvature away from the origin.

In Case (c), we will prove the rescaled limit is a punctured flat cylinder. That is, let $s_j>0$ be a sequence of numbers such that 
\begin{align}
s_j\to 0 \ , \frac{1}{s_j\beta_j}\to 0,
\end{align}
then we claim that 
\begin{equation}
\Big(\mathcal{M}\setminus\pi^{-1}\Big(\bigcup\limits_{m=1}^{m_0} B_{s_j}^{g_0}(p_m)\Big),\tilde{g}_j,\bm{x}_j\Big)\xrightarrow{GH}\Big((\TT\times\dR)\setminus\mathcal{P}_{m_0},g_0,\bm{x}_{\infty}\Big),
\end{equation}
where $g_0$ is a flat  product metric on $\TT\times\dR$.

To see this, we will carefully look at the convergence in a sequence of punctured domains with unbounded diameter. 
We denote  $U(a,b) \equiv \{\bm{x}\in\mathcal{M}| a \leq z(\bm{x}) \leq b\}$.
Let $\xi_j>0$ be a sequence with $\xi_j/\beta_j \to 0$ and we choose a sequence of punctured domains 
\begin{equation}
\mathring{U}_j\equiv U(z(\bm{x}_j)-\xi_j, z(\bm{x}_j)+\xi_j)\setminus \pi^{-1}\Big(\bigcup\limits_{m=1}^{m_0} B_{s_j}^{g_0}(p_m)\Big),
\end{equation}
where $B_{s_j}^{g_0}(p_m)$ are balls of radii $s_j$ in the flat  product metric $g_0$ on $\TT\times\dR$.
 It is straightforward that 
 \begin{equation}
 \diam_{\tilde{g}_j}(\mathring{U}_j)\approx C\cdot\xi_j\to \infty \label{e:infinite-diameter}
 \end{equation}
 and
 \begin{equation}
 \diam_{\tilde{g}_j} \Big(\pi^{-1}(B_{s_j}^{g_0}(p_m))\Big) \approx C \cdot s_j \to 0.\label{e:small-punctures}
 \end{equation}
The above arguments show that the $\mathring{U}_{\infty}$ is a complete space minus $m_0$ points.

On the other hand,
we will show that the metrics $\tilde{g}_j$ converge to a flat  product metric on $\TT\times\dR$. In fact,
 for every $\bm{y}\in 
\mathring{U}_j$, there is a bounded harmonic function $h_j$ such that
the Green's function $V_{\beta_j}$ satisfies\begin{equation}
\Big|V_{\beta_j}(\bm{y})-\Big(h_j(\bm{y}) + \frac{2\pi b_{-}}{A} \cdot z(\bm{y})+\beta_j\Big)\Big| \leq \frac{1}{2s_j}.
\end{equation}
By the assumption of Case (c), for every $j$,  it holds that $\gamma_j \equiv \frac{d_m(\bm{x}_j)}{\beta_j^{\frac{1}{2}}}\in[C_0, \frac{\iota_0'}{4}]$. Since $\frac{1}{s_j\beta_j} \to  0$,
the following holds for some uniform constant $C>0$,
\begin{align}
\Big|\lambda_j^2V_{\beta_j}(\bm{y}) - \frac{1}{\gamma_j^2}\Big| &= 
\frac{|V_{\beta_j}(\bm{y})-\beta_j|}{\gamma_j^2 \beta_j} \leq \frac{C + C \cdot \xi_j + \frac{1}{2s_j}}{\gamma_j^2\beta_j} = \frac{\frac{C}{\beta_j}+\frac{C\xi_j}{\beta_j}+\frac{1}{2s_j\beta_j}}{\gamma_j^2}\to 0.
\label{e:product}
\end{align}

Therefore, applying \eqref{e:infinite-diameter}, \eqref{e:small-punctures} and \eqref{e:product}, we have 
\begin{equation}
(\mathring{U}_j, \tilde{g}_j, \bm{x}_j) \xrightarrow{GH} \Big((\TT\times\dR)\setminus\mathcal{P}_{m_0}, g_0, \bm{x}_{\infty}\Big),
\end{equation}
where $g_0$ is a flat  product metric on $\TT\times\dR$ and $\mathcal{P}_{m_0}$ has $m_0$ points. Similar to Case (b), Applying Lemma \ref{l:curvature-estimates} (or Lemma~\ref{l:codim-1}), it follows that the sequence converges  with uniformly bounded curvature away from the monopoles.

\vspace{0.5cm}

\noindent
{\bf Region $\III$:}

\vspace{0.5cm}

For every fixed $\bm{x}_j$ in Region $\III$, we define $\lambda_j \equiv \beta_j^{-\frac{1}{2}}$ and $\tilde{g}_j \equiv \lambda_j^2 g_j$. Let $s_j>0$ be a sequence of numbers such that 
\begin{align}
s_j\to 0, \
\frac{1}{s_j\beta_j}\to 0,
\end{align}
then applying the arguments in Case (c) of Region $\II$, we have
\begin{equation}
\Big(\mathcal{M}\setminus\pi^{-1}\Big(\bigcup\limits_{m=1}^{m_0} B_{s_j}^{g_0}(p_m)\Big),\tilde{g}_j,\bm{x}_j\Big)\xrightarrow{GH}\Big((\TT\times\dR)\setminus\mathcal{P}_{m_0},g_0,\bm{x}_{\infty}\Big),
\end{equation}
where $g_0$ is a flat  product metric  on $\TT\times\dR$. Applying Lemma \ref{l:curvature-estimates} (or Lemma~\ref{l:codim-1}), it follows that the sequence converges  with uniformly bounded curvature away from the monopoles.

\vspace{0.5cm}

\noindent
{\bf Region $\IV_{-}$:} 

\vspace{0.5cm}

For fixed $\bm{x}_j$ in Region $\IV_{-}$, we choose the following rescaling factor 
\begin{equation}
\lambda_j \equiv (L_{-}(\bm{x}_j))^{-1}
\end{equation}
and the corresponding rescaled metric $\tilde{g}_j=\lambda_j^2 g_j$. To start with, let us estimate the lower bound of the rescaled distance from $\bm{x}_j$ to a monopole.
For every $\bm{x}_j$ in Region $\VI_{-}$, by the  definition of this region, we have that
\begin{align}
d_{\tilde{g}_j}(p_m, \bm{x}_j) \geq \frac{10T_0' \cdot \beta_j^{\frac{1}{2}}}{L_{-}(\bm{x}_j)}
=  \frac{10T_0' \cdot \beta_j^{\frac{1}{2}}}{\Big(\frac{2\pi b_{-}}{A} \cdot z(\bm{x}_j) + \beta_{-}+\beta_j \Big)^{\frac{1}{2}}}.
\end{align}
If $\beta_j>0$ is sufficiently large, then immediately
\begin{equation}
d_{\tilde{g}_j}(p_m, \bm{x}_j) \geq 5T_0'>0.
\end{equation}

Now we consider the following cases: 

\begin{enumerate}
\item[(a)] There is a constant $C_0 > 10 T_0'$ independent of $j$ such that \begin{equation}5T_0' \leq d_{\tilde{g}_j}(p_m,\bm{x}_j) \equiv \lambda_j\cdot d_m(\bm{x}_j) \leq C_0\end{equation} for each $1\leq m\leq m_0$.

\item[(b)] The reference points $\bm{x}_j$ in Region $\IV_{-}$ satisfy
\begin{equation}d_{\tilde{g}_j}(p_m,\bm{x}_j) \equiv \lambda_j\cdot d_m(\bm{x}_j) \to \infty,\end{equation}

\end{enumerate}

In Case (a), we have the convergence
\begin{equation}
(\mathcal{M}, \tilde{g}_j, \bm{x}_j) \xrightarrow{GH} \Big((\TT\times\dR)\setminus\mathcal{P}_{m_0}, g_0, \bm{x}_{\infty}\Big),
\end{equation}
where $g_0$ is a flat  product metric and the set $\mathcal{P}_{m_0}$ contains $m_0$ point.
To see this, first we notice that there is some constant $C>0$ such that 
\begin{equation}
|z(\bm{x}_j)| \leq C.
\end{equation}
Let $\xi_j>0$ be a sequence satisfying $\xi_j\to\infty$ and $\frac{\xi_j}{\beta_j}\to0$, and denote 
\begin{equation}
U(a,b) \equiv \{\bm{x}\in\mathcal{M}| a \leq z(\bm{x}) \leq b\}.
\end{equation}
For fixed $\bm{x}_j$ in Region $\IV_{-}$, we choose a punctured domain 
\begin{equation}
\mathring{U}_j\equiv U(z(\bm{x}_j)-\xi_j, z(\bm{x}_j)+\xi_j)\setminus \pi^{-1}\Big(\bigcup\limits_{m=1}^{m_0} B_{s_j}^{g_0}(p_m)\Big),
\end{equation}
where $B_{s_j}^{g_0}(p_m)$ are balls of radii $s_j$ in the flat product metric $g_0$ on $\TT\times\dR$ and $s_j>0$ is a sequence of numbers satisfying  
\begin{align}
s_j\to0, \ 
\frac{1}{s_j\beta_j}\to0.
\end{align}
 It is straightforward that 
 \begin{equation}
 \diam_{\tilde{g}_j}(\mathring{U}_j)\approx C\cdot\xi_j\to \infty \label{e:infinite-diameter-2}
 \end{equation}
 and the limit space $\mathring{U}_{\infty}$
 has two ends. Moreover, 
 \begin{equation}
 \diam_{\tilde{g}_j} \Big(\pi^{-1}(B_{s_j}^{g_0}(p_m))\Big) \approx C \cdot s_j \to 0.\label{e:small-punctures-2}
 \end{equation}
 Therefore, the limit space $\mathring{U}_{\infty}$ is a complete space minus $m_0$ points.

Next, we will show $\tilde{g}_j$ converges to a flat  product metric $g_0$ on $\TT\times\dR$. To this end, it suffices to show that
\begin{equation}
\frac{V_{\beta_j}(\bm{y})}{(L_{-}(\bm{x}_j))^2} \to 1.
\end{equation}
In fact, for every $\bm{y}\in 
\mathring{U}_j$, there is a bounded harmonic function $h_j$ such that
the Green's function $V_{\beta_j}$ satisfies\begin{equation}
\Big|V_{\beta_j}(\bm{y})-\Big(h_j(\bm{y}) + \frac{2\pi b_{-}}{A} \cdot z(\bm{y})+\beta_j\Big)\Big| \leq \frac{1}{2s_j}.
\end{equation}
Since $\frac{1}{s_j\beta_j} \to  0$,
the following holds for some uniform constant $C>0$,
\begin{align}
\Big|\frac{V_{\beta_j}(\bm{y})}{(L_{-}(\bm{x}_j))^2} - 1\Big| = 
\frac{|V_{\beta_j}(\bm{y})-(L_{-}(\bm{x}_j))^2|}{\Big|\frac{2\pi b_{-}}{A}\cdot z(\bm{x}_j) + \beta_{-} + \beta_j\Big|}
\leq \frac{C + C\xi_j +\frac{1}{2s_j}}{\beta_j-C} \to 0.
\label{e:product-2}
\end{align}
Therefore, applying \eqref{e:infinite-diameter-2}, \eqref{e:small-punctures-2} and \eqref{e:product-2}, we have 
\begin{equation}
(\mathring{U}_j, \tilde{g}_j, \bm{x}_j) \xrightarrow{GH} \Big((\TT\times\dR)\setminus\mathcal{P}_{m_0}, g_0, \bm{x}_{\infty}\Big),
\end{equation}
where $g_0$ is a flat  product metric on $\TT\times\dR$ and $\mathcal{P}_{m_0}$ has $m_0$ points.  Moreover, by Lemma \ref{l:curvature-estimates} (or Lemma~\ref{l:codim-1}), it follows that the sequence converges  with uniformly bounded curvature away from the monopoles.

In Case (b), it holds that 
\begin{equation}
(\mathcal{M}, \tilde{g}_j, \bm{x}_j) \xrightarrow{GH} (\TT\times\dR, g_0 , \bm{x}_{\infty}),
\end{equation}
where $g_0$ is a flat  product metric  on $\TT\times\dR$. The proof of this is similar to the previous case. Here we choose the domain 
\begin{equation}
U_j \equiv  U(z(\bm{x}_j)-\xi_j, z(\bm{x}_j)+\xi_j),
\end{equation}
where the sequence of numbers $\xi_j>0$ satisfy $\xi_j\to\infty$ and $\frac{\xi_j}{\beta_j}\to0$. Then the same arguments show that
\begin{equation}
(U_j, \tilde{g}_j, \bm{x}_j) \xrightarrow{GH} (\TT \times \mathbb{R}, g_0, \bm{x}_{\infty}),
\end{equation}
 where $g_0$ is a flat  product metric on $\TT\times\dR$.  By Lemma \ref{l:curvature-estimates} (or Lemma~\ref{l:codim-1}), it follows that the sequence converges  with uniformly bounded curvature in any compact subset containing $\bm{x}_j$ of bounded diameter.

\vspace{0.5cm}

\noindent
{\bf Region $\IV_{+}$:}

\vspace{0.5cm}

For every fixed reference point $\bm{x}_j$ in  Region $\IV_{+}$, we choose the rescaling factor 
\begin{equation}
\lambda_j \equiv (L_{+}(\bm{x}_j))^{-1}
\end{equation}
and the rescaled metric $\tilde{g}_j=\lambda_j^2g_j$. The rescaled limits are the same as those in Region $\IV_{-}$

\vspace{0.5cm}

\noindent
{\bf Region $\V_{-}$:} 

\vspace{0.5cm}

For every fixed reference point $\bm{x}_j$ in  Region $\V_{-}$, we choose the rescaling factor 
\begin{equation}
\lambda_j \equiv (\underline{L}_{-}(\bm{x}_j))^{-1}
\end{equation}
and the rescaled metric $\tilde{g}_j=\lambda_j^2g_j$.
Let $(X_{b_-}^4, g_{b_-}, q_-)$
be a Tian-Yau space in our context with a fixed reference point $q_{-}\in X_{b_-}^4$.
We need to analyze the following cases:
\begin{enumerate}
\item[(a)] Assume 
 $z_{-}(\bm{x}_j)\to \infty$.
\item[(b)] Assume that there is some constant $C_0>0$ independent of the index $j$ such that $10\zeta_0^{-}\leq z_{-}(\bm{x}_j) \leq C_0$.
\end{enumerate}

In Case (a), we have the convergence
\begin{align}
(\mathcal{M}, \tilde{g}_j, \bm{x}_j) \xrightarrow{GH} (\TT\times\dR, g_0, \bm{x}_{\infty}),
\end{align}
where $g_0$ is a flat  product metric on $\TT\times\dR$. 
To see this. we denote $\zeta_j\equiv z_{-}(\bm{x}_j)\to\infty$.
Let $\xi_j>0$ satisfy 
\begin{align}
\xi_j \to\infty, \ 
\frac{\xi_j}{\zeta_j}\to 0,
\end{align}
and we choose a unbounded domain
\begin{equation}
U_j^{-} \equiv U^{-}(\zeta_j -\xi_j , \zeta_j + \xi_j) = \{\bm{y}\in \mathcal{M} | \zeta_j - \xi_j \leq z_{-}(\bm{y}) \leq \zeta_j + \xi_j \}.
\end{equation}
We will show that
\begin{equation}
(U_j, \tilde{g}_j, \bm{x}_j) \xrightarrow{GH} (\TT\times\dR, g_0, \bm{x}_{\infty}), 
\end{equation}
where $g_0$ is a flat  product metric on $\TT\times\dR$. 
Applying the similar arguments as before,  we have 
 \begin{equation}
\diam_{\tilde{g}_j}(U_j) \to \infty  
  \end{equation}
and the limit space $U_{\infty}$ has two ends. It follows that $U_{\infty}$ is complete. 
In addition, we need to show that $\tilde{g}_j$ converges to a flat product metric on $\TT\times\dR$. In fact,
\begin{align}
\Big|\frac{V_{-}(\bm{y})}{(\underline{L}_{-}(\bm{x}_j))^2} - 1 \Big| = \frac{|V_{-}(\bm{y})-(\underline{L}_{-}(\bm{x}_j))^2|}{\frac{2\pi b_{-}}{A}\cdot \zeta_j } \leq \frac{C(\xi_j + e^{-\frac{\epsilon_0\zeta_j}{2}})}{\frac{2\pi b_{-}}{A}\cdot \zeta_j} \to 0.
\end{align} 
By Lemma \ref{l:curvature-estimates} (or Lemma~\ref{l:codim-1}), it follows that the sequence converges  with uniformly bounded curvature in any compact subset containing $\bm{x}_j$ of bounded diameter, and this finishes the analysis of Case (a).

In Case (b), since $d(q_{-}, \bm{x}_j) \leq C_0$, there is some constant $C_0'>0$ (depending only on the constant $C_0>0$ and the geometric data of $g_{b_-}$) such that
\begin{equation}
\frac{1}{C_0'} \leq \underline{L}_{-}(\bm{x}_j) \leq C_0'.
\end{equation}
Therefore, the limit space $(\mathcal{M}_{\infty},\tilde{g}_{\infty}, \bm{x}_{\infty})$ is a 
complete Ricci-flat Tian-Yau space which is a simple rescaling of $(X_{b_-}^4, g_{b-}, q_{-})$.
The convergence in this case is moreover smooth on compact subsets. 
\vspace{0.5cm}

\noindent
{\bf Region $\V_{+}$:} 

\vspace{0.5cm}

For every fixed reference point $\bm{x}_j$ in  Region $\V_{+}$, we choose the rescaling factor 
\begin{equation}
\lambda_j \equiv (\underline{L}_{+}(\bm{x}_j))^{-1}
\end{equation}
and the rescaled metric $\tilde{g}_j=\lambda_j^2g_j$. So the rescaling geometries are the same as those in Region $\V_{-}$.

\vspace{0.5cm}

\noindent
{\bf Region $\VI_{-}$:} 

\vspace{0.5cm}

We choose
$\lambda_j \equiv 1$ and the limit is $(X_{b_-}^4,g_{b_-}, q_-)$ which is a complete Tian-Yau space.

\vspace{0.5cm}

\noindent
{\bf Region $\VI_{+}$:} 

\vspace{0.5cm}

We choose
$\lambda_j \equiv 1$ and the limit is $(X_{b_+}^4,g_{b_+}, q_+)$ which is a complete Tian-Yau space.

\vspace{0.5cm}

The above arguments completely classify all the rescaled limit spaces.
We end this section by proving the following lemmas which will be used in the proof of Proposition \ref{p:injectivity-of-D} in Section \ref{s:existence}.
We will choose a convenient way to study the convergence of differential $1$-forms in the rescaled spaces. The lemma below shows that, in each part with a collapsing circle bundle structure, every differential $1$-form is equivalent to its $4$-tuple 
of coefficient functions. 

\begin{lemma}\label{l:representation-of-1-form} Let $(\mathcal{M}, g_j)$ be a sequence with gluing parameters $\beta_j \to \infty$. Let $\omega \in \Omega^1(\mathcal{M})$, then are $1$-forms $\theta_j^x$, $\theta_j^y$, $\theta_j^z$ and $\theta_j^c$ in each circle bundle part with 
\begin{align}
|\theta_j^x|_{\tilde{g}_j} = |\theta_j^y|_{\tilde{g}_j} = |\theta_j^z|_{\tilde{g}_j} \to 1,  \ 
|\theta_j^t|_{\tilde{g}_j} \to 0.
\end{align}
Moreover, every $1$-form $\omega\in\Omega^1(\mathcal{M})$ in the circle bundle part can be represented as 
\begin{equation}
\omega = f_x \theta_j^x +  f_y \theta_j^y +  f_z\theta_j^z + f_t \theta_j^t. \label{e:repre-omega}
\end{equation}

\end{lemma}

\begin{proof}
Based on the above discussions, there is a circle bundle structure in each of the following rescaled regions: Case (b) and Case (c) of Region $\II$, Region $\III$, Region $\IV_{\pm}$
and Case (a) of Region $\V_{\pm}$.
In all the above cases, the collapsed rescaled limit of $\mathcal{M}$ is isometric to $\dR^3$ or $\TT\times\dR$.

First, the proof of Case (a) of Region $\V_{-}$ is the same as the proof of Case (b) of Region $\V_{+}$.
We only need to discuss Region $\V_{-}$. 
The original sequence $g_j$
is a fixed Tian-Yau metric and have the asymptotic behavior
\begin{equation}
g_j = V_{-}(g_{\TT} + dz_{-}^2) + V_{-}^{-1}\theta_{b_{-}}^2 + O(e^{-\underline{\delta}_1 z_-}).
\end{equation}
In this case, the reference points $\bm{x}_j$ satisfy $\zeta_j \equiv z_{-}(\bm{x}_j) \to \infty$. In the above discusssions, we choose the rescaling factor $\lambda_j \equiv \frac{1}{\underline{L}_{-}(\bm{x}_j)}$. So under the 
rescaled metric $\tilde{g}_j$, it holds that
\begin{align}
|dx|_{\tilde{g}_j} = |dy|_{\tilde{g}_j} = 
|d\tilde{z}|_{\tilde{g}_j} &\to 1, \ |\theta_{b_-}|_{\tilde{g}_j} \to 0,
\end{align}
where we choose the $z$-coordinate translation as $\tilde{z}(\bm{x})={z}(\bm{x})-z(\bm{x}_j)$. 
In the remaining cases, the proof is very similar. We can properly rescale the $1$-forms $dx$, $dy$ and $dz$ by
\begin{align}
\theta_j^x \equiv \gamma_j \cdot dx, \ 
\theta_j^y \equiv \gamma_j \cdot dy, \
\theta_j^z \equiv \gamma_j \cdot d\tilde{z}, \
\theta_j^t \equiv \gamma_j \cdot dt.
\end{align}

In Case (b) and Case (c) of Region $\II$, $\gamma_j$ is defined by
\begin{equation}
\gamma_j \equiv \lambda_j \cdot \beta_j^{\frac{1}{2}},
\end{equation}
where $\lambda_j \equiv (d_{p_m}(\bm{x}_j))^{-1}$, then by straightforward computations, 
\begin{align}
|\theta_j^x|_{\tilde{g}_j} = |\theta_j^y|_{\tilde{g}_j} = |\theta_j^z|_{\tilde{g}_j} \to 1, \ |\theta_j^t|_{\tilde{g}_j} \to 0.
\end{align}
In Region $\III$ and $\IV_{\pm}$, by the definition of the rescaled metrics, 
\begin{align}
\begin{split}
|dx|_{\tilde{g}_j} = |dy|_{\tilde{g}_j} = |dz|_{\tilde{g}_j} &\to 1, \ |\theta|_{\tilde{g}_j}  \to 0.
\end{split}
\end{align}
So the proof is done.
\end{proof}

The following Lemma will be used throughout the following sections, and its simple proof is left to the reader.  
\begin{lemma}\label{l:D-harmonic} Let $(M^4,g)$ be a Riemannian $4$-manifold and 
let $\omega\in\Omega^1(M^4)$ satisfy $\mathscr{D}_g\omega = 0$, then 
\begin{equation}
\Delta_H\omega = 0,
\end{equation}
where $\mathscr{D}_g \equiv d^+  + d^*$ and $\Delta_H$ is the Hodge Laplacian. 
\end{lemma}
The following Lemma will also be very useful in the following sections. 
\begin{lemma}\label{l:2nd-fundamental-form}
In Case (b) of Region $\II$, 
the Gromov-Hausdorff map
\begin{equation}
F_j: (\mathcal{M},g_j, \bm{x}_j) \longrightarrow (\dR^3, g_0, \bm{x}_{\infty})
\end{equation}
can be given by the rescaled coordinate functions 
\begin{align}
x_j \equiv \gamma_j \cdot x, \
y_j \equiv \gamma_j \cdot y, \
z_j \equiv \gamma_j \cdot z,
\end{align}
where $\gamma_j>0$ is defined in the proof of Lemma \ref{l:representation-of-1-form}. Moreover, $F\equiv (x_j, y_j, z_j)$ satisfies
\begin{equation}
\Delta_{\tilde{g}_j} x_j = \Delta_{\tilde{g}_j} y_j
= \Delta_{\tilde{g}_j} z_j
=0
\end{equation}
 and satisfy 
 \begin{equation}
 |\nabla_{\tilde{g}_j} x_j|_{\tilde{g}_j} =  |\nabla_{\tilde{g}_j} y_j|_{\tilde{g}_j} =  |\nabla_{\tilde{g}_j} z_j|_{\tilde{g}_j}
\to 1
 \end{equation}
and away from the monopoles,
\begin{equation}
 |\nabla_{\tilde{g}_j}^2 x_j|_{\tilde{g}_j} =  |\nabla_{\tilde{g}_j}^2 y_j|_{\tilde{g}_j} =  |\nabla_{\tilde{g}_j}^2 z_j|_{\tilde{g}_j}
\to 0.
\end{equation}

\end{lemma}

\begin{remark} The explicitly Gromov-Hausdorff map $F_j$ in 
Lemma \ref{l:2nd-fundamental-form} in fact corresponds to Cheeger-Colding's quantitative splitting map (see \cite{ChC}).
\end{remark}

\begin{proof}
In terms of the original coframes $\{dx, dy, dz, \theta\}$, the volume form is given by 
\begin{equation}
\dvol_{g_j} = V_{\beta_j} dx \wedge dy \wedge dz \wedge \theta. 
\end{equation}
By definition,
\begin{align}
* (dx) = dy \wedge dz \wedge \theta, \
*(dy) = - dx \wedge dz \wedge \theta, \
*(dz) = dx \wedge dy \wedge \theta, 
\end{align}
which implies 
\begin{equation}
\Delta_{g_j} x = \Delta_{g_j} y = \Delta_{g_j} z = 0.
\end{equation}
After rescaling, we have that 
\begin{equation}
\Delta_{\tilde{g}_j} x_j = \Delta_{\tilde{g}_j} y_j
= \Delta_{\tilde{g}_j} z_j
=0.
\end{equation}
By Lemma \ref{l:representation-of-1-form}, the pointwise gradient estimate holds,
\begin{equation}
 |\nabla x_j|_{\tilde{g}_j} =  |\nabla y_j|_{\tilde{g}_j} =  |\nabla z_j|_{\tilde{g}_j}
\to 1.
\end{equation}

Now we estimate the Hessian of the harmonic functions $x_j$, $y_j$ and $z_j$. It suffices to check it for $x_j$.
First,
Bochner's formula gives that
\begin{equation}
\frac{1}{2}\Delta_{\tilde{g}_j}|\nabla x_j|_{\tilde{g}_j}^2 = |\nabla^2 x_j|_{\tilde{g}_j}^2.\label{e:bochner}
\end{equation}
Due to Cheeger-Colding (see \cite{ChC}), there exist cutoff functions $\varphi_j: \mathcal{M} \to [0,1]$ with
\begin{align}
\varphi_j(x)=
\begin{cases}
1, \  x \in B_{R}(p_j), \\
0, \ x\in\mathcal{M}\setminus B_{2R}(p_j) 
\end{cases}
\end{align}
and there exists an absolute constant $C_0>0$ such that
\begin{equation}
R|\nabla_{\tilde{g}_j}\varphi_j|_{\tilde{g}_j} + R^2 |\Delta_{\tilde{g}_j}\varphi_j| \leq C_0.
\end{equation}
Integrating \eqref{e:bochner} over $B_{4R}(p_j)$, 
\begin{align}
  \begin{split}
\fint_{B_{4R}(p_j)}  \varphi_j |\nabla^2x_j|_{\tilde{g}_j}^2 \dvol_{\tilde{g}_j}
& =\frac{1}{\Vol_{\tilde{g}_j}(B_{4R}(p_j))}\int_{B_{4R}(p_j)}\varphi_j |\nabla^2x_j|_{\tilde{g}_j}^2 \dvol_{\tilde{g}_j} \\
  &=\frac{1}{\Vol_{\tilde{g}_j}(B_{4R}(p_j))}\int_{B_{4R}(p_j)}\frac{1}{2}\varphi_j \Delta_{\tilde{g}_j} (|\nabla x_j|_{\tilde{g}_j}^2-1) \dvol_{\tilde{g}_j} \\
&=\frac{1}{2\Vol_{\tilde{g}_j}(B_{4R}(p_j))}\int_{B_{4R}(p_j)}(\Delta_{\tilde{g}_j}\varphi_j)\cdot  (|\nabla x_j|_{\tilde{g}_j}^2  -1)  \dvol_{\tilde{g}_j}  \rightarrow  0,
\end{split}
\end{align}
as $j \rightarrow \infty$.  Therefore, by volume comparison,
\begin{equation}
\fint_{B_{R}(p_j)}|\nabla^2 x_j|_{\tilde{g}_j}^2  \dvol_{\tilde{g}_j} \to 0,\label{e:L2-hessian-small}
\end{equation}
as $j \rightarrow \infty$. 
Let $p_{\infty}\in\dR^3\setminus\{0^3\}$ with $B_{2\bar{s}_0}(p_{\infty})\subset\dR^3\setminus\{0^3\}$ and we choose a sequence of geodesic balls $B_{2\bar{s}_0}(p_j)$ such that 
\begin{equation}
(B_{2\bar{s}_0}(p_j),\tilde{g}_j)\xrightarrow{GH}(B_{2\bar{s}_0}(p_{\infty}),g_0).
\end{equation}
By Lemma \ref{l:codim-1}, the curvatures on $B_{\bar{s}_0}(p_j)$
are uniformly bounded by $C\cdot\bar{s}_0^{-2}$ and $C>0$ is an absolute  constant. 
On the other hand, since $\Delta_{\tilde{g}_j}x_j=0$,  \eqref{e:L2-hessian-small} can be strengthened to 
\begin{equation}
\sup\limits_{B_{\bar{s}_0}(p_j)}|\nabla^2 x_j|_{\tilde{g}_j}^2 \to 0.
\end{equation}
The proof is done. 

\end{proof}

\section{Weighted Schauder estimate} 

\label{s:weight-analysis}

The main part of this section is to establish the appropriate weight analysis which is consistent with the different rescaled geometries of the manifold $(\mathcal{M},g_{\beta})$. The main result in this section is the weighted Schauder estimate in Proposition \ref{p:weighted-Schauder-estimate}.

To start with, we define the weight functions.
Let $\delta$, $\nu$ and $\mu$ be fixed constants to be determined later. Fix the gluing parameter $\beta>0$, 
the weight functions $\rho_{\delta,\nu,\mu}^{(k+\alpha)}$ with $k\in\dN$ and $0\leq \alpha<1$, are defined as follows:
\begin{align}
\rho_{\delta,\nu,\mu}^{(k+\alpha)}(\bm{x})\equiv
\begin{cases}
e^{ \delta \cdot (2T_-) } \cdot (\beta^{-\frac{1}{2}})^{2\mu+\nu+k+\alpha}, & \bm{x} \in \I 
\\
e^{ \delta \cdot (2T_-)}\cdot (\beta^{-\frac{\mu}{2}})\cdot d_{m}^{\mu+\nu+k+\alpha}(\bm{x}), &  \bm{x} \in\II 
\\
e^{ \delta \cdot (2T_-) }\cdot (\beta^{\frac{1}{2}})^{\nu+k+\alpha}, &  \bm{x} \in \III
\\
e^{\delta \cdot (z(\bm{x})+2T_-)}\cdot (L_-(\bm{x}))^{\nu+k+\alpha}  & \bm{x} \in \IV_{-}
\\ 
e^{\delta \cdot (z(\bm{x})+2T_-)}\cdot (L_+(\bm{x}))^{\nu+k+\alpha}   &  \bm{x} \in\IV_{+}
\\ 
e^{\delta \cdot (z_-(\bm{x}))}\cdot (\underline{L}_-(\bm{x}))^{\nu+k+\alpha}  , &  \bm{x} \in\V_{-}
\\
e^{\delta\cdot (-z_+(\bm{x})+2T_-+2T_+)}\cdot(\underline{L}_+(\bm{x}))^{\nu+k+\alpha}   , &  \bm{x} \in\V_{+}
\\
e^{\delta\cdot \zeta_0^-}\cdot \Big(\underline{L}_{-}(\zeta_0^{-})\Big)^{\nu+k+\alpha}   , & \bm{x} \in\VI_{-}
\\
e^{\delta(-\zeta_0^+ + 2T_- + 2T_+)} \cdot \Big(\underline{L}_{+}(\zeta_0^{+})\Big)^{\nu+k+\alpha}  , & \bm{x} \in\VI_{+}.
\end{cases}\label{def-weight-function}
\end{align}
Recall from Remark  \ref{re:subdivision} that these regions do not entirely cover $\mathcal{M}$, so 
we extend the weight functions to smooth functions on the entire manifold $\mathcal{M}$ by using appropriate cutoff functions in each gap region.  The geometry in each gap region is the same as the geometry of the adjacent regions, therefore these gap regions can be ignored in the following analysis.

\begin{definition}
[Weighted H\"older space] \label{whsdef}
The weighted H\"older space is defined by,
\begin{equation}
\|\omega\|_{C_{\delta,\nu, \mu}^{k,\alpha}(\mathcal{M})}
\equiv \sum\limits_{j=0}^k\Big\|\rho_{\delta,\nu,\mu}^{(j)} \cdot\nabla^j\omega\Big\|_{C^0(\mathcal{M})} + [\omega]_{C_{\delta,\nu,\mu}^{k,\alpha}(\mathcal{M})},
\end{equation}
where 
\begin{equation}
[\omega]_{C_{\delta,\nu,\mu}^{k,\alpha}(\mathcal{M})} \equiv \sup_{\substack{d_g(x,y)\leq r_0\\ x,y\in \mathcal{M}}} \Big\{\min\{\rho_{\delta,\nu,\mu}^{(k+\alpha)}(x),\rho_{\delta,\nu,\mu}^{(k+\alpha)}(y)\}\cdot\frac{|\nabla^k\omega(x)-\nabla^k\omega(y)|}{(d_g(x,y))^{\alpha}}\Big\},
\end{equation}
 $\omega\in T^{r,s}(\mathcal{M})$ is a tensor field of type $(r,s)$ and $r_0 \equiv \frac{1}{2}\InjRad_g(\mathcal{M})$.
 In the above definition, the difference of the two covariant derivatives is defined in terms of the parallel translation.
\end{definition}

\begin{remark} If $(X,g)$ is a $\delta$-aymptotically Calabi space for $n=2$, one can define a weight function
  \begin{align}
    \rho^{(k + \alpha)}_{\delta, \nu}(\bm{x}) = e^{ \delta z(\bm{x})}\cdot( z(\bm{x}))^{\frac{ \nu + k + \alpha}{2}},
\end{align}
and defined the weighted space $C^{k,\alpha}_{\delta, \nu}(X)$ exactly as in Definition \ref{whsdef}. On such a space the Laplace operator is a bounded linear mapping
\begin{align}
\Delta:   C^{k,\alpha}_{\delta, \nu}(X) \rightarrow C^{k-2,\alpha}_{\delta, \nu +2}(X).
\end{align}
It is expected that this operator is Fredholm if $\delta$ is sufficiently small and non-zero, and for arbitrary~$\nu$. The analysis in Section \ref{s:liouville-functions}  can likely be extended to prove this stronger result, but for the purposes of this paper we do not need this. Likewise, we expect that
\begin{align}
\mathscr{D}: C^{k,\alpha}_{\delta, \nu} (\Omega^1) \rightarrow
C^{k-1, \alpha}_{\delta, \nu + 1} ( \Omega^0 \oplus \Omega^2_+)
\end{align}
is Fredholm for $\delta$ sufficiently small, but this would require a much more elaborate separation of variables argument.
\end{remark}

\begin{proposition}[The Weighted Schauder Estimate]
\label{p:weighted-Schauder-estimate} Consider $(\mathcal{M},g_{\beta})$ with a sufficiently large gluing parameter $\beta>0$.
Then there exists a uniform constant $C>0$ (independent of $\beta$) such that 
for every $\omega\in\Omega^1(\mathcal{M})$, it holds that\begin{equation}
\|\omega\|_{C_{\delta,\nu,\mu}^{1,\alpha}(\mathcal{M})} \leq C\Big(\| \mathscr{D}_g \omega \|_{C_{\delta,\nu+1,\mu}^{0,\alpha}(\mathcal{M})} +\|\omega\|_{C_{\delta,\nu,\mu}^{0}(\mathcal{M})}\Big).
\end{equation}

\end{proposition}

\begin{proof}
We will prove by contradiction and suppose no such a uniform constant $C>0$ exists. That is, there exist the following sequences:
\begin{enumerate}
\item a sequence of numbers $\beta_j \to \infty$,

\item a sequence of gluing metrics $(\mathcal{M}, g_j)$ with weight functions $\rho_{j,\delta,\nu,\mu}^{(k+\alpha)}$ (for simplicity, we still denote by $\rho_{\delta,\nu,\mu}^{(k+\alpha)}$ because there is no ambiguity),

\item a sequence of differential $1$-forms $\omega_j \in \Omega^1(\mathcal{M})$ such that
\begin{align}
&\|\omega_j\|_{C_{\delta,\nu,\mu}^{1,\alpha}(\mathcal{M},g_j)} = 1, \\
&\| \mathscr{D}_{g_j}\omega_j \|_{C_{\delta,\nu+1,\mu}^{0,\alpha}(\mathcal{M},g_j)} + \|\omega_j\|_{C_{\delta,\nu,\mu}^0(\mathcal{M},g_j)} \to 0,
\label{e:contradicting-Schauder}
\end{align}
as $j \to \infty$.
\end{enumerate}

Our main goal is to prove a local version of the above weighted Schauder estimate.
Precisely, it suffices to show that, there is some uniform constant $C>0$ (independent of $j$) and for every $\omega\in\Omega^1(\mathcal{M})$ and for every $\bm{x}_j \in \mathcal{M}$, there is some $r_j>0$ (depending on the location of $\bm{x}_j$) such that the following estimate holds in $B_{2r_j}(\bm{x}_j)\subset(\mathcal{M},g_j)$,
\begin{equation}
\|\omega\|_{C_{\delta,\nu,\mu}^{1,\alpha}(B_{r_j}(\bm{x}_j))} \leq C\Big(\| \mathscr{D}_{g_j} \omega \|_{C_{\delta,\nu+1,\mu}^{0,\alpha}(B_{2r_j}(\bm{x}_j))} +\|\omega\|_{C_{\delta,\nu,\mu}^{0}(B_{2r_j}(\bm{x}_j))}\Big).\label{e:ball-weighted-schauder}
\end{equation}
Once \eqref{e:ball-weighted-schauder} is established, the contradiction immediately arises which completes the entire proof.
Indeed, \eqref{e:contradicting-Schauder} implies that either 
$\|\rho_{\delta,\nu,\mu}^{(1)}\cdot \nabla\omega_j\|_{C^0(\mathcal{M},g_j)} \geq \frac{1}{2}$ or 
$[\omega_j]_{C_{\delta,\nu,\mu}^{1,\alpha}(\mathcal{M},g_j)} \geq \frac{1}{2}$.
We can assume $\|\rho_{\delta,\nu,\mu}^{(1)}\cdot \nabla\omega_j\|_{C^0(\mathcal{M},g_j)}\geq \frac{1}{2}$ because the argument for the other case is exactly the same. Hence by definition, there exists some $\bm{x}_j \in \mathcal{M}_j$ with 
\begin{equation}|\rho_{\delta,\nu,\mu}^{(1)}(\bm{x}_j)\cdot\nabla\omega_j(\bm{x}_j)| \geq \frac{1}{2}.\label{e:lower-bound-C1}
\end{equation}
By \eqref{e:ball-weighted-schauder}, there is some $r_j>0$ which depends on $\bm{x}_j$ such that 
\begin{equation}
\|\omega_j\|_{C_{\delta,\nu,\mu}^{1,\alpha}(B_{r_j}(\bm{x}_j))} \leq C\Big(\| \mathscr{D}_{g_j} \omega_j \|_{C_{\delta,\nu+1,\mu}^{0,\alpha}(B_{2r_j}(\bm{x}_j))} +\|\omega_j\|_{C_{\delta,\nu,\mu}^{0}(B_{2r_j}(\bm{x}_j))}\Big) \longrightarrow 0.
\end{equation}
The above estimate implies
 that \begin{equation}
|\rho_{\delta,\nu,\mu}^{(1)}(\bm{x}_j)\cdot\nabla\omega_j(\bm{x}_j)|
\leq \|\omega_j\|_{C_{\delta,\nu,\mu}^{1}(B_{r_j}(\bm{x}_j))}  \leq \|\omega_j\|_{C_{\delta,\nu,\mu}^{1,\alpha}(B_{r_j}(\bm{x}_j))} \to 0.\label{e:small-weighted-C1}
\end{equation}
However, \eqref{e:small-weighted-C1} contradicts \eqref{e:lower-bound-C1}. The proof is done. 

So the main part of the proof of the proposition is to establish \eqref{e:ball-weighted-schauder}.
In our proof, the primary strategy is to rescale the metric $g_j$ and the differential $1$-form $\omega$. That is, we choose some correct rescaling factors $\lambda_j>0$, $\kappa_j>0$ and define 
\begin{align}
\tilde{g}_j = \lambda_j^2 g_j, \ 
\tilde{\omega} = \kappa_j \omega.
\end{align}
In the previous section, we showed that, for every reference point $\bm{x}_j\in\mathcal{M}$,
after appropriate rescaling, there is a subdomain $U_j$ which contains $\bm{x}_j$ and has uniformly bounded geometry away from at most finitely many singular points. 
Then the standard Schauder estimate in the rescaled spaces is available. That is, for every $\tilde{\omega}\in\Omega^1(\mathcal{M})$, \begin{equation}
\|\tilde{\omega}\|_{C^{1,\alpha}(B_{\tilde{r}_0}^{\tilde{g}_j}(\bm{x}_j))} \leq C\Big(\| \mathscr{D}_{\tilde{g}_j}\tilde{\omega} \|_{C^{0,\alpha}(B_{2\tilde{r}_0}^{\tilde{g}_j}(\bm{x}_j))} +\|\tilde{\omega}\|_{C^{0}(B_{2\tilde{r}_0}^{\tilde{g}_j}(\bm{x}_j))}\Big),\label{e:ball-standard-schauder}
\end{equation}
where $\tilde{r}_0>0$ is some uniform constant independent of the index $j$ such that the balls $B_{2\tilde{r}_0}^{\tilde{g}_j}(\bm{x}_j)$ converge to a smooth space and the convergence keeps curvatures uniformly bounded. 
Once we obtain \eqref{e:ball-standard-schauder}, we will get
the weighted estimate \eqref{e:ball-weighted-schauder} after an appropriate rescaling.

In the following arguments, for every fixed $\bm{x}_j\in\mathcal{M}$, we will choose the corresponding rescaled metrics $\tilde{g}_j = \lambda_j^2 g_j$ defined in Section \ref{ss:rescaled-limits}.

\vspace{0.5cm}

\noindent
{\bf Region $\I$:} 

\vspace{0.5cm}

We prove \eqref{e:ball-weighted-schauder} around the monopole $p_m\in\mathcal{P}_{m_0}$. Let $\lambda_j = \beta_j^{\frac{1}{2}}$ and we choose the rescaled metric $\tilde{g}_j \equiv \lambda_j^2 g_j$, then 
\begin{equation}
\Big(\mathcal{M},\tilde{g}_j,p_m\Big)\xrightarrow{C^{\infty}}\Big(\dR^4,\tilde{g}_{\infty},p_{m,\infty}\Big)\ \text{as}\ \beta_j\to\infty,\label{e:convergence-TN}
\end{equation}
where $\tilde{g}_{\infty}$ is the standard Taub-NUT metric such that the length of the $S^1$-fiber at infinity equals~$1$.
Since the above convergence is $C^{\infty}$,  the rescaled sequence $(\mathcal{M},\tilde{g}_j,p_m)$ has bounded geometry and thus the standard Schauder estimate holds in the geodesic ball $B_2^{\tilde{g}_j}(p_m)$ with respect to the 
rescaled metric $\tilde{g}_j$. Precisely, there is a uniform constant such that for every $\tilde{\omega}\in\Omega^1(\mathcal{M})$,
\begin{equation}
\|\tilde{\omega}\|_{C^{1,\alpha}(B_1^{\tilde{g}_j}(p_m))}\leq C\Big(\|\mathscr{D}_{\tilde{g}_j}\tilde{\omega}\|_{C^{0,\alpha}(B_2^{\tilde{g}_j}(p_m))} +\|\tilde{\omega}\|_{C^{0}(B_2^{\tilde{g}_j}(p_m))}\Big).
\end{equation}
Now we rescale back to the original metric $g_j$. First,
we choose
\begin{equation}\kappa_j=
e^{ \delta \cdot (2T_-)}\cdot(\beta_j^{-\frac{1}{2}})^{2\mu+\nu-1}
\end{equation}
and denote $\tilde{\omega}=\kappa_j \cdot \omega$.
With respect to the original metric, the above Schauder estimate is equivalent to the following 
\begin{align}
\begin{split}
\|\rho_{\delta,\nu,\mu}^{(1)}\cdot\nabla\omega\|_{C^{0}(B_{r_j}(p_m))} 
&+
\|\rho_{\delta,\nu,\mu}^{(1+\alpha)}\cdot\nabla\omega\|_{C^{\alpha}(B_{r_j}(p_m))}\\
&\leq 
C\Big(\|\rho_{\delta,\nu+1,\mu}^{(\alpha)}\cdot \mathscr{D}_{g_j}\omega\|_{C^{\alpha}(B_{2r_j}(p_m))}+\|\rho_{\delta,\nu,\mu}^{(0)} \cdot \omega\|_{C^0(B_{2r_j}(p_m))}\Big),
\end{split}
\end{align}
where $r_j = \lambda_j^{-1}$.
Therefore, by the definition of the weighted H\"older norm, 
\begin{align}
\|\omega\|_{C_{\delta,\nu,\mu}^{1,\alpha}(B_{r_j}(p_m))} 
\leq C\Big(\| \mathscr{D}_{g_j} \omega \|_{C_{\delta, \nu+1, \mu}^{0,\alpha}(B_{2r_j}(p_m))} +\|\omega\|_{C_{\delta, \nu, \mu}^{0}(B_{2r_j}(p_m))}\Big).
\end{align}
So the estimate \eqref{e:ball-weighted-schauder} has been proved in Region $\I$.

\vspace{0.5cm}

\noindent
{\bf Region $\II$:}

\vspace{0.5cm}

We will prove \eqref{e:ball-weighted-schauder} for every $\bm{x}_j$ in Region $\II$. As what is introduced in Section \ref{ss:rescaled-limits}, 
we break down this region in $3$ cases with different rescaling geometries: 
\begin{enumerate}
\item[(a)] There is a uniform constant $\sigma_0>0$ such that $2\beta_j^{-\frac{1}{2}} \leq d_{m}(\bm{x}_j) \leq \frac{1}{\sigma_0}\cdot \beta_j^{-\frac{1}{2}}$.

\item[(b)] The distance to a pole $d_m(\bm{x}_j)$ satisfies 
\begin{align}
\frac{d_m(\bm{x}_j)}{\beta_j^{-\frac{1}{2}}}\to \infty, \ \mbox{ and } 
\frac{d_m(\bm{x}_j)}{\beta_j^{\frac{1}{2}}} \to 0.
\end{align}

\item[(c)] There is some uniform constant $C_0>0$ such that
\begin{equation}
0< C_0 \cdot \beta_j^{\frac{1}{2}}\leq d_m(\bm{x}_j)\leq \frac{\iota_0'}{4}\cdot\beta_j^{\frac{1}{2}}.
\end{equation}

\end{enumerate}
 
In each of the above cases, the rescaled spaces $(\mathcal{M}, \tilde{g}_j, \bm{x}_j)$ have uniformly bounded curvatures and converge to a smooth limit space, which enables us to obtain the standard Schauder estimate in any ball of a definite radius the rescaled spaces.
Specifically, let $\bm{x}_j$ be a fixed point in Region~$\II$, then the standard Schauder estimate in $B_{1/6}^{\tilde{g}_j}(\bm{x}_j)$ states that for any $\tilde{\omega}\in\Omega^1(\mathcal{M})$,
\begin{equation}
\| \tilde{\omega} \|_{C^{1,\alpha}(B_{1/6}^{\tilde{g}_j}(\bm{x}_j))} \leq  C \Big(\|\mathscr{D}_{\tilde{g}_j}\tilde{\omega}\|_{C^{\alpha}(B_{1/3}^{\tilde{g}_j}(\bm{x}_j))} + \|\tilde{\omega}\|_{C^0(B_{1/3}^{\tilde{g}_j}(\bm{x}_j))} \Big).
\end{equation}
Now let 
\begin{equation}
\kappa_j \equiv e^{\delta \cdot 2T_-} \cdot (\beta_j)^{-\frac{\mu}{2}} \cdot (d_m(\bm{x}_j))^{2\mu+\nu-1},
\end{equation}
and denote $\omega_j = \tilde{\omega}_j /\kappa_j$, then rescaling to the original metrics $g_j$, we have 
\begin{align}
\begin{split}
\|\rho_{\delta,\nu,\mu}^{(1)}(\bm{x}_j)\cdot&\nabla\omega\|_{C^{0}(B_{r_j}(\bm{x}_j))} 
+
\|\rho_{\delta,\nu,\mu}^{(1+\alpha)}(\bm{x}_j)\cdot\nabla\omega\|_{C^{\alpha}(B_{r_j}(\bm{x}_j))}\\
&\leq 
C\Big(\|\rho_{\delta,\nu+1,\mu}^{(\alpha)}(\bm{x}_j)\cdot \mathscr{D}_{g_j}\omega\|_{C^{\alpha}(B_{2r_j}(\bm{x}_j))}+\|\rho_{\delta,\nu,\mu}^{(0)}(\bm{x}_j) \cdot \omega\|_{C^0(B_{2r_j}(\bm{x}_j))}\Big).
\end{split}
\end{align}
The above $r_j>0$ is defined as follows. If $\bm{x}_j$ is in Case (a) or (b), then
\begin{equation}
r_j \equiv \frac{1}{6}\lambda_j^{-1} = \frac{1}{6}d_m(\bm{x}_j).
\end{equation}
If $\bm{x}_j$ is in Case (c), then
\begin{equation}
r_j \equiv \frac{1}{6}d_j,
\end{equation}
where $d_j\equiv\min\limits_{1\leq m\leq m_0}d_{g_j}(p_m,\bm{x}_j)$. We need to show that the values $\rho_{\delta,\nu,\mu}^{(k+\alpha)}(\bm{y})$ for all $\bm{y}\in B_{r_j}(\bm{x}_j)$ are equivalent. Indeed, by the triangle inequality, we can see that for every $\bm{y}\in B_{r_j}(\bm{x}_j)$, 
\begin{equation}
(\frac{5}{6})^{\mu+\nu+k+\alpha} \leq \frac{\rho_{\delta,\nu,\mu}^{(k+\alpha)}(\bm{y})}{\rho_{\delta,\nu,\mu}^{(k+\alpha)}(\bm{x}_j)} \leq (\frac{7}{6})^{\mu+\nu+k+\alpha} .
\end{equation}
Therefore, by the definition of the weighted norm, the estimate \eqref{e:ball-weighted-schauder} immediately follows.

\vspace{0.5cm}

\noindent
{\bf Region $\III$:}

\vspace{0.5cm}

If we choose $\lambda_j \equiv \beta_j^{-\frac{1}{2}}$, then the remaining arguments coincide with those in Case (c) of Region~$\II$.

\vspace{0.5cm}
\noindent
{\bf Regions $\IV_{-}$ and $\IV_{+}$:}
\vspace{0.5cm}

\noindent
We only prove the estimate \eqref{e:ball-weighted-schauder}
for every fixed reference point $\bm{x}_j$ in Region $\IV_{-}$ and the proof for the other part is identical.

For fixed $\bm{x}_j$ in Region $\IV_{-}$, we define $\tilde{g}_j=\lambda_j^2 g_j$
and \begin{equation}
\lambda_j \equiv (L_{-}(\bm{x}_j))^{-1}.
\end{equation}
In Section \ref{ss:rescaled-limits}, the rescaled limits were separated in the following cases:\begin{enumerate}
\item[(a)]  There is a constant $C_0 > 10 T_0'$ independent of the index $j$ such that \begin{equation}5T_0' \leq d_{\tilde{g}_j}(p_m,\bm{x}_j) \equiv \lambda_j\cdot d_m(\bm{x}_j) \leq C_0\end{equation} for each $1\leq m\leq m_0$.

\item[(b)] The reference points $\bm{x}_j$ in Region $\IV_{-}$ satisfy
\begin{equation}d_{\tilde{g}_j}(p_m,\bm{x}_j) \equiv \lambda_j\cdot d_m(\bm{x}_j) \to \infty,\end{equation}

\end{enumerate}

We follow the notations in Section \ref{ss:rescaled-limits}.
Notice that, in each of the above cases, the rescaled spaces have uniformly bounded curvature and the standard Schauder estimate can be stated in the following way. 

In Case (a), by assumption, we can pick
some definite constant 
\begin{equation}\tilde{r}_0 \equiv \frac{T_0'}{2} \leq \frac{d_{\tilde{g}_j}(p_m,\bm{x}_j)}{10}\end{equation} 
such that for every $\tilde{\omega}\in\Omega^1(\mathcal{M})$,
 \begin{equation}
\| \tilde{\omega} \|_{C^{1,\alpha}(B_{\tilde{r}_0}^{\tilde{g}_j}(\bm{x}_j))} \leq  C \Big(\|\mathscr{D}_{\tilde{g}_j}\tilde{\omega}\|_{C^{\alpha}(B_{2\tilde{r}_0}^{\tilde{g}_j}(\bm{x}_j))} + \|\tilde{\omega}\|_{C^0(B_{2\tilde{r}_0}^{\tilde{g}_j}(\bm{x}_j))} \Big).
\end{equation}
In the above estimate, the constant $C>0$ depends only on $T_0'>0$ and the flat  product metric $g_0$ (particularly $C$ does not depend on $C_0$).

Now we rescale the $1$-form $\tilde{\omega}$ by choosing\begin{equation}
\kappa_j \equiv e^{\delta(z(\bm{x}_j) + 2T_-)} \cdot (L_{-}(\bm{x}_j))^{\nu-1} 
\end{equation}
and $\omega_j \equiv \tilde{\omega}_j/\kappa_j$.
First, we rescale the above estimate to the original metric $g_j$ and denote \begin{equation}r_j \equiv (L_{-}(\bm{x}_j)) \cdot \tilde{r}_0 > 0,\end{equation}
then
\begin{align}
\begin{split}
\|\rho_{\delta,\nu,\mu}^{(1)}(\bm{x}_j)\cdot&\nabla\omega\|_{C^{0}(B_{r_j}(\bm{x}_j))} 
+
\|\rho_{\delta,\nu,\mu}^{(1+\alpha)}(\bm{x}_j)\cdot\nabla\omega\|_{C^{\alpha}(B_{r_j}(\bm{x}_j))}\\
&\leq
C\Big(\|\rho_{\delta,\nu+1,\mu}^{(\alpha)}(\bm{x}_j)\cdot \mathscr{D}_{g_j}\omega\|_{C^{\alpha}(B_{2r_j}(\bm{x}_j))}+\|\rho_{\delta,\nu,\mu}^{(0)}(\bm{x}_j) \cdot \omega\|_{C^0(B_{2r_j}(\bm{x}_j))}\Big).
\end{split}
\end{align}
So the rest is to show the values of the weight function $\rho_{\delta,\nu,\mu}^{(k+\alpha)}$ are equivalent for every $\bm{y}\in B_{2r_j}(\bm{x}_j)$. Indeed, denote $\zeta_j \equiv z(\bm{x}_j)$,
by straightforward computations, there is a uniform constant $C_1>0$ such that for every $\bm{y}\in B_{2r_j}(\bm{x}_j)$,
\begin{equation}
|z(\bm{y})-\zeta_j|\leq C_1.
\end{equation}
Moreover, we can show that
\begin{equation}
\frac{\zeta_j}{\beta_j} \leq  C_2,
\end{equation}
for some uniform constant $C_2>0$.
By the definition of the weight function in Region $\IV_{-}$,
\begin{align}
\frac{\rho_{\delta,\nu,\mu}^{(k+\alpha)}(\bm{y})}{\rho_{\delta,\nu,\mu}^{(k+\alpha)}(\bm{x}_j)}  = e^{\delta \cdot (z(\bm{y})-\zeta_j)} \cdot \frac{(L_{-}(\bm{y}))^{\nu+k+\alpha}}{(L_{-}(\bm{x}_j))^{\nu+k+\alpha}}
\end{align}
The above estimates imply that there is a uniform constant $C_3>0$ such that
\begin{equation}
\frac{1}{C_3}\leq \Big|\frac{\rho_{\delta,\nu,\mu}^{(k+\alpha)}(\bm{y})}{\rho_{\delta,\nu,\mu}^{(k+\alpha)}(\bm{x}_j)}  \Big| \leq C_3.
\end{equation}
So the proof of \eqref{e:ball-weighted-schauder} in Case (a) is done.

The proof of the estimate in Case (b) is the same.

\vspace{0.5cm}

\noindent
{\bf Regions $\V_{-}$ and $\V_{+}$:}

\vspace{0.5cm}

\noindent
We only need to prove the estimate \eqref{e:ball-weighted-schauder} for the reference point $\bm{x}_j$ in Region $\V_{-}$ because the estimate in Region $\V_{+}$ is identical. 
As the discussion in Section \ref{ss:rescaled-limits},
there are the following two cases to be considered:
\begin{enumerate}
\item[(a)] Assume that a sequence of reference points $\bm{x}_j$ satisfy
 $z_{-}(\bm{x}_j)\to \infty$.
\item[(b)] Assume that there is a some constant $C_0>0$ such that a sequence of reference points $\bm{x}_j$ satisfy
$10\zeta_0^{-} \leq z_{-}(\bm{x}_j) \leq C_0$.
\end{enumerate}

First, we prove the weighted Schauder estimate \eqref{e:ball-weighted-schauder} in Case (a). 
Let $\lambda_j \equiv (\underline{L}(\bm{x}_j))^{-1}$ and $\tilde{g}_j \equiv \lambda_j^2 g_j$, then we have shown in Section \ref{ss:rescaled-limits} that
\begin{equation}
(\mathcal{M}, \tilde{g}_j, \bm{x}_j ) \xrightarrow{GH} (\TT \times \mathbb{R}, g_0, \bm{x}_{\infty}). 
\end{equation}
Moreover, the curvatures are uniformly bounded in the above convergence, which implies the standard Schauder estimate for every $\tilde{\omega} \in \Omega^1(\mathcal{M})$,
\begin{equation}
\| \tilde{\omega} \|_{C^{1,\alpha}(B_{1}^{\tilde{g}_j}(\bm{x}_j))} \leq  C \Big(\|\mathscr{D}_{\tilde{g}_j}\tilde{\omega}\|_{C^{\alpha}(B_{2}^{\tilde{g}_j}(\bm{x}_j))} + \|\tilde{\omega}\|_{C^0(B_{2}^{\tilde{g}_j}(\bm{x}_j))} \Big).
\end{equation}
To prove the weighted estimate \eqref{e:ball-weighted-schauder} in the original metrics $g_j$, we both rescale the metric $\tilde{g}_j$ and $\tilde{\omega}_j$ in the above estimate. First, let 
\begin{equation}\kappa_j \equiv e^{\delta z_{-}(\bm{x})}\cdot (\underline{L}_{-}(\bm{x}))^{\nu-1}\end{equation}
and denote $\omega_j \equiv \tilde{\omega}_j / \kappa_j$, 
then
\begin{align}
\begin{split}
\|\rho_{\delta,\nu,\mu}^{(1)}(\bm{x}_j)\cdot&\nabla\omega\|_{C^{0}(B_{r_j}(\bm{x}_j))} 
+
\|\rho_{\delta,\nu,\mu}^{(1+\alpha)}(\bm{x}_j)\cdot\nabla\omega\|_{C^{\alpha}(B_{r_j}(\bm{x}_j))}\\
&\leq 
C\Big(\|\rho_{\delta,\nu+1,\mu}^{(\alpha)}(\bm{x}_j)\cdot \mathscr{D}_{g_j}\omega\|_{C^{\alpha}(B_{2r_j}(\bm{x}_j))}+\|\rho_{\delta,\nu,\mu}^{(0)}(\bm{x}_j) \cdot \omega\|_{C^0(B_{2r_j}(\bm{x}_j))}\Big),
\end{split}
\end{align}
where 
\begin{equation}
r_j \equiv \underline{L}_{-}(\bm{x}_j).
\end{equation}

Now the last step is to show that all the values $ \rho_{\delta,\nu,\mu}^{(k+\alpha)}(\bm{y})$ are equivalent for every $\bm{y}\in B_{2r_j}(\bm{x}_j)$. Indeed, denote $\zeta_j \equiv z_{-}(\bm{x}_j)$ and $\xi_j \equiv z_{-}(\bm{y})$, then straightforward computations immediately imply that 
\begin{equation}
|\xi_j - \zeta_j| \leq C_0
\end{equation}
for some uniform constant $C_0>0$. By the definition of the weight function in Region $\V_{-}$,
\begin{align}
\frac{\rho_{\delta,\nu,\mu}^{(k+\alpha)}(\bm{y})}{\rho_{\delta,\nu,\mu}^{(k+\alpha)}(\bm{x}_j)}  = e^{\delta \cdot (\xi_j - \zeta_j)} \cdot \frac{(L_{-}(\bm{y}))^{\nu+k+\alpha}}{(L_{-}(\bm{x}_j))^{\nu+k+\alpha}}.
\end{align}
Therefore,
\begin{align}
\frac{1}{C_1} \leq \frac{\rho_{\delta,\nu,\mu}^{(k+\alpha)}(\bm{y})}{\rho_{\delta,\nu,\mu}^{(k+\alpha)}(\bm{x}_j)}  \leq C_1
\end{align}
and thus the proof of Case (a) is done.

Now we prove Case (b).
We showed in Section \ref{ss:rescaled-limits} that 
the limit space $(\mathcal{M}_{\infty}, \tilde{g}_{\infty}, \bm{x}_{\infty})$ is a finite rescale of $(X_{b_-}^4, g_{b-}, q_{-})$. 
Notice that, by the choice of the constant $C_0$ in Case (b), 
the geodesic ball 
 $B_2^{\tilde{g}_{\infty}}(\bm{x}_{\infty})$ is contained in the end part of $\mathcal{M}_{\infty}$ which has an $S^1$-fibration structure. It follows that $B_2^{\tilde{g}_{\infty}}(\bm{x}_{\infty})$ has uniformly bounded geometry. Then for every $\tilde{\omega}\in\Omega^1(\mathcal{M})$, the standard Schauder estimate holds and we have
\begin{equation}
\| \tilde{\omega} \|_{C^{1,\alpha}(B_{1}^{\tilde{g}_j}(\bm{x}_j))} \leq  C \Big(\|\mathscr{D}_{\tilde{g}_j}\tilde{\omega}\|_{C^{\alpha}(B_{2}^{\tilde{g}_j}(\bm{x}_j))} + \|\tilde{\omega}\|_{C^0(B_{2}^{\tilde{g}_j}(\bm{x}_j))} \Big),
\end{equation}
where $C>0$ is independent of the index $j$ and the constant $C_0$.
Now we rescale the above estimate to the original metric and we also rescale $\omega$ as in Case (a), which gives 
\begin{align}
\begin{split}
\|\rho_{\delta,\nu,\mu}^{(1)}(\bm{x}_j)\cdot&\nabla\omega\|_{C^{0}(B_{r_j}(\bm{x}_j))} 
+
\|\rho_{\delta,\nu,\mu}^{(1+\alpha)}(\bm{x}_j)\cdot\nabla\omega\|_{C^{\alpha}(B_{r_j}(\bm{x}_j))}\\
&\leq 
C\Big(\|\rho_{\delta,\nu+1,\mu}^{(\alpha)}(\bm{x}_j)\cdot \mathscr{D}_{g_j}\omega\|_{C^{\alpha}(B_{2r_j}(\bm{x}_j))}+\|\rho_{\delta,\nu,\mu}^{(0)}(\bm{x}_j) \cdot \omega\|_{C^0(B_{2r_j}(\bm{x}_j))}\Big),
\end{split}
\end{align}
where \begin{equation}
r_j \equiv \underline{L}_{-}(\bm{x}_j).
\end{equation}
It is similar to Case (a) that there is some constant $C_2>0$ which is independent of the index $j$ and the constant $C_0$, such that
\begin{equation}
\frac{1}{C_2} \leq \frac{\rho_{\delta,\nu,\mu}^{(k+\alpha)}(\bm{y})}{\rho_{\delta,\nu,\mu}^{(k+\alpha)}(\bm{x}_j)} \leq C_2.
\end{equation}
By definition, the weighted Schauder estimate \eqref{e:ball-weighted-schauder} immediately follows.

\vspace{0.5cm}

\noindent
{\bf Regions $\VI_{-}$ and $\VI_{+}$:} 

\vspace{0.5cm}

First, we consider the case that the fixed reference point $\bm{x}_j$ is in Region $\VI_{-}$. We choose a ball $B_1(\bm{x}_j)$ and the estimate \eqref{e:ball-weighted-schauder} 
is the standard Schauder estimate on $B_1(\bm{x}_j)$. Since the weight function in this region is uniformly bounded, the weighted Schauder estimate \eqref{e:ball-weighted-schauder} is equivalent to the standard one.

Next, if $\bm{x}_j$ is in Region $\VI_{+}$, then 
\begin{equation}
(\mathcal{M}, g_j, \bm{x}_j)\xrightarrow{C^{\infty}} (X_{b_+}^4, g_{b_+}, \bm{x}_{\infty})
\end{equation}
and the standard Schauder estimate states that there is a uniform constant $C>0$ such that for every $\tilde{\omega}\in\Omega^1(\mathcal{M})$,
\begin{equation}
\| \tilde{\omega} \|_{C^{1,\alpha}(B_1(\bm{x}_j))}
\leq C (\| \mathscr{D}_{g_j} \tilde{\omega} \|_{C^{\alpha}(B_1(\bm{x}_j))}
+ \| \tilde{\omega} \|_{C^{0}(B_1(\bm{x}_j))}
).\label{e:ball-standard-Schauder-ALH}
\end{equation}
So we just rescale the the $1$-form $\tilde{\omega}$ by letting 
$\kappa_j \equiv e^{\delta (2T_-+ 2T_+)}
$ and  $\omega=\tilde{\omega}/\kappa_j$, then \eqref{e:ball-standard-Schauder-ALH} is equivalent to
\begin{align}
\begin{split}
\|  e^{\delta (2T_-+ 2T_+)}
& \cdot\omega \|_{C^{1,\alpha}(B_1(\bm{x}_j))}\\
&\leq C \Big(\| e^{\delta (2T_-+ 2T_+)}
\cdot  \mathscr{D}_{g_j} \omega \|_{C^{\alpha}(B_1(\bm{x}_j))}
+ \|  e^{\delta (2T_-+ 2T_+)}\cdot 
\omega \|_{C^{0}(B_1(\bm{x}_j))}
\Big).
\end{split}
\end{align}
Up to some uniformly bounded constant, the above estimate implies the weighted Schauder estimate \eqref{e:ball-weighted-schauder}.

\end{proof}

\section{Perturbation to genuine hyperk\"ahler metrics} 
\label{s:existence}

In the previous sections, we have already defined the approximate metric $g_{\beta}$ and proved the elliptic regularity for the operator $\mathscr{D}_g$ on the manifold $\mathcal{M}$. 
This section is devided in $3$ subsections. In Section \ref{ss:injectivity}, we will prove the uniform injectivity of the linearized operator which is a crucial technical ingredient in proving the existence of a hyperk\"ahler triple. 
Theorem \ref{t:existence-hyperkaehler} is the main 
existence theorem of a hyperk\"ahler triple which will be proved in  
 Section \ref{ss:existence}. Specifically, we will  
set up the right Banach spaces and apply the implicit function theorem to Theorem \ref{t:existence-hyperkaehler}. Then in Section \ref{ss:completion} we will finish the main theorems introduced in Section \ref{ss:main-results}.

\subsection{The injectivity estimate for $\mathscr{D}_g$}
\label{ss:injectivity}

For the convenience of the arguments, we start with a standard fact concerning a Liouville theorem on a flat cylinder $\TT\times\dR$. 
\begin{lemma}\label{l:liouville-cylinder}
Let $(\TT\times\dR,g_0)$ be a cylinder with a flat  product metric $g_0$.  Denote by $\lambda_0>0$ the lowest eigenvalue of the torus $\TT$. If $u$ is a harmonic function on $\TT\times\dR$ with growth control
\begin{equation}|u(z)| = O(e^{\lambda z})\end{equation} for some $\lambda\in(0,\sqrt{\lambda_0})$, then
$u\equiv 0$.
\end{lemma}

Now we state a main technical result which gives the required effective estimate for the Dirac-type operator $\mathscr{D}_g$. 

\begin{proposition}[The Injectivity Estimate for $\mathscr{D}_g$]
\label{p:injectivity-of-D} 
Consider $(\mathcal{M},g_{\beta})$ with sufficiently large gluing parameter $\beta>0$. Assume that the 
parameters $\delta$, $\mu$ and $\nu$
satisfy 
\begin{enumerate}
\item $0<\delta<\frac{1}{10^3}\min\{\underline{\delta}_1, \underline{\delta}_2, \underline{\epsilon}_1,\underline{\epsilon}_2,\lambda_0, \delta_h,\delta_q\}$, 
\item  $\mu+\nu\in(0,1)$, 
\end{enumerate}
where $\underline{\delta}_1, \underline{\delta}_2, \underline{\epsilon}_1,\underline{\epsilon}_2$ are the fixed constants specified in Section \ref{ss:notations}, $\lambda_0>0$ is in Lemma \ref{l:liouville-cylinder}, $\delta_h>0$ is in Theorem \ref{t:liouville-1-form} and $\delta_q$ is in Corollary \ref{c:sutt}.
Then for every $\alpha\in(0,1)$,
 there exists a uniform constant $C=C(\alpha,\delta,\mu,\nu)>0$ which is independent of $\beta$ such that 
for every $\omega\in\Omega^1(\mathcal{M})$ it holds that
\begin{equation}
\|\omega\|_{C_{\delta,\nu,\mu}^{1,\alpha}(\mathcal{M})} \leq C \cdot \|\mathscr{D}_{g_{\beta}}\omega\|_{C_{\delta,\nu+1,\mu}^{0,\alpha}(\mathcal{M})}.
\end{equation}

\end{proposition}

\begin{proof}
By Proposition 8.2, it suffices to show that there exists a uniform constant $C>0$ such that 
\begin{equation}
\|\omega\|_{C_{\delta,\nu,\mu}^{0}(\mathcal{M})}\leq C\cdot\|\mathscr{D}_{g_{\beta}}\omega\|_{C_{\delta,\nu+1,\mu}^{0,\alpha}(\mathcal{M})}
\end{equation}
for all $\omega\in \Omega^1(\mathcal{M})$.
We argue by contradiction and suppose no such a uniform constant exists. Then we have  the following:
\begin{enumerate}\item a sequence of spaces $(\mathcal{M}_j,g_j)$ with the gluing parameter $\beta_j\to\infty$ such that
\begin{equation}
(\mathcal{M}_j, g_j, p_j)\xrightarrow{GH} (X_{\infty}, d_{\infty}, p_{\infty}).
\end{equation}

\item  a sequence of $1$-forms $\omega_j \in \Omega^1(\mathcal{M}_j)$ such that
\begin{align}
&\| \omega_j \|_{C_{\delta,\nu,\mu}^{0}(\mathcal{M}_j,g_j)} =  1
\\
&\|\mathscr{D}_{g_j} \omega_j \|_{C_{\delta,\nu+1,\mu}^{0,\alpha}(\mathcal{M}_j,g_j)} \to 0
\end{align}
as $j \rightarrow \infty$, 
\item 
a sequence of points $\bm{x}_j\in\mathcal{M}_j$ satisfying
\begin{equation}
|\rho_{j,\delta,\nu,\mu}^{(0)}(\bm{x}_j) \cdot  \omega_j(\bm{x}_j) |=1,
\end{equation}
where $\rho_{j,\delta,\nu,\mu}^{(0)}$ is a sequence of weight functions in $(\mathcal{M}_j, g_j)$.
\end{enumerate}

Now we are in a position to rescale the above contradicting sequences to produce a contradiction.
To start with, let $g_j$ be a sequence of contradicting metrics, and we denote the rescaling factors as follows:

\begin{enumerate}
\item Rescaling of the metrics: 

Let $\tilde{g}_j=\lambda_j^2\cdot g_j$, then with respect to the fixed reference point $\bm{x}_j\in\mathcal{M}_j$ picked as the above, we have the convergence,
\begin{equation}
(\mathcal{M}_j, \tilde{g}_j, \bm{x}_j) \xrightarrow{GH} (\mathcal{M}_{\infty}, \tilde{d}_{\infty}, \bm{x}_{\infty}). 
\end{equation}

\item Rescaling of the $1$-forms: 

Let $\kappa_j>0$ be a sequence of rescaling factors which will be determined later, such that \begin{equation}\tilde{\omega}_j \equiv \kappa_j \cdot \omega_j.\end{equation}

\item Rescaling of the weight functions: 

Since we need to distinguish between the weight functions on the sequence $\mathcal{M}_j$
and those on the limit spaces, we denote by $\rho_{j,\delta,\nu,\mu}^{(k+\alpha)}$ the weight functions on $\mathcal{M}_j$ and denote by $\rho_{\infty,\delta,\nu,\mu}^{(k+\alpha)}$ the weight functions on the limit spaces. Fix $k\in\dN$ and $\alpha\in(0,1)$, we rescale the weight function $\rho_{j,\delta,\nu,\mu}^{(k+\alpha)}$ by
\begin{equation}
\tilde{\rho}_{j,\delta,\nu,\mu}^{(k+\alpha)} = \tau_j^{(k+\alpha)} \cdot \rho_{j,\delta,\nu,\mu}^{(k+\alpha)}.
\end{equation}

\end{enumerate}
The above rescaling factors are chosen to satisfy the scale-invariance property of the weighted norm,
\begin{align}
1&\leq \tau_j^{(0)}\cdot \kappa_j \cdot \lambda_j^{-1} \leq 10  \\
1&\leq \tau_j^{(1)}\cdot \kappa_j \cdot \lambda_j^{-2} \leq 10 \\
1&\leq  \tau_j^{(1+\alpha)}\cdot \kappa_j \cdot \lambda_j^{-2-\alpha} \leq 10,
\end{align}
such that in this way we will obtain a sequence of $1$-forms $\tilde{\omega}_j \in \Omega^1(\mathcal{M}_j)$ with the property 
\begin{align}
\begin{split}
&\| \tilde{\omega}_j \|_{C_{\delta, \nu, \mu}^0(\mathcal{M}, \tilde{g}_j)} =  1 \\
&|\tilde{\rho}_{j,\delta,\nu,\mu}^{(0)}(\bm{x}_j) \cdot  \tilde{\omega}_j(\bm{x}_j) | =  1
\\
&\| \mathscr{D}_{\tilde{g}_j}\tilde{\omega}_j \|_{C_{\delta, \nu+1 , \mu}^{\alpha}(\mathcal{M}, \tilde{g}_j)} \to 0.
\end{split}
\end{align}

The basic strategy is to combine the compactness arguments and the Liouville theorems. That is, if $(\mathcal{M}_{\infty}, \tilde{g}_{\infty}, \bm{x}_{\infty})$ is non-collapsed, we apply Proposition \ref{p:weighted-Schauder-estimate} and 
 the $C^{1,\alpha}$-compactness to obtain a limiting $1$-form $\tilde{\omega}_{\infty}\in(\mathcal{M}_{\infty}, \tilde{g}_{\infty}, \bm{x}_{\infty})$ such that
\begin{align}
\begin{split}
&\| \tilde{\omega}_{\infty} \|_{C_{\delta, \nu, \mu}^0(\mathcal{M}_{\infty}, \tilde{g}_{\infty})} =  1 
\\
&|\tilde{\rho}_{\infty,\delta,\nu,\mu}^{(0)}(\bm{x}_{\infty}) \cdot  \tilde{\omega}_{\infty}(\bm{x}_{\infty}) | =  1
 \\
& \mathscr{D}_{\tilde{g}_{\infty}}\tilde{\omega}_{\infty}  \equiv 0. 
\end{split}\label{e:limiting-1-form}
\end{align}
We will apply the Liouville theorems to show that the above limiting $1$-form $\tilde{\omega}_{\infty}$ with controlled weighted norm is in fact 
vanishing on $\mathcal{M}_{\infty}$, which gives a contradiction.
Next, for a collapsed limit $(\mathcal{M}_{\infty}, \tilde{g}_{\infty}, \bm{x}_{\infty})$, to understand the limiting behavior of the operators $\mathscr{D}_{\tilde{g}_j}$ and the contradicting $1$-forms $\tilde{\omega}_j$, we will  lift everything to an appropriately chosen non-collapsed (local) normal cover such that the $C^{1,\alpha}$-compactness still applies on such a covering space.
On the other hand, by the representation lemma of the $1$-forms, see Lemma \ref{l:representation-of-1-form}, there are coefficient functions $f_{j}^x$, $f_j^y$, $f_j^z$ and $f_j^t$ such that  
\begin{equation}
\tilde{\omega}_j = f_{j}^x\theta_j^x + f_j^y\theta_j^x + f_j^z\theta_j^z + f_j^t\theta_j^t.
\end{equation}
 We will show that the $4$-tuples $(f_{j}^x,f_j^y,f_j^z,f_j^t)$ converge to a
 \begin{equation}(\tilde{\omega}_{\infty}, f_{\infty}^t)\equiv(f_{\infty}^x,f_{\infty}^y,f_{\infty}^z,f_{\infty}^t)
 \end{equation}
 which can be in effect viewed as the limits of the $1$-forms $\tilde{\omega}_j$. In addition, we will also show that at least one of $f_{\infty}^x$, $f_{\infty}^y$, $f_{\infty}^z$ and $f_{\infty}^t$ has a positive weighted H\"older norm at $\bm{x}_{\infty}$. 
Therefore, the desired contradiction just arises from various versions of Liouville theorems for harmonic functions in those different collapsed regions.

In accordance with the classification of the geometries of the rescaled limits in Section \ref{s:gluing-space}, we will proceed to produce the desired contradiction in each of the regions discussed in Section \ref{ss:rescaled-limits}. Precisely, we will correctly choose the rescaling factors such that the contradicting $1$-forms $\tilde{\omega}_j\in\Omega^1(\mathcal{M}_{j})$ will converge to some limit  which satisfies the norm control and satisfies the assumptions in the Liouville theorems in each region.

\vspace{0.5cm}

\noindent
{\bf Region $\I$:}

\vspace{0.5cm}

Assume that the reference point $\bm{x}_j$ is Region $\I$, then 
the rescaled limit is the standard Ricci-flat Taub-NUT space $(\mathcal{M}_{\infty}, \tilde{g}_{\infty}, \bm{x}_{\infty})$ with a limiting monopole $p_{m,\infty}$.
We choose the rescaling factors as follows, 
\begin{align}
\begin{split}
\lambda_j &= \beta_j^{\frac{1}{2}} 
\\
\tau_j^{(k+\alpha)} &= e^{ - \delta \cdot 2T_-} 
\cdot (\beta_j^{\frac{1}{2}})^{2\mu+\nu+k+\alpha}
 \\
\kappa_j &= e^{\delta \cdot 2T_-} 
\cdot (\beta_j^{-\frac{1}{2}})^{2\mu + \nu - 1}.
\end{split}
\end{align}
In the above way of rescaling, we have $d_{\tilde{g}_{\infty}}(p_{m,\infty},\bm{x}_{\infty})\leq C $ and the rescaled weight function in the limit space  is 
\begin{align}
\tilde{\rho}_{\infty,\delta,\nu,\mu}^{(k+\alpha)}(\bm{x})  = 
\begin{cases} 
1, & \bm{x}\in B_{1}(p_{m,\infty})
\\
 ( d_{\tilde{g}_{\infty}}(\bm{x}, p_{m,\infty}) )^{\mu + \nu + k + \alpha}, & \bm{x} \in \mathcal{M}_{\infty} \setminus B_{2}(p_{m,\infty}).
\end{cases}
\end{align}
Then the limiting $1$-form $\tilde{\omega}_{\infty}\in\Omega^1(\mathcal{M}_{\infty})$ satisfies that
\begin{align}
\begin{split}
&\mathscr{D}_{\tilde{g}_{\infty}} \tilde{\omega}_{\infty} \equiv 0 \\
&|\tilde{\rho}_{\infty,\delta,\nu,\mu}^{(0)}(\bm{x}_{\infty})\cdot \tilde{\omega}_{\infty}(\bm{x}_{\infty}) | = 1 \\
&\| \tilde{\omega}_{\infty} \|_{C_{\delta,\nu,\mu}^0(\mathcal{M}_{\infty})} =  1.\\
\end{split}
\end{align}
Notice that the above norm bound implies that for all $\bm{x}\in\mathcal{M}_{\infty}\setminus B_{2}(p_{m,\infty})$,
\begin{equation}
|\tilde{\omega}_{\infty}(\bm{x})| \leq  (d_{\tilde{g}_{\infty}}(\bm{x},p_{m,\infty}))^{-\mu-\nu}.
\end{equation}
Since $\tilde{\omega}_{\infty}$ is in the kernel of $\mathscr{D}_{\tilde{g}_{\infty}}$, immediately $\tilde{\omega}_{\infty}$ is harmonic with respect to the Taub-NUT metric $\tilde{g}_{\infty}$. 
Applying Lemma \ref{l:max}, we have $\tilde{\omega}_{\infty}\equiv 0$.

\vspace{0.5cm}

\noindent
{\bf Region $\II$:}

\vspace{0.5cm}

Now we discuss the case that the reference points $\bm{x}_j$ are in Region $\II$.
As what we discussed in Section \eqref{ss:rescaled-limits}, the rescaled geometries were separated in the following cases:
\begin{enumerate}
\item[(a)] There is a uniform constant $\sigma_0>0$ such that $2\beta_j^{-\frac{1}{2}} \leq d_{m}(\bm{x}_j) \leq \frac{1}{\sigma_0}\cdot \beta_j^{-\frac{1}{2}}$.

\item[(b)] The distance function $d_m(\bm{x}_j)$ to a pole $p_m$ satisfies 
\begin{align}
\frac{d_m(\bm{x}_j)}{\beta_j^{-\frac{1}{2}}}\to \infty, \ 
\frac{d_m(\bm{x}_j)}{\beta_j^{\frac{1}{2}}} \to 0.
\end{align}

\item[(c)] There is some uniform constant $C_0>0$ such that
\begin{equation}
0< C_0 \cdot \beta_j^{\frac{1}{2}}\leq d_m(\bm{x}_j)\leq \frac{\iota_0'}{4}\cdot\beta_j^{\frac{1}{2}}
\end{equation}
for all $1\leq m\leq m_0$.
\end{enumerate}

We start with our analysis in Case (a).
By Lemma \ref{l:rescaled-Taub-NUT}, the rescaled limit in Case (a) is a Ricci-flat Taub-NUT space $(\mathcal{M}_{\infty}, \tilde{g}_{\infty}, \bm{x}_{\infty})$ such that the $S^1$-fiber at infinity has length at least $\sigma_0>0$. The rescaling factors in this case are 
\begin{align}
\begin{split}
\lambda_j &= (d_m(\bm{x}_j))^{-1} 
\\
\tau_j^{(k+\alpha)} &= e^{ - \delta \cdot 2T_- } 
\cdot (\beta_j)^{\frac{\mu}{2}} \cdot (d_m(\bm{x}_j))^{- \mu -\nu - k - \alpha} 
\\
\kappa_j &= e^{\delta \cdot 2T_- } \cdot (\beta_j)^{-\frac{\mu}{2}}
\cdot  (d_m(\bm{x}_j))^{\mu+\nu-1}.
\end{split}
\end{align}
In the rescaled limit space, 
the limiting reference point $\bm{x}_{\infty}$ satisfies $d_{\tilde{g}_{\infty}}(\bm{x},p_{m,\infty})=1$. Moreover, the rescaled weight function in the limit space  is given by
\begin{equation}
\tilde{\rho}_{\infty,\delta,\nu,\mu}^{(k+\alpha)}(\bm{x}) = (d_{\tilde{g}_{\infty}}(p_{m,\infty}, \bm{x}))^{-\mu-\nu-k-\alpha},\ \bm{x}\in \mathcal{M}_{\infty}\setminus B_2(p_{m,\infty}),
\end{equation}
and the limiting $1$-form $\tilde{\omega}_{\infty}\in\Omega^1(\mathcal{M}_{\infty})$
satisfies 
\begin{align}
\begin{split}
&\mathscr{D}_{\tilde{g}_{\infty}}\tilde{\omega}_{\infty}\equiv0\\
&|\tilde{\rho}_{\infty,\delta,\nu,\mu}^{(0)}(\bm{x}_{\infty})\cdot \tilde{\omega}_{\infty}(\bm{x}_{\infty}) | = 1  \\
&\| \tilde{\omega}_{\infty} \|_{C_{\delta,\nu,\mu}^0(\mathcal{M}_{\infty})} \leq  1.\\
\end{split}
\end{align}
The above weighted norm bound implies that for every $\bm{x}\in\mathcal{M}_{\infty}\setminus B_2(p_{m,\infty})$, the limiting $1$-form $\tilde{\omega}_{\infty}$ satisfies the pointwise estimate 
\begin{equation}
|\tilde{\omega}_{\infty}(\bm{x})| \leq \Big(d_{\tilde{g}_{\infty}}(p_{m,\infty},\bm{x}) \Big)^{-\mu-\nu}.
\end{equation}
Applying Lemma \ref{l:max}, we have $\tilde{\omega}_{\infty}\equiv0$ on the rescaled limit $\mathcal{M}_{\infty}$, which completes the proof of Case (a).

The rescaled limit in Case (b) is the punctured Euclidean space $\dR^3\setminus\{0^3\}$. In this case, we choose the rescaling factors as follows,
\begin{align}
\begin{split}
\lambda_j &= (d_m(\bm{x}_j))^{-1} 
\\
\tau_j^{(k+\alpha)} &= e^{ - \delta \cdot 2T_- } \cdot (\beta_j)^{\frac{\mu}{2}}
\cdot  (d_m(\bm{x}_j))^{- \mu -\nu - k - \alpha} 
\\
\kappa_j &= e^{\delta \cdot 2T_- } \cdot (\beta_j)^{-\frac{\mu}{2}}
\cdot  (d_m(\bm{x}_j))^{\mu+\nu-1}.
\end{split}
\end{align}
In terms of the above rescaled metric, 
the reference point $\bm{x}_{\infty}$ satisfies $d_{g_0}(\bm{x},0^3)=1$. Moreover, the rescaled weight function in the limit space is given by
\begin{equation}
\tilde{\rho}_{\infty,\delta,\nu,\mu}^{(k+\alpha)}(\bm{x}) = (d_{g_0}(0^3, \bm{x}))^{-\mu-\nu-k-\alpha},\ \bm{x}\in \dR^3\setminus\{0^3\}.
\end{equation}

Mainly we will analyze the limiting behavior of the operator $\mathscr{D}_{\tilde{g}_j}$ under the collapsing sequence  $(\mathcal{M}, g_j, \bm{x}_j)$. Specifically, we will construct a globally defined $1$-form \begin{equation}\tilde{\omega}_{\infty}\in\Omega^1(\dR^3\setminus \{0^3\})\end{equation} and we will also show that the coefficient functions of $\tilde{\omega}_{\infty}$ are harmonic with respect to the Euclidean metric.
Our basic strategy is to apply Lemma \ref{l:representation-of-1-form} to reduce the convergence of the $1$-form $\tilde{\omega}_j$ to the convergence of the coefficient functions.
 Let 
\begin{equation}
\tilde{\omega}_j = f_{j}^x \cdot \theta_j^x + f_{j}^y \cdot \theta_j^y + f_{j}^z \cdot \theta_j^z + f_{j}^t \cdot \theta_j^t,
\end{equation}
then Lemma \ref{l:representation-of-1-form} and the circle bundle structure in this case guarantee the convergence of the frames $\{\theta_j^x,\theta_j^y, \theta_j^z,  \theta_j^t\}$.

Now we are in a position to construct the limits of the above coefficient functions.  
We start with the Gromov-Hausdorff convergence
\begin{equation}
(\mathcal{M}, \tilde{g}_j, \bm{x}_j)\xrightarrow{GH} (\dR^3, g_0, \bm{x}_{\infty})
\end{equation}
with $|\bm{x}_{\infty}|=1$.
For any fixed $R>10$,
let $A_{\frac{1}{R},R}^{g_0}(0^3)$ be an annulus in $\dR^3$ with respect to the Euclidean metric $g_0$. 
The first step is to obtain the limits of the coefficient functions $f_{j}^x$, $f_j^y$, $f_j^z$, $f_j^t$ with controlled weighted norms in the flat annulus $A_{\frac{1}{R},R}^{g_0}(0^3)$ under the above Gromov-Hausdorff convergence. Next, letting $R\to\infty$, we will apply Arzel\`a-Ascoli to obtain global limiting functions.

First, fix any $R>0$, we consider a Euclidean annulus $A_{\frac{1}{R},R}^{g_0}(0^3)\subset \dR^3$ and we claim that there are limiting functions $f_{\infty, R}^x$, $f_{\infty, R}^y$, $f_{\infty, R}^z$ and $f_{\infty, R}^t$ on $A_{\frac{1}{R},R}^{g_0}(0^3)\subset \dR^3$. 
For fixed $R>0$, there are $\bar{s}_0(R)>0$ and $N_0(R)>0$ such that 
$\{B_{2\bar{s}_0}(\bm{y}_{\infty,k})\}_{k=1}^{N}$ with $N\leq N_0$ is a finite collection of Euclidean balls which covers $A_{\frac{1}{R},R}^{g_0}(0^3)$ which satisfies
\begin{enumerate}
\item 
$A_{\frac{1}{R},R}^{g_0}(0^3) \subset \bigcup\limits_{s=1}^N B_{2\bar{s}_0}(\bm{y}_{\infty,k}) \subset A_{\frac{1}{3R},3R}^{g_0}(0^3)$

\item $\frac{\bar{s}_0}{3} \leq  d_{g_0}(\bm{y}_{\infty,k},\bm{y}_{\infty, k'}) \leq \bar{s}_0$ for all $1\leq k < k' \leq N$.
\end{enumerate}
We will verify that there exists a subsequence (still denoted by $j$) such that 
the above finite cover satisfy the following compatibility:
\begin{enumerate}
\item[(C1)]
$f_j^x$, $f_j^y$, $f_j^z$ and $f_j^t$ converge to harmonic functions $f_{\infty,k}^x$, $f_{\infty,k}^y$, $f_{\infty,k}^z$ and $f_{\infty,k}^t$ on 
every ball $B_{2\bar{s}_0}(\bm{y}_{\infty,k})$.

\item[(C2)] The above locally defined limiting functions can be patched together in the sense that if $B_{2\bar{s}_0}(\bm{y}_{\infty,k}) \cap B_{2\bar{s}_0}(\bm{y}_{\infty,k'})\neq\emptyset$, then 
\begin{align}
\begin{split}
f_{\infty,k}^{x}(\bm{y}_{\infty}) &= f_{\infty, k'}^{x}(\bm{y}_{\infty}), \
 f_{\infty,k}^{y}(\bm{y}_{\infty}) = f_{\infty, k'} ^{y}(\bm{y}_{\infty}), \\
  f_{\infty,k}^{z}(\bm{y}_{\infty}) &= f_{\infty, k'} ^{z}(\bm{y}_{\infty}), \ 
 f_{\infty,k}^{t}(\bm{y}_{\infty}) = f_{\infty, k'} ^{t}(\bm{y}_{\infty}) 
\end{split}
\end{align}
holds for all $\bm{y}_{\infty} \in B_{2\bar{s}_0}(\bm{y}_{\infty,k}) \cap B_{2\bar{s}_0}(\bm{y}_{\infty,k'})$.
\end{enumerate}
The above compatibility properties immediately imply that there are well-defined harmonic limiting functions $f_{\infty,R}^{x}$, $f_{\infty,R}^{y}$, $f_{\infty,R}^{z}$ and $f_{\infty,R}^{t}$
on $A_{\frac{1}{R},R}(0^3)$.

To show property (C1), by taking some subsequence, it suffices to show that for each ball $B_{2\bar{s}_0}(\bm{y}_{\infty, k})$ in the above finite cover,
there is some subsequence in the original sequence $\{j\}$ such that the coefficient functions $f_{j, k}^{x}$ converge to a harmonic function 
$f_{\infty,k}^{x}$. For this purpose, we need to locally unwrap the collapsed fibers and discuss the convergence of the coefficient functions $f_{j, k}^{x}$ on non-collapsed universal covers. 

Now we take a sequence of geodesic balls $B_{2\bar{s}_0}(\bm{y}_{j,k})$ with
\begin{equation}
(B_{2\bar{s}_0}(\bm{y}_{j,k}) , \tilde{g}_j) \xrightarrow{GH} (B_{2\bar{s}_0}(\bm{y}_{\infty,k}), g_0).
\end{equation}
Denote by $\ell_j$ ($\to 0$) the length of the collapsed $S^1$-fiber at $\bm{y}_j$ and define 
\begin{equation}
\Gamma_j = \Gamma_{\epsilon_j}(\bm{y}_{j,k}) \equiv \Image[\pi_1(B_{\epsilon_j}(\bm{y}_j)) \to \pi_1(B_{2\bar{s}_0}(\bm{y}_j))]
\end{equation}
where $\epsilon_j>0$ are chosen such that  $2\ell_j \leq \epsilon_j \leq  4\ell_j$. Immediately in our context, 
$\pi_1(B_{2\bar{s}_0}(\bm{y}_{j,k})) = \Gamma_j$ and $\Gamma_j$ is isomorphic to $\dZ$.
Now let 
\begin{equation}
\pr_j: (\widehat{B_{2\bar{s}_0}(\bm{y}_{j,k})}, \hat{g}_j, \hat{\bm{y}}_{j,k}) \longrightarrow (B_{2\bar{s}_0}(\bm{y}_{j,k}), \tilde{g}_j, \bm{y}_{j,k})
\end{equation}
be the universal covering map with $B_{2\bar{s}_0}(\bm{y}_{j,k}) = \widehat{B_{2\bar{s}_0}(\bm{y}_{j,k})} / \Gamma_j$. 
Now on the universal covers, we have the equivariant convergence and the following diagram,
\begin{equation}
\xymatrix{
\Big(\widehat{B_{2\bar{s}_0}(\bm{y}_{j,k})}, \hat{g}_j, \Gamma_j,\hat{\bm{y}}_{j,k}\Big)\ar[rr]^{eqGH}\ar[d]_{\pr_j} &   & \Big(\widehat{Y}_k, \hat{g}_{\infty}, \Gamma_{\infty}, \hat{\bm{y}}_{\infty,k}\Big)\ar [d]^{\pr_{\infty}}
\\
 \Big(B_{2\bar{s}_0}(\bm{y}_{j,k}), \tilde{g}_j, \bm{y}_{j,k}\Big)\ar[rr]^{GH} && \Big(B_{2\bar{s}_0}(\bm{y}_{\infty,k}), g_0, \bm{y}_{\infty,k}\Big)
}\label{e:equivariant-convergence}
\end{equation}
which satisfies the following properties:
\begin{enumerate}
\item[(e1)] the universal covers $(\widehat{B_{2\bar{s}_0}(\bm{y}_{j,k})}, \hat{g}_j, \hat{\bm{y}}_{j,k})$ are non-collapsed and have uniformly bounded curvatures,

\item[(e2)] the limiting Lie group $\Gamma_{\infty}$  is diffeomorphic to $\dR$ and acts isometrically on the limit space $(\widehat{Y}_k, \hat{g}_{\infty}, \hat{\bm{y}}_{\infty,k})$,

\item[(e3)] 
 the universal covering maps $\pr_j$ converge to a Riemannian submersion 
\begin{equation}
\pr_{\infty}: (\widehat{Y}_k, \hat{g}_{\infty}, \hat{\bm{y}}_{\infty,k})\longrightarrow(B_{2\bar{s}_0}(\bm{y}_{\infty,k}), g_0, \bm{y}_{\infty,k})
\end{equation}
with
$B_{2\bar{s}_0}(\bm{y}_{\infty,k}) = \widehat{Y}_k/\Gamma_{\infty}$,

\item[(e4)] for every $\hat{\bm{z}}_{\infty}\in\widehat{Y}_k$, the  orbit $\Gamma_{\infty}\cdot \hat{z}_{\infty}$ is a geodesic in $\widehat{Y}_k$ and isometric to $(\dR,dt^2)$. In particular, $(\widehat{Y}_k, \hat{g}_{\infty}, \hat{\bm{y}}_{\infty})$ is isometric to $B_{2\bar{s}_0}(0^3)\times\dR$ in the Euclidean space $\dR^4$.
\end{enumerate}
Indeed, property (e1) follows from Lemma \ref{l:codim-1}. Property (e2) and (e3)
follow from the definition of the equivariant convergence. Property (e4) immediately follows from Lemma \ref{l:2nd-fundamental-form}. Actually, by Lemma \ref{l:2nd-fundamental-form},
 the second fundamental form of each $\Gamma_{\infty}$-orbit is vanishing. In other words, each  $\Gamma_{\infty}$-orbit is a geodesic in $\widehat{Y}_k$. Combining with the facts that the limiting projection $\pr_{\infty}$ is a Riemannian submersion and $B_{2\bar{s}_0}(\bm{y}_{\infty, k})$ is a Euclidean ball, then 
  $\widehat{Y}_k\equiv B_{2\bar{s}_0}(0^3)\times \mathbb{R}$ and $\hat{g}_{\infty}$   is isometric to the Euclidean metric.

We will apply the above equivariant convergence to construct harmonic functions $f_{\infty,k}^x$,  $f_{\infty,k}^y$, $f_{\infty,k}^z$ and $f_{\infty,k}^t$ in $B_{2\bar{s}_0}(\bm{y}_{\infty,k})$. We only show the construction for $f_{\infty,k}^x$. 
Notice that Proposition \ref{p:weighted-Schauder-estimate} implies that the $\Gamma_j$-invariant lifted functions 
$\hat{f}_{j}^x$ satisfy the uniform weighted Schauder estimate 
\begin{equation}\|\hat{f}_{j}^x\|_{C_{\delta,\nu,\mu}^{1,\alpha}(\widehat{B_{2\bar{s}_0}(y_{j})})} 
\leq C\end{equation}
with respect to the lifted weight function.
Applying Arzel\`a-Ascoli, passing to a subsequence, there is a limiting function $\hat{f}_{\infty,k}^x$ with
\begin{equation}
\|\hat{f}_{\infty, k}^x\|_{C_{\delta,\nu,\mu}^{1,\alpha'}(\widehat{Y}_k)} 
\leq C
\end{equation}
with $0<\alpha'<\alpha<1$.
Combining with the above equivariant convergence, we obtain that the limit function $\hat{f}_{\infty,k}^x$ is $\Gamma_{\infty}$-invariant which descends to a function $f_{\infty,k}^x$ in $B_{2\bar{s}_0}(\bm{y}_{\infty})$. Now we prove that $f_{\infty,k}^x$ is a harmonic function on $B_{2\bar{s}_0}(\bm{y}_{\infty})$. 
In fact, the lifted $1$-forms $\hat{\omega}_j$ also satisfies
\begin{equation}
\|\hat{\omega}_j\|_{C_{\delta,\nu,\mu}^{1,\alpha}(\widehat{B_{2\bar{s}_0}(y_{j})})}  \leq C  
\end{equation}
and hence there is a limiting $1$-form $\hat{\omega}_{\infty,k}$ satisfying
\begin{equation}
\|\hat{\omega}_{\infty,k}\|_{C_{\delta,\nu,\mu}^{1,\alpha'}(\widehat{Y}_k)}  \leq C  
\end{equation}
for $0<\alpha'<\alpha<1$.
The contradiction assumption implies that $\hat{\omega}_{\infty,k}$
satisfies
\begin{align}
\mathscr{D}_{\hat{g}_{\infty}}\hat{\omega}_{\infty,k} \equiv 0\ \text{in}\ \widehat{Y}^k.
\end{align}
The standard elliptic regularity theory for $\mathscr{D}_{\hat{g}_{\infty}}$ shows that the $1$-form $\hat{\omega}_{\infty,k}$ is $C^{\infty}$. This implies that 
 $\hat{f}_{\infty, k}^x\in C^{\infty}(\widehat{Y}_k)$, then by Lemma \ref{l:D-harmonic} gives the equation
\begin{equation}\Delta_{\hat{g}_{\infty}}(\hat{f}_{\infty, k}^x)=0.\end{equation} Applying Property (e4),
\begin{equation}
\Delta_{g_{0}} (f_{\infty, k}^x) =\Delta_{\hat{g}_{\infty}}(\hat{f}_{\infty, k}^x) = 0.
\end{equation}
The construction of the harmonic limiting functions $f_{\infty, k}^{y}$, $f_{\infty, k}^{z}$ and $f_{\infty, k}^{t}$ is verbatim.

We are ready to prove the compatibility property (C2). 
To this end, we take the union
\begin{equation}
B_{\infty} \equiv B_{2\bar{s}_0}(y_{\infty,k}) \cup B_{2\bar{s}_0}(y_{\infty,k'})
\end{equation}
with
\begin{equation}
 B_{2\bar{s}_0}(\bm{y}_{\infty,k}) \cap B_{2\bar{s}_0}(\bm{y}_{\infty,k'})
\neq \emptyset.
\end{equation}
Let $B_{j} \equiv B_{2\bar{s}_0}(\bm{y}_{j,k}) \cup B_{2\bar{s}_0}(\bm{y}_{j,k'})
$, then 
\begin{equation}
(B_{j}, \tilde{g}_j) \xrightarrow{GH} (B_{\infty}, g_0).
\end{equation}
By the same arguments as the above, $B_{j}$ has uniformly bounded curvatures and the universal covering space $(\widetilde{B}_{j}, \hat{g}_j)$
is non-collapsed. Moreover, the equivariant convergence with property (e1)-(e5) as the above still holds in this case. By passing to some subsequence, 
the lifted coefficient functions $\hat{f}_{j}^{x}$  are $C^{1,\alpha'}$-converging to some invariant limiting function $\hat{f}_{\infty, k, k'}^{x}$ on $\widehat{B}_{\infty}$ such that
\begin{equation}
\hat{f}_{\infty, k, k'}^x|_{\widehat{Y}_k} = \hat{f}_{\infty,k}^{x} 
\end{equation}
 Therefore, $\hat{f}_{\infty, k, k'}^x$ descends to a function
$f_{\infty, k, k'}^x$ on $B_{\infty}$ such that
\begin{equation}
f_{\infty, k, k'}^x|_{B_{2\bar{s}_0}(\bm{y}_{\infty,k})} = f_{\infty,k}^x.
\end{equation}
In addition, $f_{\infty, k, k'}^x$ is harmonic on $B_{\infty}$, so we have managed to extend the local harmonic limiting function $f_{\infty,k}^x$ to the union $B_{\infty}$. Repeating the above arguments, we can extend the limiting functions to the whole annulus $A_{\frac{1}{R},R}^{g_0}(0^3)$. 

Consider the $4$-tuple of harmonic functions $(f_{\infty,R}^x, f_{\infty,R}^y, f_{\infty,R}^z, f_{\infty,R}^t)$ in the flat annulus $A_{\frac{1}{R},R}^{g_0}(0^3)$ obtained from the above construction, and we write
\begin{equation}
\tilde{\omega}_{\infty,R} = f_{\infty, R}^x dx + f_{\infty, R}^y dy + f_{\infty, R}^z dz.
\end{equation}
Immediately, we have the weighted norm control
\begin{align}
\begin{split}
|\tilde{\rho}_{\infty,\delta,\nu,\mu}^{(0)}(\bm{x}_{\infty})\cdot \tilde{\omega}_{\infty,R}({\bm{x}}_{\infty})|  + |\tilde{\rho}_{\infty,\delta,\nu,\mu}^{(0)}(\bm{x}_{\infty})\cdot f_{\infty,R}^t(\bm{x}_{\infty})|  &\geq \frac{1}{30}
\\
\| \tilde{\omega}_{\infty,R} \|_{C_{\delta,\nu,\mu}^{1,\alpha'}(A_{\frac{1}{R},R}^{g_0}(0^3))} + \| f_{\infty,R}^t\|_{C_{\delta,\nu,\mu}^{1,\alpha'}(A_{\frac{1}{R},R}^{g_0}(0^3))} &\leq C,
\label{e:norm-control-quotient}
\end{split}
\end{align}
where $0<\alpha'<\alpha<1$.
The above construction enables us to define a global harmonic $4$-tuple 
on the punctured Euclidean space $\dR^3\setminus\{0^3\}$ by applying the standard exhaustion arguments. 
Let $R\to +\infty$, by applying \eqref{e:norm-control-quotient} and Arzel\`a-Ascoli, there is a global $4$-tuple of harmonic functions $(f_{\infty}^x, f_{\infty}^y, f_{\infty}^z, f_{\infty}^t)$  in $\dR^3\setminus\{0^3\}$ 
and we denote 
\begin{equation}
\tilde{\omega}_{\infty} = f_{\infty}^x dx + f_{\infty}^y dy + f_{\infty}^z dz.
\end{equation} 
Then we have the weighted norm control, 
\begin{align}
\begin{split}
|\tilde{\rho}_{\infty,\delta,\nu,\mu}^{(0)}(\bm{x}_{\infty})\cdot  \tilde{\omega}_{\infty}(\bm{x}_{\infty}) | + |\tilde{\rho}_{\infty,\delta,\nu,\mu}^{(0)}(\bm{x}_{\infty})\cdot f_{\infty}^t(\bm{x}_{\infty})|  &\geq \frac{1}{30}
\\
\| \tilde{\omega}_{\infty} \|_{C_{\delta,\nu,\mu}^{1,\gamma}(\dR^3 \setminus \{0^3\})} + \| f_{\infty}^t \|_{C_{\delta,\nu,\mu}^{1,\gamma}}&\leq  C,
\end{split}
\end{align}
where $0<\gamma<\alpha'<\alpha<1$.
The weighted norm bound implies that the limiting functions have the following controlled behavior,
\begin{equation}
(|f_{\infty}^{x}| + |f_{\infty}^{y}| + |f_{\infty}^{z}| + |f_{\infty}^{t}|)(\bm{x})
\leq C\Big(d_{g_0}(\bm{x},0^3) \Big)^{-\mu-\nu},\ \forall \bm{x}\in\dR^3\setminus\{0^3\}.
\end{equation}
By the standard removable singularity theorem, the $4$-tuple of harmonic functions $(f_{\infty}^{x}, f_{\infty}^{y}, f_{\infty}^{z}, f_{\infty}^{t})$ extend to the entire Euclidean space $\dR^3$. 
 Applying the standard Liouville theorem for harmonic functions, we conclude that $\tilde{\omega}_{\infty} \equiv 0$ and $f_{\infty}^t \equiv 0$. So the contradiction arises, which completes the proof of Case (b).

Now we consider Case (c).
We have shown in Section \ref{ss:rescaled-limits} that the rescaled limit in Case (c) is a punctured flat cylinder $(\TT \times \mathbb{R})\setminus\mathcal{P}_{m_0}$. More precisely, we chose a sequence of punctured  unbounded domains $\mathring{U}_j$ containing $\bm{x}_j$ such that
\begin{equation}
(\mathring{U}_j, \tilde{g}_j, \bm{x}_j) \xrightarrow{GH} \Big((\TT \times \mathbb{R})\setminus\mathcal{P}_{m_0}, g_0, \bm{x}_{\infty}\Big).
\end{equation}
Moreover, the curvatures of the above rescaled spaces are uniformly bounded away from the singular points in $\mathcal{P}_{m_0}$.

Next we study the limiting weight functions.  For any fixed reference point $\bm{x}_j$ in this case, denote $d_j\equiv \min\limits_{1\leq m\leq m_0} \{d_m(p_m, \bm{x}_j)\}$ and we choose the following rescaling factors
\begin{align}
\begin{split}
\lambda_j &= d_j^{-1}\\
\tau_j^{(k+\alpha)} &= e^{ - \delta \cdot 2T_- }  \cdot (\beta_j)^{\frac{\mu}{2}}  
\cdot  (d_j^{-1})^{\mu + \nu + k + \alpha} 
\\
\kappa_j &= e^{\delta \cdot 2T_- }  \cdot (\beta_j)^{-\frac{\mu}{2}}
\cdot  (d_j)^{\mu+\nu-1}.
\end{split}
\end{align}
So the rescaled weight function in the limit space satisfies 
\begin{align}
\tilde{\rho}_{\infty,\delta,\nu,\mu}^{(k+\alpha)}(\bm{x}) 
=\begin{cases}
(d_{g_0}(p_m, \bm{x}))^{\mu+\nu+k+\alpha}, & \bm{x} \in B_{\iota_0''}^{g_0}(p_m)\ \text{for some} \ 1\leq m\leq m_0, 
\\
Q_{\nu,\mu,k,\alpha}, & \bm{x}\in  \bigcap\limits_{m=1}^{m_0}A_m^{g_0}(2\iota_0'',T_0''),
\\
 e^{\delta z(\bm{x})}, &  \bm{x}\in (\TT\times\dR)\setminus \bigcup\limits_{m=1}^{m_0} B_{T_0''}^{g_0}(p_m),
\end{cases}
\end{align}
where $\iota_0''\in[1, \frac{\iota_0'}{C_0}]$ is some definite constant and $Q_{\nu,\mu, k, \alpha}$ is some uniform constant depending on $\nu$, $\mu$, $k$ and $\alpha$.

Similar to Case (b), in order to apply the Liouville theorem in the collapsed limit, we need to construct a global defined $1$-form in the collapsed limit and deduce the corresponding equation.

Fix $R>0$, denote by $T_R(S)$ the $R$-tubular neighborhood of a compact set $S$, applying Lemma \ref{l:codim-1}, then 
we have the following curvature estimate on the sequence of annuli $T_{3R}^{\tilde{g}_j}(\mathcal{P}_{m_0})\setminus T_{\frac{1}{R}}^{\tilde{g}_j}(\mathcal{P}_{m_0})$:
\begin{equation}
\|\Rm_{\tilde{g}_j}\|_{L^{\infty}\Big(T_{3R}^{\tilde{g}_j}(\mathcal{P}_{m_0})\setminus T_{\frac{1}{R}}^{\tilde{g}_j}(\mathcal{P}_{m_0})\Big)} \leq K_0 \cdot R^2,\end{equation}
where $K_0>0$ is an absolute constant.  
Applying the same arguments as in Case (b), one can construct a limiting pair
\begin{equation}
(\tilde{\omega}_{\infty}, f_{\infty}^t)\in\Omega^1\Big((\TT\times\dR)\setminus\mathcal{P}_{m_0})\Big)\oplus\Omega^0\Big((\TT\times\dR)\setminus\mathcal{P}_{m_0})\Big)
\end{equation}
such that
\begin{align}
\begin{split}
|\tilde{\rho}_{\infty,\delta,\nu,\mu}^{(0)}(\bm{x}_{\infty})\cdot\tilde{\omega}_{\infty}({\bm{x}}_{\infty})| + |\tilde{\rho}_{\infty,\delta,\nu,\mu}^{(0)}(\bm{x}_{\infty})\cdot f_{\infty}^t(\bm{x}_{\infty})|  &\geq \frac{1}{30}
\\
\|\tilde{\omega}_{\infty}\|_{C_{\delta,\nu,\mu}^{1,\gamma}((\TT\times\dR)\setminus\mathcal{P}_{m_0})} &\leq C,\label{e:norm-control-quotient-case-c}
\end{split}
\end{align}
where $0<\gamma<\alpha<1$.
Let $\theta_{\infty}^x$, $\theta_{\infty}^y$ and $\theta_{\infty}^z$ be the canonical  parallel $1$-forms with unit length on $\TT\times\dR$, then 
\begin{equation}
\tilde{\omega}_{\infty}  = f_{\infty}^x \theta_{\infty}^x +
f_{\infty}^y \theta_{\infty}^y + f_{\infty}^z \theta_{\infty}^z.
\end{equation}
Moreover, it holds in the punctured cylinder $(\TT\times\dR)\setminus\mathcal{P}_{m_0}$ that \begin{equation}
\Delta_{g_0} f_{\infty}^x = \Delta_{g_0} f_{\infty}^y
= \Delta_{g_0} f_{\infty}^z = \Delta_{g_0} f_{\infty}^t
\equiv 0.
\end{equation}
Next, the weighted norm bound implies that the limiting $1$-form $\tilde{\omega}_{\infty}\in(\TT\times\dR)\setminus\mathcal{P}_{m_0} 
$ has the following controlled behavior,
\begin{align}
\begin{split}
|f_{\infty}^x(\bm{x})| + |f_{\infty}^y(\bm{x})| + |f_{\infty}^z(\bm{x})| + |f_{\infty}^t(\bm{x})| 
&\leq C\Big(d_{g_0}(\bm{x},p_m) \Big)^{-\mu-\nu}, \ \bm{x} \in B_{\iota_0''}^{g_0}(p_m),
\\ 
|f_{\infty}^x(\bm{x})| + |f_{\infty}^y(\bm{x})| + |f_{\infty}^z(\bm{x})| + |f_{\infty}^t(\bm{x})| 
&\leq C  e^{-\delta z(\bm{x})}, \ |z(\bm{x})| \geq Z_0,  \label{e:exp-growth-cyl}
\end{split}
\end{align}
for some sufficiently large $Z_0>0$.
Since we have required that
\begin{equation}
0<\mu + \nu <1,
\end{equation}
it is standard that
 the singularities in $\mathcal{P}_{m_0}$ are removable. 
It follows that the harmonic functions $f_{\infty}^x$, $f_{\infty}^y$, $f_{\infty}^z$ and $f_{\infty}^t$ extend to the entire flat cylinder $\TT\times\dR$ and they satisfy the above asymptotic behavior.
Applying Lemma \ref{l:liouville-cylinder} to the coefficient functions with the growth condition \eqref{e:exp-growth-cyl}, we conclude that $\tilde{\omega}_{\infty} \equiv 0$ and $f_{\infty}^t \equiv 0$. So the proof of Case (c) is complete.

\vspace{0.5cm}

\noindent
{\bf Region $\III$:}

\vspace{0.5cm}

The proof for Region $\III$ is identical to Case (c) of Region $\II$.

\vspace{0.5cm}

\noindent
{\bf Regions $\IV_{-}$ and Region $\IV_{+}$:}

\vspace{0.5cm}

We only focus on the case that the reference points $\bm{x}_j$ are located in Region $\IV_{-}$. The proof for Region $\IV_{+}$ is verbatim. Region $\IV_{-}$ has two different types of rescaling geometries (see Section \ref{ss:rescaled-limits}) which are given by the following two cases:

\begin{enumerate}
\item[(a)] There is a uniform constant $C_0>0$ independent of $j$ such that \begin{equation}5T_0' \leq d_{\tilde{g}_j}(p_m,\bm{x}_j) \equiv \lambda_j\cdot d_m(\bm{x}_j) \leq C_0\end{equation} for each $1\leq m\leq m_0$.

\item[(b)] The reference points $\bm{x}_j$ in Region $\IV_{-}$ satisfy
\begin{equation}d_{\tilde{g}_j}(p_m,\bm{x}_j) \equiv \lambda_j\cdot d_m(\bm{x}_j) \to \infty.\end{equation}
\end{enumerate}

For fixed reference points $\bm{x}_j$ satisfying Case (a), we have the following convergence 
\begin{equation}
(\mathring{U}_j, \tilde{g}_j, \bm{x}_j)\xrightarrow{GH} \Big((\dT^2\times\dR) \setminus \mathcal{P}_{m_0}, g_0, \bm{x}_{\infty}\Big),
\end{equation} 
where $g_0$ is a flat  product metric on $\TT\times\dR$.
We choose the rescaling factors as follows, 
\begin{align}
\begin{split}
\lambda_j &= (L_{-}(\bm{x}_j))^{-1}
\\
\tau_j^{(k+\alpha)} &= e^{-\delta(2T_-)}  \cdot  (L_{-}(\bm{x}_j))^{-\nu  - k - \alpha}  
\\
\kappa_j &= e^{\delta (2T_-)} \cdot (L_{-}(\bm{x}_j))^{\nu  - 1},
\end{split}
\end{align}
and the limiting weight function is 
\begin{align}\tilde{\rho}_{\infty,\delta,\nu,\mu}^{(k+\alpha)}=
\begin{cases}
(d_{g_0}(p_m, \bm{x}))^{\nu+\mu+k+\alpha}, &   \bm{x} \in B_{\iota_0''}^{g_0}(p_m)\ \text{for some} \ 1\leq m\leq m_0, 
\\ 
Q_{\nu,\mu,k,\alpha} , & \bm{x}\in  \bigcap\limits_{m=1}^{m_0}A_m^{g_0}(2\iota_0'',T_0''),
\\
e^{\delta z(\bm{x})}, &   \bm{x}\in (\TT\times\dR)\setminus \bigcup\limits_{m=1}^{m_0} B_{T_0''}^{g_0}(p_m),
\end{cases}
\end{align}
where $\iota_0''\in[1, \frac{\iota_0}{C_0}]$ is some definite constant and $Q_{\nu,\mu, k, \alpha}$ is some uniform constant depending on $\nu$, $\mu$, $k$, $\alpha$. So the rest of the proof is identical to the proof of Case (c) in Region $\II$.

Now we prove Case (b). We showed in Section \ref{ss:rescaled-limits} that in this case we have the convergence
\begin{equation}
(U_j, \tilde{g}_j, \bm{x}_j) \xrightarrow{GH} (\TT \times \mathbb{R}, g_0, \bm{x}_{\infty}),
\end{equation}
where $g_0$ is a flat  product metric on $\TT\times\dR$.
The rescaling factors are chosen as the following
\begin{align}
\begin{split}
\lambda_j &= (L_{-}(\bm{x}_j))^{-1}
\\
\tau_j^{(k+\alpha)} &= e^{-\delta(2T_--z_j)}  \cdot  (L_{-}(\bm{x}_j))^{-\nu  - k - \alpha}  
\\
\kappa_j &= e^{\delta (2T_-+z_j)} \cdot (L_{-}(\bm{x}_j))^{\nu  - 1},
\end{split}
\end{align}
where $z_j\equiv z(\bm{x}_j)$.
We also translate the $z$-coordinate by $\tilde{z}(\bm{x}) = z(\bm{x}) - z_j$.
It gives the limiting weight function 
\begin{equation}
\tilde{\rho}_{\infty,\delta,\nu,\mu}^{(k+\alpha)}(\bm{x})= e^{\delta \tilde{z}(\bm{x})}, \ \forall \bm{x}\in\TT\times\dR.
\end{equation}

The proof of the next stage is similar to the proof of Case (c) of Region $\III$. We follow all the notations there.
Applying exactly the same arguments,  we obtain the limiting pair $(\tilde{\omega}_{\infty}, f_{\infty}^t)\in\Omega^1(\TT\times\dR)\oplus C^{\infty}(\TT\times\dR)$
which satisfy 
\begin{align}
\begin{split}
&\Delta_{g_0}f_{\infty}^x = \Delta_{g_0}f_{\infty}^y = \Delta_{g_0}f_{\infty}^z=\Delta_{g_0}f_{\infty}^t \equiv0\\
&|\tilde{\rho}_{\infty,\delta,\nu,\mu}^{(0)}(\bm{x}_{\infty})\cdot \tilde{\omega}_{\infty}(\bm{x}_{\infty}) | + | \tilde{\rho}_{\infty,\delta,\nu,\mu}^{(0)}(\bm{x}_{\infty})\cdot f_{\infty}^t(\bm{x}_{\infty}) |  \geq  \frac{1}{30} \\
&\| \tilde{\omega}_{\infty} \|_{C_{\delta,\nu,\mu}^{1,\alpha'}(\TT\times\dR)} \leq  C,\ 0<\alpha'<\alpha<1,\\
\end{split}
\end{align}
which implies that
\begin{align}
|f_{\infty}^x(\bm{x})| + |f_{\infty}^y(\bm{x})| + | f_{\infty}^z(\bm{x})| + | f_{\infty}^t(\bm{x})| \leq C  e^{-\delta \tilde{z}(\bm{x})},\ \bm{x}\in \TT\times\dR.
\end{align}
Applying Lemma \ref{l:liouville-cylinder}, we conclude that $\tilde{\omega}_{\infty} \equiv 0$ and $f_{\infty}^t\equiv0$. So we complete the proof of Case (b).

\vspace{0.5cm}

\noindent
{\bf Regions $\V_{-}$ and $\V_{+}$:}

\vspace{0.5cm}

First, we assume that the reference points $\bm{x}_j$ are located in $\V_{-}$.
As what was discussed in Section \ref{ss:rescaled-limits}, it is natural to separate Region $\V_{-}$ in the following cases
\begin{enumerate}
\item[(a)] Assume 
 $z_{-}(\bm{x}_j)\to \infty$.
\item[(b)] Assume that there is some constant $C_0>0$ independent of the index $j$ such that $10\zeta_0^{-}\leq z_{-}(\bm{x}_j) \leq C_0$.
\end{enumerate}

The rescaled limit in Case (a) is the flat cylinder $\dT^2 \times \mathbb{R}$  and we have the convergence (see Section \ref{ss:rescaled-limits}), 
\begin{equation}
(U_j, \tilde{g}_j, \bm{x}_j) \xrightarrow{GH} (\TT\times\dR, g_0, \bm{x}_{\infty}).
\end{equation}
We choose the corresponding rescaling factors 
\begin{align}
\begin{split}
\lambda_j &= (\underline{L}_-(\bm{x}_j))^{-1} 
\\
\tau_j^{(k+\alpha)} &= e^{-\delta z_j}\cdot (\underline{L}_-(\bm{x}_j))^{-\nu-k-\alpha} \\
\kappa_j &= e^{\delta z_j}\cdot (\underline{L}_-(\bm{x}_j))^{\nu-1},
\end{split}
\end{align}
where $z_j\equiv z_-(\bm{x}_j)$.
Hence, under the $z$-coordinate translation $\tilde{z}_-(\bm{x})=z_-(\bm{x})-z_j$, the limiting weight function is
\begin{equation}
\tilde{\rho}_{\infty,\delta,\nu,\mu}^{(k+\alpha)}(\bm{x}) =  e^{\delta \cdot \tilde{z}_{-}(\bm{x})}.
\end{equation}
The remaining arguments are exactly the same as that in Case (b) of Region $\IV_{-}$, and  the proof of this case is complete.

Next, we prove Case (b) of Region $\V_{-}$. If the reference points satisfy $10\zeta_0^{-}\leq d(\bm{x}_j, q_-)\leq C_0$, we still choose the same rescaling factors
\begin{align}
\begin{split}
\lambda_j &= (\underline{L}_-(\bm{x}_j))^{-1} 
\\
\tau_j^{(k+\alpha)} &= (\underline{L}_-(\bm{x}_j))^{-\nu-k-\alpha} \\
\kappa_j &= (\underline{L}_-(\bm{x}_j))^{\nu-1}
\end{split}
\end{align}
and we have the convergence
\begin{equation}
(\mathcal{M}, \tilde{g}_j, \bm{x}_j) \xrightarrow{C^{\infty}}  (\mathcal{M}_{\infty}, \tilde{g}_{\infty}, \bm{x}_{\infty})\end{equation}
where  $(\mathcal{M}_{\infty}, \tilde{g}_{\infty}, \bm{x}_{\infty})$ is a finite rescaling of $(X_{b_{-}}^4, g_{b_{-}}, q_-)$.
 
So limiting weight function, up to some definite constant, has the form
\begin{equation}
\rho_{\infty,\delta, \nu, \mu}^{(k+\alpha)}(\bm{x}) = e^{\delta \cdot z_{-}(\bm{x})}\cdot \frac{(\underline{L}_{-}(\bm{x}))^{\nu+k+\alpha}}{(\underline{L}_{-}(\bm{x}_j))^{\nu+k+\alpha}}. 
\end{equation}
Moreover, the limiting $1$-form $\tilde{\omega}_{\infty}\in \Omega^1(X_{b_{-}}^4)$
such that
\begin{align}
\begin{split}
&\mathscr{D}_{g_{b_{-}}}\tilde{\omega}_{\infty}\equiv0\\
&|\tilde{\rho}_{\infty,\delta,\nu,\mu}^{(0)}(\bm{x}_{\infty})\cdot \tilde{\omega}_{\infty}(\bm{x}_{\infty}) | = 1 \\
&\| \tilde{\omega}_{\infty} \|_{C_{\delta,\nu,\mu}^0(X_{b_{-}}^4)} =  1,\\
\end{split}
\end{align}
which implies that for some constant $C_1>0$,
\begin{align}
|\tilde{\omega}_{\infty}(\bm{x})| \leq C_1 \cdot  e^{-\delta z_{-}(\bm{x})}\cdot (z_{-}(\bm{x}))^{-\frac{\nu}{2}}, \  \bm{x}\in X_{b_{-}}^4\setminus B_{2D_0^{-}}(q_{-}).
\end{align}
Since $\mathscr{D}_{g_{b_-}}\tilde{\omega}_{\infty}\equiv0$, by Lemma \ref{l:D-harmonic}, $\tilde{\omega}_{\infty}$ is harmonic with respect to the complete Tian-Yau metric $g_{b_-}$.
Applying Lemma \ref{l:max} to the harmonic $1$-form $\tilde{\omega}_{\infty}$, we conclude that 
$\tilde{\omega}_{\infty}\equiv 0$ on $X_{b_{-}}^4$.
So the proof of Case (b) is done.

Now we consider the case that the reference points $\bm{x}_j$ belong to Region $\V_{+}$. As the above, we still separate this region in two different pieces:
\begin{enumerate}
\item[(a)] Assume 
 $z_{+}(\bm{x}_j)\to \infty$.
\item[(b)] Assume that there is some constant $C_0>0$ independent of the index $j$ such that $10\zeta_0^{+}\leq z_{+}(\bm{x}_j) \leq C_0$.
\end{enumerate}

We skip the argument in Case (a) because it coincides with Case (a) in Region $V_{-}$.

So we start to prove Case (b) of Region $\V_{+}$.
If the reference points satisfy $10\zeta_0^{+}\leq z_{+}(\bm{x}_j) \leq C_0$, we choose the rescaling factors as follows,
\begin{align}
\begin{split}
\lambda_j &= (\underline{L}_+(\bm{x}_j))^{-1} 
\\
\tau_j^{(k+\alpha)} &= e^{-\delta(2T_- + 2T_+ + z_j)}\cdot (\underline{L}_+(\bm{x}_j))^{-\nu-k-\alpha} 
\\
\kappa_j &= e^{\delta(2T_- + 2T_+ +z_j)} \cdot (\underline{L}_+(\bm{x}_j))^{\nu-1},
\end{split}
\end{align}
where $z_j\equiv z_+(\bm{x}_j)$.
Then we have the convergence
\begin{equation}
(\mathcal{M}, \tilde{g}_j, \bm{x}_j) \xrightarrow{C^{\infty}}  (\mathcal{M}_{\infty}, \tilde{g}_{\infty}, \bm{x}_{\infty})\end{equation}
where  $(\mathcal{M}_{\infty}, \tilde{g}_{\infty}, \bm{x}_{\infty})$ is a finite rescaling of $(X_{b_{+}}^4, g_{b_{+}}, q_+)$.
Hence, under the $z$-coordinate translation $\tilde{z}_+(\bm{x})=z_+(\bm{x})-z_j$,
 the limiting weight function has the form
\begin{align}
\rho_{\infty,\delta, \nu, \mu}^{(k+\alpha)}(\bm{x}) =
e^{-\delta \cdot \tilde{z}_{+}(\bm{x})},\ \bm{x}\in X_{b_+}^4 \setminus B_{2D_0^{+}}(q_+).
\end{align}
On the other hand,  the limiting $1$-form $\tilde{\omega}_{\infty}\in \Omega^1(X_{b_{-}}^4)$ satisfies
\begin{align}
\begin{split}
&\mathscr{D}_{g_{b_{+}}}\tilde{\omega}_{\infty}\equiv0\\
&|\tilde{\rho}_{\infty,\delta,\nu,\mu}^{(0)}(\bm{x}_{\infty})\cdot \tilde{\omega}_{\infty}(\bm{x}_{\infty}) | = 1
\\
&\| \tilde{\omega}_{\infty} \|_{C_{\delta,\nu,\mu}^0(X_{b_{+}}^4)} =  1,\\
\end{split}
\end{align}
which implies which implies the $C^0$-estimate 
\begin{align}
|\tilde{\omega}_{\infty}(\bm{x})| \leq  C_2 \cdot e^{\delta \tilde{z}_+(\bm{x})}, \ \bm{x}\in X_{b_+}^4 \setminus B_{2D_0^{+}}(q_+).
\end{align}
Now we are in a position to apply the Liouville theorem for half-harmonic $1$-forms. If we choose $\delta\in(0,\delta_h)$, then Theorem \ref{t:liouville-1-form} shows that
\begin{equation}\tilde{\omega}_{\infty}\equiv 0 \ \text{on}\ X_{b_+}^4.\end{equation}

\vspace{0.5cm}

\noindent
{\bf Regions $\VI_{-}$ and $\VI_{+}$:}

\vspace{0.5cm}

If $\bm{x}_j$ are located in Region $\VI_{-}$, the proof is identical to  Case (b) of Region $\V_{-}$. If $\bm{x}_j$ are located in Region $\VI_{+}$, the proof is the same as Case (b) of Region $V_+$.

\vspace{0.5cm}

Combining all of the above regions, the proof of Proposition \ref{p:injectivity-of-D} is complete.

\end{proof}

\subsection{The existence of a hyperk\"ahler triple}
\label{ss:existence}

Now we are in a position to prove the existence of the hyperk\"ahler triple. 
For any sufficiently large gluing parameter $\beta\gg1$, denote by $\bm{\omega}_{\beta}^{\mathcal{M}}=(\omega_1,\omega_2,\omega_3)$ the approximate definite triple on $\mathcal{M}$ which was constructed in Section \ref{s:approx-triple}. 
To prove the existence of a hyperk\"ahler triple,  we  will solve the gauge-fixed elliptic system, 
\begin{align}
d^+ \bm{\eta} + \bm{\xi} = \mathfrak{F}_0\Big(\TF(-Q_{\beta}-S_{d^{-}\bm{\eta}})\Big), \
d^* \bm{\eta} = 0,
\label{e:hkt-system}
\end{align} 
where the renormalized coefficient matrix $Q_{\beta}=(Q_{ij})$ is defined by
\begin{equation}
\frac{1}{2}\omega_i\wedge\omega_j = Q_{ij} \dvol_{\bm{\omega}_{\beta}^{\mathcal{M}}}, 
\end{equation} 
see Section \ref{ss:outline} for more details about the setup.
A basic tool of solving the elliptic system \eqref{e:hkt-system} is the following version of the implicit function theorem, see for example \cite[theorem 4.4.2]{RollinSinger}.
\begin{lemma} \label{l:implicit-function}
Let $\mathscr{F}:\mathfrak{A} \to \mathfrak{B}$ be a $C^1$-map between two Banach spaces such that 
$\mathscr{F}(x)-\mathscr{F}(0)=\mathscr{L}(x)+\mathscr{N}(x)$, where the operator $\mathscr{L}:\mathfrak{A}\to\mathfrak{B}$ is linear and $\mathscr{N}(0)=0$. Assume that
\begin{enumerate}
\item $\mathscr{L}$ is an isomorphism with $\| \mathscr{L}^{-1} \| \leq C_1$,

\item there are constants $r>0$ and $C_2>0$ with $r<\frac{1}{3C_1 C_2}$ such that 
\begin{enumerate}\item  $\| \mathscr{N}(x) - \mathscr{N}(y) \|_{\mathfrak{B}} \leq C_2\cdot ( \|x\|_{\mathfrak{A}} + \|y\|_{\mathfrak{A}} ) \cdot  \| x - y \|_{\mathfrak{A}} $ for all $x,y\in B_r(0)\subset {\mathfrak{A}}$,

\item $\| \mathscr{F}(0) \|_{\mathfrak{B}} \leq \frac{r}{2C_1}$,
\end{enumerate}
\end{enumerate}
then there exists a unique solution to $\mathscr{F}(x)=0$ in $\mathfrak{A}$   such that
\begin{equation}
\|x\|_{\mathfrak{A}} \leq 2C_1 \cdot \|\mathscr{F}(0)\|_{\mathfrak{B}}.
\end{equation}

\end{lemma}

To apply the above implicit function theorem, we need to verify the above properties in our context. 
To start with, we define the following Banach spaces,
\begin{equation}
\mathfrak{A} \equiv \Big(C_{\delta,\nu,\mu}^{1,\alpha}(\mathring{\Omega}^1(\mathcal{M}))\oplus \mathcal{H}^+(\mathcal{M})\Big)\otimes\dR^3
\end{equation}
and
\begin{equation}
\mathfrak{B} \equiv\Big(C_{\delta,\nu+1,\mu}^{0,\alpha}(\Lambda^+ (\mathcal{M}))\Big)\otimes\dR^3,
\end{equation}
where $\mathcal{H}^+(\mathcal{M})$ is the space of self-dual $2$-forms on $\mathcal{M}$, $\Lambda^+ (\mathcal{M})$ is the space of self-dual $2$-forms on $\mathcal{M}$ and
$\mathring{\Omega}^1(\mathcal{M})\equiv\{\eta\in\Omega^1(\mathcal{M})|d^*\eta = 0\}
$. Notice that Proposition \ref{p:topological-invariants} implies that 
\begin{equation}\dim(\mathcal{H}^+(\mathcal{M}))=b_2^+(\mathcal{M})=3.\end{equation}
Now we give a basis of $\mathcal{H}^+(\mathcal{M})$. Let $\bm{\omega}_{\beta}^{\mathcal{M}}\equiv(\omega_1,\omega_2,\omega_3)$ be the gluing definite triple on $\mathcal{M}$ constructed in Section \ref{s:approx-triple} which induces a Riemannian metric $g$ such that the triple $\bm{\omega}^{\mathcal{M}}$ is self-dual with respect to $g$. Immediately, $d^*\omega_k = d\omega_k = 0$ and hence for every $1\leq k\leq 3$, $\omega_k$ is a self-dual harmonic $2$-form. Then by Corollary \ref{c:sutt}, 
$\{\omega_1,\omega_2,\omega_3\}$
is actually a basis of $\mathcal{H}^+(\mathcal{M})$.

Let $\mathfrak{A}$ and $\mathfrak{B}$ equipped with the following weighted H\"older norms: Let $(\bm{\eta}, \bm{\bar{\xi}}^+)\in \mathfrak{A}$ and $\bm{\xi}^+\in\mathfrak{B}$,
then
\begin{equation}
\|(\bm{\eta}, \bm{\bar{\xi}}^+)\|_{\mathfrak{A}} \equiv 
\| \bm{\eta}\|_{C_{\delta,\nu,\mu}^{1,\alpha}(\mathcal{M})}
 +
  \| \bm{\bar{\xi}}^+ \|_{L^2}
\end{equation}
and
 \begin{equation}
\|\bm{\xi}^+\|_{\mathfrak{B}} \equiv   \| \bm{\xi}^+ \|_{C_{\delta,\nu+1,\mu}^{0,\alpha}(\mathcal{M})},
\end{equation}
where the above $L^2$ norm is defined with respect to a fixed basis $\{\omega_1,\omega_2,\omega_3\}\subset\mathcal{H}^+(\mathcal{M})$.
The operator $\mathscr{F}:\mathfrak{A}\to\mathfrak{B}$ is defined by 
\begin{equation}
\mathscr{F}(\bm{\eta},\bar{\bm{\xi}}^+)\equiv d^+ \bm{\eta} + \bar{\bm{\xi}}^+ - \mathfrak{F}_0\Big(\TF(-Q_{\beta}-S_{d^{-}\bm{\eta}})\Big),
\end{equation}
which is given by the system \eqref{e:hkt-system}. The corresponding linearization is 
\begin{equation}
\mathscr{L} \equiv (d^+ \oplus \Id)\otimes\dR^3: \mathfrak{A}\longrightarrow \mathfrak{B}.
\end{equation}
So the nonlinear part is given by
\begin{equation}
\mathscr{N}(\bm{\eta},\bar{\bm{\xi}}^+) \equiv\mathfrak{F}_0\Big(\TF(-Q_{\beta})\Big)-\mathfrak{F}_0\Big(\TF(-Q_{\beta}- 
S_{d^{-}\bm{\eta}})\Big).
\end{equation}

First, we will check Property (1) in Lemma \ref{l:implicit-function} and we will prove that the linearized operator $\mathscr{L}_g$ is an isomorphism from
$\mathfrak{A}$
to $\mathfrak{B}$.

\begin{proposition}
\label{p:injectivity-for-L} For $(\mathcal{M},g_{\beta})$ with sufficiently large gluing parameter $\beta\gg1$, then there exists some constant $C>0$, independent of $\beta$, such that for every triple 
 \begin{equation}
\bm{\xi}^+ \equiv
 (\xi_1^+, \xi_2^+, \xi_3^+) \in\mathfrak{B},
 \end{equation}
 there exists a unique pair 
\begin{equation}(\bm{\eta},\bar{\bm{\xi}}^+)\equiv \Big((\eta_1,\eta_2,\eta_3), (\bar{\xi}_1^+,\bar{\xi}_2^+,\bar{\xi}_3^+)\Big)\in \mathfrak{A}\end{equation}
which satisfies
 \begin{equation}\mathscr{L}_g(\bm{\eta}, \bar{\bm{\xi}}^+) = \bm{\xi}^+
\end{equation} and
\begin{equation}
\|\bm{\eta}\|_{C_{\delta,\nu,\mu}^{1,\alpha}(\mathcal{M})}+\|\bar{\bm{\xi}}^+\|_{L^2}\leq 
 C e^{10 \delta \cdot\beta} \cdot  \|\bm{\xi}^+\|_{C_{\delta,\nu+1,\mu}^{0,\alpha}(\mathcal{M})},\label{e:L-uniform-estimate}
\end{equation}
where $\delta$, $\nu$ and $\mu$ are the constants in  Proposition \ref{p:injectivity-of-D}.\end{proposition}

\begin{proof}

First, we prove the surjectivity of the linear operator $\mathscr{L}_g$.
By standard Hodge theory, it holds that
\begin{align}
\Omega^2_+(\mathcal{M}) &=\mathcal{H}^+(\mathcal{M}) \oplus d^+(\Omega^1(\mathcal{M}))\\
 \Omega^1(\mathcal{M}) &= d ( \Omega^0(\mathcal{M})) \oplus \mathring{\Omega}^1(\mathcal{M}),
\end{align}
where $\mathring{\Omega}^1(\mathcal{M})$ denotes the space of divergence-free $1$-forms on $\mathcal{M}$, therefore 
\begin{equation}
\Omega^2_+(\mathcal{M}) = \mathcal{H}^+(\mathcal{M}) \oplus d^+(\mathring{\Omega}^1(\mathcal{M})).
\end{equation}
This clearly implies that
\begin{equation}
\mathscr{L}_{g} = (d^+\oplus \Id) \otimes\dR^3 : \mathfrak{A} \longrightarrow \mathfrak{B}.
\end{equation}
is surjective. 

The remainder of the proof is a contradiction argument. We will argue on the level of forms, and this will imply the result for triples. 
If \eqref{e:L-uniform-estimate} does not hold for a uniform constant, then there exists a 
sequence of gluing parameters $\beta_j \rightarrow \infty$ and 
$ \eta_j$, $\bar{\xi}^+_j$  with 
\begin{align}
\label{0limit}
e^{10 \delta\cdot \beta_j} \| d^+ \eta_j +\bar{\xi}^+_j \|_{C_{\delta,\nu+1,\mu}^{0,\alpha}(\mathcal{M})} &\rightarrow 0,\\
\|\eta_j \|_{C_{\delta,\nu,\mu}^{1,\alpha}(\mathcal{M})}+\|\bar{{\xi}}^+_j\|_{L^2(\mathcal{M})} &= 1,
\end{align}
as $j \to \infty$.  Pairing $d^+ \eta_j + \bar{\xi}^+_j$ with $\bar{\xi}^+_j$ and 
integrating, and using \eqref{0limit}, we obtain that 
\begin{align}
\begin{split}
\Vert \bar{\xi}^+_j \Vert_{L^2(\mathcal{M})}^2 &\leq \epsilon_j 
e^{- 10 \delta\cdot \beta_j} \int_{\mathcal{M}} | \bar{\xi}_j^+ | (\rho_{\delta,\nu+1,\mu}^{(0+\alpha)})^{-1} \dvol_{g_{\beta_j}}\\
&\leq \epsilon_j e^{- 10\delta\cdot \beta_j} \Vert \bar{\xi}^+_j \Vert_{L^2(\mathcal{M})}
\Big\{
\int_{\mathcal{M}}  (\rho_{\delta,\nu+1,\mu}^{(0+\alpha)})^{-2} \dvol_{g_{\beta_j}} 
\Big\}^{\frac{1}{2}},
\end{split}
\end{align}
where $\epsilon_j \to 0$ as $ j \to \infty$. It is easy to check that 
\begin{align}
\int_{\mathcal{M}}  (\rho_{\delta,\nu+1,\mu}^{(0+\alpha)})^{-2} \dvol_{g_{\beta_j}}  < C, 
\end{align}
where $C$ is independent of $\beta$, so this implies that 
\begin{equation}e^{ 10 \delta\cdot \beta_j} \Vert \bar{\xi}^+_j \Vert_{L^2(\mathcal{M})} \to 0\label{e:exp-L2}\end{equation} as $j \to \infty$.

Next, since the triple $\bm{\omega}^{\mathcal{M}}$ is harmonic and spans $\mathcal{H}_+(\mathcal{M})$ at every point, we can write 
\begin{equation}
\bar{\xi}^+ = \lambda_1 \omega_1 + \lambda_2 \omega_2 + \lambda_3 \omega_3.
\label{e:proj}\end{equation} 
Recall by the definition of the triple $\bm{\omega}_{\beta}^{\mathcal{M}}$, for every $1\leq p,q\leq 3$,
\begin{equation}
\frac{1}{2}\int_{\mathcal{M}}\omega_p\wedge\omega_q = \int_{\mathcal{M}} Q_{pq} \dvol_{\bm{\omega}_{\beta}^{\mathcal{M}}},\end{equation}
and so for any self-dual harmonic form $\bar{\xi}^+\in\mathcal{H}_+(\mathcal{M}) $, 
\begin{align}
\Vert \bar{\xi}^+ \Vert_{L^2(\mathcal{M})}^2 &= 2 \sum_{p,q=1}^3 \lambda_p \lambda_q \int_{\mathcal{M}} Q_{pq} \dvol_{\bm{\omega}_{\beta}^{\mathcal{M}}}, 
\end{align}
so applying the volume estimate 
\begin{equation}C^{-1}\beta^2\leq \Vol_g(\mathcal{M}) \leq C \beta^2,\end{equation}
and Proposition \ref{p:gluing-definite-triple},
we have the estimate 
\begin{align}
C^{-1} \beta_j^2 (\lambda_{1,j}^2 +\lambda_{2,j}^2 + \lambda_{3,j}^2)
\leq \Vert \bar{\xi}_j^+ \Vert_{L^2(\mathcal{M})}^2.
\end{align}
The above and \eqref{e:exp-L2}  imply that $\beta_j \lambda_{k,j} e^{10\delta\cdot \beta_j}\to 0$ as $j\to \infty$ for $k = 1,2, 3$.  
We then have
\begin{align}
\begin{split}
\Vert \bar{\xi}^+_j \Vert_{C_{\delta,\nu+1,\mu}^{0,\alpha}(\mathcal{M})}
&= \Vert \lambda_{1,j} \omega_1 +\lambda_{2,j} \omega_2 + \lambda_{3,j} \omega_3
\Vert_{C_{\delta,\nu+1,\mu}^{0,\alpha}(\mathcal{M})}\\
& \leq \lambda_{1,j}\Vert \omega_1\Vert_{C_{\delta,\nu+1,\mu}^{0,\alpha}(\mathcal{M})}
+ \lambda_{2,j}\Vert \omega_2\Vert_{C_{\delta,\nu+1,\mu}^{0,\alpha}(\mathcal{M})}
+ \lambda_{3,j}\Vert \omega_3\Vert_{C_{\delta,\nu+1,\mu}^{0,\alpha}(\mathcal{M})}.
\end{split}
\end{align}
Since 
\begin{align}
\Vert \omega_k\Vert_{C_{\delta,\nu+1,\mu}^{0,\alpha}(\mathcal{M})} \leq C e^{5  \delta \cdot \beta_j}, 
\end{align}
for $1 \leq k \leq 3$, the above implies that
\begin{align}
\Vert \bar{\xi}^+_j \Vert_{C_{\delta,\nu+1,\mu}^{0,\alpha}(\mathcal{M})} \leq C \epsilon_j \beta_j^{-1}e^{-5  \delta \cdot \beta_j}, 
\end{align}
for some sequence $\epsilon_j \to 0$ as $ j \to \infty$, so we have proved that
\begin{align}
\Vert \bar{\xi}^+_j \Vert_{C_{\delta,\nu+1,\mu}^{0,\alpha}(\mathcal{M})} \rightarrow 0,
\end{align}
as $j \to \infty$. Consequently, our sequence satisfies 
\begin{align}
\label{0limit2}
 \| d^+ \eta_j \|_{C_{\delta,\nu+1,\mu}^{0,\alpha}(\mathcal{M})} &\rightarrow 0, \\
\|\eta_j \|_{C_{\delta,\nu,\mu}^{1,\alpha}(\mathcal{M})} &\to 1,
\end{align}
as $j \to \infty$, which contradicts Proposition \ref{p:injectivity-of-D}.  
 \end{proof}

In the following proposition, we will prove the nonlinear error estimate which corresponds to Property (2) in Lemma \ref{l:implicit-function}.  
\begin{lemma}
[Nonlinear Errors]\label{l:nonlinear-errors} Consider $(\mathcal{M},g_{\beta})$ with sufficiently large gluing parameter $\beta\gg1$. Let $\delta$, $\nu$ and $\mu$ be the constants in Proposition 
\ref{p:injectivity-of-D}, then
there are constants $r_0 > 0$ and $C>0$ which are independent $\beta$, such that for every 
$\bm{v}_1\equiv(\bm{\eta}_1,\bar{\bm{\xi}}_1^+) \in B_r(0)\subset\mathfrak{A}$ and $\bm{v}_2\equiv(\bm{\eta}_2,\bar{\bm{\xi}}_2^+) \in B_r(0)\subset\mathfrak{A}$,
where $r < r_0$, we have
\begin{equation}
\| \mathscr{N}(\bm{v}_1) - \mathscr{N}(\bm{v}_2) \|_{\mathfrak{B}} \leq C (\|\bm{v}_1\|_{\mathfrak{A}} + \|\bm{v}_2\|_{\mathfrak{A}} ) \cdot \|\bm{v}_1 - \bm{v}_2\|_{\mathfrak{A}}.\label{e:nonlinear-error}
\end{equation}

\end{lemma}

\begin{proof} 
By definition, for any $\bm{v}\equiv(\bm{\omega},\bar{\bm{\xi}}^+)$,
\begin{equation}
\mathscr{N}(\bm{v}) \equiv\mathfrak{F}_0\Big(\TF(-Q_{\beta})\Big)-\mathfrak{F}_0\Big(\TF(-Q_{\beta}- 
S_{d^{-}\bm{\eta}})\Big).
\end{equation}
and hence
\begin{equation}
\mathscr{N}(\bm{v}_1) - \mathscr{N}(\bm{v}_2) = \mathfrak{F}_0\Big(\TF(-Q_{\beta}- S_{d^{-}\bm{\eta}_2})\Big)-\mathfrak{F}_0\Big(\TF(-Q_{\beta}- S_{d^{-}\bm{\eta}_1})\Big).
\end{equation}
Since $\mathfrak{F}_0:\mathscr{S}_0(\dR^3)\to \mathscr{S}_0(\dR^3)$ is a smooth map on the space of trace-free symmetric $(3\times 3)$-matrices, there is some universal constant $C>0$ such that
\begin{align}
\begin{split}
 |\mathscr{N}(\bm{v}_1) - \mathscr{N}(\bm{v}_2) |
&\leq   C | 
 d^{-}\bm{\eta}_1 *  d^{-}\bm{\eta}_1   - d^{-}\bm{\eta}_2 *  d^{-}\bm{\eta}_2| \\
&\leq C ( |d^{-}\bm{\eta}_1 | + |d^{-}\bm{\eta}_2|) \cdot |d^{-}(\bm{\eta}_1 - \bm{\eta}_2)|.
\end{split}
\end{align}
Multiplying by the weight function, 
\begin{align}
\begin{split}
\rho_{\delta, \nu +1 , \mu}^{(0)}(x)\cdot | \mathscr{N}(v_1) &- \mathscr{N}(v_2)| 
\leq C\cdot \rho_{\delta, \nu +1 , \mu}^{(0)}(x) \cdot ( |d^{-}\bm{\eta}_1| + |d^{-}\bm{\eta}_2|) \cdot |d^{-}(\bm{\eta}_1 - \bm{\eta}_2)| \\
& \leq  C \Big(\rho_{\delta, \nu  , \mu}^{(1)}(x)\cdot ( |d^-\bm{\eta}_1| + |d^-\bm{\eta}_2|)\Big)\cdot \Big(
 \rho_{\delta, \nu  , \mu}^{(1)}(x)\cdot | d^- (\bm{\eta}_1 - \bm{\eta}_2)|\Big)
.
\end{split}
\end{align}
Taking sup norms,
\begin{align}
\Vert \mathscr{N}(\bm{v}_1) - \mathscr{N}(\bm{v}_2) \Vert_{ C^0_{\delta, \nu +1 , \mu}(\mathcal{M})}\leq C \Big(  \Vert \bm{v}_1 \Vert_{ C^1_{\delta, \nu , \mu}(\mathcal{M})}  +  
\Vert \bm{v}_2 \Vert_{ C^1_{\delta, \nu , \mu}(\mathcal{M})} \Big)\cdot \Big(
\Vert \bm{v}_1 - \bm{v}_2 \Vert_ { C^1_{\delta, \nu , \mu}(\mathcal{M})}\Big).
\end{align}
By similar computations, 
we also have the estimate for the H\"older seminorm 
\begin{align}
\Big[ \mathscr{N}(\bm{v}_1) - \mathscr{N}(\bm{v}_2) \Big]_{ C^{0,\alpha}_{\delta, \nu +1 , \mu}(\mathcal{M})} \leq C  \Big(  \Vert \bm{v}_1 \Vert_{ C^{1,\alpha}_{\delta, \nu , \mu}(\mathcal{M})}  +  
\Vert \bm{v}_2 \Vert_{ C^{1,\alpha}_{\delta, \nu , \mu}(\mathcal{M})} \Big)\cdot \Big(
\Vert \bm{v}_1 - \bm{v}_2 \Vert_ { C^{1,\alpha}_{\delta, \nu , \mu}(\mathcal{M})}\Big).
\end{align}
So we obtain the effective estimate \eqref{e:nonlinear-error} for the nonlinear errors. 
\end{proof}

\begin{proposition}
\label{p:small-error}  Consider $(\mathcal{M},g_{\beta})$ with sufficiently large gluing parameter $\beta\gg1$. Let $\delta$, $\nu$ and $\mu$ be the constants in Proposition 
\ref{p:injectivity-of-D}, then there exists some constant $C>0$ which is independent of 
$\beta$ such that
\begin{equation}
\|\mathscr{F}(0)\|_{\mathfrak{B}} \leq C e^{-\frac{\delta_q \beta}{2}},
\end{equation}
where $\delta_q>0$ is the constant in Corollary \ref{c:sutt}.
\end{proposition}

\begin{proof}In our context, it holds that
\begin{equation}
\mathscr{F}(0)= - \mathfrak{F}_0\Big(\TF(-Q_{\beta})\Big).
\end{equation}
Since $\mathfrak{F}_0:\mathscr{S}_0(\dR^3)\to \mathscr{S}_0(\dR^3)$ is a smooth map on the space of trace-free symmetric $(3\times 3)$-matrices,  and $\mathfrak{F}_0(0)=0$, so we have 
\begin{equation}
\|\mathscr{F}(0)\|_{\mathfrak{B}}\leq C \|\TF(Q_{\beta})\|_{\mathfrak{B}}.
\end{equation}
The proof immediately follows from the estimate in Corollary \ref{c:sutt}. 
\end{proof}

Now we are ready to prove the existence of a hyperk\"ahler triple on $\mathcal{M}$ which implies that $\mathcal{M}$ is diffeomorphic to the $\K3$ surface.

\begin{theorem}\label{t:existence-hyperkaehler}
Consider $(\mathcal{M}, g_{\beta})$ with sufficiently large gluing parameter $\beta\gg1$. Denote by $\bm{\omega}_{\beta}^{\mathcal{M}}$ the gluing definite triple which is constructed by Proposition \ref{p:gluing-definite-triple}. Let $\delta$, $\nu$ and $\mu$ be the constants in Proposition \ref{p:injectivity-of-D}, then there exists a hyperk\"ahler triple 
$\bm{\omega}_{\beta}^{\HK}$ with the effective estimate
\begin{equation}
\|\bm{\omega}_{\beta}^{\mathcal{M}}-\bm{\omega}_{\beta}^{\HK}\|_{C_{\delta,\nu+1,\mu}^{0,\alpha}(\mathcal{M})}\leq Ce^{-\delta_0\beta}\label{e:hoelder-error}
\end{equation}
for some constants $C>0$ and $\delta_0>0$ independent of $\beta$.
In particular, $\mathcal{M}$ is diffeomorphic to the $\K3$ surface. 
\end{theorem}

\begin{proof}

It suffices to verify the conditions in Lemma \ref{l:implicit-function}.
In fact, Proposition \ref{p:injectivity-for-L},  Lemma \ref{l:nonlinear-errors} and Proposition \ref{p:small-error} verify Property (1), Property (2a)
and Property (2b) in Lemma \ref{l:implicit-function} respectively.
So applying the implicit function theorem given by Lemma \ref{l:implicit-function}, the existence of the hyperk\"ahler triple $\bm{\omega}_{\beta}^{\HK}$ just follows. 
The H\"older type error estimate \eqref{e:hoelder-error} follows directly from the implicit function theorem and the definition of the weight functions.

Since the hyperk\"ahler triple $\bm{\omega}_{\beta}^{\HK}$ determines a hyperk\"ahler metric on $\mathcal{M}$. By Proposition \ref{p:topological-invariants}, $\chi(\mathcal{M})=24$ and hence $\mathcal{M}$ is diffeomorphic to the $\K3$ surface.

\end{proof}

\subsection{Completion of main proofs}
\label{ss:completion}

In this subsection, we prove Theorems \ref{t:codim-3} and \ref{t:domain-wall-crossing}. 
\begin{proof}
[Proof of Theorem \ref{t:codim-3}]  
Recall that by Theorem \ref{t:existence-hyperkaehler}, $\mathcal{M}$ is diffeomorphic to the $\K3$ surface.

First, we consider the simpler case that there is only one cluster of monopoles, i.e., $m=1$. Without loss of generality, one can assume that all the monopoles in the neck region are located on the same torus fiber of $\TT \times \mathbb{R}$. 

We start the proof by describing the hyperk\"ahler metrics $\hat{h}_{\beta}$ and the continuous map $F_{\beta}:\K3\to[0,1]$.  
Given any sufficiently large parameter $\beta\gg1$, denote by 
$g_{\beta}$ the approximate metric which is almost Ricci-flat and determined by the approximate triple constructed in Section \ref{s:approx-triple} such that
\begin{equation}
C^{-1}\beta^{\frac{3}{2}}\leq 
\diam_{g_{\beta}}(\mathcal{M})\leq C\beta^{\frac{3}{2}}
\end{equation}
for some constant $C>0$ independent of $\beta$.
 By Theorem \ref{t:existence-hyperkaehler}, there is a hyperk\"ahler metric $\hat{g}_{\beta}$ such that 
\begin{equation}
\|\hat{g}_{\beta}-g_{\beta}\|_{C^{0,\alpha}(\mathcal{M})} \leq Ce^{-\delta_0\beta}
\end{equation}
for some $C>0$ and $\delta_0>0$ independent of $\beta$.
Let $\hat{h}_{\beta}$ be the rescaling of the hyperk\"ahler metric $\hat{g}_{\beta}$ with  
$\diam_{\hat{h}_{\beta}}(\mathcal{M})=1$. Denote by $h_{\beta}$ the rescaling of $g_{\beta}$ with $\diam_{h_{\beta}}(\mathcal{M})=1$, then
\begin{equation}
\|\hat{h}_{\beta}-h_{\beta}\|_{C^{0,\alpha}(\mathcal{M})} \leq Ce^{-\frac{\delta_0\beta}{2}}.
\end{equation}
Now we are ready to define the map $F_{\beta}:\mathcal{M}\to[0,1]$. First, recalling the notation in Section \ref{s:approx-triple}, we extend the function $z$ on the neck region to $\mathcal{M}$ as follows
\begin{align}
  \tilde{z}(\bm{x}) =
  \begin{cases}
    \zeta_0^- - 2 T_-    & \bm{x} \in X^4_{b_-} \setminus \{ z_- \geq \zeta_0^- \}\\
    z_-(\bm{x}) - 2T_-   &   \bm{x} \in  X^4_{b_-} \cap \{  \zeta_0^- \leq z_- \leq T_- \} \\
    z(\bm{x})   &    \bm{x} \in \mathcal{N}(T_-, T_+) \\
   2T_+  - z_+(\bm{x})  &   \bm{x} \in  X^4_{b_+} \cap \{  \zeta_0^+ \leq z_+ \leq T_+ \} \\   
   2 T_+ - \zeta_0^+    & \bm{x} \in X^4_{b_+} \setminus \{ z_+ \geq \zeta_0^+ \}\\
 \end{cases},
  \end{align}
and then define
  \begin{align}
  F_{\beta}(\bm{x}) = \frac{    \tilde{z}(\bm{x}) - \zeta_0^- + 2 T_-}{ 2 (T_+ + T_-) - \zeta_0^- -  \zeta_0^+}.
    \end{align}

Then it follows directly from the gluing construction that there is some point $t_1\in(0,1)$  such that $F_{\beta}^{-1}(t_1)$ 
is a singular $S^1$-bundle over $\TT$ with exactly $(b_-+b_+)$
vanishing circles. In fact, the vanishing circles occur at the monopoles of the neck region $\mathcal{N}_{m_0}^4$ constructed in Section \ref{ss:neck-region} which is a  Gibbons-Hawking space over $\TT\times \dR$.
 Moreover, for each $t\in(0,t_1)\cup(t_1,1)$, the fiber $F_{\beta}^{-1}(t)$ is diffeomorphic to a Heisenberg nilmanifold with
\begin{align}
\deg(F_{\beta}^{-1}(t))
=
\begin{cases}
b_-, & t\in (0,t_1),
 \\
b_+, & t\in (t_1,1).
\end{cases}
\end{align}
By the explicit construction in Section \ref{s:approx-triple}, there is some uniform constant $C_0>0$ such that for each regular fiber,
 \begin{align}
C_0^{-1} \beta^{-1}\leq \diam_{\hat{h}_{\beta}}(F_{\beta}^{-1}(t))\leq C_0 \beta^{-1}, \
C_0^{-1} \beta^{-2}\leq \diam_{\hat{h}_{\beta}}(S^1)\leq C_0 \beta^{-2}.
\end{align}
With these diameter estimates, we are ready to prove the uniform curvature estimates by 
applying theorem \ref{t:collapsed-eps-reg}.
Fix any $\epsilon\in(0,10^{-2})$, let $\beta>0$ sufficiently large such that 
\begin{equation}\diam_{\hat{h}_{\beta}}(F_{\beta}^{-1}(t))< \frac{\delta_0\cdot \epsilon}{10},\end{equation} where 
$\delta_0>0$ is the dimensional constant in theorem \ref{t:collapsed-eps-reg}.
Now for a ball around each regular point $B_{\epsilon}(x)\subset  F_{\beta}^{-1}([0,1]\setminus T_{2\epsilon}(\mathcal{S}))$ with $\mathcal{S}\equiv\{0,t_1,1\}$,
then  
\begin{equation}\Gamma_{\delta_0\epsilon}(x)\equiv \Image [\pi_1(B_{\delta_0\epsilon}(x))\to B_{\epsilon}(x)] \cong \pi_1(\Nil^3)\end{equation}
 and hence $\rank(\Gamma_{\delta_0\epsilon}(x))=3$.
Then by theorem \ref{t:collapsed-eps-reg},
\begin{equation}
\sup\limits_{B_{\epsilon/2}(x)}|\Rm_{\hat{h}_{\beta}}|\leq C_{0,\epsilon},
\end{equation}
 where $C_{0,\epsilon}>0$ depends only on $\epsilon$ and is independent of $\beta$. The higher order curvature estimates can be proved by considering a local universal cover and 
applying the standard regularity theory for non-collapsing Einstein metrics.  
This completes (1) of Theorem \ref{t:codim-3}.

 Now we proceed to prove (2).
 We still apply theorem \ref{t:collapsed-eps-reg} to prove curvatures blowing-up behavior around the singular fiber. In fact,  
if $x\in T_{\epsilon/2}(F_{\beta}^{-1}(t_1))$, it suffices to  show 
$\sup\limits_{B_{\epsilon/2}(x)}|\Rm_{\hat{h}_{\beta}}| \to \infty$ as $\beta\to\infty$. In fact, 
notice that
\begin{equation}
\Gamma_{\epsilon/2}(x)\equiv\Image[\pi_1(B_{\epsilon/2}(x))\to B_{1/10}(x)]\cong \dZ \oplus \dZ
\end{equation}
and hence
$\rank(\Gamma_{\epsilon/2}(x))
=2<3$. Therefore, theorem \ref{t:collapsed-eps-reg} implies that 
\begin{equation}\sup\limits_{B_{\epsilon/2}(x)}|\Rm_{\hat{h}_{\beta}}| \to \infty\end{equation} as 
$\epsilon\to0$. 

The next part  is to prove the classification of the bubble limits in (2) of statement of the theorem.
Fix the gluing parameter $\beta \gg 1$, we analyze the curvature behavior of the approximate metric $g_{\beta}$ in the gluing construction at the scale such that 
\begin{equation}C^{-1}\beta^{\frac{3}{2}}\leq \diam_{g_{\beta}}(\mathcal{M},g)\leq C \beta^{\frac{3}{2}}.\label{e:large-d-scale}\end{equation}
There are two cases to analyze. 

First, let the reference point $\bm{x}_{\beta}$ be a curvature maximum point of a Tian-Yau piece. 
It follows directly from the construction that, as $\beta\to+\infty$, the curvature $|\Rm_{g_{\beta}}|(\bm{x}_{\beta})$ is uniformly bounded but not going to $0$. 
So $(\mathcal{M}, g_{\beta}, \bm{x}_{\beta})$
converges to a complete hyperk\"ahler Tian-Yau space $(X^4,g_{TY},\bm{x}_{\infty})$ in the pointed $C^k$-topology for any $k\in\dZ_+$. We will show that $(\mathcal{M},\hat{g}_{\beta},\bm{x}_{\beta})$ also converges to the same Tian-Yau space $(X^4,g_{TY},\bm{x}_{\infty})$ in the pointed $C^k$-topology for any $k\in\dZ_+$.
In fact, by Theorem \ref{t:existence-hyperkaehler},
\begin{equation}
\|\hat{g}_{\beta}-g_{\beta}\|_{C^{0,\alpha}(\mathcal{M})} \leq Ce^{-\delta \beta},\label{e:error-est}
\end{equation}
which implies that 
$(\mathcal{M},\hat{g}_{\beta},x_{\beta})$ converges to the same Tian-Yau space $(X^4,g_{TY},\bm{x}_{\infty})$ in the pointed $C^{0,\alpha}$-topology. The stronger convergence follows from 
a regularity result for non-collapsed Einstein metrics
in \cite{AnCh}.
Since the rescaling factor $\beta^{\frac{3}{2}}$ is much smaller than exponential, so the bubble limit of $(\mathcal{M},\hat{h}_{\beta})$ around $\bm{x}_{\beta}$ is a complete hyperk\"ahler Tian-Yau space. 

Next, we consider the case in which the reference point $\bm{x}_{\beta}$ is very close to one of monopoles, i.e. $\bm{x}_{\beta} \in B_{\beta^{-\frac{1}{2}}}(p_m)$ in terms of the metric $\hat{h}_{\beta}$, where
\begin{equation}  p_m\in\mathcal{P}_{b_- + b_+}\equiv \{p_1, \ldots, p_{b_- + b_+}\}.
\end{equation}
Applying Lemma \ref{l:rescaled-Taub-NUT}, then 
\begin{equation}(\mathcal{M}, \beta\cdot g_{\beta}, \bm{x}_{\beta})\longrightarrow (\dR^4, g_{TN}, \bm{x}_{\infty}),\end{equation}
where $g_{TN}$ is the Taub-NUT metric and the convergence is with respect to the pointed $C^k$-topology for any $k\in\dZ_+$. Applying the error estimate \eqref{e:error-est} and the same arguments as the above, 
$(\mathcal{M}, \beta\cdot h_{\beta}, \bm{x}_{\beta})$ converges to  $(\dR^4, g_{TN}, \bm{x}_{\infty}),$ in the pointed $C^k$-topology for any $k\in\dZ_+$.
This implies that in terms of the hyperk\"ahler metric $\hat{h}_{\beta}$, we have the pointed $C^k$-convergence for any $k\in\dZ_+$,
\begin{equation}
(\mathcal{M}, \beta^{4}\cdot \hat{h}_{\beta},  \bm{x}_{\beta})  \longrightarrow (\dR^4, g_{TN}, \bm{x}_{\infty}).
\end{equation}
So the proof of (2) is done.

The above completes the proof in the case with 1 singular point of convergence in the interior of the interval. 
Next we are in a position to give  a  generalization of the gluing construction in Section~\ref{s:approx-triple} to produce multiple singular points of convergence in the interior of the interval.

First, we fix two hyperk\"ahler Tian-Yau spaces $(X_{b_-}^4, g_{b_-}, p_-)$ and $(X_{b_+}^4, g_{b_+}, p_+)$ with $b_-, b_+ \in \{1,\ldots, 9\}$. 
Let $\{w_j\}_{j=1}^m$ be positive integers satisfying 
\begin{equation}w_1 + \ldots + w_m =b_- + b_+.\end{equation} For each $1\leq j\leq m$, we choose the
neck region
$\mathcal{N}_{w_j}^4$ as a Gibbons-Hawking space over a finite flat cylinder $(\TT\times[-T_j,T_{j+1}], g_0)$ with $w_j$-monopoles. As in the construction of Section \ref{s:gluing-space}, each pair of monopoles in $\mathcal{N}_{w_j}^4$ has a definite and bounded distance. 
Now let $G_{j}:\TT\times\dR\to\dR$ be a global sign-changing Green's function which satisfies
\begin{equation}
-\Delta_{g_0} G_{j} = 2\pi\sum\limits_{s=1}^{w_j}\delta_{p_s}
\end{equation}
and there are constants $\beta_{j}^{-},\beta_{j}^{+}\in\dR$ and $k_{j}^{-} >0$, $k_{j}^{+} <0$ such that
\begin{align}
\begin{split}
&|\nabla_{g_0}^k( G_{j} - (k_{j}^{-}  z + \beta_{j}^{-}))|\leq C_k e^{\lambda_1 z}, \  z<-100\beta, \\
&|\nabla_{g_0}^k( G_{j} - (k_{j}^{+}  z + \beta_{j}^{+}))|\leq C_k e^{-\lambda_1 z}, \ z>100\beta, \\
&k_j^- =- k_j^+ = \frac{\pi w_j}{\Area(\TT)}.
\end{split}
\end{align}
Note that the first step of gluing is to modify the above Green's function by adding a linear function, i.e. let
\begin{equation}
V_{j} \equiv G_{j} + (\ell_j z + \beta_j) 
\end{equation}
such that two adjacent neck regions have compatible slopes, that is,
\begin{align}
\begin{split}
k_{j+1}^- + \ell_{j+1} &= k_j^+ + \ell_j
\\
k_1^{-} + \ell_1 &= \frac{2\pi b_-}{A}.
\end{split}
\end{align} Immediately, we have  $k_1^+ + \ell_1 = \frac{2\pi (b_- - w_1)}{A}$
where $A=\Area(\TT)$.
Eventually, one can check that at the right end of the last neck region $\mathcal{N}_{w_m}^4$,
\begin{equation}
k_m^+ + \ell_m = \frac{2\pi b_- - \sum\limits_{j=1}^m w_j}{A}=-\frac{2\pi b_+}{A}.
\end{equation}
Applying the construction in Section \ref{s:approx-triple}, we obtain a manifold
\begin{equation}
\mathcal{M} = X_{b_-}^4(T_1) \bigcup_{\Psi_1} \mathcal{N}_{w_1}^4(-T_1-1, T_2 )
\bigcup_{\Psi_2} \ldots
\bigcup_{\Psi_{m}}\mathcal{N}_{w_m}^4(-T_m-1, T_{m+1})
\bigcup_{\Psi_{m+1}} X_{b_+}^4(T_{m+1} + 1 ),\label{e:general-gluing}
\end{equation}
where the attaching maps $\Psi_1, \dots \Psi_m$ are chosen analogously to $\Psi_-$,
and $\Psi_{m+1}$ is chosen analogously to $\Psi_+$. Furthermore, there is
an approximate hyperk\"ahler triple $\bm{\omega}^{\mathcal{M}}$ on $\mathcal{M}$ which is hyperk\"ahler away from the damage zones, and satisfies the conclusions of Proposition \ref{p:gluing-definite-triple}.  The weight function on $\mathcal{M}$ is defined in an analogous way to \eqref{def-weight-function}, and the arguments in the previous sections are easily modified to prove the existence of a hyperk\"ahler metric $\hat{g}_{\beta}$, close to $g_{\beta}$. 

Next, choose the parameters so that $\beta_j = \beta$. The parameters $T_j$ are then all proportional to $\beta$, and the diameter of the neck region $\mathcal{N}^4_{w_j}(-T_j -1, T_{j+1})$ in the metric $\hat{g}_{\beta}$ is proportional to $\beta^{3/2}$. Therefore, for the sequence of unit diameter hyperk\"ahler metrics $\hat{h}_{\beta}$,  
these neck regions limit to nontrivial intervals, and thus there are exactly $m$ 
distinct singular points of convergence $t_j \in (0,1), j = 1 \dots m,$ in the interior of the interval.  The analysis of the regular collapsing regions and the bubbling regions is the same as above. 
\end{proof}
\begin{proof}[Proof of Theorem \ref{t:domain-wall-crossing}] This is a consequence of the above construction. To see this, let \begin{equation}\mathcal{N}_{w_1}^4(-T_1-1, T_2 ),\ldots,\mathcal{N}_{w_m}^4(-T_m-1, T_{m+1}) \end{equation} be 
the neck regions in \eqref{e:general-gluing}
such that  for each $1\leq j \leq m$, the neck region    
 $\mathcal{N}_{w_j}^4(-T_j-1, T_{j+1})$ has exactly $w_j$-monopoles which have the same $z$-coordinate.
 Notice that the degree of the nilmanifold fiber
is determined by the ending slope of the Green's function. Corollary \ref{coro:gf} implies that 
the degree of the nilpotent fibers will jump by $w_j$ when crossing a singular fiber in $\mathcal{N}_{w_j}^4(-T_j-1,T_{j+1})$. It is also easy to see from the construction that there are $w_j$ Taub-NUT bubbles at each singular point $t_j \in (0,1), j = 1 \dots m$. 
\end{proof}

\begin{remark}\label{rem:ALF}
If we take each collection of $w_j$ monopole points in $\mathcal{N}_{w_j}^4(-T_j-1, T_{j+1})$
to have distances exactly proportional to $\beta^{-1}$ (in the flat metric on $\TT \times \RR$) from each other, then the corresponding bubble limit will be a multi-Taub-NUT ALF-$A_{w_j-1}$ metric instead of having $w_j$ Taub-NUT bubbles. It is also possible to obtain nontrivial bubble-trees. For example, if the distances of the monopole points in a collection of monopole points from each other is proportional to $\beta^{-2}$, then there will be a first bubble which is a ALF orbifold with an orbifold point which is cyclic of order $w_j$, and the deepest bubble will then be an ALE-$A_{w_j-1}$ metric.
  \end{remark}

\bibliographystyle{amsalpha} 
\bibliography{HSVZ}
\end{document}